\pdfoutput=1
\documentclass{amsart}
\usepackage{bbm}
\usepackage{amsmath, amsthm, amssymb}
\usepackage{verbatim}
\usepackage{mathrsfs}
\usepackage{longtable,array}

\newtheorem{theorem}{Theorem}[section]
\newtheorem{lemma}[theorem]{Lemma}
\newtheorem{proposition}[theorem]{Proposition}
\newtheorem{corollary}[theorem]{Corollary}

\theoremstyle{definition}

\theoremstyle{remark}
\newtheorem{remark}[theorem]{Remark}

\numberwithin{equation}{section}

\renewcommand{\Re}{\operatorname{Re}}
\usepackage[
    backend=biber,
    style=alphabetic,
    sorting=nyt
]{biblatex}

\appto{\bibsetup}{\emergencystretch=1em}

\usepackage{xcolor}
\usepackage[
    colorlinks=true,
    linkcolor=blue!50!black,
    citecolor=green!40!black,
    urlcolor=blue!50!black
]{hyperref}
\hypersetup{
    pdftitle={Waldspurger's formula via a prehomogeneous zeta integral on F²⊗∧²F⁴},
    pdfauthor={Jun Wen}
}

\begin{document}

\title[Waldspurger's formula via $F^2\otimes\bigwedge^2F^4$]{Waldspurger's formula via a prehomogeneous zeta integral on $F^2 \otimes \bigwedge^2 F^4$}

\author{Jun Wen}
\address{136 Madison Ave\\
New York, NY 10016, USA}
\email{jun.wen@prismquant.capital}


\begin{abstract}
We give a new proof of Waldspurger's formula for toric periods of
cuspidal automorphic forms on $\mathrm{PGL}_2$ over a number field,
including the explicit constant. The proof compares two expansions of
the residue at $s_0=1$ of a two-variable zeta integral on the
twelve-dimensional prehomogeneous representation
$F^2\otimes\bigwedge^2F^4$ of a parabolic subgroup of
$\mathrm{GL}_4\times\mathrm{GL}_2$, paired with an Eisenstein series
induced from $\pi\boxtimes\pi$. The first expansion, arising from the
boundary geometry and the constant term of the Eisenstein series,
identifies the residue with a Suzuki--Wakatsuki zeta integral;
unfolding that integral along its generic orbits, whose connected
stabilizers are tori, produces the square of the toric period. The second
expansion, obtained by unfolding along a generic orbit whose
stabilizer has identity component the diagonal $\mathrm{PGL}_2$,
expresses the residue as the Petersson norm times an Euler product of
local orbital integrals. Local residue identities at every place and
an explicit computation of the unramified factors show that the two
expressions share the same local factors; comparing them yields the
formula, with the central $L$-value arising from the unramified
computation.
\end{abstract}

\maketitle
\setcounter{tocdepth}{2}
\tableofcontents

\section{Introduction}
\label{sec:intro}

Waldspurger's formula relates the square of a toric period of a cuspidal
automorphic form on $\mathrm{PGL}_2$ to the central value of the base-change
$L$-function. Waldspurger proved the formula using the theta
correspondence \cite{Wal85}; see also \cite{YZZ13}.
Jacquet reproved the associated nonvanishing
criterion by comparing relative trace formulas
\cite{Jac86,Jac87}. Relative trace formulas were
subsequently used to obtain explicit versions of
the identity \cite{MartinWhitehouse,FMP17}, while Sakellaridis
gave a proof for the trivial toric character by
a nonstandard comparison of relative trace
formulas based on Poisson summation
\cite{SakBE1,SakBE2}. The purpose of this paper is to
obtain the formula by a different mechanism: we construct a two-variable
zeta integral on the prehomogeneous representation
\[
             V'=F^2\otimes\bigwedge^2F^4
\]
and identify two expressions for its residue. One expression is governed
by a boundary term and the Suzuki--Wakatsuki zeta integral \cite{SW1}; the other
comes from the generic orbit and its diagonal $\mathrm{PGL}_2$ stabilizer.
The comparison yields the toric-period identity, including its global
normalization, without using the theta correspondence or a comparison
of relative trace formulas.

The choice of $V'$ is dictated by two complementary geometric
structures within the same representation. The quadratic space
$\bigwedge^2 F^4$, equipped with its Pfaffian form, decomposes as
the orthogonal sum of a hyperbolic plane and
$(\mathrm M_2(F),-\det)$. Accordingly, elements of $V'$ may be
written in coordinates $(\ell,x,\lambda)$, where
$\ell,\lambda\in F^2$ and
$x\in\mathrm M_2(F)\oplus\mathrm M_2(F)$ is a cube in the sense
of the prehomogeneous vector space
$F^2\otimes F^2\otimes F^2$ studied by
Suzuki--Wakatsuki \cite{SW1}.
Restricting the action of $\mathrm{GL}_4\times\mathrm{GL}_2$
to the parabolic group $G'$ yields a second relative
invariant $R$, which vanishes on the $G'$-stable
boundary $\{\ell=0\}$.
On this boundary, the first relative invariant $\Delta$
reduces to the cube discriminant, the Levi action on the
cube coordinate factors through the Suzuki--Wakatsuki
cube group, whose generic stabilizers are tori, and the
remaining coordinate $\lambda$ forms a linear fibre.
This boundary geometry underlies the appearance of toric
periods in the first expansion. In contrast, the generic
stabilizer for the $G'$-action on $V'$ has identity
component isomorphic to $\mathrm{PGL}_2$, whose image
in the cube group is diagonally embedded. Its invariant
pairing on $\pi\otimes\pi$ appears in the second expansion.
The Eisenstein series on $\mathrm{GL}_4$
induced from $\pi\boxtimes\pi$ connects the two:
its constant term along the parabolic of type $(2,2)$
enters the boundary analysis, while its inducing data
give rise to the diagonal pairing on the generic orbit.

Let $F$ be a number field, let $E/F$ be a quadratic field extension with
quadratic character $\eta=\eta_{E/F}$, and let $T=R_{E/F}\mathbb G_m/\mathbb G_m$,
embedded in $\mathrm{PGL}_2$ via a choice of $F$-basis $1,\delta_E$ of $E$, with $\delta_E^2=d_E\in F$. Let
$\pi$ be an irreducible cuspidal automorphic representation of
$\mathrm{PGL}_2(\mathbb A)$ and $\pi_E$ its base change to $\mathrm{GL}_2(\mathbb A_E)$.
For $\phi\in\pi$ put
\[
P_E(\phi)=\int_{T(F)\backslash T(\mathbb A)}\phi(t)\,d\tau_T(t),
\qquad
\langle\phi,\phi\rangle_{\mathrm{Tam}}
=\int_{\mathrm{PGL}_2(F)\backslash\mathrm{PGL}_2(\mathbb A)}|\phi(g)|^2\,dg,
\]
with Tamagawa measures on $T$ and $\mathrm{PGL}_2$, so that the two quotients have
volume $2$. Fix a factorization
$\langle\ ,\ \rangle_{\mathrm{Tam}}=C_\pi\prod_v\langle\ ,\ \rangle_v$ with
$\langle\phi_v^\circ,\phi_v^\circ\rangle_v=1$ for almost all $v$, and define
the normalized local periods
\[
\alpha^\#_{E_v}(\phi_v,\phi_v)
=\frac{L_v(1,\pi,\mathrm{Ad})\,L_v(1,\eta_v)}{\zeta_v(2)\,L_v(\frac12,\pi_E)}
\int_{F_v^\times\backslash E_v^\times}\langle\pi_v(t)\phi_v,\phi_v\rangle_v\,dt_v ,
\]
where $dt_v$ is the quotient of the measures
$\zeta_{E_v}(1)|x|_{E_v}^{-1}dx^{\mathrm{sd}}_{E_v}$ and
$\zeta_{F_v}(1)|x|_v^{-1}dx^{\mathrm{sd}}_v$ built from self-dual additive
measures. Then $\alpha^\#_{E_v}(\phi_v^\circ,\phi_v^\circ)=1$ for almost all $v$.
Completed $L$-functions are denoted by $\Lambda$; they are formed from the Euler products of the local factors over all places, including the archimedean ones, and are understood through meromorphic continuation outside their domains of convergence. No discriminant or conductor factor is included. Thus $\Lambda_F(s)=\prod_v\zeta_v(s)$.
At archimedean places, all local zeta and $L$-factors
are taken in the normalization arising from Tate's
local zeta integrals with the self-dual measures
specified in Section~\ref{subsec:measures}.
In particular,
\[
 \zeta_{\mathbb R}(s)=\pi^{-s/2}\Gamma(s/2),
 \qquad
 \zeta_{\mathbb C}(s)=(2\pi)^{1-s}\Gamma(s).
\]
Throughout the paper, the local Hecke $L$-factor
associated with the quadratic character
$\eta_v$ is normalized by
\[
L_v(s,\eta_v)
=\frac{\zeta_{E_v}(s)}{\zeta_{F_v}(s)},
\]
where the local zeta factors are those
used in the completed Dedekind zeta
functions. This convention applies at
all places, including the archimedean
ones.

\begin{theorem}\label{thm:main}
For every factorizable $\phi=\bigotimes_v\phi_v\in\pi$,
\[
\frac{|P_E(\phi)|^2}{\langle\phi,\phi\rangle_{\mathrm{Tam}}}
=\frac12\cdot
\frac{\Lambda_F(2)\,\Lambda(\frac12,\pi_E)}
{\Lambda(1,\pi,\mathrm{Ad})\,\Lambda(1,\eta)^2}
\prod_v\frac{\alpha^\#_{E_v}(\phi_v,\phi_v)}{\langle\phi_v,\phi_v\rangle_v}.
\]
\end{theorem}

Theorem~\ref{thm:main} is Waldspurger's formula
\cite{Wal85}. Its constant agrees with the case
$n=2$ of the Ichino--Ikeda formula
\cite[\S6]{II10}. Their Haar-measure constant
$C_0=\Lambda(1,\eta)^{-1}$ has the same form
as our constant $C_T$ (\S\ref{subsec:invariant-form}), with each completed
$L$-function understood in its respective
archimedean normalization. Their factor
$\frac14$ becomes $\frac12$ after multiplication
by the norm
$\langle1,1\rangle_T
=\operatorname{vol}(T(F)\backslash T(\mathbb A))=2$
of the trivial toric character. The final constant is independent of
these normalization choices when the measures and local factors are
adjusted consistently (cf.\ Remark~\ref{rem:normalization-infinity};
the factors not entering the toric measure cancel directly between the
completed $L$-functions and the normalization of $\alpha^\#_{E_v}$).
The product is finite, and the right side
vanishes when some $\pi_v$ is not
$E_v^\times$-distinguished \cite{Tunnell,Saito}.

\subsection*{Strategy of the proof}
The proof compares two computations of the residue at
$s_0=1$ of a single two-variable zeta integral
$Z'_E(\Psi,f;s,s_0)$. For $\Re s$ sufficiently
large, the two computations give
\[
\operatorname*{Res}_{s_0=1}Z'_E(\Psi,f;s,s_0)
=c_1 C_P\mathcal A(s)
=c_2 C_\pi\mathcal A(s),
\]
where $c_1$ and $c_2$ are explicit scalars
depending on the fixed global data,
$\mathcal A(s)$ is a common Euler product,
$C_P$ is the constant in the factorization
of the toric functional \eqref{eq:CP},
and $C_\pi$ is the constant in the
factorization of the diagonal invariant pairing
\eqref{eq:second-global-pairing}, which agrees
with the constant above
(Remark~\ref{rem:hermitian-datum}).
The first expression involves the toric period
functional, while the second involves the
diagonal invariant pairing. Under the local distinguishedness hypothesis of
Theorem~\ref{thm:CP-Cpi}, local data can be chosen
for which $\mathcal A(s)\ne0$, and comparing
the two expressions determines the relation
between $C_P$ and $C_\pi$.

\emph{The zeta integral.}
In Section~\ref{sec:fusion} we show that the parabolic group
$G'=(P\times\mathrm{GL}_2)/Z'_\rho$, with $P=P_{2,2}$, acts
prehomogeneously on $V'$ with the two relative invariants $\Delta$ and $R$
described above, and that the generic $G'(F)$-orbits are in bijection with
$F^\times/(F^\times)^2$ through the square class of $\Delta$
(Theorem~\ref{thm:generic-orbits}); the orbit attached to $E$ is the one
of class $[d_E]$. For a Schwartz function $\Psi$ on $V'(\mathbb A)$ and a
section $f_{s_0}$ of the Eisenstein series on $\mathrm{GL}_4$ induced from
$\pi\boxtimes\pi$ on the Levi of $P$, we consider in
Section~\ref{sec:zeta}
\[
Z'_E(\Psi,f;s,s_0)=\int_{G'(F)\backslash G(\mathbb A)}
|\chi_\Delta(g)|^{s+s_0}\,\Theta_E^{\ell\ne0}(\Psi)(g)\,f_{s_0}(g)\,dg,
\]
where $\Theta^{\ell\ne0}_E$ is the part of the theta series supported on the
orbits of discriminant class $E$ and on the open $P$-stable set
$\{\ell\ne0\}$. The combined character is $|\chi_\Delta|^s|\chi_R|^{s_0}$,
so $Z'_E$ is a two-variable zeta integral attached to the pair of
invariants $(\Delta,R)$. It converges absolutely in an explicit region
(Proposition~\ref{prop:absolute-convergence}).

\emph{First computation: the Eisenstein side.}
Applying Arthur's truncation to the Eisenstein series and taking the
$N$-constant term of the theta series, whose boundary part $\{\ell=0\}$ is
the cube theta series (Lemma~\ref{lem:N-constant-term}), we continue
$Z'_E$ meromorphically to a neighbourhood of $s_0=1$, where it has at most
a simple pole (Theorem~\ref{thm:first-expansion}). The residue is
expressed through $f$ and $M(0)f$, where $M(z)$ is the standard
intertwining operator; since $M(0)=-\mathrm{Id}$
(Proposition~\ref{prop:intertwining-sign}), it is a $K$-average of the
Suzuki--Wakatsuki zeta integral for the split quaternion algebra,
$D\oplus D$ with $D=\mathrm M_2$ (Corollary~\ref{cor:residue-SW}).
Unfolding along the generic cube orbits, whose connected stabilizers are the tori
$T_E\times T_E$, produces the toric period functional
$\mathcal P_E=P_E\otimes P_E$ on $\pi\otimes\pi$, and hence the first
expression above (Proposition~\ref{prop:first-residue}).

\emph{Second computation: the orbit side.}
Unfolding $\Theta^{\ell\ne0}_E$ over the single generic orbit of class
$[d_E]$ and integrating over its stabilizer produces the invariant pairing
on $\pi\otimes\pi$, and expresses $Z'_E$ as $C_\pi$ times an Euler
product of local orbital zeta integrals
(Proposition~\ref{prop:second-expansion}). After extracting the local Tate factors, the
generic-orbit expression contains the global
factor $\Lambda_F(2s_0-2)$, whose simple pole
at $s_0=1$ produces the second residue
(Proposition~\ref{prop:second-residue}).
The remaining normalized local orbital integrals
are evaluated at $s_0=1$.

\emph{Local analysis.}
The two expressions involve the same local factors because of the local
residue identity (Theorem~\ref{thm:local-residue},
Corollary~\ref{cor:averaged-local-residue}): at every place, the residue
of the local orbital integral is the $K_v$-averaged local
Suzuki--Wakatsuki zeta integral $\mathcal J_v(s)$, which is the factor
of $\mathcal A(s)$ at the places $v\in S$. At unramified places the orbital integral is
computed in closed form (Theorem~\ref{thm:unramified-factor},
Appendix~\ref{app:unramified}), and this is where the central
$L$-value enters: at almost all places the local orbital zeta integral
$\mathscr Z'_v$, multiplied by $2c_v^2$ with $c_v=\zeta_v(1)$ and
divided by the Tate factor $\zeta_v(2s_0-2)$, has value
$L_v(1,\eta_v)\,r_v\,Z^\#_{E_v}(s+\frac12)$ at $s_0=1$
(Corollary~\ref{cor:unramified-normalization}), where
$Z^\#_{E_v}$ is the normalized unramified Suzuki--Wakatsuki factor and
$r_v=\zeta_v(2)L_v(\frac12,\pi_E)/\bigl(L_v(1,\pi,\mathrm{Ad})L_v(1,\eta_v)\bigr)$
is the value of the unnormalized local toric period at the spherical
vector. At the
places of $S$ we choose data with
$\mathcal J_v(s)\ne0$, which is possible because
$\pi_v$ is $E_v^\times$-distinguished
(Lemma~\ref{lem:nonvanishing}), while the
unramified factors form an absolutely
convergent, nonzero Euler product; comparing the two residues then gives
Theorem~\ref{thm:CP-Cpi}. Specializing to the Hermitian pairing and
converting to Tamagawa and self-dual measures
(\S\ref{subsec:invariant-form}) yields Theorem~\ref{thm:main}.

The argument uses the unfolding and absolute
convergence of the global Suzuki--Wakatsuki zeta
integral \cite[Lemmas~4.9 and~4.13]{SW1}, together
with the local unramified calculations
\cite[Theorem~5.7]{SW1} and
\cite[Theorem~1.2]{II10}. The results used here concern
the zeta integral and its local factors and
do not depend on Waldspurger's global formula.
The unramified local period integral is computed
directly in \cite[Theorem~1.2]{II10}; it is also
used, via \cite[Lemmes~2 et~3]{Wal85}, in the proof of
\cite[Corollary~34]{SW2}. The theta correspondence and the
Shimura--Waldspurger lift are not used.

\subsection*{Relation to earlier work}
Suzuki and Wakatsuki use zeta integrals associated with the space of
Bhargava's cubes to study nonvanishing and mean values of toric periods
\cite{SW1,SW2}. Waldspurger's formula is an input to their applications:
it is used to express their zeta function as a sum of Euler products and
to translate period mean values into mean values of $L$-values. In the
present work, the Suzuki--Wakatsuki zeta integral appears in the
boundary contribution of the larger prehomogeneous vector space $V'$,
and the direction is reversed: only its analytic theory and local
calculations are used, and Waldspurger's formula is obtained from the
comparison of two residue expansions.

Unlike the relative trace formula approaches recalled above, the local
comparison here uses residue identities and unramified calculations
rather than transfer identities for orbital integrals. The use of
Eisenstein truncation and residues also bears a methodological
resemblance to the residue method for period integrals
\cite{JR92,Jia98,PWZ21}, which relates periods of Eisenstein series to
poles and nonvanishing of automorphic $L$-functions; here the
toric-period contribution comes from the boundary of a prehomogeneous
zeta integral and leads to an exact period identity.

\subsection*{A non-split analogue}
{\emergencystretch=1.5em
For a quadratic extension $K/F$, the construction
suggests an analogue in which the quadratic space
$(\mathrm{M}_2,-\det)$ is replaced by
$(\mathrm{Herm}_2(K),-\det)$.
The resulting six-dimensional quadratic space
\[
V_6=\mathbb H\perp(\mathrm{Herm}_2(K),-\det)
\]
is associated, through the exceptional isogeny in
dimension six, with a quasi-split unitary group of
signature $(2,2)$. The corresponding boundary
prehomogeneous vector space is the space of pairs
of binary Hermitian forms, and the relevant
intertwining operators are governed by Asai
$L$-functions. We expect that a comparison of the
two expansions of the associated zeta integral
will lead to relations between toric periods and
periods over quaternionic subgroups. We hope to
return to this construction elsewhere.\par}

\subsection*{Organization}
A list of notation and measure conventions is provided
at the end of the paper for ease of reference.
Section~\ref{sec:fusion} treats the geometry of $V'$, the relative invariants,
generic orbits and measures. Section~\ref{sec:zeta} defines $Z'_E$ and proves
absolute convergence. Section~\ref{sec:first-expansion} gives the meromorphic
continuation and the first expansion of the residue. Section~\ref{sec:second-expansion}
gives the second expansion. Sections~\ref{sec:local-residue} and
\ref{sec:unramified} contain the local residue identity and the unramified
computation. Section~\ref{sec:comparison} compares the two residues and proves
Theorem~\ref{thm:main}. Appendix~\ref{app:levi-measure} proves the Levi
measure decomposition of Section~\ref{sec:fusion}, Appendix~\ref{app:convergence} contains the
lattice-point and slice estimates behind the absolute convergence in
Section~\ref{sec:zeta}, Appendix~\ref{app:intertwining} the local and global
computation of the central intertwining operator in
Section~\ref{sec:first-expansion}, Appendix~\ref{app:local-continuation} the local
continuation at $s_0=1$ with uniform bounds, and Appendix~\ref{app:unramified}
the proof of the unramified formula.

\section{The fusion space and its relative invariants}
\label{sec:fusion}

\subsection{The Pfaffian quadratic space and the fusion embedding}
\label{subsec:pfaffian}

Let $F$ be a field of characteristic different
from $2$. We identify $\bigwedge\nolimits^2F^4$
with the space of alternating $4\times4$ matrices,
and equip it with the Pfaffian quadratic form
\[
q(v)=\operatorname{Pf}(v)
=v_{12}v_{34}-v_{13}v_{24}+v_{14}v_{23}.
\]
Write
\[
B(v,v')=q(v+v')-q(v)-q(v')
\]
for its associated symmetric bilinear form. In particular,
$B(v,v)=2q(v)$.

Fix a decomposition $F^4=F^2\oplus F^2$ and put
\[
J=\begin{pmatrix}0&1\\-1&0\end{pmatrix}.
\]
Every $v\in\bigwedge\nolimits^2F^4$ has a unique expression
\[
v=\mathfrak m(\ell,x,\lambda)
=
\begin{pmatrix}
\ell J&x\\
-{}^tx&\lambda J
\end{pmatrix},
\qquad
\ell,\lambda\in F,\quad x\in M_2(F),
\]
and a direct computation gives
\[
q(v)=\ell\lambda-\det(x).
\]
Consequently,
\[
(\bigwedge\nolimits^2F^4,q)\simeq
\mathbb H\perp(M_2(F),-\det),
\]
and $\bigwedge\nolimits^2F^4$ is a split quadratic space of dimension six.

The right action
\[
v\cdot h={}^thvh,\qquad h\in\mathrm{GL}_4,
\]
satisfies
\[
q(v\cdot h)=\det(h)q(v),
\]
and thus induces the exceptional similitude homomorphism
\[
\mathrm{GL}_4\longrightarrow\mathrm{GSO}(\bigwedge\nolimits^2F^4,q).
\]

Consider the twelve-dimensional space
\[
V'=F^2\otimes\bigwedge\nolimits^2F^4.
\]
We regard $w\in V'$ as an ordered pair $(w_1,w_2)$ and write
\[
w=(\ell,x,\lambda),
\qquad
\ell={}^t(\ell_1,\ell_2),
\qquad
\lambda={}^t(\lambda_1,\lambda_2),
\qquad
x=(x_1,x_2).
\]
Thus $\ell$ and $\lambda$ are column vectors in $F^2$,
while $x_i\in M_2(F)$.

The Witt decomposition gives the fusion embedding
\[
V=M_2(F)\oplus M_2(F)\hookrightarrow V',
\qquad
(x_1,x_2)\longmapsto(0,(x_1,x_2),0).
\]
In particular,
\[
V\simeq F^2\otimes F^2\otimes F^2
\]
is the space of Bhargava's cubes \cite{Bha04}, and its associated
binary quadratic form is
\[
(u,v)\longmapsto-\det(ux_1+vx_2).
\]

\subsection{The parabolic action and the effective group}
\label{subsec:group}

Let $P=P_{2,2}\subset\mathrm{GL}_4$ be the standard
parabolic subgroup, with Levi subgroup
$M\simeq\mathrm{GL}_2\times\mathrm{GL}_2$
and unipotent radical
\[
N=
\left\{
n(X)=
\begin{pmatrix}
I_2&X\\
0&I_2
\end{pmatrix}
:X\in M_2
\right\}.
\]

The group $\mathrm{GL}_4\times\mathrm{GL}_2$
acts on $V'$ by
\[
(w_1,w_2)\cdot(h,k)
=
({}^thw_1h,{}^thw_2h)k.
\]
Here the first factor acts on the alternating
matrices by congruence, while the outer
$\mathrm{GL}_2$ acts on the index of the ordered
pair. Denote this representation by
\[
\rho_{V'}:
\mathrm{GL}_4\times\mathrm{GL}_2
\longrightarrow\mathrm{GL}(V').
\]

For $h=\operatorname{diag}(a,b)$, the Witt
coordinates transform as
\[
\ell\longmapsto\det(a)\ell,\qquad
x_i\longmapsto{}^tax_i b,\qquad
\lambda\longmapsto\det(b)\lambda,
\]
with the outer $\mathrm{GL}_2$ acting on the
index $i$.

For $n(X)\in N$, direct matrix multiplication
gives
\[
\ell_i\longmapsto\ell_i,\qquad
x_i\longmapsto x_i+\ell_iJX,
\]
and
\[
\lambda_i
\longmapsto
\lambda_i+\ell_i\det(X)+\tau(x_i,X),
\]
where
\[
\tau(x,X)
=
X_{11}x_{12}+X_{21}x_{22}
-x_{11}X_{12}-x_{21}X_{22}.
\]
More precisely, if
\[
p(x,Y)=\det(x+Y)-\det(x)-\det(Y)
\]
denotes the polar form of the determinant
on $M_2$, then
\[
\tau(x,X)=p(x,JX).
\]
Since $p$ is nondegenerate and multiplication
by $J$ is invertible, $\tau$ is a perfect
bilinear pairing on $M_2$. This observation
will also be used in Section~\ref{sec:first-expansion}.

The kernel of $\rho_{V'}$ is precisely
\[
Z'_\rho
=
\{(zI_4,z^{-2}I_2):z\in\mathbb G_m\}.
\]
Indeed, if $(h,k)$ acts trivially on
$F^2\otimes\bigwedge^2F^4$, then
$\bigwedge^2h$ and $k$ must both be scalar
operators. Since $\bigwedge^2h$ is scalar
only when $h$ is scalar, we have $h=zI_4$.
Writing $k=\gamma I_2$, the condition
that the action be trivial becomes
$z^2\gamma=1$, so $\gamma=z^{-2}$.

We consequently consider the effective
quotient groups
\[
G=
(\mathrm{GL}_4\times\mathrm{GL}_2)/Z'_\rho,
\qquad
G'=
(P\times\mathrm{GL}_2)/Z'_\rho.
\]
Since $Z'_\rho\subset M\times\mathrm{GL}_2$,
the Levi subgroup of $G'$ is
\[
\widetilde M
=
(M\times\mathrm{GL}_2)/Z'_\rho.
\]

Since $Z'_\rho\simeq\mathbb G_m$ is a split
central torus, Hilbert's theorem~90 gives
\[
H^1(L,Z'_\rho)=1
\]
for every field extension $L/F$.
Consequently, the quotient maps induce
identifications
\[
G(L)=
\bigl(\mathrm{GL}_4(L)\times\mathrm{GL}_2(L)\bigr)
/Z'_\rho(L),
\]
\[
G'(L)=
\bigl(P(L)\times\mathrm{GL}_2(L)\bigr)
/Z'_\rho(L),
\]
and similarly for $\widetilde M(L)$.

When $F$ is a number field, at every finite
place $v$ the corresponding identifications
hold over $\mathfrak o_v$, since
\[
H^1(\mathfrak o_v,Z'_\rho)
=
\operatorname{Pic}(\mathfrak o_v)=0.
\]
Thus every local point admits a lift to
the corresponding product group, and at
almost all finite places the lift may be
chosen inside the integral subgroup.
Passing to restricted products gives
\[
G(\mathbb A)=
\bigl(\mathrm{GL}_4(\mathbb A)
\times\mathrm{GL}_2(\mathbb A)\bigr)
/Z'_\rho(\mathbb A),
\]
with analogous descriptions of
$G'(\mathbb A)$ and
$\widetilde M(\mathbb A)$.

To identify the induced representation on
the cube subspace, we parameterize the
Levi elements by
\[
h(g_1,g_2)
=
\operatorname{diag}({}^tg_1^{-1},g_2),
\qquad
g_1,g_2\in\mathrm{GL}_2.
\]
The identity
\[
{}^tgJg=\det(g)J
\]
gives the explicit formula
\[
\mathfrak m(\ell,x,\lambda)
\cdot h(g_1,g_2)
=
\mathfrak m\bigl(
\det(g_1)^{-1}\ell,\,
g_1^{-1}xg_2,\,
\det(g_2)\lambda
\bigr).
\]

Thus the cube subspace
\[
V=\{\ell=\lambda=0\}\simeq M_2^2
\]
is stable under the Levi action.
On this subspace the two inner factors
act by
\[
(x_1,x_2)
\longmapsto
(g_1^{-1}x_1g_2,g_1^{-1}x_2g_2),
\]
and the outer $\mathrm{GL}_2$ mixes the
two matrices.

Let
\[
\rho_V:
\mathrm{GL}_2^3
\longrightarrow\mathrm{GL}(V)
\]
denote this three-factor cube representation.
Its kernel is
\[
\ker\rho_V
=
\left\{
(\alpha I_2,\beta I_2,\gamma I_2):
\gamma=\alpha/\beta
\right\}.
\]

There is an additional one-dimensional torus in $\widetilde M$ which acts trivially on
the cube subspace. For $c\in\mathbb G_m$, put
\[
e_c=\bigl[\bigl(\operatorname{diag}(I_2,c^{-1}I_2),\,cI_2\bigr)\bigr]\in\widetilde M .
\]
The map $e\colon\mathbb G_m\to\widetilde M$, $c\mapsto e_c$, is a homomorphism of
algebraic groups. Its kernel is trivial: if
$(\operatorname{diag}(I_2,c^{-1}I_2),cI_2)=(zI_4,z^{-2}I_2)$, then $z=1$ and $c=1$.
Hence $e$ is a closed immersion, and we denote its image by
\[
A=\{e_c\}\simeq\mathbb G_m .
\]
Since $\operatorname{diag}(I_2,c^{-1}I_2)$ is central in $M$ and $cI_2$ is central in
$\mathrm{GL}_2$, the torus $A$ is central in $\widetilde M$. Its action on the Witt
coordinates is
\[
(\ell,x,\lambda)\cdot e_c=(c\ell,\,x,\,c^{-1}\lambda).
\]
Because $e$ is an isomorphism onto $A$ defined over $F$, we have
\[
A(F)=\{e_c:c\in F^\times\},\qquad A(\mathbb A)=\{e_c:c\in\mathbb A^\times\}.
\]

\begin{remark}\label{rem:A0}
The torus $A$ contains the image of
$A_0=\{a_t=\operatorname{diag}(tI_2,t^{-1}I_2)\}\subset M$. Indeed, modulo $Z'_\rho$,
\[
(a_t,I_2)\equiv\bigl(\operatorname{diag}(I_2,t^{-2}I_2),\,t^2I_2\bigr)=e_{t^2},
\]
since the two representatives differ by $(tI_4,t^{-2}I_2)\in Z'_\rho$. Thus
$A_0\to A$ is the squaring map $t\mapsto t^2$. It is a surjective isogeny of tori
with kernel $\mu_2=\{a_{\pm1}\}$, so $A\simeq A_0/\mu_2$ as algebraic groups. On
rational points, however, the image of $A_0(F)$ is only $\{e_c:c\in(F^\times)^2\}$:
for $c\notin(F^\times)^2$, the element $e_c\in A(F)$ does not come from $A_0(F)$.
For this reason we always work with the parametrization $c\mapsto e_c$ rather than
with $a_t$.
\end{remark}

The action of $\widetilde M$ on $V$
has kernel precisely $A$.
Indeed, the kernel of the cube action
in $\mathrm{GL}_2^3$ is $\ker\rho_V$,
as described above, and its image in
$\widetilde M$ is exactly $A$. Consequently,
\[
H=\widetilde M/A
\simeq\mathrm{GL}_2^3/\ker\rho_V.
\]
This is the cube group $H$ used by
Suzuki--Wakatsuki, with the same
three-factor action on $V=M_2^2$.
Thus $\widetilde M$ is an extension
of $H$ by the central torus $A$.
The additional coordinate $c$ will
provide the scaling variable in the
fusion-space zeta integral.

\subsection{The relative invariants}
\label{subsec:invariants}

Retain the column-vector convention for $\ell$
and the Witt coordinates
$w=(w_1,w_2)=(\ell,x,\lambda)$
from Section~\ref{subsec:pfaffian}.

For $w=(w_1,w_2)\in V'$, let
\[
G_w=
\begin{pmatrix}
B(w_1,w_1)&B(w_1,w_2)\\
B(w_2,w_1)&B(w_2,w_2)
\end{pmatrix},
\qquad
Q_w(u,v)=q(uw_1+vw_2).
\]
Since
\[
Q_w(u,v)
=
\frac12
\begin{pmatrix}u&v\end{pmatrix}
G_w
\begin{pmatrix}u\\v\end{pmatrix},
\]
define
\[
\Delta(w)=\operatorname{disc}(Q_w)=-\det G_w,
\qquad
R(w)={}^t\ell\,\operatorname{adj}(G_w)\ell.
\]
For example, for
$w_\mu=\bigl(\mathfrak m(0,I_2,0),\,\mathfrak m(1,0,\mu)\bigr)$ one has
$q(w_1)=-1$, $q(w_2)=\mu$ and $B(w_1,w_2)=0$, so
$G_{w_\mu}=\operatorname{diag}(-2,2\mu)$ and $\Delta(w_\mu)=4\mu$.
This element reappears as the normal form of
Section~\ref{subsec:generic-orbits}.
The factor $4$ is retained in the local discriminant normalization,
including at places dividing $2$.

\begin{proposition}\label{prop:relative-invariants}
For every $(h,k)\in\mathrm{GL}_4\times\mathrm{GL}_2$,
one has
\[
\Delta(w\cdot(h,k))
=
\det(h)^2\det(k)^2\Delta(w).
\]
If $(h,k)\in P\times\mathrm{GL}_2$, with
\[
h=
\begin{pmatrix}
a&*\\
0&b
\end{pmatrix},
\]
then
\[
R(w\cdot(h,k))
=
\det(a)^3\det(b)\det(k)^2R(w).
\]
Consequently, $\Delta$ is a relative invariant
of $(G,V')$ with character
\[
\chi_\Delta(h,k)=(\det h\,\det k)^2,
\]
and $R$ is a relative invariant of $(G',V')$
with character
\[
\chi_R(h,k)
=
\det(a)^3\det(b)\det(k)^2
=
\chi_\Delta(h,k)\frac{\det a}{\det b}.
\]
\end{proposition}

\begin{proof}
The Pfaffian identity
\[
q({}^thvh)=\det(h)q(v),
\qquad h\in\mathrm{GL}_4,
\]
implies
\[
G_{w\cdot(h,k)}
=
\det(h)\,{}^tkG_wk
\]
for arbitrary $(h,k)\in\mathrm{GL}_4\times\mathrm{GL}_2$.
Taking determinants gives
\[
\Delta(w\cdot(h,k))
=
\det(h)^2\det(k)^2\Delta(w).
\]

Now suppose $h\in P$.
The $\ell$-coordinate transforms by
\[
\ell\longmapsto\det(a)\,{}^tk\ell.
\]
For a $2\times2$ matrix $X$, the adjugate identities give
\[
\operatorname{adj}({}^tkXk)
=
\det(k)^2k^{-1}
\operatorname{adj}(X){}^tk^{-1},
\qquad
\operatorname{adj}(cX)=c\,\operatorname{adj}(X).
\]
Consequently,
\[
\begin{aligned}
R(w\cdot(h,k))
&=
\det(a)^2\det(h)\det(k)^2
\,{}^t\ell\operatorname{adj}(G_w)\ell\\
&=
\det(a)^3\det(b)\det(k)^2R(w).
\end{aligned}
\]

Finally, for an element
$(zI_4,z^{-2}I_2)\in Z'_\rho$,
we have $a=b=zI_2$ and $k=z^{-2}I_2$. Hence
\[
\chi_\Delta(zI_4,z^{-2}I_2)
=
(z^2z^2z^{-4})^2=1,
\]
and
\[
\chi_R(zI_4,z^{-2}I_2)
=
z^6z^2z^{-8}=1.
\]
Thus both transformation characters are trivial on
$Z'_\rho$. In particular, $\chi_\Delta$ defines
a character of $G$, while $\chi_R$ defines
a character of $G'$.
\end{proof}

\begin{lemma}\label{lem:R-determinant}
Put
\[
y_w=\ell_2x_1-\ell_1x_2\in M_2(F).
\]
Then
\[
R(w)=-2\det(y_w).
\]
In particular, $R(w)\ne0$ if and only if
$y_w$ is invertible.
\end{lemma}

\begin{proof}
Set
\[
v=\ell_2w_1-\ell_1w_2\in\bigwedge\nolimits^2F^4.
\]
Its $\ell$-coordinate is zero and its middle
coordinate is $y_w$; thus
\[
q(v)=-\det(y_w).
\]
On the other hand,
\[
\begin{aligned}
2q(v)
&=
\begin{pmatrix}\ell_2&-\ell_1\end{pmatrix}
G_w
\begin{pmatrix}\ell_2\\-\ell_1\end{pmatrix}\\
&=
{}^t\ell\,\operatorname{adj}(G_w)\ell
=
R(w).
\end{aligned}
\]
The assertion follows.
\end{proof}

We write
\[
V'_{\mathrm{gen}}
=
\{w\in V':\Delta(w)R(w)\ne0\}.
\]
This is a $G'$-invariant Zariski-open subset.

In the Levi coordinates $a={}^tg_1^{-1}$ and $b=g_2$,
the relative characters defined above and the
parabolic modulus are
\[
\chi_\Delta=(\det a\det b\det k)^2,
\qquad
\chi_R=\det(a)^3\det(b)\det(k)^2,
\]
and
\[
\delta_P
=
\left|\frac{\det a}{\det b}\right|^2.
\]

The determinant character of $\rho_{V'}$ is
\[
\det(\rho_{V'})(h,k)
=
(\det h)^6(\det k)^6
=
\chi_\Delta(h,k)^3.
\]
Indeed, the action of $h$ on $\bigwedge^2F^4$
has determinant $(\det h)^3$, while the outer
$\mathrm{GL}_2$ acts on the two copies of
$\bigwedge^2F^4$ with determinant $(\det k)^6$.

Define the positive character
\[
u=
\left|\frac{\det a}{\det b}\right|^{1/2}.
\]
For the central torus $A$ and its coordinate
$e_c$ introduced in Section~\ref{subsec:group},
one has
\[
u(e_c)=|c|,
\qquad
\delta_P=u^4.
\]
The identity
\[
\chi_R
=
\chi_\Delta\frac{\det a}{\det b}
\]
will relate the two zeta parameters in Section~\ref{sec:zeta}.

For later use, we record the action of the remaining
scalar direction in $G$. The element represented by
$(zI_4,I_2)$ acts on $V'$ by multiplication by $z^2$.
Since $\Delta$ is homogeneous of degree four,
the transformation law above gives
\[
\chi_\Delta(zI_4,I_2)=z^8.
\]

On the boundary $\ell=0$, one has
\[
Q_{(0,x,\lambda)}(u,v)
=
-\det(ux_1+vx_2),
\]
so the discriminant is independent of $\lambda$
and restricts to the cube discriminant.

Finally, the matrix $y_w$ is invariant under $N$,
as follows immediately from
$x_i\longmapsto x_i+\ell_iJX$. Indeed,
\[
\begin{aligned}
y_{w\cdot n(X)}
&=
\ell_2(x_1+\ell_1JX)
-\ell_1(x_2+\ell_2JX)\\
&=y_w.
\end{aligned}
\]

\subsection{Generic orbits and stabilizers}
\label{subsec:generic-orbits}

For $\mu\in F^\times$, set
\[
w_\mu=
\bigl((0,I_2,0),(1,0,\mu)\bigr).
\]
Directly from the definitions,
\[
\Delta(w_\mu)=4\mu,
\qquad
R(w_\mu)=-2.
\]

\begin{proposition}\label{prop:generic-stabilizer}
The stabilizer of $w_\mu$ in
$P\times\mathrm{GL}_2$ is
\[
\left\{
\left(
\begin{pmatrix}
a&0\\
0&\varepsilon\det(a)a^{-t}
\end{pmatrix},
\begin{pmatrix}
\varepsilon\det(a)^{-1}&0\\
0&\det(a)^{-1}
\end{pmatrix}
\right):
a\in\mathrm{GL}_2,\ 
\varepsilon\in\mu_2
\right\}.
\]
It is isomorphic to $\mathrm{GL}_2\times\mu_2$.
In the effective group $G'$, its stabilizer is
\[
G'_{w_\mu}\simeq\mathrm{PGL}_2\times\mu_2.
\]
In terms of $h(g_1,g_2)$ and $e_c$ of
\S\ref{subsec:group}, the identity component
($\varepsilon=1$, $a={}^tt^{-1}$) is the image of
\[
\iota(t)=h(t,t)\,e_{\det t}
=\bigl(\operatorname{diag}({}^tt^{-1},\det(t)^{-1}t),\,
\det(t)I_2\bigr),
\qquad t\in\mathrm{GL}_2,
\]
and $u(\iota(t))=1$. In particular, its image in
$H=\widetilde M/A$ is the diagonal $\mathrm{PGL}_2$ in the
coordinates $(g_1,g_2)$ of $h(g_1,g_2)$.
\end{proposition}

\begin{proof}
Write
\[
h=n(X)\operatorname{diag}(a,b),
\qquad
k=\begin{pmatrix}p&q\\r&s\end{pmatrix}.
\]
Comparing the $\ell$-coordinates in
$w_\mu\cdot(h,k)=w_\mu$ gives
\[
r=0,\qquad s=\det(a)^{-1}.
\]
The $x$-coordinates then give
\[
p\,{}^tab=I_2,
\qquad
q\,{}^tab+s\,{}^taJXb=0.
\]
Thus
\[
b=p^{-1}a^{-t},
\qquad
X=q\det(a)J.
\]
The first $\lambda$-coordinate, using
\[
\tau(I_2,X)=-2q\det(a),
\]
forces $q=0$, and hence $X=0$.
The second $\lambda$-coordinate yields
\[
\det(b)=\det(a).
\]
Consequently,
\[
p=\varepsilon\det(a)^{-1},
\qquad
b=\varepsilon\det(a)a^{-t},
\qquad
\varepsilon\in\mu_2.
\]
Conversely, the displayed elements fix $w_\mu$.

The parametrization
\[
(a,\varepsilon)\longmapsto
\left(
\operatorname{diag}
\bigl(a,\varepsilon\det(a)a^{-t}\bigr),
\operatorname{diag}
\bigl(\varepsilon\det(a)^{-1},
\det(a)^{-1}\bigr)
\right)
\]
is a homomorphism
$\mathrm{GL}_2\times\mu_2$
onto this stabilizer. Indeed,
\[
\det(aa')\,(aa')^{-t}
=
\bigl(\det(a)a^{-t}\bigr)
\bigl(\det(a')a'^{-t}\bigr),
\]
because for $2\times2$ matrices
$\det(a)a^{-t}=JaJ^{-1}$.

Moreover, $Z'_\rho$ lies in the
$\varepsilon=1$ component.
For $(zI_4,z^{-2}I_2)\in Z'_\rho$,
the corresponding parameters are
$a=zI_2$ and $\varepsilon=1$.
Thus quotienting by $Z'_\rho$ gives
\[
G'_{w_\mu}
\simeq
(\mathrm{GL}_2/\mathbb G_m)\times\mu_2
\simeq
\mathrm{PGL}_2\times\mu_2.
\]
By Hilbert's theorem~90, the same
description holds on $F$-points:
\[
G'_{w_\mu}(F)
\simeq
\mathrm{PGL}_2(F)\times\mu_2(F).
\]
The projective Levi image of the
connected stabilizer is diagonal.

To identify the connected stabilizer
explicitly, write
\[
\iota(t)=h(t,t)e_{\det t}
=
\left[
\left(
\operatorname{diag}
\bigl({}^tt^{-1},(\det t)^{-1}t\bigr),
(\det t)I_2
\right)
\right].
\]
This agrees with the $\varepsilon=1$
component above upon setting
$a={}^tt^{-1}$. Moreover,
\[
\det({}^tt^{-1})
=
\det((\det t)^{-1}t)
=
(\det t)^{-1},
\]
so $u(\iota(t))=1$.
For $z\in\mathbb G_m$,
\[
\iota(zt)
=
(z^{-1}I_4,z^2I_2)\iota(t),
\]
and the prefactor belongs to $Z'_\rho$.
Conversely, if $\iota(t)\in Z'_\rho$, then
${}^tt^{-1}$ is scalar, so $t$ is scalar.
Thus $\iota$ induces an embedding
$\mathrm{PGL}_2\hookrightarrow G'$
whose image is the identity component
of $G'_{w_\mu}$.
\end{proof}

\begin{theorem}\label{thm:generic-orbits}
The discriminant square class induces
a bijection
\[
G'(F)\backslash V'_{\mathrm{gen}}(F)
\xrightarrow{\ \sim\ }
F^\times/(F^\times)^2,
\qquad
[w]\longmapsto[\Delta(w)].
\]
Every $w\in V'_{\mathrm{gen}}(F)$
is conjugate to $w_\mu$ for some
$\mu\in F^\times$, and $w_\mu$
and $w_\nu$ are conjugate if and
only if $\mu/\nu$ is a square.
\end{theorem}

\begin{proof}
We give an explicit reduction to normal form.
Since
\[
R(w)=-2\det(y_w)\ne0,
\qquad
y_w=\ell_2x_1-\ell_1x_2,
\]
we have $\ell\ne0$.
An outer $\mathrm{GL}_2(F)$ transformation
brings $\ell$ to ${}^t(0,1)$.
Then $x_1=y_w$ is invertible.

Using $N(F)$, eliminate $x_2$.
We obtain
\[
w=\bigl((0,y,\alpha),(1,0,\beta)\bigr),
\qquad
y\in\mathrm{GL}_2(F).
\]
Put
\[
d=\det(y),\qquad
c=\frac{\alpha}{2d}.
\]
Since
\[
\tau(y,cJy)=-2cd=-\alpha,
\]
applying $n(cJy)$ changes the first
$\lambda$-coordinate to zero.
The second $x$-coordinate becomes
\[
x_2=J(cJy)=-cy.
\]
Applying the outer shear
\[
k=\begin{pmatrix}
1&c\\
0&1
\end{pmatrix},
\]
which replaces $w_2$ by $cw_1+w_2$
and fixes $\ell={}^t(0,1)$, since
${}^tk\ell=\ell$, gives
\[
w=
\bigl((0,y,0),(1,0,\beta+c^2d)\bigr).
\]
Indeed, the second $\lambda$-coordinate
after applying $n(cJy)$ is
\[
\beta+\det(cJy)=\beta+c^2d.
\]
The discriminant identity
\[
\Delta(w)=\alpha^2+4d\beta
\]
therefore gives
\[
\beta+c^2d=\frac{\Delta(w)}{4d}.
\]
Applying the Levi element
$\operatorname{diag}(I_2,y^{-1})$
now yields
\[
w_\mu,\qquad
\mu=\frac{\Delta(w)}{4d^2}.
\]
Thus $[\mu]=[\Delta(w)]$.

If $w_\mu$ and $w_\nu$ are conjugate,
relative invariance of $\Delta$
forces $\mu/\nu$ to be a square.
Conversely, if $\nu=t^2\mu$, then
\[
\left(
\operatorname{diag}(I_2,tI_2),
\operatorname{diag}(t^{-1},1)
\right)
\]
sends $w_\mu$ to $w_\nu$.
\end{proof}

For a quadratic field extension
$E/F$, with $E=F+F\delta_E$ and $\delta_E^2=d_E$, put
\[
V'_E(F)=
\{w\in V'(F):
\Delta(w)\in d_E(F^\times)^2\}.
\]
The binary quadratic form $Q_w$
is anisotropic for $w\in V'_E(F)$.
Consequently, if $\ell\ne0$,
then $R(w)\ne0$: indeed,
$\ell_2w_1-\ell_1w_2$ is a nonzero
vector in the anisotropic plane
spanned by $w_1,w_2$.
It follows that
\[
V'_E(F)\cap\{\ell\ne0\}
\]
is a single $G'(F)$-orbit.

When $F$ is a number field, define,
at every place $v$ of $F$,
\[
V'_E(F_v)=
\{w\in V'(F_v):
\Delta(w)\in d_E(F_v^\times)^2\}.
\]
Theorem~\ref{thm:generic-orbits} holds over
any field of characteristic different from $2$,
in particular over $F_v$. Hence
\[
V'_E(F_v)\cap\{R\ne0\}
\]
is a single $G'(F_v)$-orbit, namely that of
$w_\mu$ with $[4\mu]=[d_E]$ in
$F_v^\times/(F_v^\times)^2$. This orbit is contained in
$\{\ell\ne0\}$, since the condition $\ell\ne0$ is
$G'$-stable and $\ell(w_\mu)\ne0$.
If $E_v$ is a field, $Q_w$ is anisotropic for
$w\in V'_E(F_v)$, and the argument above gives
$V'_E(F_v)\cap\{\ell\ne0\}=V'_E(F_v)\cap\{R\ne0\}$.
If $E$ splits at $v$, so that
$E_v\simeq F_v\times F_v$, the form $Q_w$ is
isotropic, and the locus
\[
\{w\in V'_E(F_v):
\ell(w)\ne0,\ R(w)=0\}
\]
is nonempty. The single-orbit assertion then
holds for $R\ne0$ but not for the whole
locus $\ell\ne0$.

\begin{remark}[Orthogonal interpretation]
Let $w\in V'_{\mathrm{gen}}(F)$, and put
\[
U=Fw_1+Fw_2\subset\bigwedge\nolimits^2F^4.
\]
Since $\Delta(w)\ne0$, the plane $U$
is nondegenerate. Let $e=(0,0,1)$
be the isotropic vector whose line
determines $P$, and decompose
\[
e=e_U+e_\perp
\]
according to $\bigwedge\nolimits^2F^4=U\perp U^\perp$.
Since $B(e,w_i)=\ell_i$, the
orthogonal projection of $e$
onto $U$ is
\[
e_U=\sum_{j=1}^2
(G_w^{-1}\ell)_j w_j.
\]
Consequently,
\[
q(e_U)
=\frac12\,{}^t\ell G_w^{-1}\ell
=\frac{R(w)}{2\det G_w}
=-\frac{R(w)}{2\Delta(w)}.
\]
As $q(e)=0$, we obtain
\[
q(e_\perp)=\frac{R(w)}{2\Delta(w)}.
\]
In particular, $e_\perp$ is
anisotropic because $R(w)\ne0$.
Hence
\[
Y=(U\oplus Fe_\perp)^\perp
=U^\perp\cap e^\perp
\]
is a nondegenerate ternary
quadratic space.

The second description makes
the ternary space explicit in
Witt coordinates. If
$v=(L,X,T)\in\bigwedge\nolimits^2F^4$, then
\[
B(v,e)=L,\qquad
B(v,w_i)=L\lambda_i+T\ell_i-p(X,x_i).
\]
Therefore
\[
Y=
\left\{
(0,X,T):
T\ell_i=p(X,x_i),\quad i=1,2
\right\},
\qquad
q|_Y(0,X,T)=-\det X.
\]
For the representative
\[
w_\mu=\bigl((0,I_2,0),(1,0,\mu)\bigr),
\]
these equations are simply
\[
\operatorname{tr}X=T=0.
\]
Thus
\[
(Y,q|_Y)
\simeq
(\mathfrak{sl}_2(F),-\det),
\qquad
-\det
\begin{pmatrix}
a&b\\c&-a
\end{pmatrix}
=a^2+bc.
\]
In particular, $Y$ is split,
and conjugation on traceless
matrices identifies
$\mathrm{SO}(Y)$ with
$\mathrm{PGL}_2$.
This gives a geometric explanation
for the connected factor in
Proposition~\ref{prop:generic-stabilizer};
the extra $\mu_2$ accounts for
the disconnected component.
The full stabilizer, its rational points over $F$ and $F_v$, and the degree-two orbit map follow from Proposition~\ref{prop:generic-stabilizer} and Theorem~\ref{thm:generic-orbits}; the gauge-form constant $\kappa_0$ enters the proof of Proposition~\ref{prop:second-expansion}, where only $\kappa_0\in F^\times$ is needed.
\end{remark}

\subsection{Theta series}
\label{subsec:theta}
From now on let $F$ be a number field and $\mathbb A=\mathbb A_F$
its ring of adeles. Let $E/F$ be a quadratic field extension and
fix $d_E\in F^\times$ representing its discriminant square class.
As above, set
\[
 V'_E(F)=\{w\in V'(F):\Delta(w)\in d_E(F^\times)^2\}.
\]
For $\Psi\in\mathcal S(V'(\mathbb A))$, define
\[
\Theta_E(\Psi)(g)
=
\sum_{w\in V'_E(F)}\Psi(wg),
\qquad g\in G(\mathbb A).
\]
The series converges absolutely for each fixed
$g\in G(\mathbb A)$. Indeed, the translate
$w\mapsto\Psi(wg)$ is again a Schwartz--Bruhat
function on $V'(\mathbb A)$, and such a function
is absolutely summable over $V'(F)$.
Since $V'_E(F)\subset V'(F)$, the assertion follows.

By Proposition~\ref{prop:relative-invariants},
the relative invariant $\Delta$ transforms under
$G(F)$ by the square-valued character
$\chi_\Delta$. Consequently, the square class
of $\Delta(w)$ is preserved by $G(F)$, so that
$V'_E(F)$ is $G(F)$-stable. It follows that
\[
\Theta_E(\Psi)(\gamma g)
=
\Theta_E(\Psi)(g),
\qquad
\gamma\in G(F).
\]

The parabolic subgroup preserves the condition $\ell=0$, whereas
the full group $G$ need not do so. We accordingly write
\[
 \Theta_E^{\ell\ne0}(\Psi)(g)
 =\sum_{\substack{w\in V'_E(F)\\\ell(w)\ne0}}\Psi(wg).
\]
This partial theta series is left $G'(F)$-invariant,
and the orbit classification in
Theorem~\ref{thm:generic-orbits}
identifies its summation set with a single
rational $G'$-orbit. The complementary
stratum has $\ell=0$ and consists of points $(0,x,\lambda)$ with
$x$ of discriminant class $E$ and arbitrary $\lambda\in F^2$.
This separation is the starting point for the two expansions of
the zeta integral.

\subsection{Measures and normalizations}
\label{subsec:measures}

We fix the following Haar measure conventions, following
Suzuki--Wakatsuki \cite[\S3.1]{SW1}.

Let $\mathbb A=\mathbb A_F$ be the adele ring of $F$,
and let $\Delta_F$ denote the absolute discriminant
of $F$. Fix the standard additive character
\[
\psi_F=\psi_{\mathbb Q}\circ
\operatorname{Tr}_{F/\mathbb Q}
\]
of $F\backslash\mathbb A$, where $\psi_{\mathbb Q}$
is the standard additive character of
$\mathbb Q\backslash\mathbb A_{\mathbb Q}$.

At every finite place $v$, let $dx_v$ be the additive
Haar measure on $F_v$ satisfying
\[
\operatorname{vol}(\mathfrak o_v,dx_v)=1.
\]
At the archimedean places, take Lebesgue measure
at real places and twice Lebesgue measure
$2\,dx\,dy$ at complex places. Then
\[
dx=|\Delta_F|^{-1/2}\prod_v dx_v
\]
is the self-dual Haar measure on $\mathbb A$
with respect to $\psi_F$. In particular,
\[
\operatorname{vol}(F\backslash\mathbb A,dx)=1.
\]

For each place $v$, let $dx_v^{\mathrm{sd}}$
denote the Haar measure on $F_v$ self-dual
with respect to $\psi_v$. At a finite place,
\[
\operatorname{vol}
(\mathfrak o_v,dx_v^{\mathrm{sd}})
=
N(\mathfrak d_v)^{-1/2},
\]
where $\mathfrak d_v$ is the local different.
At real places $dx_v^{\mathrm{sd}}$ is Lebesgue
measure, and at complex places it is
$2\,dx\,dy$. Thus
$dx_v^{\mathrm{sd}}=dx_v$ at the
archimedean places. Since
\[
\prod_{v<\infty}N(\mathfrak d_v)=|\Delta_F|,
\]
we have
\[
\prod_v dx_v^{\mathrm{sd}}
=
|\Delta_F|^{-1/2}\prod_v dx_v
=
dx.
\]
Thus the two conventions give the same
global additive measure, differing only
in the distribution of the discriminant
factor among the local measures. We will
use this observation when comparing local
torus measures in Section~\ref{sec:comparison}.

For a finite-dimensional $F$-vector space $\mathcal V$
with a fixed $F$-basis, we give $\mathcal V(F_v)$ the
product $dy_v$ of the measures $dx_v$ in the
coordinates, and $\mathcal V(\mathbb A)$ the measure
\[
dy=|\Delta_F|^{-\dim\mathcal V/2}\prod_v dy_v .
\]
This is the self-dual measure for the pairing
$\psi_F\circ\langle\ ,\ \rangle$ given by the
coordinate dot product, and
$\operatorname{vol}(\mathcal V(F)\backslash\mathcal V(\mathbb A),dy)=1$.
We apply this to $V'$ in the Witt coordinates
$(\ell_i,x_i,\lambda_i)_{i=1,2}$ of
\S\ref{subsec:pfaffian}, so that $\dim V'=12$, writing
$dw_v$ and $dw=|\Delta_F|^{-6}\prod_vdw_v$ in this case, and
to $\mathbb A^2$ in the coordinates $\lambda=(\lambda_1,\lambda_2)$,
so that the self-dual measure is
$d\lambda=|\Delta_F|^{-1}\prod_v d\lambda_v$.

On $\mathrm{PGL}_2$, and on the identity component
$S^\circ\simeq\mathrm{PGL}_2$ of $G'_{w_\mu}$
(Proposition~\ref{prop:generic-stabilizer}), we use the
Tamagawa measure
\[
dt=|\Delta_F|^{-3/2}\prod_v|\omega_v|,
\]
where $\omega$ is a nonzero invariant $F$-rational
gauge form; the product formula makes $dt$
independent of the choice of $\omega$. Since
$\mathrm{PGL}_2$ has no nontrivial $F$-rational
characters, no convergence factors are needed.
With this Tamagawa normalization,
\[
\operatorname{vol}\bigl(
\mathrm{PGL}_2(F)\backslash
\mathrm{PGL}_2(\mathbb A),dt
\bigr)=2
\]
by the classical Tamagawa-number formula
for $\mathrm{PGL}_2$; see \cite{Weil82,Ono65}.

At archimedean places, all local zeta and $L$-factors
are taken in the normalization arising from Tate's
local zeta integrals with the self-dual measures
fixed above, as in the introduction.
In particular,
\[
 \zeta_{\mathbb R}(s)=\pi^{-s/2}\Gamma(s/2),
 \qquad
 \zeta_{\mathbb C}(s)=(2\pi)^{1-s}\Gamma(s).
\]
As fixed in the introduction, we use
$L_v(s,\eta_v)=\zeta_{E_v}(s)/\zeta_{F_v}(s)$
at every place, including the archimedean ones.
Set $c_v=\zeta_{F_v}(1)$ at every place. Thus
$c_v=(1-q_v^{-1})^{-1}$ at finite places, and $c_v=1$
at archimedean places, since
$\zeta_{\mathbb R}(1)=\zeta_{\mathbb C}(1)=1$.
Define
\[
d^\times x_v
=
c_v\frac{dx_v}{|x|_v},
\qquad
d^\times x
=
|\Delta_F|^{-1/2}\prod_v d^\times x_v.
\]
Here $|\,\cdot\,|_v$ denotes the normalized
absolute value on $F_v$, and the product is
understood in the restricted-product sense.

Let
\[
\mathbb A^1
=
\{x\in\mathbb A^\times:|x|_{\mathbb A}=1\}.
\]
We use the Haar measure $d^1x$ on $\mathbb A^1$
determined by the decomposition
\[
\int_{\mathbb A^\times}\phi(x)\,d^\times x
=
\int_0^\infty
\int_{\mathbb A^1}
\phi(x_t y)\,d^1y\,\frac{dt}{t},
\]
where $x_t\in\mathbb A^\times$ is any measurable
choice with $|x_t|_{\mathbb A}=t$.
Thus the quotient measure on
$\mathbb A^\times/\mathbb A^1\simeq\mathbb R_{>0}$
is $dt/t$. Put
\[
c_F
=
\operatorname{vol}
\bigl(F^\times\backslash\mathbb A^1,d^1x\bigr).
\]
This agrees with the normalization of
\cite[(3.1)]{SW1}. At every finite place, the measure
$c_v\,dx_v/|x|_v$ gives $\mathfrak o_v^\times$ volume one,
and at the archimedean places $d^\times x_v$ is the
measure of Tate's local integrals above.

Let $f=\bigotimes_vf_v$ be the standard factorizable
Schwartz function: $f_v=\mathbf 1_{\mathfrak o_v}$ at
finite places, $f_v(x)=e^{-\pi x^2}$ at real places, and
$f_v(x)=e^{-2\pi|x|^2}$ at complex places. With our
additive and multiplicative Haar measures, its global
Tate zeta integral is
\[
Z(f,s)=|\Delta_F|^{-1/2}\Lambda_F(s).
\]
By Tate's residue formulas \cite[Main Theorem~4.4.1]{Tate},
\[
\operatorname*{Res}_{s=0}Z(f,s)
=-c_F f(0)=-c_F,
\qquad
\operatorname*{Res}_{s=1}Z(f,s)
=c_F\widehat f(0)
=c_F|\Delta_F|^{-1/2},
\]
where the Fourier transform is taken with respect to
the self-dual measure $dx$. Since the archimedean
factors of $\Lambda_F$ equal $1$ at $s=1$, it follows
that
\[
c_F=\operatorname*{Res}_{s=1}
\zeta_F^{\mathrm{fin}}(s),
\qquad
\operatorname*{Res}_{u=0}
\bigl(|\Delta_F|^{-1/2}\Lambda_F(u)\bigr)
=-c_F,
\]
where $\zeta_F^{\mathrm{fin}}$ is the finite part of the
Dedekind zeta function.

Let $H=\widetilde M/A$ be the effective cube
group defined in Section~\ref{subsec:group}.
Choose a nonzero left-invariant $F$-rational
gauge form $\omega_H$ on $H$. We equip
$H(\mathbb A)$ with the measure
\[
dh
=
c_F^{-1}|\Delta_F|^{-5}
\prod_v c_v|\omega_{H,v}|.
\]
This is the Tamagawa measure of
Suzuki--Wakatsuki \cite[\S4.3]{SW1}
(see part~(4) of Lemma~\ref{lem:levi-measure}),
rather than the Iwasawa-coordinate measure
used in \cite[\S3.2]{SW1}.

On the central torus $A\simeq\mathbb G_m$,
we use the coordinate $c$ introduced
in \S\ref{subsec:group}, with the
action on Witt coordinates specified
there.

We transport the multiplicative measure
$d^\times x$ to $A(\mathbb A)$ through
this coordinate and denote it by
$d^\times c$. Explicitly,
\[
d^\times c
=
|\Delta_F|^{-1/2}
\prod_v c_v\frac{dc_v}{|c|_v}.
\]
This measure is defined in the coordinate
$c$, not in the parameter $t$ of $A_0$,
where $c=t^2$. In particular, no factor
$c_F^{-1}$ is included in $d^\times c$.

We equip $N(\mathbb A)\simeq M_2(\mathbb A)$
with the product of the globally self-dual
additive measures on its four coordinates:
\[
dn=
|\Delta_F|^{-2}
\prod_v|\omega_{N,v}|,
\]
where $\omega_N$ is the standard
translation-invariant $F$-rational gauge
form on $N\simeq M_2$. Consequently,
\[
\operatorname{vol}
\bigl(N(F)\backslash N(\mathbb A),dn\bigr)=1.
\]

For $m\in\widetilde M(\mathbb A)$, put
\[
\chi_P(m)=\det\bigl(\operatorname{Ad}(m)|_{\mathfrak n}\bigr),
\qquad
\delta_P(m)=|\chi_P(m)|_{\mathbb A},
\]
so that $\delta_P$ is the parabolic modulus.
This is the adelic form of the modulus
computed in \S\ref{subsec:invariants}.
We extend $\chi_P$ and $\delta_P$ to
$G'=N\rtimes\widetilde M$ trivially on $N$.

The left Haar measure on $G'(\mathbb A)$,
in the coordinates $p=nm$, is
\[
d\mu_{G'}(nm)
=
\delta_P(m)^{-1}\,dn\,dm,
\]
where $dm$ is the measure on
$\widetilde M(\mathbb A)$ characterized
in the following lemma.
In particular, $\delta_P=u^4$, with
$u$ as defined in
\S\ref{subsec:invariants}; both are
independent of the representative of $m$,
since $Z'_\rho$ scales $a$ and $b$ by the
same scalar.

\begin{lemma}\label{lem:levi-measure}
\begin{enumerate}
\item The sequence
\[
1\longrightarrow A\longrightarrow
\widetilde M\longrightarrow H
\longrightarrow1
\]
is exact on $F$-, $F_v$-, and
$\mathbb A$-points, and on
$\mathfrak o_v$-points for almost
all finite places, with respect
to the evident smooth integral
models.

\item There is a unique Haar measure
$dm$ on $\widetilde M(\mathbb A)$
such that
\[
\int_{\widetilde M(\mathbb A)}
\phi(m)\,dm
=
\int_{H(\mathbb A)}
\int_{A(\mathbb A)}
\phi(e_c\widetilde h)\,
d^\times c\,dh,
\]
where $\widetilde h$ is any
measurable lift of
$h\in H(\mathbb A)$.
We abbreviate this
quotient-integral formula by
\[
dm=dh\,d^\times c.
\]

\item Let $\omega_A=dc/c$ be the
invariant $F$-rational gauge
form on $A$. Let
$\omega_{\widetilde M}$ be the
left-invariant gauge form
corresponding to
$\omega_A\otimes\omega_H$
under the canonical
determinant-line isomorphism
\[
\det\widetilde{\mathfrak m}^{*}
\simeq
\det\mathfrak a^{*}\otimes
\det\mathfrak h^{*}
\]
induced by
\[
0\longrightarrow\mathfrak a
\longrightarrow\widetilde{\mathfrak m}
\longrightarrow\mathfrak h
\longrightarrow0.
\]
We use the indicated order of
the factors. Changing the
ordering convention changes
the resulting gauge form at
most by a sign, which does not
affect its associated local
measures. Then
\[
dm=
c_F^{-1}|\Delta_F|^{-11/2}
\prod_v c_v^2
|\omega_{\widetilde M,v}|.
\]

Let $\omega_{G'}$ be the
left-invariant gauge form
corresponding to
$\omega_N\otimes\omega_{\widetilde M}$
under $G'=N\rtimes\widetilde M$.
Then
\[
d\mu_{G'}=
c_F^{-1}|\Delta_F|^{-15/2}
\prod_v c_v^2
|\omega_{G',v}|.
\]

\item The measure $dh$ does not
depend on the choice of nonzero
left-invariant $F$-rational
gauge form $\omega_H$.
Under the identification
\[
H\simeq\mathrm{GL}_2^3/\ker\rho_V,
\]
it agrees with the Tamagawa
measure of Suzuki--Wakatsuki
\cite[\S4.3]{SW1}; that is,
the measure $dh$ defined above is
\[
dh=
c_F^{-1}|\Delta_F|^{-5}
\prod_v c_v|\omega_{H,v}|.
\]
This is the normalization
used in \cite[\S\S1.1, 4.3]{SW1},
and differs from the measure
used in \cite[\S3.2]{SW1}.
\end{enumerate}
\end{lemma}
The proof is given in Appendix~\ref{app:levi-measure}.

Although $\widetilde M$ has a
central one-dimensional torus
in addition to the effective
cube group $H$, the measure
$d\mu_{G'}$ contains only one
factor $c_F^{-1}$. This is
because $dh$ is the
Tamagawa-normalized measure
on $H$, whereas $d^\times c$
is the unnormalized
multiplicative measure on $A$.
In particular, no second factor
$c_F^{-1}$ is introduced when
passing from $H$ to
$\widetilde M$.

\section{The zeta integral and its convergence}\label{sec:zeta}

\subsection{Definition of the zeta integral}
\label{subsec:zeta-integral}

Let $F$ be a number field and let $E/F$ be a quadratic field extension.
We retain the notation and measure conventions of Section~\ref{sec:fusion}. Let $\pi$
be an irreducible cuspidal automorphic representation of
$\mathrm{PGL}_2(\mathbb A)$, realized in the space $\mathcal A_\pi$
of cusp forms.

Let $P=P_{2,2}=MN\subset\mathrm{GL}_4$. We use the Levi
parametrization
\[
h(g_1,g_2)=\operatorname{diag}({}^tg_1^{-1},g_2),
\qquad
g_1,g_2\in\mathrm{GL}_2(\mathbb A).
\]
Thus, writing $h=\operatorname{diag}(a,b)$, we have
$a={}^tg_1^{-1}$ and $b=g_2$. Recall the positive character
\[
u(h)
=
\left|\frac{\det a}{\det b}\right|_{\mathbb A}^{1/2}
=
|\det g_1\det g_2|_{\mathbb A}^{-1/2},
\qquad
\delta_P(h)^{1/2}=u(h)^2.
\]

The inverse-transpose automorphism causes no ambiguity in the
automorphic realization. Indeed, with
\[
J=
\begin{pmatrix}
0&1\\
-1&0
\end{pmatrix},
\]
one has ${}^ta^{-1}=(\det a)^{-1}JaJ^{-1}$.
Since $\pi$ has trivial central character and
$J\in\mathrm{GL}_2(F)$, for every $\phi\in\mathcal A_\pi$,
\[
\phi({}^ta^{-1})
=
\phi(JaJ^{-1})
=
\phi(aJ^{-1})
=
(\rho(J^{-1})\phi)(a).
\]
In particular, the transported function belongs to
$\mathcal A_\pi$ itself. Moreover,
$\mathcal A_\pi=\mathcal A_{\widetilde\pi}$ and
$\widetilde\pi\simeq\pi$, because the central character is
trivial. By multiplicity one for the cuspidal spectrum of $\mathrm{GL}_2$
\cite[Proposition~11.1.1]{JL70}, these representations have the
same realization in the cuspidal automorphic spectrum.
Consequently, the function
\[
h(g_1,g_2)
\longmapsto
\phi_1(g_1)\phi_2(g_2)
\]
belongs to the automorphic realization of
$\pi\boxtimes\widetilde\pi$ on $M(\mathbb A)$. Since the two
realizations coincide, we also write $\sigma=\pi\boxtimes\pi$
for this space, as in the rest of the paper.

Consider the normalized induced representation
\[
I(s')
=
\operatorname{Ind}_{P(\mathbb A)}^{\mathrm{GL}_4(\mathbb A)}
\bigl(
\pi|\det|^{s'}
\boxtimes
\widetilde\pi|\det|^{-s'}
\bigr).
\]
We use its scalar-valued automorphic realization: a smooth
function $f$ on $\mathrm{GL}_4(\mathbb A)$ belongs to $I(s')$
if it is left $N(\mathbb A)$-invariant and, for every fixed $g$,
the function
\[
m\longmapsto
\delta_P(m)^{-1/2}u(m)^{-2s'}f(mg)
\]
belongs to the automorphic realization of
$\sigma=\pi\boxtimes\pi$ on $M(\mathbb A)$,
under the identification
$\widetilde\pi\simeq\pi$ fixed above. Since
\[
\delta_P(m)^{1/2}
\left|
\frac{\det a}{\det b}
\right|_{\mathbb A}^{s'}
=
u(m)^{2+2s'},
\]
we introduce the shifted parameter
\[
s_0=s'+1.
\]
Thus the inducing character on the Levi subgroup is
$u^{2s_0}$.

Let $K_4\subset\mathrm{GL}_4(\mathbb A)$ be the standard
maximal compact subgroup, whose finite local components are
$\mathrm{GL}_4(\mathfrak o_v)$ and whose archimedean components
are $\mathrm O(4)$ or $\mathrm U(4)$. By Iwasawa decomposition, $\mathrm{GL}_4(\mathbb A)=P(\mathbb A)K_4$.
Extend $u$ to $\mathrm{GL}_4(\mathbb A)$ by
\[
u(nmk)=u(m),
\qquad
n\in N(\mathbb A),\quad
m\in M(\mathbb A),\quad
k\in K_4.
\]
This is well defined: regarded as a character of
$P(\mathbb A)$ trivial on $N(\mathbb A)$, $u$ is trivial
on $P(\mathbb A)\cap K_4$. Indeed, this intersection is
compact, and both $|\det a|$ and $|\det b|$ equal $1$
on its Levi projection.

Choose a smooth, $K_4$-finite section $f^\circ\in I(0)$
whose inducing data are finite sums of pure tensors in
$\mathcal A_\pi\otimes\mathcal A_{\widetilde\pi}$. For
$s_0\in\mathbb C$, define
\[
f_{s_0}(nmk)
=
u(m)^{2s_0-2}f^\circ(nmk).
\]
Then $f_{s_0}\in I(s_0-1)$, and this defines a
holomorphic family of flat sections, with restriction
to $K_4$ independent of $s_0$. In particular, its
restriction to $M(\mathbb A)$ has the form
\[
f_{s_0}(h(g_1,g_2))
=
u(h(g_1,g_2))^{2s_0}
\sum_j
\phi_{1,j}(g_1)\phi_{2,j}(g_2),
\]
for finitely many pairs of cusp forms
$\phi_{1,j},\phi_{2,j}\in\mathcal A_\pi$.

The section $f_{s_0}$ is left $P(F)$-invariant. Indeed,
it is left $N(\mathbb A)$-invariant; for
$\gamma\in M(F)$ the product formula gives
$u(\gamma)=1$, while the inducing automorphic functions
are left $M(F)$-invariant. It is also invariant under
the centre of $\mathrm{GL}_4(\mathbb A)$. Indeed, $zI_4=h(z^{-1}I_2,zI_2)$ with $u(zI_4)=1$,
and the inducing representations have trivial central
characters.

Extend $f_{s_0}$ to
$\mathrm{GL}_4(\mathbb A)\times\mathrm{GL}_2(\mathbb A)$
independently of the outer $\mathrm{GL}_2$-factor.
Its invariance under the scalar matrices $zI_4$
implies invariance under
\[
Z'_\rho(\mathbb A)
=
\{
(zI_4,z^{-2}I_2):z\in\mathbb A^\times
\}.
\]
It therefore descends to a well-defined function on
$G(\mathbb A)$, which we continue to denote by $f_{s_0}$.

Let $\Psi\in\mathcal S(V'(\mathbb A))$. As in
\S\ref{subsec:theta}, define
\[
\Theta_E^{\ell\neq0}(\Psi)(g)
=
\sum_{\substack{
w\in V'_E(F)\\
\ell(w)\neq0
}}
\Psi(wg).
\]
The series converges absolutely for each fixed
$g\in G(\mathbb A)$, as established in \S\ref{subsec:theta}.
By Proposition~\ref{prop:relative-invariants}, the discriminant square class is
preserved under $G(F)$. Moreover, as in the proof of
that proposition, an element
$(h,k)\in P\times\mathrm{GL}_2$, with Levi component
$\operatorname{diag}(a,b)$, acts on the column vector
$\ell$ by $\ell\longmapsto \det(a)\,{}^tk\ell$.
The unipotent radical $N$ fixes $\ell$. Thus the
condition $\ell\neq0$ is preserved by $G'(F)$,
and $\Theta_E^{\ell\neq0}(\Psi)$ is left
$G'(F)$-invariant.

Let $K_2\subset\mathrm{GL}_2(\mathbb A)$ be the
standard maximal compact subgroup, and let
$K\subset G(\mathbb A)$ be the image of
$K_4\times K_2$. By Iwasawa decomposition and
the adelic quotient identifications established
in \S\ref{subsec:group}, $G(\mathbb A)=G'(\mathbb A)K$.
We normalize $dk$ by $\operatorname{vol}(K)=1$
and fix the Haar measure $dg$ on $G(\mathbb A)$
by the Iwasawa integration formula
\[
\int_{G(\mathbb A)}
\varphi(g)\,dg
=
\int_K
\int_{\widetilde M(\mathbb A)}
\int_{N(\mathbb A)}
\varphi(nmk)
\delta_P(m)^{-1}
\,dn\,dm\,dk.
\]
Here $dn$ and $dm$ are the measures fixed in
\S\ref{subsec:measures}, with $dm$ defined on the effective
Levi quotient $\widetilde M(\mathbb A)$.
Equivalently, using the left Haar measure
$d\mu_{G'}$ of Lemma~\ref{lem:levi-measure}, this formula reads
\begin{equation}\label{eq:iwasawa-Gprime}
\int_{G(\mathbb A)}
\varphi(g)\,dg
=
\int_K
\int_{G'(\mathbb A)}
\varphi(pk)
\,d\mu_{G'}(p)\,dk.
\end{equation}
This is the standard Iwasawa integration formula,
with $\operatorname{vol}(K)=1$; it does not require
uniqueness of the decomposition $g=pk$.

We use the measure on
$G'(F)\backslash G(\mathbb A)$ induced by $dg$
and counting measure on the discrete subgroup
$G'(F)$. Since $G$ is unimodular, this quotient
measure is right $G(\mathbb A)$-invariant.
Recall from Proposition~\ref{prop:relative-invariants} that
\[
\chi_\Delta(h,k)
=
(\det h\det k)^2
\]
defines a character of $G$, trivial on $Z'_\rho$.

For $s,s_0\in\mathbb C$, define
\begin{equation}\label{eq:zeta-integral}
Z'_E(\Psi,f;s,s_0)
=
\int_{G'(F)\backslash G(\mathbb A)}
|\chi_\Delta(g)|^{s+s_0}
\Theta_E^{\ell\neq0}(\Psi)(g)
f_{s_0}(g)\,dg,
\end{equation}
initially in the region of absolute convergence
established in
Proposition~\ref{prop:absolute-convergence}
of \S\ref{subsec:absolute-convergence}. The integrand is left
$G'(F)$-invariant: the theta series has this
invariance, the section $f_{s_0}$ is left
$G'(F)$-invariant, and
$|\chi_\Delta(\gamma)|=1$ for
$\gamma\in G'(F)$ by the product formula.
Thus the integral is well defined whenever
it converges absolutely. The quotient by
$G'(F)$ is natural: although the discriminant
square class is preserved by $G(F)$, the
condition $\ell\neq0$ is preserved only by
the parabolic subgroup.

Recall that the relative invariant $R$ has character
\[
\chi_R(h,k)
=
\chi_\Delta(h,k)
\frac{\det a}{\det b},
\qquad
h=\operatorname{diag}(a,b).
\]
This identity is on $G'$; the character
$\chi_R$ is not a character of $G$. Since
\[
|\chi_R(h,k)|
=
|\chi_\Delta(h,k)|u(h)^2,
\]
we obtain
\begin{equation}\label{eq:zeta-characters}
|\chi_\Delta(h,k)|^{s+s_0}
u(h)^{2s_0}
=
|\chi_\Delta(h,k)|^s
|\chi_R(h,k)|^{s_0}.
\end{equation}
Hence the character in the zeta integral,
combined with the inducing character of
$f_{s_0}$, is precisely the character
associated with the two relative invariants
$\Delta$ and $R$ of the prehomogeneous action
of $G'$ on $V'$.

The shifted parameter $s_0=s'+1$ will be
used throughout the paper. In particular,
the normalized Eisenstein parameter is
$s'=s_0-1$, and the same convention will
be used for the standard intertwining operator.

We now fix the scaling convention used in the
convergence estimates below. Put $d=[F:\mathbb Q]=r_1+2r_2$.
For $u>0$, define an idele $c(u)\in\mathbb A^\times$
by
\[
c(u)_v=
\begin{cases}
u^{1/d},&v\mid\infty,\\
1,&v\nmid\infty.
\end{cases}
\]
We use the normalized absolute values, so that
$|z|_v=|z|_{\mathbb C}^2$ at a complex place.
Consequently, $|c(u)|_{\mathbb A} = u^{(r_1+2r_2)/d} =u$.

Let
\[
a_u=e_{c(u)}\in A(\mathbb A)
\subset\widetilde M(\mathbb A),
\]
where $e_c$ is the one-parameter subgroup
introduced in Section~\ref{sec:fusion}. Then
\[
u(a_u)=u,
\qquad
\delta_P(a_u)=u^4,
\]
and
\[
|\chi_\Delta(a_u)|_{\mathbb A}=1,
\qquad
|\chi_R(a_u)|_{\mathbb A}=u^2.
\]
In particular,
\[
f_{s_0}(a_um)
=
u^{2s_0}f_{s_0}(m),
\qquad
m\in\widetilde M(\mathbb A).
\]

With the coordinates
$w=(\ell,x,\lambda)$ of Section~\ref{sec:fusion},
the action of $a_u$ at each archimedean
place is
\begin{equation}\label{eq:au-action}
(\ell,x,\lambda)
\longmapsto
(u^{1/d}\ell,x,u^{-1/d}\lambda).
\end{equation}
At every finite place, $a_u$ acts trivially.
The element $a_u$ is central in
$\widetilde M(\mathbb A)$, although it
need not commute with the unipotent
radical $N(\mathbb A)$.

Define
\[
\widetilde M(\mathbb A)^1
=
\ker\bigl(
u:\widetilde M(\mathbb A)
\longrightarrow\mathbb R_{>0}
\bigr).
\]
Since $u(a_u)=u$, multiplication gives
an isomorphism of topological groups
\[
\mathbb R_{>0}
\times\widetilde M(\mathbb A)^1
\xrightarrow{\ \sim\ }
\widetilde M(\mathbb A),
\qquad
(u,m^1)\longmapsto a_um^1.
\]
Its inverse is
\[
m\longmapsto
\bigl(
u(m),a_{u(m)}^{-1}m
\bigr).
\]
The group structure is a direct product
because $a_u$ is central in
$\widetilde M(\mathbb A)$.

Since $u$ is trivial on
$\widetilde M(F)$ by the product formula,
we may write
\[
[\widetilde M]^1
=
\widetilde M(F)
\backslash\widetilde M(\mathbb A)^1.
\]
We use the corresponding Haar decomposition
\[
dm=d^\times u\,dm^1,
\qquad
d^\times u=\frac{du}{u},
\]
where $dm^1$ is the Haar measure on
$\widetilde M(\mathbb A)^1$ determined by
this decomposition.

This convention is compatible with the
measure decomposition on $A(\mathbb A)$
fixed in \S\ref{subsec:measures}:
\[
d^\times c
=
d^\times u\,d^1c,
\qquad
u=|c|_{\mathbb A}.
\]
Here $d^1c$ is the Haar measure on
$A(\mathbb A)^1\simeq\mathbb A^1$
obtained by transporting the measure
$d^1x$ of \S\ref{subsec:measures}
via $c\mapsto e_c$. In particular, the
constant
\[
c_F
=
\operatorname{vol}
\bigl(
F^\times\backslash\mathbb A^1,
d^1c
\bigr)
\]
is unchanged. No additional Jacobian
or factor of $2$ is introduced by
the choice of $u$.

\subsection{The \texorpdfstring{$R$}{R}-cutoff}
\label{subsec:R-cutoff}

Let $F$ be a number field of degree $d=[F:\mathbb Q]$.
Let $K$ be the compact subgroup of $G(\mathbb A)$ fixed
in \S\ref{subsec:zeta-integral}, namely the image
of $K_4\times K_2$, and write $K=K_\infty K_f$.
Fix a norm $\|\cdot\|$ on the real vector space
$V'(F_\infty)$.

Choose a compact subset $\Omega_0\subset V'(\mathbb A_f)$
such that
\[
\operatorname{supp}\Psi
\subset V'(F_\infty)\times\Omega_0,
\]
and put $\Omega_f=\Omega_0K_f$.
Since $K_f$ is compact, $\Omega_f$ is compact and
right-$K_f$-invariant.

For $B\geq0$, define
\[
\Psi_B^+(v)
=
(1+\|v_\infty\|)^{-B}
\mathbf 1_{\Omega_f}(v_f),
\qquad v\in V'(\mathbb A).
\]
All implied constants below may depend on $\Psi$,
the chosen norm, and $\Omega_f$, but are uniform
in $\kappa\in K$.

\begin{lemma}\label{lem:R-cutoff}
For every $A,B\geq0$, there is a constant $C_{A,B}>0$
such that
\[
|\Psi(v\kappa)|
\leq
C_{A,B}
(1+|R(v)|_{\mathbb A})^{-A}
\Psi_B^+(v)
\]
for all $v\in V'(\mathbb A)$ and $\kappa\in K$.
\end{lemma}

\begin{proof}
Since $R$ is homogeneous of degree four and the
normalized absolute value at a complex place is
the square of the usual one, we have
\[
|R(v_\infty)|_\infty
\leq C(1+\|v_\infty\|)^{4d}.
\]
Moreover, compactness of $\Omega_f$ gives $\sup_{v_f\in\Omega_f}|R(v_f)|_f<\infty$.
Consequently, for $v_f\in\Omega_f$,
\[
1+|R(v)|_{\mathbb A}
\ll(1+\|v_\infty\|)^{4d}.
\]

On the other hand, compactness of $K_\infty$
implies $\|v_\infty\kappa_\infty\| \asymp\|v_\infty\|$
uniformly in $\kappa_\infty\in K_\infty$.
Since $\Psi$ is Schwartz--Bruhat, for every
$N\geq0$,
\[
|\Psi(v\kappa)|
\ll_N
(1+\|v_\infty\|)^{-N}
\mathbf 1_{\Omega_f}(v_f),
\]
uniformly in $\kappa\in K$. Taking
$N=B+4dA$, we obtain
\[
\begin{aligned}
|\Psi(v\kappa)|
&\ll_{A,B}
(1+\|v_\infty\|)^{-B-4dA}
\mathbf 1_{\Omega_f}(v_f)\\
&\ll_{A,B}
(1+|R(v)|_{\mathbb A})^{-A}
\Psi_B^+(v),
\end{aligned}
\]
as required.
\end{proof}

We shall also use the compact-uniformity estimate
\begin{equation}\label{eq:compact-majorant}
\Psi_B^+(v\kappa)\ll_B\Psi_B^+(v),
\qquad \kappa\in K.
\end{equation}
Indeed, the archimedean norms are uniformly
equivalent under $K_\infty$, while the
finite-support condition follows from the
right-$K_f$-invariance of $\Omega_f$.

Define
\[
\Theta_B^{+,N}(g)
=
\int_{[N]}
\sum_{\substack{w\in V'_E(F)\\\ell(w)\neq0}}
\Psi_B^+(wng)\,dn.
\]
Let $D(g)$ denote the $N$-constant term of the
theta series restricted to the stratum
$\ell\neq0$.

\begin{corollary}\label{cor:R-cutoff}
For every $A,B\geq0$, every $u>0$, and every
$m\in[\widetilde M]^1$, one has
\[
\begin{aligned}
&\int_{[N]}
\sum_{\substack{w\in V'_E(F)\\\ell(w)\neq0}}
|\Psi(wna_um)|\,dn\\
&\qquad\ll_{A,B}
\bigl(1+u^2|\chi_\Delta(m)|_{\mathbb A}\bigr)^{-A}
\Theta_B^{+,N}(a_um).
\end{aligned}
\]
Moreover, uniformly in $\kappa\in K$,
\[
\begin{aligned}
&\int_{[N]}
\sum_{\substack{w\in V'_E(F)\\\ell(w)\neq0}}
|\Psi(wna_um\kappa)|\,dn\\
&\qquad\ll_{A,B}
\bigl(1+u^2|\chi_\Delta(m)|_{\mathbb A}\bigr)^{-A}
\Theta_B^{+,N}(a_um).
\end{aligned}
\]
In particular, $|D(a_um\kappa)|$ satisfies
the same bound.
\end{corollary}

\begin{proof}
Let $w\in V'_E(F)$ with $\ell(w)\neq0$.
By the anisotropy argument following
Theorem~\ref{thm:generic-orbits}, the vector $\ell_2w_1-\ell_1w_2$
is a nonzero vector in an anisotropic quadratic
plane. Consequently, $R(w)\in F^\times$.
Since $R$ is a relative invariant of $G'$,
for $g\in G'(\mathbb A)$ we have
\[
R(wg)=\chi_R(g)R(w).
\]
The product formula gives $|R(w)|_{\mathbb A}=1$,
and therefore
\[
|R(wg)|_{\mathbb A}
=
|\chi_R(g)|_{\mathbb A}.
\]

Recall that the central element $e_c$ may be
represented by
\[
\bigl(
\operatorname{diag}(I_2,c^{-1}I_2),
cI_2
\bigr).
\]
For this representative, $\det h=c^{-2}$ and $\det k=c^2$.
Thus the relative-invariant characters satisfy
\[
\chi_\Delta(e_c)=1,
\qquad
\chi_R(e_c)=c^2.
\]
Since
\[
\chi_R
=
\chi_\Delta\frac{\det a}{\det b},
\qquad
u=
\left|
\frac{\det a}{\det b}
\right|_{\mathbb A}^{1/2},
\]
we obtain
\[
|\chi_R(m)|_{\mathbb A}
=
|\chi_\Delta(m)|_{\mathbb A}u(m)^2.
\]
On $[\widetilde M]^1$, we have $u(m)=1$.

By \S\ref{subsec:zeta-integral},
$a_u=e_{c(u)}$ with $|c(u)|_{\mathbb A}=u$,
so the identity $\chi_R(e_c)=c^2$ gives
\[
|\chi_R(a_u)|_{\mathbb A}=u^2.
\]
Consequently, for $m\in[\widetilde M]^1$,
\[
\begin{aligned}
|\chi_R(a_um)|_{\mathbb A}
&=
|\chi_R(a_u)|_{\mathbb A}
|\chi_R(m)|_{\mathbb A}\\
&=
u^2|\chi_\Delta(m)|_{\mathbb A}.
\end{aligned}
\]

The character $\chi_R$ is trivial on the
unipotent subgroup $N$. Hence, for every
$n\in N(\mathbb A)$,
\[
|R(wna_um)|_{\mathbb A}
=
u^2|\chi_\Delta(m)|_{\mathbb A}.
\]
Applying Lemma~\ref{lem:R-cutoff}, we obtain
\[
\begin{aligned}
|\Psi(wna_um)|
&\ll_{A,B}
\bigl(1+|R(wna_um)|_{\mathbb A}\bigr)^{-A}
\Psi_B^+(wna_um)\\
&=
\bigl(1+u^2|\chi_\Delta(m)|_{\mathbb A}\bigr)^{-A}
\Psi_B^+(wna_um).
\end{aligned}
\]
Summing these inequalities over $w$ and
integrating over $[N]$ proves the first
assertion, with the absolute value taken
before summation and integration.

For $\kappa\in K$, apply
Lemma~\ref{lem:R-cutoff} with $v=wna_um$
and the element $\kappa\in K$.
Since the bound in that lemma involves
$R(v)$ rather than $R(v\kappa)$, we obtain
\[
|\Psi(wna_um\kappa)|
\ll_{A,B}
\bigl(1+u^2|\chi_\Delta(m)|_{\mathbb A}\bigr)^{-A}
\Psi_B^+(wna_um),
\]
uniformly in $\kappa\in K$.
Summing over $w$ and integrating over $[N]$
proves the uniform assertion. The corresponding
bound for $|D(a_um\kappa)|$ follows immediately
from the triangle inequality.
\end{proof}

Corollary~\ref{cor:R-cutoff} separates the
relative-invariant decay
\[
\bigl(1+u^2|\chi_\Delta(m)|_{\mathbb A}\bigr)^{-A}
\]
from the lattice-point majorant. Since $A$
is arbitrary, this provides the additional
decay needed for the range $u\geq1$.
For $0<u\leq1$, we instead use the grading
\[
V'=(\ell)\oplus(x)\oplus(\lambda)
\]
and the fibration $V'\to V'/W$, where
$W=\{\ell=0\}$. We estimate the lattice sum
along the fibres uniformly in the quotient
coordinate, and then sum over the quotient
lattice. This gives the uniform $u^{-2}$
loss established in Appendix~\ref{app:convergence}.

\subsection{Slice convergence and absolute convergence}
\label{subsec:absolute-convergence}

Recall the decomposition
\[
\widetilde M(\mathbb A)
=
\{a_u:u>0\}
\times\widetilde M(\mathbb A)^1
\]
and the Haar measure decomposition $dm=d^\times u\,dm^1$
from \S\ref{subsec:zeta-integral}.

Let $\Phi_f$ denote the restriction of the section
$f$ under consideration to
$\widetilde M(\mathbb A)^1K$.
For each $\kappa\in K$, the function
$m\mapsto\Phi_f(m\kappa)$ is the corresponding
cuspidal inducing datum on the Levi.
Its rapid decay on Siegel sets, used in the proofs of
Lemma~\ref{lem:slice-convergence} and
Proposition~\ref{prop:absolute-convergence},
is recorded in Appendix~\ref{app:convergence}.

For $h\in\mathrm{GL}_4(F_\infty)$, put
\begin{equation}\label{eq:height-H}
H(h)
=
\max\left\{
1,\left\|(\textstyle\bigwedge^2h)^{-1}\right\|
\right\},
\end{equation}
the operator norm being taken on the full archimedean
representation with fixed compatible norms.

We first establish convergence on the norm-one slice.

\begin{lemma}\label{lem:slice-convergence}
Let
\[
\sigma_1>3,
\qquad
B_+>16d\sigma_1.
\]
Then
\[
\begin{aligned}
K_{\sigma_1}:=
\int_K\int_{[\widetilde M]^1}
&|\chi_\Delta(m)|_{\mathbb A}^{\sigma_1}\\
&\times
\Theta_{B_+}^{+,N}(m\kappa)
|\Phi_f(m\kappa)|\,dm^1\,d\kappa
<\infty.
\end{aligned}
\]
\end{lemma}
The proof is given in Appendix~\ref{app:convergence}.

For the absolute convergence of the truncated integral
in Lemma~\ref{lem:truncated-pairing}, we also need an
outer-cusp estimate for the full theta series, without
taking an $[N]$-average. For
$g\in G(\mathbb A)$, define
\[
\Theta_B^+(g)
=
\sum_{w\in V'_E(F)}
\Psi_B^+(wg).
\]
Unlike $\Theta_B^{+,N}$, this is the full theta majorant:
no $[N]$-average is taken, and no restriction $\ell(w)\neq0$
is imposed.

\begin{lemma}\label{lem:outer-cusp}
Let $I=[\sigma_-,\sigma_+]$
be a compact interval with $\sigma_->3$, and choose $B>16d\sigma_+$.

Let $h\in\mathrm{GL}_4(\mathbb A)^1$ be a reduced
representative, and let $H(h)=H(h_\infty)$ be the
height \eqref{eq:height-H}.

Then the integral of $\Theta_B^+(hkz)$ over the outer
Siegel variables and the positive central variable,
with weight $|\chi_\Delta|_{\mathbb A}^{\sigma}$,
is bounded by
\[
\ll_I H(h)^{c_I},
\qquad \sigma\in I,
\]
where one may take $c_I=4d\sigma_+$.

In particular, the outer cusp produces at most
polynomial growth in the inner $\mathrm{GL}_4$ height.
\end{lemma}
The proof is given in Appendix~\ref{app:convergence}.

We now combine these estimates to establish absolute
convergence of the zeta integral. Write $\sigma=\Re(s+s_0)$.
By Corollary~\ref{cor:R-cutoff}, the integral over $[N]$
of the absolute values of the theta summands is bounded
by the relative-invariant cutoff times $\Theta_B^{+,N}$.
Thus the estimates below control the absolute value
of the integrand before summation and integration.
We treat the ranges $0<u\leq1$ and $u\geq1$ separately.
We use separate Schwartz-decay exponents $B_-$ and $B_+$
for these two ranges, respectively; $B_+$ is the exponent
of Lemma~\ref{lem:slice-convergence}. They may be chosen
independently, since by Lemma~\ref{lem:R-cutoff} the
majorants $\Psi_B^+$ are available for every $B\geq0$.

We use the balanced Siegel coordinates of
\S\ref{subsec:scaling-lattice}, explicitly
displayed in \eqref{eq:siegel-coordinates}. Here $r>0$ denotes the
remaining positive central parameter of the Levi
representative $m\in\widetilde M(\mathbb A)^1$, after
separating the $a_u$-direction, normalized by
$|\chi_\Delta(m)|_{\mathbb A}=r^{4d}$
(see \eqref{eq:chi-Delta-siegel}). The parameter $Y\gg1$
is the outer $\mathrm{GL}_2$ Siegel height: in the chosen
balanced representatives, the outer torus coordinate
satisfies $|y|_{\mathbb A}\asymp Y^d$, and the outer
Siegel measure is $\ll Y^{-d}\,d^\times Y$
\eqref{eq:outer-siegel-measure}. The parameters $r$ and $Y$
are independent noncompact directions. The residual
inner-block determinants have uniformly bounded local
moduli; the noncentral inner Siegel directions are measured
by the height $H_\infty=H(h_\infty)\geq1$ of the inner Levi
component $h=\operatorname{diag}(a,b)$, see
\eqref{eq:height-H}. All implied constants are uniform in
the chosen representatives.

For $0<u\leq1$, Lemma~\ref{lem:small-u-general} gives
\[
\begin{aligned}
\Theta_{B_-}^{+,N}(a_um)
\ll{}&
u^{-2}(1+rY^{1/2})^{-B_-/3}\\
&\times
\left(1+\frac{Y^{1/2}}r\right)^{2d}
\left(1+\frac{H_\infty Y^{1/2}}r\right)^{10d}.
\end{aligned}
\]
Multiplying by
$|\chi_\Delta(m)|_{\mathbb A}^{\sigma}=r^{4d\sigma}$, we note
that the rapid-decay factor depends on the combination
$rY^{1/2}$. We therefore set
\[
\rho=rY^{1/2},
\qquad
r=Y^{-1/2}\rho,
\qquad
d^\times r=d^\times\rho,
\]
so that
\[
r^{4d\sigma}=Y^{-2d\sigma}\rho^{4d\sigma},
\qquad
\frac{Y^{1/2}}r=\frac Y\rho.
\]
The $r$-integral is then bounded by
\[
\begin{aligned}
Y^{-2d\sigma}
\int_0^\infty&
\rho^{4d\sigma}
(1+\rho)^{-B_-/3}\\
&\times
(1+Y/\rho)^{2d}
(1+H_\infty Y/\rho)^{10d}
\,d^\times\rho.
\end{aligned}
\]
Since $\sigma>3$ and $B_->12d\sigma$ throughout the parameter set,
splitting the $\rho$-integral at $\rho=1$ gives
\[
\ll
H_\infty^{10d}
Y^{12d-2d\sigma};
\]
the powers governing convergence at $\rho=0$ and
$\rho=\infty$ are $4d\sigma-12d$ and $4d\sigma-B_-/3$,
respectively.

Including the outer Siegel measure, the remaining outer
integral is
\[
\ll
\int_{Y\gg1}
Y^{11d-2d\sigma}\,d^\times Y,
\]
which converges when
\[
\sigma>\frac{11}{2}.
\]
The inner integrations converge by cuspidal
rapid decay.

The original section contributes the factor
\[
u^{2\Re s_0-4}.
\]
Together with the $u^{-2}$ loss from the quotient
counting, the remaining integral is
\[
\int_0^1
u^{2\Re s_0-6}
\,d^\times u,
\]
which converges if
\[
\Re s_0>3.
\]

For $u\geq1$, consider the slice integral
before insertion of the section factor
$u^{2\Re s_0-4}$,
\[
\begin{aligned}
J(u)
:={}&\int_K\int_{[\widetilde M]^1}
|\chi_\Delta(m)|_{\mathbb A}^{\sigma}
\int_{[N]}
\sum_{\substack{w\in V'_E(F)\\\ell(w)\ne0}}
|\Psi(wna_um\kappa)|\,dn\\
&\qquad\times
|\Phi_f(m\kappa)|\,dm^1\,d\kappa.
\end{aligned}
\]
Write $\mathcal X=|\chi_\Delta(m)|_{\mathbb A}$, and choose
an auxiliary cutoff exponent $D\geq\sigma-\sigma_1$.
When $\sigma\geq\sigma_1$,
\[
(1+u^2\mathcal X)^{-D}\mathcal X^\sigma
\leq
u^{-2(\sigma-\sigma_1)}
\mathcal X^{\sigma_1}.
\]
Corollary~\ref{cor:R-cutoff}, with $A=D$, and
Lemma~\ref{lem:large-u-scaling} therefore give
\[
J(u)
\ll
u^{B_+/d+2\sigma_1-2\sigma}
K_{\sigma_1}.
\]

Consequently, the full large-$u$ contribution
is bounded by
\[
K_{\sigma_1}
\int_1^\infty
u^{2\Re s_0-4
+B_+/d+2\sigma_1-2\sigma}
\,d^\times u.
\]
The exponent simplifies to
\[
2\sigma_1+\frac{B_+}{d}
-2\Re s-4.
\]
Thus the integral converges if
\[
\Re s>
\sigma_1+\frac{B_+}{2d}-2.
\]

We summarize the resulting convergence region.

\begin{proposition}\label{prop:absolute-convergence}
Let $E/F$ be a quadratic field extension, and put
$d=[F:\mathbb Q]$ and $\sigma=\Re(s+s_0)$.

Choose real parameters $\sigma_1,B_+$ satisfying
\[
\sigma_1>3,
\qquad
B_+>16d\sigma_1.
\]

Assume
\[
\Re s_0>3
\]
and
\[
\Re s>
\sigma_1+\frac{B_+}{2d}-2.
\]

Then the zeta integral $Z'_E(\Psi,f;s,s_0)$
converges absolutely. For a fixed Schwartz function
and a holomorphic flat family of sections as in
\S\ref{subsec:zeta-integral}, the convergence
is locally uniform on compact subsets of
this open region. In particular, the integral
defines a holomorphic function there.

Moreover, the large-$u$ estimate
\begin{equation}\label{eq:large-u-majorant}
J(u)
\ll
u^{B_+/d+2\sigma_1-2\sigma}
K_{\sigma_1},
\qquad u\geq1,
\end{equation}
holds whenever $\sigma_1>3$,
$B_+>16d\sigma_1$, and
$\sigma=\Re(s+s_0)\geq\sigma_1$,
independently of the other restrictions
on $s$ and $s_0$ imposed above.
The same estimate, as well as
Lemma~\ref{lem:slice-convergence},
applies to any section whose restriction
$\Phi_f$ satisfies the required
rapid-decay estimates.
\end{proposition}

\begin{proof}
Adding the hypotheses
$\Re s_0>3$ and
$\Re s>
\sigma_1+B_+/(2d)-2$
gives
\[
\sigma
=
\Re s+\Re s_0
>
\sigma_1+\frac{B_+}{2d}+1
>
9\sigma_1+1,
\]
since $B_+/(2d)>8\sigma_1$; in particular
$\sigma>\sigma_1$ and $\sigma>11/2$.
On each compact subset of the stated region,
choose the auxiliary cutoff exponent $D$ so that $D\geq\sup(\sigma-\sigma_1)$,
where the supremum is taken over that compact subset.
By the absolute-value estimate of
Corollary~\ref{cor:R-cutoff}, the bounds above
apply to the integral of the absolute values
of the theta summands, rather than merely
to the absolute value of their sum.

For the range $0<u\leq1$, choose a
Schwartz-decay exponent $B_-$ satisfying
\[
B_->\max\{30d,12d\sigma\}.
\]
Such a choice is always possible by
Lemma~\ref{lem:R-cutoff}. On a compact subset
$K$ of the stated region, we choose a single
exponent satisfying
\[
B_->\max\left\{
30d,\,
12d\sup_{(s,s_0)\in K}\Re(s+s_0)
\right\}.
\]
The estimates for $0<u\leq1$ then establish
convergence under
\[
\Re s_0>3,
\qquad
\sigma>\frac{11}{2}.
\]
For $u\geq1$, the large-$u$ estimate and
Lemma~\ref{lem:slice-convergence} give an
integrable majorant under $\Re s> \sigma_1+\frac{B_+}{2d}-2$.
The Schwartz exponents $B_-$ and $B_+$
are independent, so their requirements
are compatible. Combining the two ranges
proves absolute convergence and the
asserted decay estimate. On a compact subset
of the stated open region, all strict inequalities
have a positive margin; the same majorants can
therefore be chosen uniformly in the parameters.
The holomorphic dependence of the flat section,
together with dominated convergence (or Morera's
theorem), proves the asserted holomorphy.

The final assertion concerns only the slice
estimate and $J(u)$, neither of which includes
the $u$-weight of the section.
\end{proof}

In particular, allowing the auxiliary
parameters $\sigma_1$ and $B_+$
to vary with compact subsets, we see that
the proposition gives absolute convergence
and holomorphy throughout the explicit
region
\[
\Re s_0>3,\qquad \Re s>25.
\]

For the application in Section~\ref{sec:first-expansion}, we will also need to consider
the analogous integral over $0<u\leq U$, for a fixed truncation parameter $U\geq1$ as in Section~\ref{sec:first-expansion} (not to be confused with the matrix $U$ in $n(U)$), involving the full theta
series $\Theta_E(\Psi)$, rather than only its $\ell\neq0$
stratum. The additional contribution from $\ell=0$ is treated
separately. Its $N$-constant term is the cube theta integral
described in Section~\ref{sec:first-expansion}. By
Lemma~\ref{lem:N-constant-term}, the triviality of $\chi_\Delta$ on $A$,
and the flatness of the section, its integrand is the product of a
norm-one Levi integrand and the central weight
$u^{2\Re s_0-2}$. The norm-one Levi integral converges
absolutely for sufficiently large $\sigma$ by
Lemma~\ref{lem:cube-convergence}, and the central integral
\[
\int_0^Uu^{2\Re s_0-2}\,d^\times u
=\frac{U^{2\Re s_0-2}}{2\Re s_0-2}
\]
is finite when $\Re s_0>1$. By Tonelli's theorem the
boundary contribution over $0<u\leq U$ therefore converges absolutely,
for each fixed $U$; no uniformity as $U\to\infty$ is asserted.
In this range of $s_0$, the same weight is not integrable over $u>U$, so this contribution is
controlled separately and does not enter the definition
or the convergence assertion for $Z'_E$. In particular,
Proposition~\ref{prop:absolute-convergence} does not
assert absolute convergence for the full-theta integral.

\section{Meromorphic continuation and the residue at \texorpdfstring{$s_0=1$}{s0=1}}
\label{sec:first-expansion}

\subsection{Standard inputs}
\label{subsec:standard-inputs}

We begin by recalling the standard results on Eisenstein
series and Arthur truncation that will be used in the
first expansion. Let $E_P(f,s_0)$ be the Eisenstein series
attached to the section $f_{s_0}$ of
\S\ref{subsec:zeta-integral}, with
normalized inducing parameter $s'=s_0-1$.

Since $f_{s_0}$ is invariant under the centre of
$\mathrm{GL}_4(\mathbb A)$, the Eisenstein series
descends to
\[
\mathrm{GL}_4(F)Z_4(\mathbb A)
\backslash\mathrm{GL}_4(\mathbb A),
\]
and Arthur's truncation is well defined on
$[\mathrm{GL}_4]^1$.

\medskip
\noindent
{\emergencystretch=1em
\textbf{(S1) Langlands' theory of Eisenstein series
\cite[\S\S~II.1.5, IV.1.8, IV.1.10, IV.1.11]{MW}.}
The Eisenstein series $E_P(f,s_0)$ converges absolutely
for $\Re s_0>2$, admits meromorphic
continuation, and is holomorphic on the line
$\Re s_0=1$. It satisfies the functional
equation
\[
E_P(f,s_0)
=
E_P\bigl(M(s_0-1)f,2-s_0\bigr),
\]
where $M(s_0-1)$ denotes the standard intertwining
operator.\par}

The constant term of $E_P(f,s_0)$ along $P$ is
\[
f_{s_0}+M(s_0-1)f_{s_0}.
\]
Its constant term along every other proper standard
parabolic subgroup of $\mathrm{GL}_4$ vanishes.
Indeed, by cuspidality of the inducing data and the
constant-term formula \cite[\S~II.1.7]{MW}, the
constant term along a standard parabolic $Q$
vanishes unless its Levi subgroup $M_Q$ contains
a conjugate of $M\simeq\mathrm{GL}_2\times\mathrm{GL}_2$.
Among proper standard parabolic subgroups of
$\mathrm{GL}_4$, this occurs only for $Q=P$.

The normalizing factor of $M(s_0-1)$ is
\[
r(s_0-1)
=
\frac{
\Lambda(2s_0-2,\pi\times\widetilde\pi)
}{
\Lambda(2s_0-1,\pi\times\widetilde\pi)
},
\]
up to the conventional $\varepsilon$-factor.
Since $\omega_\pi=1$, we have
$\pi\simeq\widetilde\pi$.

In the half-plane $\Re s_0\geq1$,
$E_P(f,s_0)$ is holomorphic except for at most
a simple pole at $s_0=3/2$, the residual
(Speh) pole, which occurs for suitable $f$
since $\pi\simeq\widetilde\pi$.
Indeed, the poles in this region are controlled
by the standard intertwining operator
\[
M(s')=r(s')N(s'),
\qquad s'=s_0-1,
\]
where the normalized intertwining operator
$N(s')$ is holomorphic for
$\Re s'\geq0$. Indeed, $N(s')$ is the restricted
tensor product of the local normalized operators; on a
decomposable $K$-finite section only finitely many of them act
nontrivially, since the unramified ones fix the spherical vectors;
since the local components $\pi_v$ are unitary and generic,
each local operator is holomorphic and nonzero for
$\Re s'\geq0$ by \cite[Proposition~I.10]{MW89}
(which requires only unitarity, not temperedness), and is an
isomorphism on the unitary axis $\Re s'=0$ by
\cite[\S~I.1(3)]{MW89}.
The completed Rankin--Selberg $L$-function
$\Lambda(s,\pi\times\widetilde\pi)$ has simple
poles at $s=0$ and $s=1$ and is holomorphic
elsewhere \cite[Appendice, Corollaire~(ii)]{MW89}.
It does not vanish for $\Re s>1$, where its Euler
product converges absolutely \cite{JS81}, nor on
the line $\Re s=1$ away from $s=1$ \cite{Sh81},
since the local factors have no zeros.
For $\Re s'\geq0$ the denominator is evaluated in
$\Re s\geq1$, where it does not vanish, and the
numerator has poles only at $2s'\in\{0,1\}$; at $s'=0$ the poles
of the numerator and denominator cancel, consistently with the
holomorphy of $M(0)$. These standard analytic properties
therefore show that $r(s')$ is holomorphic for
$\Re s'\geq0$ except for a possible simple pole at
$s'=1/2$. Since $N(s')$ is holomorphic there, the same holds for
$M(s')$, and hence, by the constant-term description above, for
$E_P(f,s_0)$: its only possible pole in $\Re s_0\geq1$
is at $s_0=3/2$.

\medskip
\noindent
\textbf{(S2) Arthur's truncation
\cite{Ar78,Ar80}.}
For a sufficiently regular truncation parameter $T$,
Arthur's truncation operator $\Lambda^T$, applied in
the $\mathrm{GL}_4$ variable, makes
$\Lambda^T E_P(f,s_0)$ rapidly decreasing on
$[\mathrm{GL}_4]^1$, locally uniformly in $s_0$
away from the poles of the Eisenstein series.

The preceding inputs simplify Arthur's truncation
considerably in the present setting. Because the
inducing data are cuspidal on
$\mathrm{GL}_2\times\mathrm{GL}_2$, only the constant
terms associated with $P$ and $\mathrm{GL}_4$
contribute to the alternating sum.

\begin{lemma}\label{lem:rank-one-truncation}
Arthur's truncation of $E_P(f,s_0)$ is given by
\[
\begin{aligned}
\Lambda^T E_P(f,s_0)(g)
={}&E_P(f,s_0)(g)\\
&-
\sum_{\gamma\in
P(F)\backslash\mathrm{GL}_4(F)}
\bigl(
f_{s_0}
+
M(s_0-1)f_{s_0}
\bigr)(\gamma g)\\
&\qquad\qquad\times
\widehat\tau_P
\bigl(H_P(\gamma g)-T\bigr).
\end{aligned}
\]
\end{lemma}

\begin{proof}
By (S1), the constant terms of $E_P(f,s_0)$
along all proper standard parabolic subgroups
other than $P$ vanish. Consequently, only the
terms indexed by $Q=\mathrm{GL}_4$ and $Q=P$
remain in Arthur's alternating sum.

The former contributes $E_P(f,s_0)$, while the
latter contributes the $P$-constant term $f_{s_0}+M(s_0-1)f_{s_0}$
multiplied by the truncation function
$\widehat\tau_P(H_P(\gamma g)-T)$.
Since $P$ is maximal, the sign of this term
is negative, giving the asserted formula.
\end{proof}

The truncation in Lemma~\ref{lem:rank-one-truncation}
acts only in the $\mathrm{GL}_4$ variable; the
outer $\mathrm{GL}_2$ factor is not truncated.
Its noncompact cusp will instead be controlled
by Lemma~\ref{lem:outer-cusp}. This separation
of the two directions is used in establishing
the convergence of the truncated pairing in
the proof of Lemma~\ref{lem:truncated-pairing}.

We do not require the supports of
$\widehat\tau_P(H_P(\gamma g)-T)$,
as $\gamma$ varies in
$P(F)\backslash\mathrm{GL}_4(F)$,
to be disjoint. The unfolding below
is justified by absolute convergence
of the sum of the integrals of the
absolute values, independently of
any such disjointness property.

\subsection{The \texorpdfstring{$N$}{N}-constant term and the cube theta series}
\label{subsec:N-constant-term}

We next identify the contribution of the boundary
$\ell=0$ to the $N$-constant term of the
fusion-space theta series. This is the geometric
step that relates the first expansion to the
cube zeta integral of Suzuki--Wakatsuki.

Recall that
\[
V'=F^2\oplus V\oplus F^2,
\qquad
V=M_2(F)\oplus M_2(F),
\]
and write
\[
w=(\ell,x,\lambda),
\qquad
x=(x_1,x_2).
\]
The subspace $V$ is the space of Bhargava's
cubes, embedded in $V'$ by $x\longmapsto(0,x,0)$.

For $\Psi\in\mathcal S(V'(\mathbb A))$, define
the Schwartz function on $V(\mathbb A)$
\[
\Phi_\Psi(x)
=
\int_{\mathbb A^2}
\Psi(0,x,\lambda)\,d\lambda.
\]
The measure $d\lambda$ is the self-dual
additive measure fixed in \S\ref{subsec:measures}.

For $m\in\widetilde M(\mathbb A)$, write
\[
\Theta_E^N(\Psi;m)
=
\int_{N(F)\backslash N(\mathbb A)}
\Theta_E(\Psi)(nm)\,dn.
\]
Here $dn$ is normalized so that
\[
\operatorname{vol}
\bigl(N(F)\backslash N(\mathbb A)\bigr)=1.
\]

Split the summation defining $\Theta_E$
according to whether $\ell=0$ or $\ell\ne0$.
We denote the corresponding contributions to
the $N$-constant term by $C_0$ and $D$,
respectively.

\begin{lemma}\label{lem:N-constant-term}
Assume that $E/F$ is a quadratic field
extension. Then
\[
\Theta_E^N(\Psi;m)
=
C_0(\Psi;m)+D(\Psi;m),
\]
where $D$ is the $N$-constant term of the
theta series restricted to $\ell\ne0$.

Let $m$ be represented by a Levi element
whose $\mathrm{GL}_4$ component is
$\operatorname{diag}(a,b)$ and whose outer
$\mathrm{GL}_2$ component is $k$.
Then
\[
C_0(\Psi;m)
=
|\det b|_{\mathbb A}^{-2}
|\det k|_{\mathbb A}^{-1}
\Theta_{\mathrm{cube},E}(\Phi_\Psi;m),
\]
where
\[
\Theta_{\mathrm{cube},E}(\Phi_\Psi;m)
=
\sum_{x\in V_E(F)}
\Phi_\Psi(x\cdot m).
\]
Here $\Delta(x):=\Delta(0,x,\lambda)$, which is independent
of $\lambda$ (see below), and
\[
V_E(F)=\{x\in V(F):\Delta(x)\in d_E(F^\times)^2\}.
\]
The action on the right is the induced
Levi action on the cube space $V$.

Moreover, for the central subgroup $A$
of \S\ref{subsec:group},
\[
C_0(\Psi;e_cm)
=
|c|_{\mathbb A}^2C_0(\Psi;m).
\]
In particular,
\[
C_0(\Psi;a_um)=u^2C_0(\Psi;m).
\]
\end{lemma}

\begin{proof}
We first examine the action of $N$ on the
boundary $\ell=0$.

Identify $N$ with the additive group $M_2$.
By the parabolic action formula of \S\ref{subsec:group},
for $n(X)\in N$ we have
\[
(0,x,\lambda)\cdot n(X)
=
\bigl(
0,x,\lambda+L_x(X)
\bigr),
\]
where
\[
L_x:M_2(F)\longrightarrow F^2,
\qquad
L_x(X)
=
\begin{pmatrix}
\tau(x_1,X)\\
\tau(x_2,X)
\end{pmatrix}.
\]

We claim that $L_x$ is surjective for every
$x\in V_E(F)$.

Indeed, the binary quadratic form attached
to $x$ is $Q_x(u,v) = -\det(ux_1+vx_2)$.
Since $E/F$ is a quadratic field extension,
the discriminant of $Q_x$ belongs to the
nonsquare class defining $E$. In particular,
$Q_x$ is nondegenerate and anisotropic
over $F$.

Suppose that $x_1$ and $x_2$ were linearly
dependent over $F$. Then there would exist
$(u,v)\ne(0,0)$ such that $ux_1+vx_2=0$.
Consequently, $Q_x(u,v)=0$,
contradicting anisotropy. Thus $x_1$ and
$x_2$ are linearly independent.

The pairing
\[
\tau:M_2(F)\times M_2(F)\longrightarrow F
\]
is nondegenerate. Therefore the two linear
functionals
\[
X\longmapsto\tau(x_1,X),
\qquad
X\longmapsto\tau(x_2,X)
\]
are linearly independent. It follows that
\[
L_x:M_2(F)\longrightarrow F^2
\]
is surjective.

Since $L_x$ is an $F$-linear surjection,
it induces a continuous surjective
homomorphism of compact adelic quotients
\[
\overline L_x:
N(F)\backslash N(\mathbb A)
\longrightarrow
F^2\backslash\mathbb A^2.
\]
The pushforward of normalized Haar measure
under $\overline L_x$ is translation-invariant
and has total mass one. Hence it is the
normalized Haar probability measure on
$F^2\backslash\mathbb A^2$.

With the self-dual additive measures,
\[
\operatorname{vol}
(F^2\backslash\mathbb A^2)=1.
\]
Thus the pushforward measure agrees with
the quotient measure induced by $d\lambda$.

It follows that, for fixed $x\in V_E(F)$,
the integral over $[N]$ of the
$\lambda$-summation is
\[
\begin{aligned}
&\int_{[N]}
\sum_{\lambda\in F^2}
\Psi\bigl(
(0,x,\lambda)\cdot n(X)m
\bigr)\,dn(X)\\
&\qquad=
\int_{F^2\backslash\mathbb A^2}
\sum_{\lambda\in F^2}
\Psi\bigl(
(0,x,\lambda+t)\cdot m
\bigr)\,dt\\
&\qquad=
\int_{\mathbb A^2}
\Psi\bigl(
(0,x,\lambda)\cdot m
\bigr)\,d\lambda.
\end{aligned}
\]
The last equality is the usual unfolding
of the additive quotient.

For fixed $m$, all interchanges are justified
by Tonelli's theorem. Indeed, $g\longmapsto \sum_{v\in V'(F)}|\Psi(vg)|$
is finite and locally bounded on
$G(\mathbb A)$, as recalled in
Section~\ref{subsec:theta}.
Moreover, $nm$ ranges over a compact set
as $n$ ranges over a compact fundamental
domain for $[N]$. Hence
\[
\int_{[N]}
\sum_{x\in V_E(F)}
\sum_{\lambda\in F^2}
\bigl|
\Psi((0,x,\lambda)nm)
\bigr|\,dn
<\infty.
\]
Since
$\Phi_\Psi\in\mathcal S(V(\mathbb A))$,
the cube theta series
$\Theta_{\mathrm{cube},E}(\Phi_\Psi;m)$
also converges absolutely.

We now compute the Jacobian of the Levi
action.
Write $m$ as in the statement. On the
boundary $\ell=0$, the Levi action has
the form
\[
(0,x,\lambda)\cdot m
=
\bigl(
0,x\cdot m,\det(b)\,{}^tk\lambda
\bigr).
\]
Thus
\[
\begin{aligned}
&\int_{\mathbb A^2}
\Psi\bigl(
(0,x,\lambda)\cdot m
\bigr)\,d\lambda\\
&\quad=
\int_{\mathbb A^2}
\Psi\bigl(
0,x\cdot m,\det(b)\,{}^tk\lambda
\bigr)\,d\lambda.
\end{aligned}
\]
The linear transformation $\lambda\longmapsto \det(b)\,{}^tk\lambda$
on the two-dimensional adelic vector
space has modulus
\[
|\det b|_{\mathbb A}^2
|\det k|_{\mathbb A}.
\]
Changing variables therefore gives
\[
\begin{aligned}
&\int_{\mathbb A^2}
\Psi\bigl(
0,x\cdot m,\det(b)\,{}^tk\lambda
\bigr)\,d\lambda\\
&\qquad=
|\det b|_{\mathbb A}^{-2}
|\det k|_{\mathbb A}^{-1}
\Phi_\Psi(x\cdot m).
\end{aligned}
\]
Summing over $x\in V_E(F)$ yields
\[
C_0(\Psi;m)
=
|\det b|_{\mathbb A}^{-2}
|\det k|_{\mathbb A}^{-1}
\Theta_{\mathrm{cube},E}(\Phi_\Psi;m).
\]

The scalar factor is well defined on
$\widetilde M$. Indeed, for an element $(zI_4,z^{-2}I_2)\in Z'_\rho$,
we have $\det b=z^2$ and $\det k=z^{-4}$,
and therefore
\[
|z^2|_{\mathbb A}^{-2}
|z^{-4}|_{\mathbb A}^{-1}=1.
\]

Notice that the $E$-restriction is
preserved throughout this calculation.
Indeed, on $\ell=0$ the discriminant
satisfies $\Delta(0,x,\lambda)=\Delta(x)$,
independently of $\lambda$.
Thus the $\ell=0$ stratum of
$V'_E(F)$ is exactly $V_E(F)\times F^2$
in the $(x,\lambda)$ coordinates.

Finally, recall that
\[
e_c
=
\bigl(
\operatorname{diag}(I_2,c^{-1}I_2),
cI_2
\bigr).
\]
For this element, $\det b=c^{-2}$ and $\det k=c^2$.
Moreover, $e_c$ acts trivially on the
cube coordinates $x$. Consequently,
\[
\begin{aligned}
C_0(\Psi;e_cm)
&=
|c^{-2}|_{\mathbb A}^{-2}
|c^2|_{\mathbb A}^{-1}
C_0(\Psi;m)\\
&=
|c|_{\mathbb A}^2C_0(\Psi;m).
\end{aligned}
\]
Since $a_u=e_{c(u)}$ and
$|c(u)|_{\mathbb A}=u$, we obtain $C_0(\Psi;a_um) = u^2C_0(\Psi;m)$.
This completes the proof.
\end{proof}

The preceding lemma identifies the boundary
contribution by unfolding the additive
$\lambda$-summation along $N$, using the
surjectivity of $L_x$, rather than by Poisson
summation. The essential observation is that the
discriminant is independent of $\lambda$
on the boundary $\ell=0$, so that the
$E$-restriction is preserved throughout
the unfolding. For $\ell\ne0$, the
discriminant generally depends on $\lambda$,
and the same argument does not apply.

The cube theta series
$\Theta_{\mathrm{cube},E}$ is the discriminant-$E$ part
of the theta series over $V^0(F)$
considered by Suzuki--Wakatsuki \cite[\S\S2.2, 3.2]{SW1}.
Its contribution to the global zeta integral
will be identified in the next subsection.

The contribution $D$ is precisely the
$N$-constant term associated with the
$\ell\ne0$ summation studied in
Sections~\ref{subsec:R-cutoff} and~\ref{subsec:absolute-convergence}
(see Corollary~\ref{cor:R-cutoff}).
The estimates established there will
control its contribution to the truncated
integral.

\subsection{The truncated pairing and the first expansion}
\label{subsec:first-expansion}

We now apply Arthur's truncation to the Eisenstein series in the first expansion. Throughout this
subsection, $E/F$ is a quadratic field extension.

Let $T$ be a sufficiently regular truncation parameter for $\mathrm{GL}_4$. Define
\[
I_T(\Psi,f;s,s_0)
=
\int_{[G]}
|\chi_\Delta(g)|_{\mathbb A}^{s+s_0}
\Theta_E(\Psi)(g)
\Lambda^T E_P(f,s_0)(g)\,dg.
\]
The truncation acts only in the $\mathrm{GL}_4$ variable.

\begin{lemma}\label{lem:truncated-pairing}
Let $\mathcal K$ be a compact subset of the domain of holomorphy of $E_P(f,s_0)$. For
$\Re s$ sufficiently large, depending on $\mathcal K$, the integral $I_T(\Psi,f;s,s_0)$
converges absolutely and locally uniformly for $s_0\in\mathcal K$. In particular, it is holomorphic in
$s_0$ on a neighbourhood of $\mathcal K$.
\end{lemma}

\begin{proof}
Put
\[
\sigma_-=
\inf_{s_0\in\mathcal K}
\Re(s+s_0),
\qquad
\sigma_+=
\sup_{s_0\in\mathcal K}
\Re(s+s_0).
\]
We take $\Re s$ sufficiently large that $\sigma_->3$.

By Lemma~\ref{lem:outer-cusp}, applied with $B>16d\sigma_+$
to the majorant $|\Psi(\cdot\,\kappa)|\ll\Psi_B^+$ of
Lemma~\ref{lem:R-cutoff}, the integral of
\[
|\chi_\Delta(g)|_{\mathbb A}^{
\Re(s+s_0)}
|\Theta_E(\Psi)(g)|
\]
over the outer $\mathrm{GL}_2$ cusp and the positive central direction is bounded by
\[
C_{\mathcal K} H(g_4)^{4d\sigma_+},
\]
where $H(g_4)=H(g_{4,\infty})$ is the height \eqref{eq:height-H}.
The bound is uniform for $s_0\in\mathcal K$.

In applying Lemma~\ref{lem:outer-cusp}, we use
the reduced representatives
$g_4\in\mathrm{GL}_4(\mathbb A)^1$
appearing in its statement; a scalar in the
$\mathrm{GL}_4$-coordinate can be moved to these representatives
and the central parameter by means of $Z'_\rho$.
The positive central parameter $r$ and the
outer $\mathrm{GL}_2$ Siegel variables have
already been integrated out in that lemma.
Thus the remaining integration is over
the reduced quotient $[\mathrm{GL}_4]^1$, with the
measure of the decomposition in \S\ref{subsec:scaling-lattice}.
On reduced Siegel sets, the height $H(g_4)$
is bounded by a fixed power of a standard automorphic height
$\operatorname{ht}(g_4)$ controlling $\|g_{4,\infty}\|$ and
$\|g_{4,\infty}^{-1}\|$: indeed
$\|(\bigwedge^2g_{4,\infty})^{-1}\|=\|\bigwedge^2(g_{4,\infty}^{-1})\|
\ll\|g_{4,\infty}^{-1}\|^2$. The rapid-decay assertion (S2)
is understood with respect to such a height.

By (S2) of Section~\ref{subsec:standard-inputs}, for every $A>0$,
\[
\sup_{s_0\in\mathcal K}
|\Lambda^T E_P(f,s_0)(g_4)|
\ll_{\mathcal K,A,T}
(1+\operatorname{ht}(g_4))^{-A}
\]
on Siegel sets.

Choosing $A$ sufficiently large absorbs the polynomial factor
$H(g_4)^{4d\sigma_+}$ and yields an integrable majorant with respect to the
Siegel-set measure, uniformly for $s_0\in\mathcal K$. Hence
the pairing converges absolutely and locally uniformly.

Since the truncated Eisenstein series is holomorphic in $s_0$ on a neighbourhood of $\mathcal K$,
Morera's theorem proves the asserted holomorphy.
\end{proof}

We next make the truncation cutoff explicit.

For $m=\operatorname{diag}(a,b)\in M(\mathbb A)$,
our convention for the parabolic height is
\[
H_P(m)
=
\frac12
\bigl(
\log|\det a|_{\mathbb A},
\log|\det a|_{\mathbb A},
\log|\det b|_{\mathbb A},
\log|\det b|_{\mathbb A}
\bigr).
\]
The fundamental weight of $P=P_{2,2}$ is
\[
\varpi_P
=
\frac12(e_1+e_2-e_3-e_4).
\]
Consequently,
\[
\begin{aligned}
\langle\varpi_P,H_P(m)\rangle
&=
\frac12
\left(
\log|\det a|_{\mathbb A}
-
\log|\det b|_{\mathbb A}
\right)\\
&=\log u(m).
\end{aligned}
\]
For $g=nmk$ in the parabolic Iwasawa decomposition, define $H_P(g)=H_P(m)$ and
$u(g)=u(m)$; these are well defined because $|\det a|_{\mathbb A}$ and
$|\det b|_{\mathbb A}$ are trivial on the Levi projection of $P(\mathbb A)\cap K_4$
(\S\ref{subsec:zeta-integral}). The preceding identity then gives
\[
\langle\varpi_P,H_P(g)\rangle
=
\log u(g).
\]

Since $P$ is maximal,
\[
\widehat\tau_P(H-T)=1
\quad\Longleftrightarrow\quad
\langle\varpi_P,H-T\rangle>0.
\]
Thus
\begin{equation}\label{eq:truncation-cutoff}
\widehat\tau_P\bigl(H_P(g)-T\bigr)
=
\mathbf 1_{\{u(g)>U\}},
\qquad
U=U(T)= e^{\langle\varpi_P,T\rangle}.
\end{equation}
In particular, the truncation threshold depends only on the positive central coordinate $u$ and not on
the determinant-one Levi variables. We take $T$ sufficiently regular that $U\geq1$.

We now define the cube integral that occurs in the boundary contribution.

For $\kappa\in K$, put
\[
\Psi_\kappa(v)=\Psi(v\kappa),
\qquad
\Phi_{\Psi_\kappa}(x)
=
\int_{\mathbb A^2}
\Psi_\kappa(0,x,\lambda)\,d\lambda.
\]
Since
\[
\Theta_E(\Psi)(nm\kappa)
=
\Theta_E(\Psi_\kappa)(nm),
\]
we have $C_0(\Psi;m\kappa) = C_0(\Psi_\kappa;m)$.

For a section $f$ (later $f_{s_0}$ or $M(s_0-1)f_{s_0}$), let $\Phi_f$
denote its restriction to $\widetilde M(\mathbb A)^1K$. For each fixed
$\kappa\in K$, the function
\[
m\longmapsto\Phi_f(m\kappa)=(\rho(\kappa)f)(m)
\]
is the corresponding cuspidal inducing datum on the Levi.
For a complex parameter $\mathfrak{s}\in\mathbb C$, define
\begin{equation}\label{eq:boundary-cube-integral}
\mathcal Z(\Phi_f;\mathfrak{s})
=
\int_K
\int_{[\widetilde M]^1}
|\chi_\Delta(m)|_{\mathbb A}^{\mathfrak{s}}
\Theta_{\mathrm{cube},E}
(\Phi_{\Psi_\kappa};m)
\Phi_f(m\kappa)\,dm^1\,d\kappa.
\end{equation}
We initially take $\Re \mathfrak{s}$ sufficiently large.

\begin{lemma}\label{lem:cube-convergence}
For $\Re \mathfrak{s}$ sufficiently large, the integral
\eqref{eq:boundary-cube-integral} converges absolutely, with integrand
majorized uniformly in $\kappa\in K$, and locally uniformly in $\mathfrak{s}$. In particular,
$\mathcal Z(\Phi_f;\mathfrak{s})$ is holomorphic in this right half-plane.
\end{lemma}

\begin{proof}
As established in
\S\ref{subsec:group}, the central torus
$A$ acts trivially on the cube space.
The effective action of
$\widetilde M$ on
$V=M_2\oplus M_2$ therefore
factors through $H=\widetilde M/A$
and identifies with the split-quaternion
prehomogeneous representation of
\cite{SW1}. In the coordinates
$h(g_1,g_2)=\operatorname{diag}({}^tg_1^{-1},g_2)$, its discriminant
character satisfies
\[
\chi_\Delta(h(g_1,g_2),g_3)
=\bigl(\det(g_1)^{-1}\det(g_2)\det(g_3)\bigr)^2
=\omega(g_1,g_2,g_3).
\]
Thus the complex parameter $\mathfrak{s}$ agrees with that of \cite{SW1}.

The integrand is invariant under $A(\mathbb A)^1$: this torus acts
trivially on the cube, $\chi_\Delta$ is trivial on $A$, and the
inducing automorphic data have trivial central character. The
identifications
\[
\widetilde M(\mathbb A)^1/A(\mathbb A)^1\simeq H(\mathbb A),
\qquad \widetilde M(F)/A(F)\simeq H(F)
\]
follow from the exactness established in Lemma~\ref{lem:levi-measure}
and $u(e_c)=|c|_{\mathbb A}$. In particular, integration over the
central fibre $A(F)\backslash A(\mathbb A)^1$ contributes its finite
volume $c_F$; the quotient-measure normalization is recalled again
before Corollary~\ref{cor:residue-SW}.

For fixed $\kappa\in K$, the inducing automorphic data in
$\Phi_f(m\kappa)$, after the inverse-transpose identification in
Section~\ref{subsec:zeta-integral}, belong to the cuspidal
realization of $\pi\boxtimes\pi$ and are independent of
the outer $\mathrm{GL}_2$-factor. The family of these data has
uniformly bounded cuspidal seminorms for $\kappa\in K$.
By compactness of $K$ and continuity of
$\kappa\mapsto\Phi_{\Psi_\kappa}$ in the Schwartz topology, there
is a single nonnegative Schwartz majorant $\Phi^+$ such that
\[
|\Phi_{\Psi_\kappa}(x)|\leq\Phi^+(x)
\qquad (x\in V(\mathbb A),\ \kappa\in K).
\]
Since $E/F$ is a field, $\Theta_{\mathrm{cube},E}$ involves only
nonsplit regular orbits. Its absolute-value majorant is bounded by
the nonsplit-orbit majorant used in
\cite[(4.11) and Lemma~4.13]{SW1}, where absolute values are taken
inside the sum over rational points. Set
\[
J(\sigma)
=
\int_{[\widetilde M]^1}
|\chi_\Delta(m)|_{\mathbb A}^{\sigma}
\sum_{x\in V_E(F)}
\Phi^+(x\cdot m)\,dm^1.
\]
That estimate, together with the quotient identification above,
gives $J(\sigma)<\infty$ for $\sigma$ sufficiently large; the
integration over the central fibre contributes the finite factor $c_F$.

The uniform boundedness of the cuspidal inducing data
established above gives a constant $C_f$ such that
\[
|\Phi_f(m\kappa)|\leq C_f
\qquad(m\in\widetilde M(\mathbb A)^1,\ \kappa\in K).
\]
As $|\Theta_{\mathrm{cube},E}(\Phi_{\Psi_\kappa};m)|\leq
\sum_{x\in V_E(F)}\Phi^+(x\cdot m)$, Tonelli's theorem gives, for
$\sigma=\Re \mathfrak{s}$,
\[
\begin{aligned}
&\int_K\int_{[\widetilde M]^1}
|\chi_\Delta(m)|_{\mathbb A}^{\sigma}
\left|
\Theta_{\mathrm{cube},E}
(\Phi_{\Psi_\kappa};m)\Phi_f(m\kappa)
\right|\,dm^1\,d\kappa\\
&\qquad\leq
C_f\operatorname{vol}(K)J(\sigma)
<\infty.
\end{aligned}
\]
This proves absolute convergence of
the integral over
$K\times[\widetilde M]^1$.
The majorant is independent of
$\kappa$, as required.

The convergence is locally uniform in $\mathfrak{s}$: for real $a<b$ in the
region of absolute convergence and $a\leq\Re \mathfrak{s}\leq b$,
\[
|\chi_\Delta(m)|^{\Re \mathfrak{s}}
\leq |\chi_\Delta(m)|^a+|\chi_\Delta(m)|^b,
\]
so the integrand is dominated, uniformly on the strip, by
\[
C_f\bigl(|\chi_\Delta(m)|_{\mathbb A}^{a}+|\chi_\Delta(m)|_{\mathbb A}^{b}\bigr)
\sum_{x\in V_E(F)}\Phi^+(x\cdot m).
\]
This pointwise majorant is integrable over $K\times[\widetilde M]^1$,
since its integral is bounded by
\[
C_f\operatorname{vol}(K)\bigl(J(a)+J(b)\bigr)<\infty.
\]
Thus the convergence is locally uniform in $\mathfrak{s}$, and holomorphy
follows by Morera's theorem.
\end{proof}

On $[\widetilde M]^1$ we have
\[
u(m)=1,
\qquad
\delta_P(m)=1,
\qquad
|\det a|_{\mathbb A}
=
|\det b|_{\mathbb A}.
\]
Since
\[
|\chi_\Delta(m)|_{\mathbb A}
=
|\det a\det b\det k|_{\mathbb A}^2,
\]
it follows that
\[
|\chi_\Delta(m)|_{\mathbb A}
=
|\det b|_{\mathbb A}^4 |\det k|_{\mathbb A}^2.
\]
Hence
\begin{equation}\label{eq:boundary-character-shift}
\delta_P(m)^{-1}
|\det b|_{\mathbb A}^{-2} |\det k|_{\mathbb A}^{-1}
=
|\chi_\Delta(m)|_{\mathbb A}^{-1/2}.
\end{equation}

By Lemma~\ref{lem:N-constant-term}, the boundary contribution at $u=1$ is therefore
\[
|\chi_\Delta(m)|_{\mathbb A}^{ s+s_0-\frac12}
\Theta_{\mathrm{cube},E}
(\Phi_{\Psi_\kappa};m)
\Phi_f(m\kappa).
\]
This explains the shift
\[
\mathfrak{s}=s+s_0-\frac12.
\]

For the flat section $f_{s_0}$, we have $u=1$ on
$\widetilde M(\mathbb A)^1K$, so its restriction $\Phi_{f_{s_0}}$ there
equals that of $f^\circ$ and is independent of $s_0$. Consequently, the
corresponding cube integral is $\mathcal Z(\Phi_{f_{s_0}};\mathfrak{s})$
with $\mathfrak{s}=s+s_0-\frac12$, which depends on $s_0$ only through
$\mathfrak{s}$.
For any section $F$ we also write
$\mathcal Z(F;\mathfrak{s})=\mathcal Z(\Phi_F;\mathfrak{s})$; thus
$\mathcal Z(f_{s_0};\mathfrak{s})=\mathcal Z(\Phi_{f_{s_0}};\mathfrak{s})$.

We next establish the holomorphy of the large-$u$ tail.

For a section $F_{s_0}$, write
\[
\begin{aligned}
\mathcal J(D,F_{s_0};u)
={}&
\int_K\int_{[\widetilde M]^1}
|\chi_\Delta(m)|_{\mathbb A}^{s+s_0} D(\Psi;a_um\kappa)\\
&\qquad\times
\delta_P(a_um)^{-1}
F_{s_0}(a_um\kappa)\,dm^1\,d\kappa.
\end{aligned}
\]
For the convergence estimates below, we bound $|\mathcal J(D,F_{s_0};u)|$ by the corresponding
absolute-value integral, in which $D(\Psi;a_um\kappa)$ and $F_{s_0}(a_um\kappa)$ are replaced by
their absolute values and $|\chi_\Delta(m)|_{\mathbb A}^{s+s_0}$ by
$|\chi_\Delta(m)|_{\mathbb A}^{\Re(s+s_0)}$.

\begin{lemma}\label{lem:holomorphic-tail}
Let $\mathcal K$ be a compact subset of the domain on which
$M(s_0-1)f_{s_0}$ is holomorphic. For $\Re s$ sufficiently large,
depending on $\mathcal K$, the integrals
\[
\int_{u>U}\mathcal J(D,f_{s_0};u)\,d^\times u,
\qquad
\int_{u>U}\mathcal J\bigl(D,M(s_0-1)f_{s_0};u\bigr)\,d^\times u
\]
converge absolutely and locally uniformly for $s_0\in\mathcal K$.
They, and hence the large-$u$ tail
$\int_{u>U}\mathcal J\bigl(D,f_{s_0}+M(s_0-1)f_{s_0};u\bigr)\,d^\times u$,
are holomorphic in $s_0$ on a neighbourhood of $\mathcal K$.
\end{lemma}

\begin{proof}
We use the large-$u$ estimates of Section~\ref{subsec:absolute-convergence}.

Choose $\sigma_1>3$ and $B_+>16d\sigma_1$,
and take $\Re s$ sufficiently large.

Put $\sigma=\Re(s+s_0)$.
Applying the large-$u$ majorant
\eqref{eq:large-u-majorant} of
Proposition~\ref{prop:absolute-convergence}
to the relevant cuspidal inducing datum,
we obtain
\[
J(u)
\ll
u^{B_+/d+2\sigma_1-2\sigma} K_{\sigma_1},
\qquad u\geq1,
\]
where the constant depends on that datum.

For the original section, the additional $u$-weight is
\[
u^{2\Re s_0-4}.
\]
Consequently, using the absolute-value integral,
\[
\begin{aligned}
|\mathcal J(D,f_{s_0};u)|
&\ll
u^{2\Re s_0-4} u^{B_+/d+2\sigma_1-2\sigma} K_{\sigma_1}\\
&=
u^{2\sigma_1+B_+/d -2\Re s-4} K_{\sigma_1}.
\end{aligned}
\]
The exponent is independent of $\Re s_0$.

For the intertwined section $M(s_0-1)f_{s_0}\in I(1-s_0)$,
the corresponding $u$-weight is
\[
u^{-2\Re s_0}.
\]
Thus
\[
\begin{aligned}
&|\mathcal J(
D,M(s_0-1)f_{s_0};u)|\\
&\qquad\ll
u^{-2\Re s_0} u^{B_+/d+2\sigma_1-2\sigma} K_{\sigma_1}(s_0)\\
&\qquad=
u^{B_+/d+2\sigma_1 -2\Re s -4\Re s_0} K_{\sigma_1}(s_0),
\end{aligned}
\]
where $K_{\sigma_1}(s_0)$ is the constant of the slice estimate for the Levi restriction of
$M(s_0-1)f_{s_0}$.

The final assertion of Proposition~\ref{prop:absolute-convergence} applies to the Levi restriction of
$M(s_0-1)f_{s_0}$. Indeed, this restriction takes values in the cuspidal inducing representation, and
its relevant $K$-finite seminorms are locally bounded in $s_0$ away from the poles of the intertwining
operator.

Since $\mathcal K$ is compact and lies in the holomorphy domain of the intertwining operator, we may
enlarge it slightly to a compact neighbourhood $\mathcal K'$ on which these seminorms remain uniformly
bounded; thus $K_{\sigma_1}(s_0)$ is bounded uniformly for $s_0\in\mathcal K'$. We then choose
$\Re s$ so large that both exponents above are strictly negative, uniformly on $\mathcal K'$.

The resulting bounds are integrable against $d^\times u$ on $u>U$, since $U\geq1$. Thus both integrals
converge absolutely and locally uniformly in $s_0$ on an open neighbourhood of $\mathcal K$. Since the
integrand is holomorphic in $s_0$, Morera's theorem gives the assertion.
\end{proof}

We can now state the first expansion.

\begin{theorem}\label{thm:first-expansion}
Assume that $E/F$ is a quadratic field extension.

For $\Re s_0>3$ and $\Re s$ sufficiently large, as in
Proposition~\ref{prop:absolute-convergence}, the zeta integral $Z'_E(\Psi,f;s,s_0)$ satisfies
\begin{equation}\label{eq:first-expansion}
\begin{aligned}
Z'_E(\Psi,f;s,s_0)
={}&I_T(\Psi,f;s,s_0)\\
&+\int_{u>U}
\mathcal J\bigl(
D,f_{s_0}+M(s_0-1)f_{s_0};u
\bigr)\,d^\times u\\
&-\frac{U^{2s_0-2}}{2s_0-2}
\mathcal Z\left(
f_{s_0};s+s_0-\frac12
\right)\\
&+\frac{U^{2-2s_0}}{2s_0-2}
\mathcal Z\left(
M(s_0-1)f_{s_0}; s+s_0-\frac12
\right).
\end{aligned}
\end{equation}

For every compact set $\mathcal K\subset\mathbb C$, there exists $c(\mathcal K)$ such that, for
$\Re s\geq c(\mathcal K)$, the right-hand side of
\eqref{eq:first-expansion} is meromorphic
on a neighbourhood of $\mathcal K$ and gives a continuation of $Z'_E(\Psi,f;s,\cdot)$ there.

If $\mathcal K\subset \{\Re s_0\geq1\}$,
its only possible poles in $\mathcal K$ are $s_0=1$, which is at most simple, and $s_0=3/2$.

Moreover, for $\Re s$ sufficiently large,
\begin{equation}\label{eq:first-expansion-residue}
\operatorname*{Res}_{s_0=1}
Z'_E(\Psi,f;s,s_0)
=
-\frac12
\left[
\mathcal Z\left(f;s+\frac12\right)
-
\mathcal Z\left(M(0)f;s+\frac12\right)
\right].
\end{equation}
\end{theorem}

\begin{proof}
We first establish the expansion in the region
$\Re s_0>3$, with $\Re s$
sufficiently large, where all the integrals below
converge absolutely.

By (S1), the Eisenstein series $E_P(f,s_0)$
converges absolutely in this region. Applying
Lemma~\ref{lem:rank-one-truncation}, we obtain
\[
\begin{aligned}
\Lambda^T E_P(f,s_0)(g)
={}&
\sum_{\gamma\in P(F)\backslash\mathrm{GL}_4(F)}
f_{s_0}(\gamma g)
\bigl(1-\widehat\tau_P(H_P(\gamma g)-T)\bigr)\\
&-
\sum_{\gamma\in P(F)\backslash\mathrm{GL}_4(F)}
(M(s_0-1)f_{s_0})(\gamma g)
\widehat\tau_P(H_P(\gamma g)-T).
\end{aligned}
\]
No disjointness of the supports associated with
distinct rational cosets is required.

We justify unfolding the two sums separately.
For the absolute values in both sums, note that by
compactness of $K$ and continuity of
$\kappa\mapsto\Psi_\kappa$ in the Schwartz topology,
there exists a single nonnegative Schwartz function
$\Psi^+$ such that
\[
|\Psi_\kappa(w)|\leq\Psi^+(w)
\qquad
(w\in V'(\mathbb A),\ \kappa\in K).
\]
The computation of
Lemma~\ref{lem:N-constant-term}
then applies to $\Psi^+$, and
Lemma~\ref{lem:cube-convergence}
controls the resulting boundary integral.

For the first sum, the contribution of the
$\ell\ne0$ part $D$ is absolutely integrable by
Proposition~\ref{prop:absolute-convergence},
while the boundary contribution $C_0$ is controlled
by Lemmas~\ref{lem:N-constant-term} and
\ref{lem:cube-convergence}, together with
\[
\int_0^U u^{2\Re s_0-2}
\,d^\times u<\infty
\qquad
(\Re s_0>1).
\]
For the second sum, the boundary contribution
is absolutely integrable by
Lemmas~\ref{lem:N-constant-term} and
\ref{lem:cube-convergence}, applied to the
intertwined section, together with
\[
\int_U^\infty u^{2-2\Re s_0}
\,d^\times u<\infty
\qquad
(\Re s_0>1),
\]
and the $\ell\ne0$ contribution is controlled
by Lemma~\ref{lem:holomorphic-tail}.
These estimates apply to the absolute-value
integrals, locally uniformly on compact subsets
of the initial domain. Tonelli's theorem and
Fubini's theorem therefore justify the unfolding
and the subsequent interchange of integrations.

Unfolding the two rational-coset sums replaces $[G]$ by
$G'(F)\backslash G(\mathbb A)$, which we integrate using the
Iwasawa decomposition \eqref{eq:iwasawa-Gprime}. By
\eqref{eq:truncation-cutoff}, the truncation
functions become the characteristic functions
of $u\leq U$ and $u>U$ after unfolding.
Integrating over $[N]$, which produces the
$N$-constant term $C_0+D$, and decomposing the
Levi variable as $a_um$, we therefore obtain
\[
\begin{aligned}
I_T(\Psi,f;s,s_0)
={}&
\int_{0<u\leq U}
\mathcal J(C_0+D,f_{s_0};u)
\,d^\times u\\
&-
\int_{u>U}
\mathcal J(C_0+D,M(s_0-1)f_{s_0};u)
\,d^\times u.
\end{aligned}
\]
On the other hand, the definition of the original
zeta integral gives
\[
Z'_E(\Psi,f;s,s_0)
=
\int_0^\infty
\mathcal J(D,f_{s_0};u)
\,d^\times u.
\]
Subtracting the preceding identities and using
linearity in the section variable, we obtain
\[
\begin{aligned}
Z'_E(\Psi,f;s,s_0)
={}&
I_T(\Psi,f;s,s_0)\\
&+
\int_{u>U}
\mathcal J\bigl(
D,f_{s_0}+M(s_0-1)f_{s_0};u
\bigr)\,d^\times u\\
&-
\int_{0<u\leq U}
\mathcal J(C_0,f_{s_0};u)
\,d^\times u\\
&+
\int_{u>U}
\mathcal J(C_0,M(s_0-1)f_{s_0};u)
\,d^\times u.
\end{aligned}
\]

We next evaluate the two boundary integrals.
By \S\ref{subsec:zeta-integral},
\[
f_{s_0}(a_um)=u^{2s_0}f_{s_0}(m),
\qquad
\delta_P(a_u)^{-1}=u^{-4},
\qquad
|\chi_\Delta(a_u)|_{\mathbb A}=1.
\]
Lemma~\ref{lem:N-constant-term} gives $C_0(\Psi;a_um)=u^2C_0(\Psi;m)$.
These identities remain valid with a right
$K$-variable $\kappa$, since $C_0(\Psi;a_um\kappa) = C_0(\Psi_\kappa;a_um)$.
Consequently, the $u$-dependence of the first
boundary integrand is
\[
u^{2s_0}u^{-4}u^2=u^{2s_0-2}.
\]
For the intertwined section
$M(s_0-1)f_{s_0}\in I(1-s_0)$, the corresponding
dependence is
\[
u^{4-2s_0}u^{-4}u^2=u^{2-2s_0}.
\]

Using \eqref{eq:boundary-character-shift},
the remaining integral over
$K\times[\widetilde M]^1$ is the cube zeta
integral with parameter $\mathfrak{s}=s+s_0-\frac12$,
in the notation $\mathcal Z(F;\mathfrak{s})=\mathcal Z(\Phi_F;\mathfrak{s})$ of
\S\ref{subsec:first-expansion}. In particular, the
restriction of the flat section $f_{s_0}$ is
independent of $s_0$.

Since
\[
\int_0^U u^{2s_0-2}\,d^\times u
=
\frac{U^{2s_0-2}}{2s_0-2},
\qquad
\int_U^\infty u^{2-2s_0}\,d^\times u
=
\frac{U^{2-2s_0}}{2s_0-2},
\]
we obtain
\[
\begin{aligned}
Z'_E(\Psi,f;s,s_0)
={}&
I_T(\Psi,f;s,s_0)\\
&+
\int_{u>U}
\mathcal J\bigl(
D,f_{s_0}+M(s_0-1)f_{s_0};u
\bigr)\,d^\times u\\
&-
\frac{U^{2s_0-2}}{2s_0-2}
\mathcal Z\left(
f_{s_0};s+s_0-\frac12
\right)\\
&+
\frac{U^{2-2s_0}}{2s_0-2}
\mathcal Z\left(
M(s_0-1)f_{s_0};s+s_0-\frac12
\right).
\end{aligned}
\]
This proves \eqref{eq:first-expansion} in the
initial region of absolute convergence.

We now establish its meromorphic continuation
in $s_0$. Let $\mathcal K\subset\mathbb C$ be
compact. By Lemma~\ref{lem:truncated-pairing},
the truncated pairing
$I_T(\Psi,f;s,s_0)$ is holomorphic on compact
subsets of the holomorphy domain of $E_P(f,s_0)$,
provided $\Re s$ is sufficiently
large. Similarly, Lemma~\ref{lem:holomorphic-tail}
shows that the large-$u$ tail is holomorphic
on compact subsets of the holomorphy domain of
$M(s_0-1)f_{s_0}$, under the same condition
on $\Re s$.

To extend these assertions across poles,
let $s_0^*$ be a pole of $E_P(f,s_0)$ or
$M(s_0-1)f_{s_0}$, and choose $k$ at least
as large as the orders of both poles.
On a sufficiently small neighbourhood of
$s_0^*$, the families
\[
(s_0-s_0^*)^k\Lambda^T E_P(f,s_0)
\]
and $(s_0-s_0^*)^kM(s_0-1)f_{s_0}$
extend holomorphically across $s_0^*$.
The first satisfies the locally uniform
rapid-decay estimates used in
Lemma~\ref{lem:truncated-pairing}, while the
Levi restrictions of the second have locally
bounded cuspidal seminorms, as required in
Lemma~\ref{lem:holomorphic-tail}.

Applying the proofs of these lemmas to the
pole-cleared families shows that $(s_0-s_0^*)^kI_T(\Psi,f;s,s_0)$
and
\[
(s_0-s_0^*)^k
\int_{u>U}
\mathcal J\bigl(
D,f_{s_0}+M(s_0-1)f_{s_0};u
\bigr)\,d^\times u
\]
are holomorphic near $s_0^*$, provided
$\Re s$ is sufficiently large.
Thus both contributions extend meromorphically,
with poles of order at most $k$.
On compact sets containing several poles,
the same argument applies after multiplying
by a finite product of the corresponding
pole-clearing factors.

The two cube terms are likewise meromorphic
in $s_0$. Indeed, by
Lemma~\ref{lem:cube-convergence}, the cube
integrals converge absolutely and locally
uniformly when
$\Re(s+s_0-\tfrac12)$ is
sufficiently large. The first cube integral
is then holomorphic in $s_0$, since the Levi
restriction of $f_{s_0}$ is independent of
$s_0$. For the second, the meromorphic
dependence of the Levi restriction of
$M(s_0-1)f_{s_0}$ is handled by the same
pole-clearing argument, using the locally
uniform majorants in
Lemma~\ref{lem:cube-convergence}.
The remaining factors are meromorphic, with
the explicit denominator $2s_0-2$.

It follows that, for
$\Re s\geq c(\mathcal K)$
sufficiently large, the right-hand side of
\eqref{eq:first-expansion} is meromorphic
on a neighbourhood of $\mathcal K$.

To identify this meromorphic family with
the continuation of the original zeta
integral, choose a compact connected set
containing $\mathcal K$ and meeting the
initial half-plane $\Re s_0>3$.
Enlarging this set slightly, the preceding
estimates, with $\Re s$
sufficiently large, show that the right-hand
side of \eqref{eq:first-expansion} is
meromorphic on a connected open neighbourhood
of this set. On its intersection with the
initial half-plane, the formula agrees with
$Z'_E(\Psi,f;s,s_0)$. The identity theorem
therefore identifies the resulting meromorphic
function with the continuation of the original
zeta integral and establishes its uniqueness.

Finally, we determine the possible poles
and the residue at $s_0=1$. By (S1), in the
half-plane $\Re s_0\geq1$ the
only possible pole contributed by the
Eisenstein series or the intertwining operator
is $s_0=3/2$. The explicit boundary
denominator $2s_0-2$ gives an additional
possible simple pole at $s_0=1$.

Choose a sufficiently small closed disc
$\mathcal K$ around $s_0=1$. For
$\Re s$ sufficiently large,
both the truncated pairing and the
large-$u$ tail are holomorphic on a
neighbourhood of $\mathcal K$.
Thus the residue at $s_0=1$ comes entirely
from the two boundary terms.

Since
\[
\lim_{s_0\to1}U^{2s_0-2}
=
\lim_{s_0\to1}U^{2-2s_0}
=1,
\qquad
\operatorname*{Res}_{s_0=1}
\frac{1}{2s_0-2}
=\frac12,
\]
and the cube integrals are holomorphic
at $s_0=1$, we conclude that
\[
\operatorname*{Res}_{s_0=1}
Z'_E(\Psi,f;s,s_0)
=
-\frac12
\left[
\mathcal Z\left(f;s+\frac12\right)
-
\mathcal Z\left(M(0)f;s+\frac12\right)
\right],
\]
which is \eqref{eq:first-expansion-residue}.
This completes the proof.
\end{proof}

\subsection{The central intertwining operator and the residue}
\label{subsec:central-intertwining}

We retain the notation of the preceding subsections.
Let $P$ denote the parabolic subgroup whose projection
to the $\mathrm{GL}_4$-factor of
$G=(\mathrm{GL}_4\times\mathrm{GL}_2)/Z'_\rho$
is the standard parabolic of type $(2,2)$.
In the coordinates of \S\ref{subsec:zeta-integral}, its relevant
Levi factor is identified with $M\simeq\mathrm{GL}_2\times\mathrm{GL}_2$.
The inducing representation is realized as $\sigma=\pi\boxtimes\pi$,
where $\pi$ is a cuspidal automorphic
representation of $\mathrm{PGL}_2(\mathbb A)$.

Recall that the Levi coordinates of \S\ref{subsec:zeta-integral} are
\[
h(g_1,g_2)
=
\operatorname{diag}({}^tg_1^{-1},g_2).
\]
Since ${}^ta^{-1} = (\det a)^{-1}JaJ^{-1}$,
the inverse-transpose action preserves the
automorphic space $\mathcal A_\pi$ itself: since
$\omega_\pi=1$ and $J\in\mathrm{GL}_2(F)$, one has
$\phi({}^ta^{-1})=(\rho(J^{-1})\phi)(a)$ for
$\phi\in\mathcal A_\pi$ (\S\ref{subsec:zeta-integral}).
Thus the two inducing factors are realized on the same
automorphic space.

Let $I(\sigma,z)$ denote the normalized parabolically
induced representation, and let
\[
M(z):I(\sigma,z)\longrightarrow I(w\sigma,-z)
\]
be the standard (unnormalized) intertwining operator,
defined initially in its domain of convergence by
the intertwining integral of \cite[\S~II.1.6]{MW}
and extended meromorphically as in
\cite[Theorem~IV.1.8]{MW}.
In the notation of (S1), we write
\[
M(z)=r(z)N(z),
\]
where $r(z)$ is the scalar normalizing factor and
$N(z)$ is the normalized intertwining operator.

We normalize the Haar measure on $N(\mathbb A)$ by
\[
\operatorname{vol}
\bigl(N(F)\backslash N(\mathbb A)\bigr)=1,
\]
and use the same global measure conventions in (S1).
With this convention, $M(z)$ is the operator
appearing in the Eisenstein constant-term formula
in (S1) and in the functional equation of
Eisenstein series
\cite[Theorem~IV.1.10(a)]{MW}.

We henceforth write $z=s'=s_0-1$
for the normalized inducing parameter
introduced in
\S\ref{subsec:zeta-integral}.
Thus $s_0=1$
corresponds to the central point $z=0$ of the
normalized induced representation. A section
$f_z\in I(\sigma,z)$ is called standard, or flat,
if its restriction to the maximal compact subgroup
$K$ is independent of $z$. The compact-picture
realization thereby identifies the underlying
spaces of $I(\sigma,z)$ as $z$ varies.

Recall that $M(z)$ takes values in $I(w\sigma,-z)$.
The Weyl element $w$ interchanges the two
$\mathrm{GL}_2$ blocks $a$ and $b$ of
$h=\operatorname{diag}(a,b)$; in the coordinates
$h(g_1,g_2)=\operatorname{diag}({}^tg_1^{-1},g_2)$ this is
$(g_1,g_2)\mapsto({}^tg_2^{-1},{}^tg_1^{-1})$. Since the
inverse-transpose automorphism preserves $\mathcal A_\pi$, as
shown above, on $\sigma=\pi\boxtimes\pi$ this amounts to
interchanging the two identical tensor factors. Explicitly,
$w\sigma$ is realized on the functions $m\mapsto F(w^{-1}mw)$,
and the identification $w\sigma=\sigma$ is
\[
F(a,b)\longmapsto F(b,a)
\]
in the block coordinates $m=\operatorname{diag}(a,b)$; in the
coordinates $(g_1,g_2)$ it reads, on decomposable vectors,
\[
\phi_1\otimes\phi_2\longmapsto
\rho(J^{-1})\phi_2\otimes\rho(J^{-1})\phi_1,
\]
which is an involution since $\rho(J^{-2})=\rho(-I_2)=1$.
This is the identification implicit in the functional equation
of the Eisenstein series in (S1), where $E_P$ on $w\sigma$ is
formed from the functions $m\mapsto F(w^{-1}mw)$.

At the central point, the intertwining operator therefore defines an endomorphism
\[
M(0):I(\sigma,0)\longrightarrow I(\sigma,0).
\]
Since $\mathcal A_{w\sigma}=\mathcal A_\sigma$ as spaces of functions on
$M(F)\backslash M(\mathbb A)$, the spaces $I(w\sigma,-z)$ and
$I(\sigma,-z)$ coincide as spaces of functions on
$\mathrm{GL}_4(\mathbb A)$. Thus $M(0)$ maps $I(\sigma,0)$ into itself,
and the identity $M(0)=-\mathrm{Id}$ established below is an identity of
functions; no further map is applied. The swap formula $F(a,b)\mapsto F(b,a)$ above records how a
vector of $\sigma$ is matched with its realization in $w\sigma$.
The sign determined below is understood with
respect to this fixed identification.

We first isolate the local analytic input
needed for the proof of
Proposition~\ref{prop:intertwining-sign}.

\begin{lemma}\label{lem:local-jacquet}
Let $v$ be a place of $F$, and let
$h_{v,z}\in I_v(\sigma_v,z)$ be a standard
local section. Let
\[
N(F_v)
=
\left\{
n(X)=
\begin{pmatrix}
I_2&X\\
0&I_2
\end{pmatrix}
:X\in M_2(F_v)
\right\},
\qquad
w=\operatorname{antidiag}(I_2,I_2).
\]
Choose a generic character of the standard
maximal unipotent subgroup $N_0=N_M\ltimes N$
whose restriction to $N(F_v)$ is
\[
\psi_v(n(X))=\psi_v(X_{21}),
\]
and whose restrictions to the two simple-root
subgroups of $N_M(F_v)$ are given by the same
additive character $\psi_v$.

For $g$ in the $\mathrm{GL}_4$-factor, put
\[
h_{v,z}^{\,g}(m)
=
\delta_P(m)^{-1/2}
u_v(m)^{-2z}h_{v,z}(mg),
\]
where
\[
u_v(m)
=
\left|\frac{\det a}{\det b}\right|_v^{1/2},
\qquad
m=\operatorname{diag}(a,b).
\]
Fix a nonzero Whittaker functional
$\ell_{\sigma_v}$ on $\sigma_v$ with respect
to the restriction of the chosen character
to $N_M(F_v)$, independently of $z$, and set
\[
\ell_{\sigma_v}(h_{v,z};g)
=
\ell_{\sigma_v}(h_{v,z}^{\,g}).
\]
Define
\[
\lambda_{v,z}(h_v)
=
\int_{N(F_v)}
\ell_{\sigma_v}(h_{v,z};w^{-1}n)
\psi_v(n)^{-1}\,dn
\]
in its initial region of convergence.

Then the following assertions hold.

\begin{enumerate}
\item
For every standard local section $h_v$,
the function
\[
z\longmapsto\lambda_{v,z}(h_v)
\]
extends to an entire function.

\item
The functional $\lambda_{v,0}$ is nonzero
on the $K_v$-finite vectors of
$I_v(\sigma_v,0)$.

\item
Suppose that $v$ is nonarchimedean and
that $\pi_v$, $\psi_v$, and the inducing
data are unramified. Let
$\{\alpha_1,\alpha_2\}$ be the Satake
parameters of $\pi_v$. Normalize
$\psi_v$ to have conductor $\mathfrak o_v$,
choose measures such that
\[
\operatorname{vol}(N(\mathfrak o_v))
=
\operatorname{vol}(N_M(\mathfrak o_v))
=1,
\]
and normalize the spherical section
$h_v^\circ$ by
\[
\ell_{\sigma_v}(h_v^\circ;1)=1.
\]
Then
\[
\lambda_{v,z}(h_v^\circ)
=
L_v(1+2z,\pi_v\times\widetilde\pi_v)^{-1}.
\]
\end{enumerate}
\end{lemma}
The proof is given in Appendix~\ref{app:intertwining}.

We now determine the action of the global standard
intertwining operator at the central point.

\begin{proposition}\label{prop:intertwining-sign}
The standard intertwining operator satisfies
\[
M(0)=-\operatorname{Id}
\]
on the $K$-finite vectors of $I(\sigma,0)$.
The identity extends by continuity to the
smooth induced representation.
\end{proposition}
The proof is given in Appendix~\ref{app:intertwining}.

\begin{remark}
The preceding proposition also implies
that the Eisenstein series vanishes at
the central point. Indeed, since
$\sigma$ is unitary, the inducing
parameter $\sigma\otimes u^{2z}$
has real part zero at $z=0$.
By \cite[Proposition~IV.1.11(b)]{MW},
both $E_P(f_z,z)$ and $M(z)$
are holomorphic there.
Evaluating the functional equation
\cite[Theorem~IV.1.10(a)]{MW}
at $z=0$ gives
\[
E_P(f_0,0)
=
E_P(M(0)f_0,0)
=
-E_P(f_0,0).
\]
Consequently, $E_P(f_0,0)=0$,
or, in the $s_0$-notation,
\[
E_P(f,1)=0.
\]
\end{remark}

We now apply
Proposition~\ref{prop:intertwining-sign}
to the residue identity established
in the preceding subsection. By that
identity,
\[
\begin{aligned}
\operatorname*{Res}_{s_0=1}
Z'_E(\Psi,f;s,s_0)
&=
-\frac12
\left[
\mathcal Z\left(f;s+\frac12\right)
-
\mathcal Z\left(M(0)f;s+\frac12\right)
\right]\\
&=
-\mathcal Z\left(f;s+\frac12\right).
\end{aligned}
\]

We next identify $\mathcal Z$
with the zeta integral of
Suzuki--Wakatsuki. Recall from
Lemma~\ref{lem:levi-measure} that the central
extension
\[
1\longrightarrow A
\longrightarrow\widetilde M
\longrightarrow H
\longrightarrow1
\]
is exact on rational and
adelic points. For the
central element
$e_c\in A(\mathbb A)$, we have $u(e_c)=|c|_{\mathbb A}$.
By exactness, every element
of $H(\mathbb A)$ lifts to
$\widetilde M(\mathbb A)$.
Since $u(e_c)=|c|_{\mathbb A}$
takes every value in
$\mathbb R_{>0}$, multiplying
by a suitable $e_c$ gives
a lift in
$\widetilde M(\mathbb A)^1$.
The kernel of
\[
\widetilde M(\mathbb A)^1
\longrightarrow H(\mathbb A)
\]
is $A(\mathbb A)^1$.

Consequently, the fibre over
the rational quotient is
\[
A(F)\backslash A(\mathbb A)^1
\simeq F^\times\backslash
\mathbb A^1.
\]
With the measure conventions
of Lemma~\ref{lem:levi-measure} and \S\ref{subsec:measures},
we have
\[
dm=dh\,d^\times c,
\qquad
d^\times c=d^\times u\,d^1c,
\]
whereas the decomposition
used in \S\ref{subsec:zeta-integral} is $dm=d^\times u\,dm^1$.
It follows that
\[
dm^1=dh\,d^1c.
\]
Here $dh$ agrees with the
Tamagawa measure used by
Suzuki--Wakatsuki, by part~(4) of
Lemma~\ref{lem:levi-measure}. Put
\[
c_F=
\operatorname{vol}
\left(
F^\times\backslash\mathbb A^1,
d^1c
\right).
\]

The integrand defining
$\mathcal Z$ is invariant
under $A(\mathbb A)^1$.
Indeed, $\chi_\Delta(e_c)=1$,
the action on the underlying
cube is trivial, and
$\Phi_f(e_cm)=\Phi_f(m)$
because $\omega_\pi=1$;
moreover $u(e_c)=1$ for
$c\in\mathbb A^1$.
Thus integration along
the central fibre
contributes exactly $c_F$.

The character $\chi_\Delta$
descends to the character
used in the Suzuki--Wakatsuki
integral. More precisely,
for the indicated
coordinates on $H$,
\[
\chi_\Delta
(h(g_1,g_2),g_3)
=
\left(
\det g_1^{-1}
\det g_2
\det g_3
\right)^2,
\]
which agrees with their
character $\omega$.

Finally, the definition of $\mathcal Z$ involves
averaging both the Schwartz data and the inducing
section over $K$. This averaging must be retained
for general sections. Moreover, since
$u=\delta_P=1$ on $\widetilde M(\mathbb A)^1$,
the restriction $(\rho(\kappa)f)|_{\widetilde M(\mathbb A)^1}$
agrees with
$\xi_f(\kappa)|_{\widetilde M(\mathbb A)^1}$,
where $\xi_f(\kappa)$ is defined in
\S\ref{sec:second-expansion}. In particular,
this restriction is independent of $s_0$.

We use the Suzuki--Wakatsuki zeta integral
with its natural linear extension to inducing
data in $\pi\boxtimes\pi$. For decomposable
data, the automorphic factor is the product
of the two corresponding cusp forms; the
integral is extended linearly to general
inducing data. Explicitly, $Z_E^{\mathrm{SW}}(\Phi,\xi,\mathfrak s)$
is the integral defining $Z_E$ in \cite[\S4.5]{SW1}, with the
automorphic product $\phi(g_1)\phi(g_2)$ replaced by
$\xi(g_1,g_2)$ and with $V_E(F)$ as above.

We therefore obtain the following residue formula.

\begin{corollary}\label{cor:residue-SW}
Under the hypotheses and measure
normalizations above, for $\Re s$
sufficiently large,
\[
\begin{aligned}
\operatorname*{Res}_{s_0=1}
Z'_E(\Psi,f;s,s_0)
={}&-c_F\int_K
Z_E^{\mathrm{SW}}
\left(
\Phi_{\Psi_\kappa},
\xi_f(\kappa),
s+\frac12
\right)
\,d\kappa.
\end{aligned}
\]
Here $\xi_f(\kappa)$ is the inducing
datum identified with
$(\rho(\kappa)f)|_{\widetilde M(\mathbb A)^1}$
above (see Section~\ref{sec:second-expansion} for its
definition),
\[
c_F=
\operatorname{vol}
\left(
F^\times\backslash\mathbb A^1,
d^1c
\right),
\]
and $Z_E^{\mathrm{SW}}$ is the
Suzuki--Wakatsuki zeta integral with
the Tamagawa measure normalization
fixed in \S\ref{subsec:measures}.
\end{corollary}

\section{The second expansion}\label{sec:second-expansion}

The first expansion expresses the residue of the zeta integral
in terms of the boundary contribution associated with the
parabolic subgroup. We now obtain a complementary expression
by unfolding the theta series along its generic rational
orbit. The essential geometric feature is that the connected
stabilizer of this orbit is a diagonal copy of
$\mathrm{PGL}_2$. Since the inducing representation is
$\sigma=\pi\boxtimes\pi$, integration along this stabilizer
produces a diagonal invariant bilinear pairing on the two
copies of $\pi$. The remaining orbit variables act by
translation on one factor, giving rise to local matrix
coefficients.

This construction provides an Eulerian expression for the
same global zeta integral, which will subsequently be
compared with the first expansion through their residues.
We retain the bilinear formulation throughout, allowing
arbitrary inducing data in $\pi\boxtimes\pi$. The Hermitian
specialization is postponed until
Remark~\ref{rem:hermitian-datum}.

Let
\[
 w_\mu=((0,I_2,0),(1,0,\mu)),\qquad
 [4\mu]=[d_E]
 \quad\text{in }F^\times/(F^\times)^2.
\]
By Proposition~\ref{prop:generic-stabilizer}, the
stabilizer $S$ of $w_\mu$ in $G'$ is
\[
 S\simeq\mathrm{PGL}_2\times\mu_2,
\]
and its identity component $S^\circ$ is identified
with the diagonal $\mathrm{PGL}_2$ through
\[
 \iota(t)=h(t,t)\,e_{\det t},
 \qquad t\in\mathrm{GL}_2.
\]
The factor $e_{\det t}$ is needed: $h(t,t)$ alone does not fix
$w_\mu$, and $u(h(t,t))=|\det t|^{-1}$, whereas $u(\iota(t))=1$.
Since $\iota(zt)=(z^{-1}I_4,z^2I_2)\,\iota(t)$, the map $\iota$ is
well defined on $\mathrm{PGL}_2$ modulo $Z'_\rho$, and its image in
the projective Levi is diagonal.

The distinction between $S$ and $S^\circ$ will be
important for the normalization. Unfolding with respect
to the full rational stabilizer contributes a global
factor of $1/2$, whereas the orbit map from
$S^\circ\backslash G'$ has degree two and contributes
a separate local factor. We keep these factors
distinct until the final Euler product.

For each place $v$, choose a nonzero bilinear
functional, invariant under $\mathrm{PGL}_2(F_v)$,
\[
 B_v\in
 \operatorname{Hom}_{\mathrm{PGL}_2(F_v)}
   (\pi_v\otimes\pi_v,\mathbb C).
\]
Since $\pi_v$ is irreducible and self-dual, with
trivial central character, this space is
one-dimensional. At almost all finite places $v$, the representation
$\pi_v$ is unramified. Fix a nonzero spherical
vector $e_v^\circ\in\pi_v$ and put
$\xi_v^\circ=e_v^\circ\otimes e_v^\circ$.
These vectors determine the restricted tensor
product realization of $\sigma$.
Relative to these fixed reference tensors, normalize
the local invariant functional $B_v$ by
\[
 B_v(\xi_v^\circ)=1.
\]
This is possible because any nonzero $B_v$ is nonzero on
$\xi_v^\circ$: it corresponds to an isomorphism
$\pi_v\simeq\pi_v^\vee$, which carries $e_v^\circ$ to a nonzero
spherical vector of $\pi_v^\vee$, and the canonical pairing of the
spherical lines of $\pi_v^\vee$ and $\pi_v$ is nonzero.

The corresponding global functional is the diagonal
period
\[
 B(\xi)=
 \int_{\mathrm{PGL}_2(F)\backslash
       \mathrm{PGL}_2(\mathbb A)}
 \xi(h(t,t))\,dt,
\]
where $dt$ is the Tamagawa measure fixed in
Section~\ref{subsec:measures}. The integral converges
absolutely by cuspidality.

For example, if
$\xi=\phi_1\otimes\phi_2$ is a pure tensor of
automorphic forms, then
\[
 B(\phi_1\otimes\phi_2)
 =
 \int_{[\mathrm{PGL}_2]}
 \phi_1(t)\phi_2(t)\,dt.
\]
In particular, $B$ is bilinear, rather than
Hermitian. Since the global diagonal period is invariant under
$\mathrm{PGL}_2(\mathbb A)$, local multiplicity one
implies that its restriction to factorizable vectors
is proportional to the product of the chosen local
invariant functionals. Thus, under the restricted
tensor product identification
$\sigma\simeq\bigotimes_v'\sigma_v$, there is a
scalar $C_\pi$ such that
\begin{equation}\label{eq:second-global-pairing}
 B\left(\bigotimes_v\xi_v\right)
 =
 C_\pi\prod_v B_v(\xi_v).
\end{equation}

The unfolding will also involve translates of the
second inducing factor. Accordingly, for
$y\in\mathrm{GL}_2(F_v)$ define
\begin{equation}\label{eq:second-beta}
 \beta_{\xi,v}(y)
 =
 B_v((1\otimes\pi_v(y))\xi).
\end{equation}
For a pure tensor $\xi=\phi_1\otimes\phi_2$, this
is the invariant bilinear pairing
\[
 \beta_{\xi,v}(y)
 =
 B_v(\phi_1\otimes\pi_v(y)\phi_2).
\]
Thus $\beta_{\xi,v}$ is a matrix coefficient
associated with the invariant pairing $B_v$.
Since $\pi_v$ has trivial central character,
$\beta_{\xi,v}$ depends only on the image of $y$
in $\mathrm{PGL}_2(F_v)$.

We next associate to a local inducing section a
vector in $\sigma_v$ to which the functional $B_v$
can be applied. Let $f_{v,s_0}$ be a flat local
section in the vector-valued realization of the normalized induced
representation, and let $\kappa\in K_v$. Define
\[
 \xi_{f_v}(\kappa)
 :=f_{v,s_0}(\kappa)\in\sigma_v.
\]
By flatness, this vector is independent of $s_0$.
Equivalently, the normalized inducing law gives
\[
 f_{v,s_0}(m\kappa)
 =
 \delta_P(m)^{1/2}
 u(m)^{2(s_0-1)}
 \sigma_v(m)\xi_{f_v}(\kappa),
 \qquad m\in\widetilde M(F_v).
\]
Globally, in the scalar-valued automorphic realization of
\S\ref{subsec:zeta-integral}, the corresponding vector
$\xi_f(\kappa)\in\sigma$ is the automorphic form
\[
 \xi_f(\kappa)(m)
 =
 \delta_P(m)^{-1/2}
 u(m)^{-2(s_0-1)}
 f_{s_0}(m\kappa),
 \qquad m\in\widetilde M(\mathbb A).
\]
This is the scalar-valued realization of the underlying
Levi representation vector. For factorizable
$f=\bigotimes_v f_v$ and $\kappa=(\kappa_v)$,
it corresponds to
$\bigotimes_v\xi_{f_v}(\kappa_v)$
under the restricted tensor product identification
$\sigma\simeq\bigotimes_v'\sigma_v$
used in \eqref{eq:second-global-pairing}.
Throughout the global unfolding and convergence
arguments below, inducing sections are evaluated
in the scalar-valued automorphic realization of
\S\ref{subsec:zeta-integral}; the vector-valued
realization is used to identify their underlying
local inducing vectors and to express the resulting
periods through the functionals $B_v$.
As in \S\ref{subsec:zeta-integral}, the local sections $f_v$ are
taken $K_v$-finite, so that a factorizable $f$ is $K_4$-finite with
inducing data given by finite sums of pure tensors of cusp forms.
Thus the absolute-convergence estimates of
Proposition~\ref{prop:absolute-convergence}
apply without modification.

The remaining variables in the generic-orbit
unfolding will be expressed in terms of the
relative invariants $\Delta$ and $R$, together
with the matrix
\[
 y_w=\ell_2x_1-\ell_1x_2.
\]
The diagonal period then enters through
$\beta_{\xi_{f_v}(\kappa),v}(y_w)$.

With the coordinate measure $dw_v$ fixed in
Section~\ref{subsec:measures}, define
\begin{equation}\label{eq:local-second}
 \begin{split}
 \mathscr Z'_v(\Psi_v,f_v;s,s_0)
 =\int_{K_v}\int_{V'_E(F_v)}&
  (\rho(\kappa)\Psi_v)(w)
  |\Delta(w)|_v^{s-1}
  |R(w)|_v^{s_0-2}\\[-2pt]
 &\hspace{12mm}\cdot
  \beta_{\xi_{f_v}(\kappa),v}(y_w)
  \,dw_v\,d\kappa.
 \end{split}
\end{equation}
By Lemma~\ref{lem:R-determinant},
$R(w)=-2\det(y_w)$, so $y_w\in\mathrm{GL}_2(F_v)$
whenever $R(w)\ne0$. The complement of
$\{\ell\ne0,\,R\ne0\}$ in $V'_E(F_v)$
has measure zero, so the integrand is understood
on this full-measure locus.

The powers $s-1$ and $s_0-2$ arise from the
relative-invariant characters, the parabolic
modulus, and the Jacobian of the generic-orbit
map. They are therefore determined by the
unfolding, rather than imposed as an independent
normalization. The following proposition
establishes the resulting Eulerian expression
for the global zeta integral.

\begin{proposition}\label{prop:second-expansion}
Let $\Psi=\bigotimes_v\Psi_v$ and
$f=\bigotimes_v f_v$ be factorizable,
with $K_v$-finite local inducing sections
and standard spherical data at almost all places.  In the absolute
convergence region of Proposition~\ref{prop:absolute-convergence},
\begin{equation}\label{eq:second-expansion-product}
 Z'_E(\Psi,f;s,s_0)
 =\frac12c_F^{-1}|\Delta_F|^{-6}C_\pi
   \prod_v\left[2c_v^2\mathscr Z'_v(\Psi_v,f_v;s,s_0)\right].
\end{equation}
The product of the corresponding absolute local integrals is finite.
\end{proposition}

\begin{proof}
\emph{1. Unfolding and the connected stabilizer.}
Apply the Iwasawa integration formula~\eqref{eq:iwasawa-Gprime}, noting that
$|\chi_\Delta(\kappa)|=1$.  By the anisotropy observation following Theorem~\ref{thm:generic-orbits}, on the rational
$E$-stratum the condition $\ell\ne0$ implies $R\ne0$.
Theorem~\ref{thm:generic-orbits} therefore identifies the rational summation set with
$w_\mu G'(F)$. Hence
\[
 \Theta_E^{\ell\ne0}(\rho(\kappa)\Psi)(p)
 =\sum_{\gamma\in S(F)\backslash G'(F)}
   (\rho(\kappa)\Psi)(w_\mu\gamma p).
\]
Proposition~\ref{prop:absolute-convergence} justifies exchanging the sum and integral
in the initial region of absolute convergence.
Since $S(F)=S^\circ(F)\times\{\pm1\}$, unfolding over $S(F)$
is half of unfolding over $S^\circ(F)$:
\begin{equation}\label{eq:second-unfolded}
 Z'_E=\frac12\int_K\int_{S^\circ(F)\backslash G'(\mathbb A)}
 (\rho(\kappa)\Psi)(w_\mu p)
 |\chi_\Delta(p)|^{s+s_0}
 (\rho(\kappa)f)(p)\,d\mu_{G'}(p)\,d\kappa.
\end{equation}
This factor $1/2$ is global; it must be kept distinct from the local
degree-two orbit map below.

\emph{2. Integration over the diagonal.}
We next integrate along the connected stabilizer
$S^\circ\simeq\mathrm{PGL}_2$. Recall that
$u|_{S^\circ}=\delta_P|_{S^\circ}=1$ and that
$S^\circ$ is unimodular. Moreover, for
$t\in\mathrm{PGL}_2(\mathbb A)$, with $\iota$ the embedding
recalled below,
\[
 w_\mu\iota(t)p=w_\mu p,
 \qquad
 |\chi_\Delta(\iota(t)p)|=|\chi_\Delta(p)|.
\]
Thus the Schwartz factor and the relative-invariant
weight in \eqref{eq:second-unfolded} are invariant
under left translation by $S^\circ(\mathbb A)$.
Weil's quotient integration formula, with the
compatible measures fixed above, gives
\[
 \begin{aligned}
 Z'_E=\frac12\int_K
 \int_{S^\circ(\mathbb A)\backslash G'(\mathbb A)}
 &(\rho(\kappa)\Psi)(w_\mu p)
 |\chi_\Delta(p)|^{s+s_0}\\
 &\quad\cdot
 \left(
 \int_{S^\circ(F)\backslash S^\circ(\mathbb A)}
   (\rho(\kappa)f)(\iota(t)p)\,dt
 \right)
 \,d\dot\mu(p)\,d\kappa.
 \end{aligned}
\]

We calculate the inner integral explicitly. Write
$p=nm$, where $n\in N(\mathbb A)$ and the Levi
component of $m$ is $(h(g_1,g_2),k)$, with
\[
 h(g_1,g_2)
 =\operatorname{diag}({}^tg_1^{-1},g_2).
\]
Put $\xi=\xi_f(\kappa)$, regarded as an automorphic
vector in the realization of
$\sigma=\pi\boxtimes\pi$, and write
$\xi(g_1,g_2)=\xi(h(g_1,g_2))$.

By Proposition~\ref{prop:generic-stabilizer},
the connected stabilizer is embedded through
\[
 \iota(t)=h(t,t)e_{\det t}
 =
 \bigl(h(t,\det(t)^{-1}t),\det(t)I_2\bigr),
 \qquad t\in\mathrm{GL}_2,
\]
which descends to $\mathrm{PGL}_2$ in $G'$.
In particular, $u(\iota(t))=1$.

Since $\iota(t)$ belongs to the Levi subgroup,
we may write
\[
 \iota(t)nm
 =
 \bigl(\iota(t)n\iota(t)^{-1}\bigr)
 \bigl(\iota(t)m\bigr).
\]
The first factor lies in $N(\mathbb A)$.
By left $N(\mathbb A)$-invariance of the
inducing section, it contributes nothing.
The Levi component of $\iota(t)m$ is
\[
 \left(
 h(tg_1,\det(t)^{-1}tg_2),
 \det(t)k
 \right).
\]
Using the defining transformation law of the
normalized section and
$\delta_P^{1/2}=u^2$, we obtain
\[
 \begin{aligned}
 (\rho(\kappa)f)(\iota(t)p)
 &=
 \delta_P(m)^{1/2}
 u(m)^{2(s_0-1)}
 \xi(tg_1,\det(t)^{-1}tg_2)\\
 &=
 u(p)^{2s_0}
 \xi(tg_1,\det(t)^{-1}tg_2).
 \end{aligned}
\]
Here the outer $\mathrm{GL}_2$ component does
not act on $\sigma$. Since $\pi$ has trivial
central character, the scalar
$\det(t)^{-1}$ in the second argument acts
trivially. Hence
\[
 (\rho(\kappa)f)(\iota(t)p)
 =
 u(p)^{2s_0}\xi(tg_1,tg_2).
\]

Recall the diagonal period
\[
 B(\eta)
 =
 \int_{[\mathrm{PGL}_2]}\eta(t,t)\,dt,
 \qquad
 [\mathrm{PGL}_2]
 =
 \mathrm{PGL}_2(F)\backslash
 \mathrm{PGL}_2(\mathbb A).
\]
With the right-regular action of
$\pi\boxtimes\pi$, the preceding identity
therefore gives
\[
 \begin{aligned}
 &\int_{[\mathrm{PGL}_2]}
   (\rho(\kappa)f)(\iota(t)p)\,dt\\
 &\qquad=
 u(p)^{2s_0}
 \int_{[\mathrm{PGL}_2]}
   \xi(tg_1,tg_2)\,dt\\
 &\qquad=
 u(p)^{2s_0}
 B((\pi(g_1)\otimes\pi(g_2))\xi).
 \end{aligned}
\]
The diagonal invariance of $B$ implies
\[
 \begin{aligned}
 B((\pi(g_1)\otimes\pi(g_2))\xi)
 &=
 B((\pi(g_1^{-1})\otimes\pi(g_1^{-1}))
   (\pi(g_1)\otimes\pi(g_2))\xi)\\
 &=
 B((1\otimes\pi(g_1^{-1}g_2))\xi).
 \end{aligned}
\]

We now identify the relative group element
$g_1^{-1}g_2$ with the Witt coordinate of the
orbit point. Recall that
\[
 y_w=\ell_2x_1-\ell_1x_2.
\]
For the chosen base point, $y_{w_\mu}=I_2$.
The right action of the Levi component
$(h(g_1,g_2),k)$ gives
\[
 y_{w_\mu p}
 =
 \det(a)\det(k)\,g_1^{-1}g_2,
 \qquad a={}^tg_1^{-1}.
\]
Since the scalar
$\det(a)\det(k)$ acts trivially under $\pi$,
we have
\[
 \pi(y_{w_\mu p})
 =
 \pi(g_1^{-1}g_2).
\]
Consequently,
\[
 \int_{[\mathrm{PGL}_2]}
   (\rho(\kappa)f)(\iota(t)p)\,dt
 =
 u(p)^{2s_0}
 B((1\otimes\pi(y_{w_\mu p}))
   \xi_f(\kappa)).
\]

Finally, the character identity
$|\chi_R|=|\chi_\Delta|u^2$ gives
\[
 \begin{aligned}
 |\chi_\Delta(p)|^{s+s_0}u(p)^{2s_0}
 &=
 |\chi_\Delta(p)|^s
 \bigl(|\chi_\Delta(p)|u(p)^2\bigr)^{s_0}\\
 &=
 |\chi_\Delta(p)|^s|\chi_R(p)|^{s_0}.
 \end{aligned}
\]
Substituting into the quotient integral yields
\begin{equation}\label{eq:after-inner}
 \begin{split}
 Z'_E=\frac12\int_K
 \int_{S^\circ(\mathbb A)\backslash G'(\mathbb A)}
 &(\rho(\kappa)\Psi)(w_\mu p)
 |\chi_\Delta(p)|^s|\chi_R(p)|^{s_0}\\
 &\cdot
 B((1\otimes\pi(y_{w_\mu p}))
   \xi_f(\kappa))
 \,d\dot\mu(p)\,d\kappa.
 \end{split}
\end{equation}

\emph{3. The quotient measure.}
We now identify the quotient measure on
$S^\circ(\mathbb A)\backslash G'(\mathbb A)$, denoted $d\dot\mu$,
appearing in \eqref{eq:after-inner}.
Recall from Lemma~\ref{lem:levi-measure}(3) that
the Haar measure on $G'(\mathbb A)$ is
\[
 d\mu_{G'}
 =
 c_F^{-1}|\Delta_F|^{-15/2}
 \prod_v c_v^2|\omega_{G',v}|.
\]
On the other hand, under the identification
$S^\circ\simeq\mathrm{PGL}_2$, the measure
used for the diagonal period is
\[
 dt
 =
 |\Delta_F|^{-3/2}
 \prod_v|\omega_{S^\circ,v}|.
\]

The quotient measure is characterized by
Weil's integration formula
\[
 \int_{G'(\mathbb A)}\Phi(g)\,d\mu_{G'}(g)
 =
 \int_{S^\circ(\mathbb A)\backslash G'(\mathbb A)}
 \int_{\mathrm{PGL}_2(\mathbb A)}
 \Phi(\iota(t)g)\,dt\,d\dot\mu(g),
\]
for compactly supported continuous $\Phi$.
Since the two Haar measures above are expressed
in terms of compatible local gauge forms,
the quotient measure is obtained by taking
their ratio. In particular, the global
normalization factor is
\[
 \frac{
 c_F^{-1}|\Delta_F|^{-15/2}
 }{
 |\Delta_F|^{-3/2}
 }
 =
 c_F^{-1}|\Delta_F|^{-6}.
\]
At each place $v$, the corresponding quotient
measure is
\[
 d\dot\mu_v
 =
 c_v^2
 \left|
 \frac{\omega_{G',v}}{\omega_{S^\circ,v}}
 \right|.
\]
Consequently,
\begin{equation}\label{eq:second-quotient-measure}
 d\dot\mu
 =
 c_F^{-1}|\Delta_F|^{-6}
 \prod_v d\dot\mu_v,
 \qquad
 d\dot\mu_v
 =
 c_v^2
 \left|
 \frac{\omega_{G',v}}{\omega_{S^\circ,v}}
 \right|.
\end{equation}
The restricted product of these local quotient
measures is well defined; the convergence of
its unramified base volumes is verified
explicitly in Step~6 below.
Here the quotient of gauge forms is understood
as the relative top-degree form on
$S^\circ\backslash G'$, characterized locally
by the quotient integration formula.

The exponent $-6$ is consistent with
$\dim G'=15$ and $\dim S^\circ=3$:
the quotient has dimension $12$, equal to
$\dim V'$. The comparison of this quotient
gauge with the gauge on the open orbit in
$V'$ will be carried out in the next step.

\emph{4. Orbit Jacobian and its multiplicity.}
We now compare the quotient measure on
$S^\circ(\mathbb A)\backslash G'(\mathbb A)$
with the measure on the corresponding open orbit
in $V'$. There are two points to account for:
the Jacobian of the orbit map and its degree.
We work throughout with rational gauge forms,
so that the undetermined scalar in the Jacobian
identity disappears globally by the product formula.

Let $\mathcal O=w_\mu G'$ and consider the
orbit map
\[
 \varpi:S^\circ\backslash G'
 \longrightarrow\mathcal O,
 \qquad
 S^\circ p\longmapsto w_\mu p.
\]
Since the full stabilizer is
$S=S^\circ\times\mu_2$, the orbit is
identified with $S\backslash G'$, and
$\varpi$ is a finite etale morphism of
degree two. Both $S^\circ\backslash G'$
and $\mathcal O$ have dimension twelve,
equal to $\dim V'$.

We first determine an invariant rational
top form on $\mathcal O$. Let $dW$ denote
the fixed coordinate top form on the
twelve-dimensional vector space $V'$.
The determinant character of the
representation of $G'$ on $V'$ is
$\chi_\Delta^3$, while the relative
invariant $\Delta$ satisfies
\[
 \Delta(wp)=\chi_\Delta(p)\Delta(w).
\]
Consequently,
\[
 \begin{aligned}
 \frac{d(Wp)}{\Delta(wp)^3}
 &=
 \frac{\chi_\Delta(p)^3dW}
      {\chi_\Delta(p)^3\Delta(w)^3}\\
 &=
 \frac{dW}{\Delta(w)^3}.
 \end{aligned}
\]
It follows that
\[
 \omega_V=\frac{dW}{\Delta^3}
\]
is a nonvanishing right-$G'$-invariant
rational top form on the open orbit
$\mathcal O$.

We next examine the quotient gauge form.
Recall that
\[
 \chi_P(p)
 =
 \det\bigl(\operatorname{Ad}(p)
              |_{\mathfrak n}\bigr).
\]
For a Levi element with relevant blocks
$a$ and $b$, the adjoint action on
$\mathfrak n$ gives
\[
 \chi_P(p)
 =
 \left(\frac{\det a}{\det b}\right)^2.
\]
In particular, $\chi_P|_S=1$.
The quotient of the left-invariant
rational gauge forms
$\omega_{G'}/\omega_{S^\circ}$
has right-translation character
$\chi_P^{-1}$. Thus, for
$g\in G'$,
\[
 R_g^*
 \left(
 \frac{\omega_{G'}}{\omega_{S^\circ}}
 \right)
 =
 \chi_P(g)^{-1}
 \frac{\omega_{G'}}{\omega_{S^\circ}}.
\]
Since $\chi_P$ is a character, we obtain
\[
 R_g^*
 \left(
 \chi_P\frac{\omega_{G'}}{\omega_{S^\circ}}
 \right)
 =
 \chi_P\frac{\omega_{G'}}{\omega_{S^\circ}}.
\]
Hence
\[
 \chi_P\frac{\omega_{G'}}{\omega_{S^\circ}}
\]
is a right-$G'$-invariant rational top
form on $S^\circ\backslash G'$.

On the other hand, the orbit map
$\varpi$ is right-$G'$-equivariant.
Therefore
\[
 \varpi^*\omega_V
 =
 \varpi^*\left(\frac{dW}{\Delta^3}\right)
\]
is another nonvanishing right-invariant
rational top form on the same homogeneous
space. The ratio of these two forms is
a right-$G'$-invariant rational function.
Since the right action of $G'$ on
$S^\circ\backslash G'$ is transitive,
this ratio is constant. Moreover,
both forms and the orbit map are defined
over $F$, so the constant belongs to
$F^\times$. We conclude that
\begin{equation}\label{eq:second-gauge-identity}
 \chi_P\frac{\omega_{G'}}{\omega_{S^\circ}}
 =
 \kappa_0\,
 \varpi^*\left(\frac{dW}{\Delta^3}\right),
 \qquad \kappa_0\in F^\times.
\end{equation}
This identity determines the orbit
Jacobian up to a rational scalar.
Its precise numerical value is unnecessary:
the product formula gives
\[
 \prod_v|\kappa_0|_v=1.
\]

It remains to account for the degree
of the orbit map on local rational points.
Applying Theorem~\ref{thm:generic-orbits}
over $F_v$, as in the discussion following it,
we identify the rational orbit of $w_\mu$ with
\[
 \mathcal O(F_v)
 =
 V'_E(F_v)\cap\{R\ne0\}.
\]
By Proposition~\ref{prop:generic-stabilizer},
the stabilizer has the rational splitting
$S=S^\circ\times\mu_2$.
If $w=w_\mu p_w$ with
$p_w\in G'(F_v)$, its two preimages
under $\varpi$ are represented by
\[
 S^\circ(F_v)p_w
 \quad\text{and}\quad
 S^\circ(F_v)\epsilon p_w,
\]
where $\epsilon$ represents the
nonidentity element of
$S(F_v)/S^\circ(F_v)$.
These two classes are distinct and
have the same image, because
$\epsilon$ stabilizes $w_\mu$.
Conversely, any two lifts of $w$
differ by an element of $S(F_v)$.
Thus
\[
 \#\varpi^{-1}(w)(F_v)=2.
\]
In particular, the local multiplicity
is established on rational points,
rather than inferred merely from the
geometric degree of $\varpi$.

Taking absolute values in
\eqref{eq:second-gauge-identity}
and using $|\chi_P|=\delta_P$,
we obtain
\[
 \left|
 \frac{\omega_{G',v}}{\omega_{S^\circ,v}}
 \right|
 =
 |\kappa_0|_v\,
 \delta_P(p)^{-1}
 \varpi^*\left(
 \frac{dw_v}{|\Delta(w)|_v^3}
 \right).
\]
Recall from
\eqref{eq:second-quotient-measure}
that
\[
 d\dot\mu_v
 =
 c_v^2
 \left|
 \frac{\omega_{G',v}}{\omega_{S^\circ,v}}
 \right|.
\]
The change-of-variables formula for
the degree-two orbit map therefore
gives
\begin{equation}\label{eq:second-local-orbit-measure}
 \begin{aligned}
 &\int_{S^\circ(F_v)\backslash G'(F_v)}
   \Phi(w_\mu p)\,d\dot\mu_v(p)\\
 &\qquad=
 2c_v^2|\kappa_0|_v
 \int_{V'_E(F_v)}
   \Phi(w)\,
   \delta_P(p_w)^{-1}
   \frac{dw_v}{|\Delta(w)|_v^3}.
 \end{aligned}
\end{equation}
Here $p_w\in G'(F_v)$ is any element
such that $w=w_\mu p_w$.
The expression $\delta_P(p_w)$
is independent of this choice,
since two such elements differ
by an element of $S(F_v)$ and
$\delta_P|_S=1$.

The factor $2$ in
\eqref{eq:second-local-orbit-measure}
is the local multiplicity of the
orbit map. It is distinct from
the global factor $1/2$ introduced
in \eqref{eq:second-unfolded}
when passing from $S(F)$ to
$S^\circ(F)$.

\emph{5. Conversion of the weights.}
We now express the character weights in the unfolded integral
in terms of the relative invariants $\Delta$ and $R$ on $V'$.
The powers of these invariants in the local zeta integral
are determined by the original zeta weights together with
the modulus factor arising from the orbit Jacobian in Step~4.

Let $w=w_\mu p$, with $p\in G'(F_v)$.
By the relative-invariance properties of $\Delta$ and $R$,
\[
 \Delta(w_\mu p)=\chi_\Delta(p)\Delta(w_\mu),
 \qquad
 R(w_\mu p)=\chi_R(p)R(w_\mu).
\]
Since the chosen base point satisfies
\[
 \Delta(w_\mu)=4\mu,
 \qquad
 R(w_\mu)=-2,
\]
we obtain
\[
 |\chi_\Delta(p)|_v
   =\frac{|\Delta(w)|_v}{|4\mu|_v},
 \qquad
 |\chi_R(p)|_v
   =\frac{|R(w)|_v}{|2|_v}.
\]

Recall also the relation between the modulus character
and the two relative-invariant characters:
\[
 \delta_P(p)
 =
 \left|\frac{\chi_R(p)}{\chi_\Delta(p)}\right|_v^2.
\]
Consequently,
\[
 \begin{aligned}
 &|\chi_\Delta(p)|_v^s
  |\chi_R(p)|_v^{s_0}
  \delta_P(p)^{-1}
  |\Delta(w)|_v^{-3}\\
 &\qquad=
  |\chi_\Delta(p)|_v^{s+2}
  |\chi_R(p)|_v^{s_0-2}
  |\Delta(w)|_v^{-3}\\
 &\qquad=
  |4\mu|_v^{-s-2}|2|_v^{2-s_0}
  |\Delta(w)|_v^{s-1}
  |R(w)|_v^{s_0-2}.
 \end{aligned}
\]
Thus the exponent $s-1$ of $|\Delta(w)|_v$
comes from the original exponent $s$,
the contribution $+2$ of $\delta_P^{-1}$,
and the contribution $-3$ of the invariant
orbit gauge $dW/\Delta^3$.
Similarly, the exponent $s_0-2$ of $|R(w)|_v$
comes from the modulus factor $\delta_P^{-1}$.

Writing
\begin{equation}\label{eq:second-basepoint-factor}
 e_v
 =
 |4\mu|_v^{-s-2}|2|_v^{2-s_0},
 \qquad
 \prod_v e_v=1,
\end{equation}
the last equality follows from the product formula,
since $4\mu,2\in F^\times$.
Likewise, the rational Jacobian scalar
$\kappa_0\in F^\times$ from
\eqref{eq:second-gauge-identity} satisfies
\[
 \prod_v|\kappa_0|_v=1.
\]
Therefore neither the base-point normalization
nor the undetermined rational scalar in the orbit
Jacobian contributes an additional global factor.

Combining this calculation with
\eqref{eq:second-local-orbit-measure},
the local scalar weight and quotient measure
become
\[
 2c_v^2|\kappa_0|_v e_v\,
 |\Delta(w)|_v^{s-1}
 |R(w)|_v^{s_0-2}\,dw_v.
\]
After taking the product over all places,
the factors $|\kappa_0|_v e_v$ cancel by
the product formula. This gives precisely
the powers of $\Delta$ and $R$ appearing
in the local zeta integral $\mathscr Z'_v$.

\emph{6. Absolute convergence and Euler factorization.}
We first establish absolute integrability of the
factorizable orbital integrand before applying Fubini.
Put $r=\Re s$ and $r_0=\Re s_0$.

Define the nonnegative summandwise integral
\[
\begin{aligned}
\mathcal A
:=\int_K\int_{G'(F)\backslash G'(\mathbb A)}
&|\chi_\Delta(p)|_{\mathbb A}^{r+r_0}
\sum_{\substack{w\in V'_E(F)\\ \ell(w)\ne0}}
|(\rho(\kappa)\Psi)(wp)|\\
&\qquad\times
|(\rho(\kappa)f)(p)|
\,d\mu_{G'}(p)\,d\kappa.
\end{aligned}
\]
By Proposition~\ref{prop:absolute-convergence}
and Corollary~\ref{cor:R-cutoff}, the summandwise
majorant estimates established in
\S\ref{subsec:absolute-convergence} give
\[
\mathcal A<\infty
\]
throughout the stated initial region.
The summandwise estimate is essential:
integrability of the absolute value of the
theta sum does not imply integrability of
the sum of the absolute values of its terms,
which is needed to obtain a finite unfolded
majorant.

Since the integrand of $\mathcal A$ is nonnegative,
Tonelli's theorem permits unfolding the
contribution of its generic orbit along
$S(F)\backslash G'(F)$, even before
integrability of the unfolded expression
has been established. The resulting
nonnegative integral is finite because
it is bounded by $\mathcal A$.
Passing from $S(F)$ to $S^\circ(F)$
multiplies this integral by $2$, since
$S/S^\circ\simeq\mu_2$.

We next apply Weil's integration formula to
$S^\circ(F)\backslash G'(\mathbb A)$,
integrating first along
\[
 [S^\circ]
 =
 S^\circ(F)\backslash S^\circ(\mathbb A)
\]
and then over
$S^\circ(\mathbb A)\backslash G'(\mathbb A)$,
with the quotient measure fixed in Step~3.
Since $S^\circ$ stabilizes $w_\mu$ and
$\chi_\Delta|_{S^\circ}
=\chi_R|_{S^\circ}=1$,
the Schwartz factor and the character weights
are invariant under left $S^\circ(\mathbb A)$
and may be taken outside the inner integral.

By the calculation in Step~2,
\[
 \int_{[S^\circ]}
   (\rho(\kappa)f)(\iota(t)p)\,dt
 =
 u(p)^{2s_0}
 B\bigl((1\otimes\pi(y_{w_\mu p}))
        \xi_f(\kappa)\bigr).
\]
Taking absolute values and using
$|\chi_R(p)|=|\chi_\Delta(p)|u(p)^2$
converts the original character weights
into
$|\chi_\Delta(p)|^r|\chi_R(p)|^{r_0}$.
The triangle inequality
\[
 \left|
 \int_{[S^\circ]}
   (\rho(\kappa)f)(\iota(t)p)\,dt
 \right|
 \le
 \int_{[S^\circ]}
   |(\rho(\kappa)f)(\iota(t)p)|\,dt
\]
therefore shows that
\[
\begin{aligned}
\mathcal I:=
\int_K\int_{S^\circ(\mathbb A)\backslash G'(\mathbb A)}
 &|(\rho(\kappa)\Psi)(w_\mu p)|
 |\chi_\Delta(p)|^r|\chi_R(p)|^{r_0}\\
 &\quad\cdot
 \left|
 B\bigl((1\otimes\pi(y_{w_\mu p}))
        \xi_f(\kappa)\bigr)
 \right|
 \,d\dot\mu(p)\,d\kappa
 \le 2\mathcal A<\infty.
\end{aligned}
\]
Thus the unfolded orbital integrand is
absolutely integrable in the stated
initial region. This provides the global
domination needed for the restricted-product
argument below.

For factorizable inducing data and
$\kappa=(\kappa_v)$, the restricted tensor product
identification gives
\[
 \xi_f(\kappa)
 =
 \bigotimes_v \xi_{f_v}(\kappa_v).
\]
Consequently, for $y=(y_v)\in\mathrm{GL}_2(\mathbb A)$,
\[
 B\bigl((1\otimes\pi(y))\xi_f(\kappa)\bigr)
 =
 C_\pi\prod_v
 \beta_{\xi_{f_v}(\kappa_v),v}(y_v),
\]
by \eqref{eq:second-global-pairing} and
\eqref{eq:second-beta}.  Moreover $C_\pi\ne0$: for a cusp form
$\phi\in\pi$ the conjugate $\overline\phi$ lies in $\pi$, since the
space of conjugates is a cuspidal representation isomorphic to
$\pi^\vee\simeq\pi$, hence equal to the space of $\pi$ by multiplicity
one, so that $\phi\otimes\overline\phi\in\sigma$ with no further
identification; and $B(\phi\otimes\overline\phi)=\langle\phi,\phi\rangle>0$.

For each $v$ define the nonnegative local integrals
\[
\begin{aligned}
a_v:=\int_{K_v}\int_{S^\circ(F_v)\backslash G'(F_v)}
 &|(\rho(\kappa_v)\Psi_v)(w_\mu p)|
 |\chi_\Delta(p)|_v^{r}|\chi_R(p)|_v^{r_0}\\
 &\quad\cdot
 |\beta_{\xi_{f_v}(\kappa_v),v}(y_{w_\mu p})|
 \,d\dot\mu_v(p)\,d\kappa_v
\end{aligned}
\]
and
\[
\begin{aligned}
\mathscr Z'_{v,\mathrm{abs}}(s,s_0)
 :=\int_{K_v}\int_{V'_E(F_v)}
 &|(\rho(\kappa_v)\Psi_v)(w)|
 |\Delta(w)|_v^{r-1}|R(w)|_v^{r_0-2}\\
 &\quad\cdot
 |\beta_{\xi_{f_v}(\kappa_v),v}(y_w)|
 \,dw_v\,d\kappa_v.
\end{aligned}
\]
Since these integrands are nonnegative, the local
orbit-measure formula
\eqref{eq:second-local-orbit-measure}
applies without any prior integrability assumption.
Combining it with the character-weight identity
of Step~5 gives
\[
a_v
=
2c_v^2|\kappa_0|_v|e_v|\,
\mathscr Z'_{v,\mathrm{abs}}(s,s_0).
\]
In particular, $a_v$ is finite if and only if
$\mathscr Z'_{v,\mathrm{abs}}(s,s_0)$ is finite.

Let $S_0$ be a finite set of places, containing the archimedean ones,
outside which $\Psi_v=\mathbf 1_{V'(\mathfrak o_v)}$, $f_v$ is the
normalized spherical section, so that $\xi_{f_v}(\kappa_v)=\xi_v^\circ$
for $\kappa_v\in K_v$, $|4\mu|_v=|2|_v=|\kappa_0|_v=1$,
$S^\circ(F_v)\cap G'(\mathfrak o_v)=S^\circ(\mathfrak o_v)$, and the
gauge forms $\omega_{G'}$, $\omega_{S^\circ}$ have good reduction.  For
$v\notin S_0$ let $X_v^\circ\subset S^\circ(F_v)\backslash G'(F_v)$ be
the image of $G'(\mathfrak o_v)$, and let $b_v$ be the contribution of
$X_v^\circ$ to $a_v$.  On $X_v^\circ$ the integrand is identically one:
$w_\mu p\in V'(\mathfrak o_v)$, $|\chi_\Delta(p)|_v=|\chi_R(p)|_v=1$, and
$y_{w_\mu p}\in\mathfrak o_v^\times\mathrm{GL}_2(\mathfrak o_v)$, so that
$\beta_{\xi_v^\circ,v}(y_{w_\mu p})=B_v(\xi_v^\circ)=1$.  Hence, with
$c_v=\zeta_{F_v}(1)$,
\[
 b_v=c_v^2\,
 \frac{\#G'(\mathbb F_{q_v})\,q_v^{-15}}
      {\#\mathrm{PGL}_2(\mathbb F_{q_v})\,q_v^{-3}}
 =(1-q_v^{-2})^2 .
\]
Here the final equality follows from the finite-field
orders of $G'$ and $\mathrm{PGL}_2$, together with
$c_v=(1-q_v^{-1})^{-1}$.  Since $Z'_\rho\simeq\mathbb G_m$ has trivial
$H^1$ by Hilbert's theorem~90, and $P=P_{2,2}$ has Levi factor
$\mathrm{GL}_2\times\mathrm{GL}_2$ and four-dimensional unipotent radical,
\[
 \#G'(\mathbb F_q)
 =\frac{q^4\,\#\mathrm{GL}_2(\mathbb F_q)^3}{q-1}
 =q^{15}(1-q^{-1})^2(1-q^{-2})^3,
\]
while $\#\mathrm{PGL}_2(\mathbb F_q)=q^3(1-q^{-2})$.
Thus $0<b_v\le a_v$ for $v\notin S_0$, and $\prod_{v\notin S_0}b_v$
converges absolutely.

The quotient $S^\circ(\mathbb A)\backslash G'(\mathbb A)$
is the restricted product of the local quotients with
respect to the $X_v^\circ$, and by Step~3 its measure is
$c_F^{-1}|\Delta_F|^{-6}$ times the restricted product
of the $d\dot\mu_v$. For a finite set $T\supset S_0$,
let $U_T$ be the set of points whose components lie in
$X_v^\circ$ for every $v\notin T$.

The integrand of $\mathcal I$ is the product of the
local integrands of the $a_v$, times $|C_\pi|$.
Thus Tonelli's theorem gives
\[
\mathcal I(U_T)
=
c_F^{-1}|\Delta_F|^{-6}|C_\pi|
\prod_{v\in T}a_v
\prod_{v\notin T}b_v.
\]
Since the $U_T$ form an increasing exhaustion,
monotone convergence yields
\begin{equation}\label{eq:second-abs-product}
\mathcal I
=
\sup_{T\supset S_0}
\left\{
c_F^{-1}|\Delta_F|^{-6}|C_\pi|
\prod_{v\in T}a_v
\prod_{v\notin T}b_v
\right\}.
\end{equation}
Here the equality is understood for nonnegative
extended integrals.

Suppose first that $a_v>0$ for every $v\in S_0$.
Since $\mathcal I<\infty$, $C_\pi\ne0$, and
$\prod_{v\notin S_0}b_v$ is positive and finite,
the preceding identity implies that every $a_v$
is finite. We may therefore rewrite it as
\[
\mathcal I
=
c_F^{-1}|\Delta_F|^{-6}|C_\pi|
\left(\prod_{v\notin S_0}b_v\right)
\left(\prod_{v\in S_0}a_v\right)
\sup_{T\supset S_0}
\prod_{v\in T\setminus S_0}\frac{a_v}{b_v}.
\]
In particular,
\[
\sup_{T\supset S_0}
\prod_{v\in T\setminus S_0}\frac{a_v}{b_v}
<\infty.
\]
Since $a_v/b_v\ge1$ for $v\notin S_0$,
it follows that
\[
\sum_{v\notin S_0}
\left(\frac{a_v}{b_v}-1\right)<\infty.
\]
Indeed, with $t_v=a_v/b_v-1\ge0$, every finite set $T$ of places
outside $S_0$ satisfies
$\sum_{v\in T}t_v\le\prod_{v\in T}(1+t_v)-1$, since all cross terms
are nonnegative.
In general, the local integrals depend only on the local data.
Fix a place $v$ and local data there satisfying
the hypotheses of the proposition.
At every other place in $S_0$, choose auxiliary
data whose nonnegative local integral is strictly
positive. Such data exist by choosing a Schwartz
function nonzero near $w_\mu$ and an inducing vector
on which the local diagonal functional is nonzero.
Indeed, one may choose a $K_v$-finite inducing
section whose inducing vector has nonzero local
diagonal period; since $y_{w_\mu}=I_2$, the value of
$\beta$ at the base point is that period, and the
corresponding local integrand is then positive on a
neighbourhood of $w_\mu$.
Proposition~\ref{prop:absolute-convergence} applies to the
resulting factorizable global data, since the auxiliary
sections may be chosen $K_v$-finite at every place
and hence belong to the class of flat sections
specified in \S\ref{subsec:zeta-integral}.  If $a_v=0$ there is nothing to prove;
otherwise, applying the preceding argument to these data shows that
$a_v<\infty$.  This conclusion includes the case in which the
original local integral might a priori have been infinite.  The same
choice shows that the displayed summability holds for the original
data, because $a_v$ and $b_v$ for $v\notin S_0$ do not depend on the
data at $S_0$.  Since $\prod_v|\kappa_0|_v=\prod_v|e_v|=1$, in every case
\[
\prod_v\left[2c_v^2
\mathscr Z'_{v,\mathrm{abs}}(s,s_0)\right]<\infty .
\]
This proves the asserted absolute local Euler factorization, rather
than merely convergence of the product of complex local integrals.
Moreover, since $b_v\le1$,
\[
 \sum_{v\notin S_0}(a_v-b_v)
 \le\sum_{v\notin S_0}\Bigl(\frac{a_v}{b_v}-1\Bigr)<\infty .
\]

Now let $z_v$ be the complex local integral obtained by removing the
absolute values from the integrand of $a_v$; by
\eqref{eq:second-local-orbit-measure} and Step~5,
$z_v=2c_v^2|\kappa_0|_v e_v\,\mathscr Z'_v(\Psi_v,f_v;s,s_0)$.  For
$v\notin S_0$ the contribution of $X_v^\circ$ to $z_v$ is again $b_v$,
so $|z_v-b_v|\le a_v-b_v$.  Together with
$\sum_{v\notin S_0}|b_v-1|<\infty$, this gives
\[
 \sum_{v\notin S_0}|z_v-1|<\infty,
\]
which is the absolute convergence of $\prod_v z_v$; finitely many
exceptional factors $z_v$, $v\in S_0$, may vanish.
The complex integrand in \eqref{eq:after-inner}
is dominated in absolute value by one half of
the nonnegative integrand defining $\mathcal I$,
which is integrable by the preceding argument.
Thus Fubini's theorem applies on each $U_T$,
and dominated convergence along the exhaustion
yields
\[
 Z'_E=\frac12c_F^{-1}|\Delta_F|^{-6}C_\pi
 \prod_v\left[
 2c_v^2|\kappa_0|_v e_v
 \mathscr Z'_v(\Psi_v,f_v;s,s_0)
 \right].
\]
The product formula removes the factors $|\kappa_0|_v$ and $e_v$,
proving \eqref{eq:second-expansion-product}.
At an unramified spherical place the $K_v$-average is trivial, so
Corollary~\ref{cor:unramified-normalization} applies and gives
$2c_v^2\mathscr Z'_v=1+O(q_v^{-2})$ at almost all places, uniformly in
its stated region. This provides an independent proof of the normal
convergence of the complex Euler product in that region. It is not used
above: the absolute integrability, the unfolding, and the Euler
factorization were established without it.
\end{proof}

\begin{remark}[Hermitian specialization of the bilinear formula]\label{rem:hermitian-datum}
Our Hermitian pairings are linear in the first argument.  Since complex conjugation and the contragredient have the same
automorphic realization for a self-dual representation with trivial
central character, choose
compatible antilinear intertwiners $j_v$ identifying $\pi_v$ with its
complex conjugate, and put
\[
 B_v(x\otimes y)=\langle x,j_v^{-1}y\rangle_v.
\]
At almost all finite places $v$, choose $j_v$ and the spherical
reference vector $e_v^\circ$ so that
\[
 j_v e_v^\circ=e_v^\circ,
 \qquad
 \langle e_v^\circ,e_v^\circ\rangle_v=1.
\]
With these normalizations, the invariant bilinear
functional $B_v$ satisfies
\[
 B_v(e_v^\circ\otimes e_v^\circ)=1,
\]
as required by the restricted tensor product
normalization in \eqref{eq:second-global-pairing}.
For $\xi_v=\phi_v\otimes j_v\phi_v$ this gives
\[
 \beta_{\xi_v,v}(y)=\langle\phi_v,\pi_v(y)\phi_v\rangle_v,
\]
with no Petersson-norm denominator.  For compatible pure tensors,
\[
 C_\pi=\frac{\langle\phi,\phi\rangle_{\mathrm{Tam}}}
                 {\prod_v\langle\phi_v,\phi_v\rangle_v}.
\]
Thus the bilinear identity specializes to the scalar Hermitian
orbital formula without changing any of its local normalization factors.
\end{remark}
\section{The local residue and the toric zeta integral}\label{sec:local-residue}

The two expansions of the global residue lead to different local objects.  The
unfolding over the generic orbit in Proposition~\ref{prop:second-expansion}
produces orbital integrals on the twelve-dimensional space $V'$, whereas the
boundary calculation of Section~\ref{sec:first-expansion} leads to the
Suzuki--Wakatsuki zeta integral on the eight-dimensional cube space $V$.
We now identify the relation between them.  The extra coordinate
$\ell\in F_v^2$ in $V'$ supplies a Mellin variable whose singularity at $s_0=1$
forces the integral onto the boundary $\ell=0$.  On this boundary, the
projective pencil of matrices associated with a cube is precisely a torus
orbit.  Keeping track of its Jacobian gives both the local constant and the
shift $s\mapsto s+\frac12$.

The local orbital integral itself need not converge at $s_0=1$.  The assertion
is instead that, after division by its Tate factor, it has a holomorphic value
there.  This requires a quantitative boundary estimate, proved in
Appendix~\ref{app:local-continuation}; in particular, it cannot be deduced from
local constancy of the test function.  We work throughout with the bilinear inducing data of
Section~\ref{sec:second-expansion}. The result applies
to every $\xi$ in a fixed finite-dimensional subspace
$W\subset\sigma_v$, consisting of $K_{M,v}$-finite vectors
at archimedean places, as in
Appendix~\ref{app:local-continuation}.
The uniform estimates there also justify the
$K_v$-average defining $\mathscr Z'_v$.

Fix a place $v$ of $F$, write $V_E(F_v)$ for the cube orbit associated with
the fixed global discriminant representative $d_E\in F^\times$, and let
$\xi\in W\subset\sigma_v=\pi_v\boxtimes\pi_v$, where $W$ is
finite-dimensional and consists of $K_{M,v}$-finite vectors if $v$ is
archimedean.  As in Section~\ref{sec:second-expansion},
put
\[
 \beta_{\xi,v}(y)=B_v\bigl((1\otimes\pi_v(y))\xi\bigr).
\]
For $\Psi_v\in\mathcal S(V'(F_v))$, define
\[
 \mathcal Z_v^{\mathrm{orb}}(\Psi_v,\xi;s,s_0)
 =\int_{V'_E(F_v)}\Psi_v(w)|\Delta(w)|_v^{s-1}
 |R(w)|_v^{s_0-2}\beta_{\xi,v}(y_w)\,dw_v.
\]
Here $y_w=\ell_2x_1-\ell_1x_2$ and $R(w)=-2\det(y_w)$ by
Lemma~\ref{lem:R-determinant}.  The local factor $\mathscr Z'_v$ of \eqref{eq:local-second}
is its compact average:
\begin{equation}\label{eq:Zprime-as-average}
 \mathscr Z'_v(\Psi_v,f_v;s,s_0)
 =\int_{K_v} \mathcal Z_v^{\mathrm{orb}}\bigl(\rho(\kappa)\Psi_v,
 \xi_{f_v}(\kappa);s,s_0\bigr)\,d\kappa.
\end{equation}

Let $V=M_2\oplus M_2$ and embed it into $V'$ by
$x\mapsto(0,x,0)$.  Choose
\[
 x_E=(1,\delta_E),\qquad \delta_E^2=d_E,
 \qquad \operatorname{tr}(\delta_E)=0,
\]
with the usual embedding of $E_v$ in $M_2(F_v)$.  The proof of
\cite[Proposition~2.3]{SW1} applies verbatim over $F_v$; with $D_v=M_2(F_v)$
it identifies
\[
 V_E(F_v)=\{x\in V(F_v):P(x)\in d_E(F_v^\times)^2\}
 =x_E\cdot\rho_V(\mathrm{GL}_2(F_v)^3).
\]
We record directly that the invariant $P$ used there is the restriction of
our discriminant $\Delta$.  On $\ell=0$, the Pfaffian pencil is
$-\det(u x_1+v x_2)$, so
\[
 \Delta(0,x,\lambda)
 =4\bigl(\langle x_1,x_2\rangle^2-\det x_1\det x_2\bigr),
\]
where $\langle x_1,x_2\rangle=\tfrac12\,p(x_1,x_2)$, with $p$ the polar
form of the determinant from \S\ref{subsec:group}, so that
$\langle x,x\rangle=\det x$; this is
independent of $\lambda$.  Write $x^\iota$ for the matrix adjugate and set
$A=x_1x_2^\iota$.  Since
$x_2x_1^\iota=A^\iota$, $\operatorname{tr}A=2\langle x_1,x_2\rangle$,
and $A-A^\iota=2A-(\operatorname{tr}A)I_2$, we obtain
\[
 \begin{aligned}
 P(x)&=-\det(x_1x_2^\iota-x_2x_1^\iota)\\
 &= (\operatorname{tr}A)^2-4\det A
 =4\bigl(\langle x_1,x_2\rangle^2-\det x_1\det x_2\bigr)
 =\Delta(x).
 \end{aligned}
\]
This also fixes the discriminant convention in the change of variables below.

Let $T_v=F_v^\times\backslash E_v^\times$, with quotient measure
$dh_v=dh_{E_v}/d^\times z$ as in \cite[(4.4)]{SW1}.  For
$x=x_E\cdot\rho_V(g_1,g_2,g_3)$, set
\[
 \alpha^B_{\xi,v}(g_1,g_2)
 =\int_{T_v}\beta_{\xi,v}(g_1^{-1}tg_2)\,dh_v(t).
\]
(The integral converges absolutely: $T_v$ is compact when
$E_v$ is a field, and in the split case this follows from the
matrix-coefficient bound (U1); see the proof of
Theorem~\ref{thm:local-residue}.)
Whenever the defining toric integral converges
absolutely, its value is independent of the choice
of $(g_1,g_2,g_3)$ representing $x$, for arbitrary
$\xi\in\sigma_v$; this is the bilinear analogue of
\cite[Remark~4.7]{SW1}. By the description of the
stabilizer of $x_E$ in \cite[Proposition~2.3]{SW1}, the
representatives differ by the connected stabilizer and
by the element $(\gamma,\gamma,\operatorname{diag}(1,-1))$
of the non-identity component, where $\gamma$ normalizes
$T_v$ and induces the nontrivial automorphism
$t\mapsto\bar t$ of $E_v/F_v$. For the connected stabilizer,
replacing $(g_1,g_2)$ by $(t_1g_1,t_2g_2)$ with
$t_1,t_2\in T_v$ changes the integral only by a translation
of $t$, which preserves $dh_v$. For the non-identity
component,
\[
\begin{aligned}
\alpha^B_{\xi,v}(\gamma g_1,\gamma g_2)
&=\int_{T_v}\beta_{\xi,v}(g_1^{-1}\gamma^{-1}t\gamma\,g_2)\,dh_v(t)\\
&=\int_{T_v}\beta_{\xi,v}(g_1^{-1}\bar t\,g_2)\,dh_v(t)
=\alpha^B_{\xi,v}(g_1,g_2),
\end{aligned}
\]
since $dh_v$ is invariant under the Galois automorphism.
Hence $\alpha^B_{\xi,v}$ descends to
$V_E(F_v)$. In the region of absolute
convergence, \cite[(4.7)]{SW1} applies to the function
$\alpha^B_{\xi,v}|P|_v^{\mathfrak{s}-2}\Phi$.  We therefore define
\[
 \mathcal Z^{\mathrm{SW}}_{E_v}(\Phi,\xi;\mathfrak{s})
 =\frac{2c_v}{L_v(1,\eta_v)^2}
 \int_{V_E(F_v)}\alpha^B_{\xi,v}(g_1,g_2)
 |P(x)|_v^{\mathfrak{s}-2}\Phi(x)\,dx_v.
\]
This is the bilinear extension, in the representation
variables, of the local zeta integral of
\cite[\S4.4]{SW1}, with exactly the same prefactor
$2c_v/L_v(1,\eta_v)^2$.
For $\xi=\phi_v\otimes j_v\phi_v$
(Remark~\ref{rem:hermitian-datum}), it agrees
with the local Suzuki--Wakatsuki integral
$Z_{E_v}$ itself, without division by a
Petersson norm.  Equivalently, if $L_v(\xi)=\int_{T_v}\beta_{\xi,v}(t)\,dh_v(t)$,
then $\alpha^B_{\xi,v}(g_1,g_2)
=L_v((\pi_v(g_1)\otimes\pi_v(g_2))\xi)$, the bilinear
functional appearing in Section~\ref{sec:comparison}.

Let
\[
 \mathcal T_v(\sigma)=\int_{F_v^\times}
 \Phi_{0,v}(z)|z|_v^\sigma\,d^\times z
\]
be the standard local Tate integral.  At finite places,
$\Phi_{0,v}=\mathbf 1_{\mathfrak o_v}$ and
$d^\times z=c_v\,dz/|z|_v$, where
$c_v=(1-q_v^{-1})^{-1}=\zeta_{F_v}(1)$; thus
$\mathcal T_v(\sigma)=\zeta_v(\sigma)$.  At archimedean places,
$d^\times z=dz/|z|_v$, where $dz$ is the
self-dual additive Haar measure with respect to
the additive character fixed in
\S\ref{subsec:measures}
(Lebesgue measure at real places and $2\,dx\,dy$
at complex places).  With $\Phi_{0,v}(x)=e^{-\pi x^2}$ at real places
and $\Phi_{0,v}(z)=e^{-2\pi|z|^2}$ at complex places,
\[
 \mathcal T_v(\sigma)=\pi^{-\sigma/2}\Gamma(\sigma/2),
 \qquad\text{respectively}\qquad
 \mathcal T_v(\sigma)=(2\pi)^{1-\sigma}\Gamma(\sigma),
\]
and we take these as $\zeta_v(\sigma)$.  This is the normalization
arising from Tate's local integrals for these measures; at complex places
it differs from
$\Gamma_{\mathbb C}(\sigma)=2(2\pi)^{-\sigma}\Gamma(\sigma)$ by the
factor $\pi$.  In particular $\zeta_{F_v}(1)=1$ at archimedean places,
so $c_v=\zeta_{F_v}(1)$ at every place.  This is the constant $c_v$ of
\cite[\S3.1]{SW1}.

\begin{theorem}[Local residue identity]\label{thm:local-residue}
Let $W\subset\sigma_v$ be a fixed finite-dimensional
subspace, consisting of $K_{M,v}$-finite vectors when $v$
is archimedean, and let $\xi\in W$.
Fix $s$ with $\Re s>s_1$, where $s_1$ is as in
Appendix~\ref{app:local-continuation}.  Then
\[
 R_v^{\mathrm{orb}}(\Psi_v,\xi;s,s_0)
 :=\frac{\mathcal Z_v^{\mathrm{orb}}(\Psi_v,\xi;s,s_0)}{\mathcal T_v(2s_0-2)}
\]
extends holomorphically to a neighbourhood of $s_0=1$, and
\[
 R_v^{\mathrm{orb}}(\Psi_v,\xi;s,1)
 =\frac{|2|_v^{-1}|4d_E|_v^{1/2}L_v(1,\eta_v)}{2c_v^2}
 \mathcal Z^{\mathrm{SW}}_{E_v}\left(\Phi^{\mathrm{loc}}_{\Psi_v},\xi;
 s+\frac12\right),
\]
where
\[
 \Phi^{\mathrm{loc}}_{\Psi_v}(x)
 =\int_{F_v^2}\Psi_v(0,x,\lambda)\,d\lambda.
\]
\end{theorem}

\begin{proof}
We first extract the boundary contribution, and then identify the remaining
projective integral with the torus integral.  Both parts use the measure
conventions fixed in Section~\ref{subsec:measures}.

\emph{1. The scaling variable and the projective measure.}
Choose a measurable representative $\widehat\ell$ of each line in
$\mathbb P^1(F_v)$ and write $\ell=z\widehat\ell$.  Define the measure
$d\mu$ on the projective lines, relative to this section, by
$d\ell=|z|_v^2d^\times z\,d\mu(\widehat\ell)$.  Equivalently, for $F$ on $F_v^2\setminus\{0\}$ satisfying
$F(zu)=|z|_v^{-2}F(u)$ for $z\in F_v^\times$
(homogeneity of degree $-2$ with respect to $|\cdot|_v$),
\begin{equation}\label{eq:P1-quotient}
 \int_{\mathbb P^1(F_v)}F(\widehat\ell)\,d\mu(\widehat\ell)
 =\int_{(F_v^2\setminus\{0\})/F_v^\times}
 \frac{F(u)\,du}{d^\times z}.
\end{equation}
Since $d(zu)=|z|_v^2\,du$, the density $F(u)\,du$ is invariant
under scalar multiplication, and the right-hand side denotes
integration against the induced quotient density relative to
$d^\times z$.  If the section is replaced by
$\widehat\ell'=c\widehat\ell$, then
$d\mu'=|c|_v^2d\mu$ and
$F(\widehat\ell')=|c|_v^{-2}F(\widehat\ell)$; hence the
projective integral is independent of the choice of section.

Put $\widehat y=\widehat\ell_2x_1-\widehat\ell_1x_2$.  Then
$y_w=z\widehat y$, $R(w)=-2z^2\det\widehat y$, and
$\beta_{\xi,v}(z\widehat y)=\beta_{\xi,v}(\widehat y)$ because the
central character of $\pi_v$ is trivial.  Combining the factor $|z|_v^2$ in
$d\ell=|z|_v^2d^\times z\,d\mu(\widehat\ell)$
with the power of $R$ gives
\[
 \mathcal Z_v^{\mathrm{orb}}=\int_{F_v^\times}|z|_v^{2s_0-2}
 \mathcal F_v(z;s,s_0)\,d^\times z,
\]
where
\[
 \begin{aligned}
 \mathcal F_v(z;s,s_0)=\int_{\mathbb P^1(F_v)}\int_{x,\lambda}
 &\Psi_v(z\widehat\ell,x,\lambda)
 |\Delta(z\widehat\ell,x,\lambda)|_v^{s-1}
 \mathbf1_{V'_E}(z\widehat\ell,x,\lambda)\\
 &\quad\times|2\det\widehat y|_v^{s_0-2}
 \beta_{\xi,v}(\widehat y)\,dx\,d\lambda\,d\mu.
 \end{aligned}
\]
On the boundary $\ell=0$, the discriminant and the
orbit condition depend only on the cube coordinate
$x$: one has $\Delta(0,x,\lambda)=P(x)$, and the
boundary orbit condition reduces to $x\in V_E(F_v)$
away from the discriminant-zero locus. Consequently,
the $\lambda$-integration is absorbed into
$\Phi^{\mathrm{loc}}_{\Psi_v}(x)
=\int_{F_v^2}\Psi_v(0,x,\lambda)\,d\lambda$,
while the projective direction $\widehat\ell$
remains in the integral $I(x;s_0)$.  Thus the boundary value is
\[
 \mathcal F_v(0;s,s_0)=\int_{V_E(F_v)}
 \Phi^{\mathrm{loc}}_{\Psi_v}(x)|P(x)|_v^{s-1}
 I(x;s_0)\,dx_v,
\]
with
\[
 I(x;s_0)=\int_{\mathbb P^1(F_v)}
 \frac{\beta_{\xi,v}(\widehat\ell_2x_1-\widehat\ell_1x_2)}
 {|2\det(\widehat\ell_2x_1-\widehat\ell_1x_2)|_v^{2-s_0}}
 \,d\mu(\widehat\ell).
\]

\emph{2. Holomorphy after removing the Tate factor.}
Fix $s$ as in the theorem. By
Lemma~\ref{lem:local-continuation}, there exist
$\varepsilon>0$ and a neighbourhood $U$ of $1$
such that
\[
 \bigl|\mathcal F_v(z;s,s_0)
       -\mathcal F_v(0;s,s_0)\bigr|
 \ll |z|_v^\varepsilon
 \qquad (|z|_v\le1,\ s_0\in U),
\]
uniformly in $s_0$. Moreover,
$\mathcal F_v(0;s,s_0)$ is holomorphic on $U$: its integrand depends
holomorphically on $s_0$, and the absolute convergence in
Lemma~\ref{lem:local-continuation} is uniform on $U$.

Put $\tau=2s_0-2$ and
\[
 A_v(\tau)
 =\int_{|z|_v\le1}|z|_v^\tau\,d^\times z,
 \qquad \Re\tau>0.
\]
Subtracting the boundary value in the small-$z$
region gives
\[
 \mathcal Z_v^{\mathrm{orb}}(\Psi_v,\xi;s,s_0)
 =\mathcal F_v(0;s,s_0)A_v(\tau)+H_v(s_0),
\]
where
\[
 \begin{aligned}
 H_v(s_0)
 ={}&
 \int_{|z|_v\le1}|z|_v^\tau
 \bigl(\mathcal F_v(z;s,s_0)
       -\mathcal F_v(0;s,s_0)\bigr)\,d^\times z\\
 &+\int_{|z|_v>1}|z|_v^\tau
       \mathcal F_v(z;s,s_0)\,d^\times z.
 \end{aligned}
\]
The first integral converges locally uniformly when
$\Re\tau>-\varepsilon$, by the boundary estimate,
and the second is holomorphic near $s_0=1$
by the Schwartz estimates. Thus $H_v$ is
holomorphic near $s_0=1$.

On the other hand, the local Tate factor satisfies
\[
 \mathcal T_v(\tau)=A_v(\tau)+h_v(\tau),
\]
where $h_v$ is holomorphic near $\tau=0$.
Indeed, at finite places
$\Phi_{0,v}=\mathbf1_{\mathfrak o_v}$, while
at archimedean places
$\Phi_{0,v}(z)-1=O(|z|_v^\delta)$ near zero
for some $\delta>0$, and the Gaussian gives
rapid decay at infinity. Hence $A_v(\tau)$
and $\mathcal T_v(\tau)$ have the same
simple pole and principal part at $\tau=0$.

Consequently,
\[
 \begin{aligned}
 R_v^{\mathrm{orb}}(\Psi_v,\xi;s,s_0)
 &=\frac{\mathcal Z_v^{\mathrm{orb}}
          (\Psi_v,\xi;s,s_0)}
         {\mathcal T_v(\tau)}\\
 &=\mathcal F_v(0;s,s_0)
   +\frac{H_v(s_0)
          -\mathcal F_v(0;s,s_0)h_v(\tau)}
         {\mathcal T_v(\tau)}.
 \end{aligned}
\]
Since $\mathcal T_v(\tau)$ has a simple pole
at $\tau=0$, its reciprocal is holomorphic
and vanishes there. It follows that
$R_v^{\mathrm{orb}}$ extends holomorphically
to $s_0=1$, with
\[
 R_v^{\mathrm{orb}}(\Psi_v,\xi;s,1)
 =\mathcal F_v(0;s,1).
\]
This is the $s_0$-dependent analogue of
Lemma~\ref{lem:tate-quotient}, with the holomorphic
dependence of $\mathcal F_v(0;s,s_0)$ and the
uniform remainder estimate taken into account.

The quantitative boundary estimate is essential
even at nonarchimedean places: although the
Schwartz function is locally constant, the
discriminant depends on $\lambda$ away from
$\ell=0$, and its square class need not be
uniformly locally constant as $z\to0$.

\emph{3. The matrix pencil and its discriminant.}
Write $x=x_E\cdot\rho_V(g_1,g_2,g_3)$.  In the cube action, the inner
factors send $x_i$ to $g_1^{-1}x_i g_2$, and $g_3$ acts linearly on
the pair.  The pencil $\{u_2x_1-u_1x_2:u\in F_v^2\}$ is therefore
$g_1^{-1}E_vg_2$.  More precisely, relative to the standard basis of
$F_v^2$ and the fixed basis $(1,\delta_E)$ of $E_v$,
the outer factor determines a linear isomorphism
$\gamma:F_v^2\to E_v$ such that
\[
 u_2x_1-u_1x_2=g_1^{-1}\gamma(u)g_2,
 \qquad
 \det(u_2x_1-u_1x_2)=\lambda_x N_{E_v/F_v}(\gamma(u)),
 \quad \lambda_x=\frac{\det g_2}{\det g_1}.
\]
We compute the discriminant of the binary quadratic form
$u\mapsto\det(u_2x_1-u_1x_2)$ in two ways. First, expanding
the determinant directly gives
\[
 \begin{aligned}
 \det(u_2x_1-u_1x_2)
 &=u_2^2\det x_1-2u_1u_2\langle x_1,x_2\rangle
 +u_1^2\det x_2,\\
 \operatorname{disc}\bigl(\det(u_2x_1-u_1x_2)\bigr)
 &=4\langle x_1,x_2\rangle^2-4\det x_1\det x_2
 =\Delta(x).
 \end{aligned}
\]
On the other hand, using the norm-form expression
$\lambda_x N_{E_v/F_v}(\gamma(u))$, we have
$N(a+b\delta_E)=a^2-d_Eb^2$, whose discriminant
is $4d_E$. Multiplication by $\lambda_x$ and the
linear substitution $\gamma$ multiply the discriminant
by $\lambda_x^2$ and $(\det\gamma)^2$, respectively.
Equating the two discriminants yields
$\Delta(x)=4d_E(\lambda_x\det\gamma)^2$, and hence
\begin{equation}\label{eq:local-pencil-Jacobian}
 |\lambda_x\det\gamma|_v
 =\left|\frac{\Delta(x)}{4d_E}\right|_v^{1/2}.
\end{equation}

\emph{4. Conversion of the projective measure.}
At $s_0=1$ the integrand of $I(x;s_0)$ is homogeneous of degree $-2$.
By \eqref{eq:P1-quotient}, the substitution $t=\gamma(u)$ and
$du=|\det\gamma|_v^{-1}\,da\,db$ give
\[
 I(x;1)=\frac1{|2|_v|\lambda_x|_v|\det\gamma|_v}
 \int_{E_v^\times/F_v^\times}
 \beta_{\xi,v}(g_1^{-1}tg_2)
 \frac{da\,db/|N_{E_v/F_v}(t)|_v}{d^\times z}.
\]
In the basis $(1,\delta_E)$, the measure of \cite[(4.4)]{SW1} is
\[
 dh_{E_v}(a+b\delta_E)
 =c_vL_v(1,\eta_v)\frac{da\,db}{|N_{E_v/F_v}(a+b\delta_E)|_v},
 \qquad dh_v=dh_{E_v}/d^\times z.
\]
Here $dh_{E_v}$ is the multiplicative Haar measure on
$E_v^\times$. The measure on $E_v^\times$ given by
$da\,db/|N_{E_v/F_v}(t)|_v$ is invariant under
multiplication by $E_v^\times$, since multiplication
by $t$ has determinant $N_{E_v/F_v}(t)$ on the
underlying $F_v$-vector space. The measure $dh_v$
is the quotient Haar measure on
$F_v^\times\backslash E_v^\times$, relative to
the fixed measure $d^\times z$ on $F_v^\times$; the quotient notation in
the preceding display is to be read in the same sense.
No separately chosen self-dual measure on $E_v$ is introduced.
Using \eqref{eq:local-pencil-Jacobian}, we obtain the precise identity
\begin{equation}\label{eq:local-projective-torus}
 I(x;1)=\frac{|2|_v^{-1}|4d_E|_v^{1/2}}
 {c_vL_v(1,\eta_v)}
 |\Delta(x)|_v^{-1/2}\alpha^B_{\xi,v}(g_1,g_2).
\end{equation}
The homogeneity of degree $-2$ ensures that
the projective integral is independent of the
choice of representatives. Moreover, the
combination $|\lambda_x\det\gamma|_v$ is
determined by $\Delta(x)$, so the resulting
identity is independent of the chosen
coordinates of the pencil.

For completeness, the integral in
\eqref{eq:local-projective-torus} is absolutely
convergent. When $E_v$ is a field, its norm form
has no projective isotropic direction, and
$T_v=F_v^\times\backslash E_v^\times$ is compact.
When $E_v$ splits, the norm form has exactly
two isotropic lines, corresponding under $\gamma$
to the two ends of $T_v\simeq F_v^\times$.

To verify convergence in the split case, identify
$E_v$ with $F_v\times F_v$ and parametrize
$T_v$ by $t(a)=(a,1)$, $a\in F_v^\times$.
For fixed $g_1,g_2$, the matrix norm satisfies
\[
 \frac{|\det(g_1^{-1}t(a)g_2)|_v}
      {\|g_1^{-1}t(a)g_2\|^2}
 \asymp_{g_1,g_2}
 \min(|a|_v,|a|_v^{-1}).
\]
Consequently, the matrix-coefficient estimate
(U1) in Appendix~\ref{app:local-continuation}
gives
\[
 \bigl|\beta_{\xi,v}(g_1^{-1}t(a)g_2)\bigr|
 \ll_{g_1,g_2,\xi}
 \min(|a|_v,|a|_v^{-1})^{\beta_0}
 (1+|\log|a|_v|)^{N_0},
 \qquad \beta_0>0.
\]
This bound is integrable against $d^\times a$
at both $a=0$ and $a=\infty$. Thus the toric
integral is absolutely convergent also in the
split case. In particular, matrix-coefficient
decay is essential here: a constant coefficient
would give a divergent integral.

As a normalization check, consider an unramified
finite place where
$\mathfrak o_{E_v}
 =\mathfrak o_v+\mathfrak o_v\delta_E$,
and normalize $da\,db$ so that
$\operatorname{vol}(\mathfrak o_v^2)=1$.
The maximal compact subgroup
$\mathfrak o_{E_v}^{\times}$ has volume $1$
with respect to $dh_{E_v}$. Indeed, in the
inert case,
\[
 c_vL_v(1,\eta_v)(1-q_v^{-2})=1,
\]
whereas in the split case,
\[
 c_vL_v(1,\eta_v)(1-q_v^{-1})^2=1.
\]
Since $\mathfrak o_v^\times$ also has volume
$1$ with respect to $d^\times z$, the induced
measure on the maximal compact subgroup of
$T_v$ has volume $1$.

\emph{5. The shift and the constant.}
Substituting \eqref{eq:local-projective-torus} into the boundary integral
multiplies $|\Delta(x)|_v^{s-1}$ by $|\Delta(x)|_v^{-1/2}$.
Since $P(x)=\Delta(x)$, the resulting exponent is
\[
 |P(x)|_v^{s-3/2}=|P(x)|_v^{(s+1/2)-2}.
\]
Thus the parameter in the Suzuki--Wakatsuki zeta integral
is $s+\frac12$. By its defining normalization,
\[
 \begin{aligned}
 \mathcal F_v(0;s,1)
 &=\frac{|2|_v^{-1}|4d_E|_v^{1/2}}
 {c_vL_v(1,\eta_v)}
 \frac{L_v(1,\eta_v)^2}{2c_v}
 \mathcal Z^{\mathrm{SW}}_{E_v}
 \left(\Phi^{\mathrm{loc}}_{\Psi_v},\xi;s+\frac12\right)\\
 &=\frac{|2|_v^{-1}|4d_E|_v^{1/2}L_v(1,\eta_v)}{2c_v^2}
 \mathcal Z^{\mathrm{SW}}_{E_v}
 \left(\Phi^{\mathrm{loc}}_{\Psi_v},\xi;s+\frac12\right).
 \end{aligned}
\]
By Step~2,
$R_v^{\mathrm{orb}}(\Psi_v,\xi;s,1)
=\mathcal F_v(0;s,1)$.
This proves the claimed identity for every $s$ with
$\Re s>s_1$.
\end{proof}

\begin{remark}[Bilinearity]\label{rem:bilinear-residue}
The proof uses $\beta_{\xi,v}$ only through its invariance
under scalar matrices, its integral along the pencil, and
the matrix-coefficient bound (U1). These properties apply
to general $\xi\in\sigma_v$ for which (U1) holds, with
uniform constants on fixed finite-dimensional subspaces
as established in Lemma~\ref{lem:local-continuation};
no polarization argument is needed. Specializing to
$\xi=\phi_v\otimes j_v\phi_v$ as in
Remark~\ref{rem:hermitian-datum}, one recovers the
Hermitian scalar form of Theorem~\ref{thm:local-residue}.
\end{remark}

\begin{remark}[Uniformity in the inducing data]\label{rem:uniformity}
Lemma~\ref{lem:local-continuation} controls the dependence
on $\xi$ through a norm on a fixed finite-dimensional
subspace of $\sigma_v$, and the dependence on $\Psi_v$
through finitely many Schwartz seminorms, including first
$\ell$-derivatives at archimedean places. In particular,
these estimates are uniform over the compact families
arising from the $K_v$-action on $\Psi_v$ and the
$K_v$-finite inducing data. These are precisely the
uniformity properties (U1) and (U2) of
Appendix~\ref{app:local-continuation} needed to justify
the subsequent compact averaging.
\end{remark}

\begin{corollary}[$K_v$-averaged local residue]
\label{cor:averaged-local-residue}
Let $\Psi_v$ and the $K_v$-finite standard section $f_v$
be as in Proposition~\ref{prop:second-expansion}.
Fix $s$ with $\Re s>s_1$. Then
\[
 R_v(s,s_0):=
 \frac{\mathscr Z'_v(\Psi_v,f_v;s,s_0)}
 {\mathcal T_v(2s_0-2)}
\]
extends holomorphically to a neighbourhood of $s_0=1$,
and
\[
 R_v(s,1)
 =
 \frac{|2|_v^{-1}|4d_E|_v^{1/2}
 L_v(1,\eta_v)}{2c_v^2}
 \int_{K_v}
 \mathcal Z^{\mathrm{SW}}_{E_v}
 \left(
 \Phi^{\mathrm{loc}}_{\rho(\kappa)\Psi_v},
 \xi_{f_v}(\kappa);s+\frac12
 \right)\,d\kappa.
\]
\end{corollary}

\begin{proof}
For $\Re s_0>1$,
equation~\eqref{eq:Zprime-as-average} gives
\[
 R_v(s,s_0)
 =
 \int_{K_v}
 R_v^{\mathrm{orb}}
 \bigl(
 \rho(\kappa)\Psi_v,\xi_{f_v}(\kappa);
 s,s_0
 \bigr)\,d\kappa.
\]
We verify that the continuation estimates of
Appendix~\ref{app:local-continuation} apply
uniformly to the data in this integral.

Since $f_v$ is $K_v$-finite, the space
\[
 U_f=\operatorname{span}_{\mathbb C}
 \{\rho(\kappa)f_v:\kappa\in K_v\}
\]
is finite-dimensional. By the definition of
$\xi_{f_v}(\kappa)$, the vectors
$\xi_{f_v}(\kappa)$ lie in the image of $U_f$
under the linear evaluation map
$f\mapsto\xi_f(1)$. Hence
\[
 W=\operatorname{span}_{\mathbb C}
 \{\xi_{f_v}(\kappa):\kappa\in K_v\}
 \subset\sigma_v
\]
is finite-dimensional.

Moreover, $W$ is $K_{M,v}$-stable. Indeed,
for $k_M\in K_{M,v}$, the normalized inducing
law, together with
$\delta_P(k_M)=u(k_M)=1$, gives
\[
 \sigma_v(k_M)\xi_{f_v}(\kappa)
 =\xi_{f_v}(k_M\kappa)\in W.
\]
Thus every vector in $W$ is $K_{M,v}$-finite,
including at archimedean places. Since
$\kappa\mapsto\xi_{f_v}(\kappa)$ is continuous
and $K_v$ is compact, these vectors have
uniformly bounded norms in $W$.

Likewise, the family
\[
 \{\rho(\kappa)\Psi_v:\kappa\in K_v\}
\]
is compact in the Schwartz space. Consequently,
the finitely many Schwartz seminorms occurring
in (U2), including the required first
$\ell$-derivatives at archimedean places,
are uniformly bounded over this family.

By Lemma~\ref{lem:local-continuation},
Corollary~\ref{cor:local-continuation}, and
Remark~\ref{rem:uniformity}, there is a
neighbourhood $D$ of $s_0=1$, independent
of $\kappa$, on which
\[
 s_0\longmapsto
 R_v^{\mathrm{orb}}
 \bigl(
 \rho(\kappa)\Psi_v,\xi_{f_v}(\kappa);
 s,s_0
 \bigr)
\]
is holomorphic. Moreover, for every compact
subset $D'\subset D$,
\[
 \sup_{\kappa\in K_v}
 \sup_{s_0\in D'}
 \left|
 R_v^{\mathrm{orb}}
 \bigl(
 \rho(\kappa)\Psi_v,\xi_{f_v}(\kappa);
 s,s_0
 \bigr)
 \right|<\infty.
\]
The integral over $K_v$ therefore defines
a holomorphic function on $D$, by Fubini's
theorem and Morera's theorem. On the overlap
with $\Re s_0>1$, it agrees with
$R_v(s,s_0)$, and hence gives its holomorphic
continuation to $D$.

In particular, evaluation at $s_0=1$
may be passed through the compact integral:
\[
 R_v(s,1)
 =
 \int_{K_v}
 R_v^{\mathrm{orb}}
 \bigl(
 \rho(\kappa)\Psi_v,\xi_{f_v}(\kappa);
 s,1
 \bigr)\,d\kappa.
\]
Applying Theorem~\ref{thm:local-residue}
to each integrand yields
\[
 R_v(s,1)
 =
 \frac{|2|_v^{-1}|4d_E|_v^{1/2}
 L_v(1,\eta_v)}{2c_v^2}
 \int_{K_v}
 \mathcal Z^{\mathrm{SW}}_{E_v}
 \left(
 \Phi^{\mathrm{loc}}_{\rho(\kappa)\Psi_v},
 \xi_{f_v}(\kappa);s+\frac12
 \right)\,d\kappa,
\]
as claimed. At nonarchimedean places,
smoothness and $K_v$-finiteness make the
relevant compact families locally constant
and hence finite-valued, so the same
conclusion also follows directly.
\end{proof}

The choice of a single global representative $x_E$ is
deliberate. As explained in \cite[(4.8)]{SW1},
changing the representative introduces a corresponding
factor in the local cube integral. With
$d_E\in F^\times$ fixed globally, the product formula
gives
\[
 \prod_v |2|_v^{-1}|4d_E|_v^{1/2}
 =\prod_v |d_E|_v^{1/2}=1.
\]
Thus the geometric factors appearing in
Theorem~\ref{thm:local-residue} cancel globally.
This cancellation is independent of the normalization
of the toric period and the Petersson pairing, whose
contributions will be accounted for separately in
Section~\ref{sec:comparison}.

Theorem~\ref{thm:local-residue} identifies the
normalized residue at $s_0=1$ for every place $v$,
but does not determine the full two-variable local
integral. In the next section, we compute the latter
at unramified places. Specializing that computation
to $s_0=1$ recovers the present local identity and
provides the Euler factors required for the global
comparison in Section~\ref{sec:comparison}.

\section{The unramified local calculation}\label{sec:unramified}

The local residue identity of Section~\ref{sec:local-residue} relates the orbital integral arising from the second expansion to the toric zeta integral of Suzuki--Wakatsuki.  That identity holds for general local test data, but it identifies only the contribution at the distinguished point $s_0=1$.  For the global comparison we need more: at almost every place we must know the orbital integral as a function of both complex variables.  The purpose of this section is to compute that function for spherical data.

There are two reasons for retaining the second variable
in the calculation. First, the theorem isolates the
local factor $\zeta_v(2s_0-2)$. Together with the
Tate factors $\mathcal T_v(2s_0-2)$ of
Section~\ref{sec:local-residue} at the remaining
places, these multiply to the completed zeta
function $\Lambda_F(2s_0-2)$, whose pole at
$s_0=1$ accounts for the residue of the second
expansion (Section~\ref{sec:comparison}).
Second, the factors depending only on $s_0$,
namely $L_v(s_0-\frac12,\pi_E)$ and
$L_v(2s_0-1,\pi,\mathrm{Ad})^{-1}$, assemble
into partial automorphic $L$-functions.
Once these are extracted, the remaining factors
involving $s$ form an Euler product that, for
$\Re s$ sufficiently large, converges
normally on a neighbourhood of $s_0=1$ as well
as in the initial region; this is established in
Section~\ref{sec:comparison}. This separation
permits the global comparison to reach $s_0=1$
and expose the central value. Neither conclusion
follows merely by specializing the local residue
identity.

Fix a finite place $v$ of odd residue characteristic.  Write $\mathfrak o=\mathfrak o_v$, $\varpi=\varpi_v$, and $q=\#(\mathfrak o/\varpi\mathfrak o)$. Let $\operatorname{ord}=\operatorname{ord}_v$ denote the normalized valuation of $F_v$, so that $\operatorname{ord}(\varpi)=1$.  We suppose that the quadratic algebra $E_v$ is unramified, either split or a field, that $d_E\in\mathfrak o^\times$, and we let $\eta_v$ denote the quadratic character of $F_v^\times$ associated with $E_v/F_v$, trivial in the split case and unramified quadratic in the inert case.  Thus $\eta_v(\varpi)=1$ in the split case and $-1$ in the inert case.  We also assume that $\pi_v$ is unramified, has trivial central character, and is equipped with its normalized spherical matrix coefficient.  The Schwartz function is $\Psi_v=\mathbf 1_{V'(\mathfrak o)}$, and the inducing vector $\xi_v^\circ$ is spherical, normalized by $B_v(\xi_v^\circ)=1$. Thus $\beta_{\xi_v^\circ,v}$ is the normalized spherical matrix coefficient of $\pi_v$.

Let $\alpha,\alpha^{-1}$ be the normalized Satake parameters of $\pi_v$, and put
\[
a_\pi=\alpha+\alpha^{-1},\qquad \lambda_v=q^{1/2}a_\pi,\qquad
c_v=\zeta_v(1)=(1-q^{-1})^{-1}.
\]
Let $f_v$ be the normalized spherical section, as in
Section~\ref{sec:second-expansion}. Then
$\rho(\kappa)\Psi_v=\Psi_v$ and $\xi_{f_v}(\kappa)=\xi_v^\circ$
for all $\kappa\in K_v$, so the integrand of
\eqref{eq:Zprime-as-average} is independent of $\kappa$.
Since $\operatorname{vol}(K_v)=1$, we obtain
\[
Z'_v(s,s_0)
:=\mathscr Z'_v(\Psi_v,f_v;s,s_0)
=\mathcal Z_v^{\mathrm{orb}}(\Psi_v,\xi_v^\circ;s,s_0).
\]  The additive Haar measures are those fixed in
\S\ref{subsec:measures}; in particular,
$\operatorname{vol}(\mathfrak o,dx_v)=1$, as at every
finite place.

\begin{theorem}\label{thm:unramified-factor}
Under the preceding hypotheses, set
\[
 c'_v=\frac12(1-q^{-1})^2(1-q^{-2})^2
\]
and
\[
\begin{aligned}
\mathcal D_E(s,s_0)
={}&\bigl(1+q^{-(2s+2s_0-1)}\bigr)
       \bigl(1+q^{-(2s+2s_0-2)}\bigr)\\
 &-\eta_v\lambda_vq^{-(2s+s_0)}
       \bigl(1+q^{-(2s_0-2)}\bigr).
\end{aligned}
\]
Then
\[
\begin{aligned}
Z'_v(s,s_0)={}&c'_v\,\zeta_v(2s_0-2)
 \frac{L_v(s_0-\frac12,\pi_E)}{L_v(2s_0-1,\pi,\mathrm{Ad})}\\
&\times\zeta_v(2s)
 \frac{\zeta_v(2s+4s_0-4)}{\zeta_v(2s+2s_0-2)}
 L_v(2s+2s_0-2,\pi,\mathrm{Ad})\,\mathcal D_E(s,s_0).
\end{aligned}
\]
The equality holds initially in the domain of absolute convergence and, equivalently, as an identity of rational functions in the local exponential variables.  In particular it gives a meromorphic continuation of $Z'_v(s,s_0)$ to $\mathbb C^2$.
\end{theorem}

Before proving the theorem, we explain the geometric
reductions underlying the calculation. The integral
defining $Z'_v$ is supported on the locus in the
prehomogeneous vector space where the discriminant
$\Delta$ belongs to the prescribed quadratic square
class. Its integrand involves two relative invariants:
$\Delta$, which carries the Mellin parameter $s$,
and $R$, which carries the parameter $s_0$.
It also involves the spherical matrix coefficient
evaluated at a $2\times2$ matrix $y_w$ associated
with $w$. In the original coordinates these three
quantities are coupled, so the main task is to
separate their contributions.

We begin with the two-dimensional coordinate
$\ell\in\mathfrak o^2$. Every nonzero $\ell$ can
be written as $\ell=\varpi^k\ell_0$, where
$k\ge0$ and $\ell_0$ is primitive. The compact
group $\mathrm{GL}_2(\mathfrak o)$ acts transitively
on primitive vectors, and the corresponding
change of variables preserves the lattice,
the discriminant square class, and the integrand.
We may therefore take
$\ell=(0,\varpi^k)$. The $k$-th shell has
volume $q^{-2k}(1-q^{-2})$, while the
contribution of $R$ is
$q^{-2k(s_0-2)}$. Their product is
\[
(1-q^{-2})q^{-k(2s_0-2)}.
\]
In these coordinates $y_w=\varpi^k x_1$.
Since $\pi_v$ has trivial central character,
the scalar $\varpi^k$ does not affect the
spherical matrix coefficient. Thus the
$\ell$-integration reduces to a sum over $k$,
with the remaining dependence on $k$ occurring
through the Mellin weight and the discriminant.

Next we decompose the matrix variable $x_1$
according to its elementary divisors:
\[
x_1\in\varpi^a K
\begin{pmatrix}
1&0\\
0&\varpi^n
\end{pmatrix}K,
\qquad a,n\ge0,\qquad
K=\mathrm{GL}_2(\mathfrak o).
\]
The compact Levi action preserves the integral
lattice and the discriminant square class,
and changes the relative invariants only by
units. Together with the bi-$K$-invariance
of the spherical matrix coefficient, this
shows that the remaining integral is constant
on each such double coset. The parameter
$a$ records the common divisibility of the
entries of $x_1$, while $n$ records the
difference between its elementary-divisor
valuations. In particular,
$\operatorname{ord}(\det x_1)=2a+n$, but this valuation alone
does not determine the contribution: the
spherical coefficient and the discriminant
distribution still depend on the Cartan
parameter $n$. The primitive Cartan
coefficients are obtained from the spherical
Hecke eigenvalues by subtracting the
contribution of matrices divisible by
$\varpi$ (Lemma~\ref{lem:unramified-cartan-coefficients}).

After this reduction the discriminant can be
completed to a square. More precisely, for
diagonal $x_1$ and on each range of the shell
and Cartan indices, it takes the form
\[
\Delta=\varpi^{2c}
\bigl(z^2+4\varpi^eW\bigr),
\]
where $z$ is Haar-uniform and independent
of the auxiliary random variable $W$.
The factor $\varpi^{2c}$ does not affect
the quadratic square class. Conditional on
the valuation of $W$, its unit part is
Haar-uniform in $\mathfrak o^\times$, since $W$
is unit-invariant
(Lemma~\ref{lem:unramified-levi-invariance}(3));
see the remark on unit-invariant variables at the
beginning of Appendix~\ref{app:unramified}.
Thus the discriminant condition reduces
to a one-dimensional square-class density,
depending only on that valuation.
The random variable $W$ is a sum of
products of independent Haar-uniform
variables, and character orthogonality
gives its valuation tails explicitly
(Lemma~\ref{lem:unramified-fourier-tails}).
These formulas, together with the
square-class densities, yield rational
generating functions for the remaining
sums. This separation is the key to
the computation: the spherical
matrix-coefficient weights, the
discriminant square class, and the
valuation distributions can be evaluated
separately and then assembled into the
final series.

\begin{proof}[Proof of Theorem~\ref{thm:unramified-factor}]
The preceding reductions are proved with their Jacobians in Appendix~\ref{app:unramified}.  We recall the form of the resulting generating series in order to identify its Euler factors and its leading constant.

Put
\[
 Y=q^{-s},\qquad T=q^{-(s_0-\frac12)},\qquad
 U=qY^2T^2,
\]
and write $L(z)=1-a_\pi z+z^2$.  The denominator polynomial of $L_v(s_0-\frac12,\pi_E)$ is
\[
 L_E(T)=
 \begin{cases}
 L(T)^2,&\eta_v(\varpi)=1,\\
 L(T)L(-T),&\eta_v(\varpi)=-1.
 \end{cases}
\]
Put $\kappa_{K,v}=(1-q^{-1})(1-q^{-2})$, the volume of $\mathrm{GL}_2(\mathfrak o)$ for the additive measure on $M_2(F_v)$. The shell decomposition, primitive Cartan formula, Fourier valuation computation, and the rational-function identities established in Appendix~\ref{app:unramified} give
\[
\begin{aligned}
\frac{Z'_v(s,s_0)}{(1-q^{-2})\kappa_{K,v}}
={}&\frac{\frac12(1-q^{-1})(1-T^2)
 (1-(a_\pi^2-2)T^2+T^4)\mathcal D'}
 {(1-qT^2)L_E(T)
 (1-(a_\pi^2-2)U+U^2)
 (1-Y^2)(1-q^2Y^2T^4)},
\end{aligned}
\]
where
\[
 \mathcal D'=(1+Y^2T^2)(1+qY^2T^2)
 -\eta_v(\varpi)a_\pi Y^2T(1+qT^2).
\]
Appendix~\ref{app:unramified} derives the
generating series directly from the orbital
integral and evaluates the valuation sums.
The final rational-function identity follows
there from explicit even--odd decompositions
and, in the split case, a coefficient comparison.

To identify the Euler factors, observe that
\[
\begin{aligned}
(1-qT^2)^{-1}
 &=\zeta_v(2s_0-2),\\
L_E(T)^{-1}
 &=L_v(s_0-\tfrac12,\pi_E),\\
(1-T^2)\bigl(1-(a_\pi^2-2)T^2+T^4\bigr)
 &=L_v(2s_0-1,\pi,\mathrm{Ad})^{-1},\\
(1-Y^2)^{-1}
 &=\zeta_v(2s),\\
(1-q^2Y^2T^4)^{-1}
 &=\zeta_v(2s+4s_0-4).
\end{aligned}
\]
Moreover,
\[
\frac{1}{1-(a_\pi^2-2)U+U^2}
=
\frac{L_v(2s+2s_0-2,\pi,\mathrm{Ad})}
{\zeta_v(2s+2s_0-2)},
\]
and $\mathcal D'=\mathcal D_E(s,s_0)$,
since $\lambda_v=q^{1/2}a_\pi$.

The remaining constant is
\[
(1-q^{-2})\kappa_{K,v}\cdot\frac{1-q^{-1}}2
=\frac12(1-q^{-1})^2(1-q^{-2})^2
=c'_v.
\]
Substitution gives the asserted formula
for $Z'_v(s,s_0)$.

The integral and the generating series
agree in their common region of absolute
convergence. Since the resulting expression
is rational in the local exponential
variables, it also provides the asserted
meromorphic continuation.
\end{proof}

The theorem has two different applications, which are worth distinguishing.  Far inside the convergence region, its Euler factors are close to $1$ after the normalization used in the second global expansion.  Near $s_0=1$, the same formula identifies the factor left after removing the pole $\zeta_v(2s_0-2)$ with the unramified Suzuki--Wakatsuki zeta integral.  The latter identity is the local bridge between the two global expansions.

\begin{corollary}\label{cor:unramified-normalization}
With the hypotheses of Theorem~\ref{thm:unramified-factor}, the following statements hold.
\begin{enumerate}
\item The leading normalization is
\[
 2c_v^2c'_v=(1-q^{-2})^2=\zeta_v(2)^{-2}.
\]
\item For every $\delta>0$, uniformly in
$\Re s_0\ge3+\delta$ and
$\Re s\ge1+\delta$,
\[
 2c_v^2Z'_v(s,s_0)=1+O_\delta(q^{-2}).
\]
The implied constant is independent of the unramified unitary generic representation $\pi_v$ and of the imaginary parts of $s,s_0$.
\item After removing the indicated zeta factor, the value at $s_0=1$ is
\[
 \left.\frac{2c_v^2Z'_v(s,s_0)}{\zeta_v(2s_0-2)}\right|_{s_0=1}
 =L_v(1,\eta_v)\,r_v\,
 Z^\#_{E_v}\left(s+\frac12\right),
\]
where $r_v$ is defined in \eqref{eq:rv} and $Z^\#_{E_v}$ is the normalized unramified toric zeta factor of Suzuki--Wakatsuki. This agrees with Theorem~\ref{thm:local-residue} at $v$, where $|2|_v=|4d_E|_v=1$.
\end{enumerate}
\end{corollary}

\begin{proof}
The first identity follows immediately from the value of $c'_v$ and $c_v=(1-q^{-1})^{-1}$.  For the second, recall that unitarity and genericity give the elementary bounds
\[
 |\alpha|,|\alpha^{-1}|\le q^{1/2},\qquad
 |\lambda_v|\le2q,\qquad |\alpha^{\pm2}|\le q.
\]
Applying the displayed Satake bounds to the formula of the theorem gives, uniformly in the stated region,
\[
\begin{aligned}
\zeta_v(2s_0-2)&=1+O_\delta(q^{-4-2\delta}),\\
L_v(s_0-\tfrac12,\pi_E)&=1+O_\delta(q^{-2-\delta}),\\
L_v(2s_0-1,\pi,\mathrm{Ad})^{-1}&=1+O_\delta(q^{-4-2\delta}),\\
\zeta_v(2s)&=1+O_\delta(q^{-2-2\delta}),\\
\frac{\zeta_v(2s+4s_0-4)}{\zeta_v(2s+2s_0-2)}&=1+O_\delta(q^{-6-4\delta}),\\
L_v(2s+2s_0-2,\pi,\mathrm{Ad})&=1+O_\delta(q^{-5-4\delta}),\\
\mathcal D_E(s,s_0)&=1+O_\delta(q^{-4-3\delta}).
\end{aligned}
\]
All denominator Euler factors are uniformly separated from zero in this region.  Thus
\[
 2c_v^2Z'_v(s,s_0)
 =(1-q^{-2})^2\bigl(1+O_\delta(q^{-2-\delta})\bigr)
 =1+O_\delta(q^{-2}).
\]
This estimate concerns the full normalized local
factor in the stated region of absolute convergence;
the normal convergence of the residual Euler product
near $s_0=1$, after extracting the indicated
automorphic $L$-factors and the zeta factor, is
established separately in
Section~\ref{sec:comparison}.

For the third assertion, the quotient of the two $s$-dependent zeta factors in the theorem becomes $1$ at $s_0=1$.  Consequently
\[
\begin{aligned}
\left.\frac{2c_v^2Z'_v(s,s_0)}{\zeta_v(2s_0-2)}\right|_{s_0=1}
={}&\frac{L_v(\frac12,\pi_E)}{L_v(1,\pi,\mathrm{Ad})}
 \frac{\zeta_v(2s)L_v(2s,\pi,\mathrm{Ad})\mathcal D_E(s,1)}{\zeta_v(2)^2}.
\end{aligned}
\]
The polynomial $\mathcal D_E(s,1)$ is exactly
$D_{E_v}(\pi_v,s+\frac12)$ in the notation of
\cite[Theorem~5.7]{SW1} (cf.\ the proof of Theorem~4.10
therein).  Their unramified formula reads
\[
 Z^\#_{E_v}(\mathfrak{s})
 =\frac{\zeta_v(2\mathfrak{s}-1)L_v(2\mathfrak{s}-1,\pi,\mathrm{Ad})
 D_{E_v}(\pi_v,\mathfrak{s})}{\zeta_v(2)^3}.
\]
Substituting $\mathfrak{s}=s+\frac12$ and using the
definition \eqref{eq:rv} of $r_v$ proves the result.  The evaluation at $s_0=1$ is made only after dividing out the factor $\zeta_v(2s_0-2)$, which has a pole there.
\end{proof}

We conclude by recording the role of this calculation in the global argument.  The estimate in part (2) ensures normal convergence of the normalized good-place product in the initial region.  It is not a continuation argument near $s_0=1$.  For that purpose, Section~\ref{sec:comparison} uses the complete factorization in Theorem~\ref{thm:unramified-factor}: after extracting the partial zeta and automorphic $L$-functions, the remaining Euler product is controlled separately.  Part~(3) identifies the specialization of the normalized local orbital factor, after removal of $\zeta_v(2s_0-2)$, with the corresponding unramified factor on the toric side.  This explains why both the two-variable closed form and its normalization are needed for the proof of the main theorem.


\section{Proof of the main theorem}\label{sec:comparison}

\subsection{Set-up}\label{subsec:comparison-setup}

The preceding sections supply two descriptions of the same meromorphic zeta
integral.  The boundary calculation of Section~\ref{sec:first-expansion}
expresses its residue in terms of the cube zeta integral, whereas the
unfolding of Section~\ref{sec:second-expansion} expresses it through
local orbital integrals.  We now compare the two descriptions.  It is
important to retain the bilinear inducing data until the final step:
only then do we specialize to the Hermitian pairing and the square of
a toric period.  We also keep the distinction between the coordinate
measures used in the local calculations and the self-dual measures of
the statement of the main theorem.  In particular, all measure-conversion factors are
retained explicitly until the two residue formulas
have been compared.

Fix $E$, $\pi$, factorizable data $\Psi=\bigotimes_v\Psi_v$, and a factorizable
standard section $f=\bigotimes_v f_v$. Fix the global representatives $x_E$
(for the cube space) and $w_\mu$, with $[4\mu]=[d_E]$ (for $V'$). Let $S$ be a
finite set of places containing:
\begin{itemize}
\item all archimedean and all dyadic places;
\item all places where $\pi_v$ or $E_v$ ramifies, or where $d_E$, $\mu$ or
$\kappa_0$ is not a unit;
\item all finite places dividing the different of $F/\mathbb Q$, and all
places at which the chosen $F$-basis $1,\delta_E$ does not identify
$\mathfrak o_v^2$ with $\mathfrak o_{E_v}$, so that $S$ contains every place where
the coordinate and locally self-dual additive measures used in
\S\ref{subsec:invariant-form} differ;
\item every place where \cite[Condition~1.1]{SW1} fails;
\item every finite place where $j_ve_v^\circ\neq e_v^\circ$ or
$\langle e_v^\circ,e_v^\circ\rangle_v\neq1$;
\item every place where $\Psi_v\neq\mathbf 1_{V'(\mathfrak o_v)}$, or where
$f_v$ is not the normalized spherical section.
\end{itemize}
For $v\notin S$ the data are spherical, $\xi_{f_v}(\kappa)=\xi_v^\circ$ for all
$\kappa\in K_v$, and $\Phi^{\mathrm{loc}}_{\Psi_v}=\mathbf 1_{V(\mathfrak o_v)}$.
We write $\mathfrak{s}_s:=s+\tfrac12$ for the cube zeta parameter, reserving $S$ for the finite set of places.

Let $dh_v$ be the measure on
$T_v=F_v^\times\backslash E_v^\times$
of \cite[(4.4)]{SW1}, namely
$dh_{E_v}/d^\times z$. For $\xi\in\sigma_v$ put
\[
L_v(\xi)
=\int_{T_v}\beta_{\xi,v}(t)\,dh_v(t),
\qquad
\alpha^B_{\xi,v}(g_1,g_2)
=L_v\bigl(
(\pi_v(g_1)\otimes\pi_v(g_2))\xi
\bigr).
\]
The diagonal invariance of $B_v$, together
with translation invariance of $dh_v$,
implies that $L_v$ is separately
$T_v$-invariant in the two tensor factors.
Indeed, translation invariance proves
invariance in the second factor, while
diagonal invariance of $B_v$ moves the
torus action in the first factor to the
second, where the same change of variables
applies. Consequently, local
multiplicity one \cite[Lemme~8(i)]{Wal85}
(see also \cite{Tunnell,Saito}) gives
\[
L_v\in
\operatorname{Hom}_{T_v}(\pi_v,\mathbb C)
\otimes
\operatorname{Hom}_{T_v}(\pi_v,\mathbb C),
\]
under the self-dual realization determined
by $B_v$. In particular, $L_v=0$ if
$\pi_v$ is not $E_v^\times$-distinguished,
while by Lemma~\ref{lem:alpha-nonzero} below
$L_v$ is a nonzero element of this
one-dimensional space if $\pi_v$ is
$E_v^\times$-distinguished.
At this stage the functional remains
bilinear; no Hermitian normalization
has been imposed.

For $\xi=\phi_v\otimes j_v\phi_v$,
the definition specializes to
\[
L_v(\xi)=\alpha_{E_v}(\phi_v,\phi_v),
\]
the local toric period of Suzuki--Wakatsuki,
formed using $\langle\ ,\ \rangle_v$.
At every $v\notin S$, the spherical
normalization
$B_v(\xi_v^\circ)=1$ is equivalent to
$\langle\phi_v^\circ,\phi_v^\circ\rangle_v=1$.
Hence the unramified computation
\cite[Theorem~1.2]{II10}, in the case
$\dim V_1=3$, $\dim V_0=2$ with trivial
character, where $\Delta_{G_1,v}=\zeta_v(2)$,
gives
\begin{equation}\label{eq:rv}
L_v(\xi_v^\circ)=r_v:=
\frac{\zeta_v(2)\,
L_v(\tfrac12,\pi_E)}
{L_v(1,\pi,\mathrm{Ad})\,
L_v(1,\eta_v)};
\end{equation}
here the maximal compact subgroup of $T_v$
has volume $1$ for $dh_v$, as computed in the
proof of \cite[Corollary~34]{SW2}, which also
derives \eqref{eq:rv} from \cite[Lemmes~2 et~3]{Wal85}.
This is a purely local identity and does
not invoke Waldspurger's global period formula.

For notational convenience, at the finitely
many places $v\in S$ we define $r_v$ by
the same ratio of local $L$-factors,
without asserting the spherical period
identity there. We also put
\[
r^S=
\frac{\zeta_F^S(2)\,
L^S(\tfrac12,\pi_E)}
{L^S(1,\pi,\mathrm{Ad})\,
L^S(1,\eta)}.
\]
This ratio is well-defined because its
denominator is nonzero (see the proofs of
Propositions~\ref{prop:first-residue} and~\ref{prop:second-residue}), although its
numerator may vanish at the central point.
The quantity $r^S$ will be used only
multiplicatively; in particular, the
comparison does not require division by
$L^S(\tfrac12,\pi_E)$.

\begin{lemma}\label{lem:alpha-nonzero}
If $\pi_v$ is $E_v^\times$-distinguished, there exist a
$K_v$-finite vector $\phi_v\in\pi_v$ and $g\in\mathrm{GL}_2(F_v)$
such that
\[
\alpha_{E_v}\bigl(\pi_v(g)\phi_v,\pi_v(g)\phi_v\bigr)>0.
\]
If $E_v$ is a field or $v$ is finite, one may take $g=1$.
In particular, $L_v\neq0$.
\end{lemma}

\begin{proof}
Here $\alpha_{E_v}(\phi,\phi)=\int_{T_v}
\langle\pi_v(t)\phi,\phi\rangle_v\,dh_v(t)$, and $\pi_v$ is
unitary for $\langle\ ,\ \rangle_v$. At a finite place every
smooth vector is $K_v$-finite.

If $E_v$ is a field, $T_v$ is compact and
$\alpha_{E_v}(\phi,\phi)=\mathrm{vol}(T_v)\,\|\mathrm{pr}\,\phi\|^2$,
where $\mathrm{pr}$ is the orthogonal projection onto
$\pi_v^{T_v}$; this space is nonzero because $\pi_v$ is
$E_v^\times$-distinguished, so by density some $K_v$-finite
$\phi$ has $\mathrm{pr}\,\phi\neq0$.

Suppose $E_v=F_v\oplus F_v$. Then $T_v=gAg^{-1}$ for some
$g\in\mathrm{GL}_2(F_v)$, where $A$ is the image of
$a(t)=\mathrm{diag}(t,1)$, $t\in F_v^\times$. Consequently,
$\alpha_{E_v}(\pi_v(g)\phi,\pi_v(g)\phi)$ is a positive multiple of
\[
\int_{F_v^\times}\langle\pi_v(a(t))\phi,\phi\rangle_v\,d^\times t,
\]
and it suffices to find a $K_v$-finite $\phi$ for which this
integral is positive. Realize $\pi_v$ in its
Kirillov model $\phi\mapsto W_\phi$, so that
$W_{\pi_v(a(t))\phi}(y)=W_\phi(yt)$.

At a finite place, the invariant inner product is
\[
\langle\phi_1,\phi_2\rangle_v
=c\int_{F_v^\times}W_{\phi_1}(y)\overline{W_{\phi_2}(y)}\,d^\times y,
\qquad c>0,
\]
the integral converging absolutely for all smooth vectors
\cite[Theorem~6.2]{Be84}.

At an archimedean place, the integral
$Z(W_{\phi_1},W_{\phi_2},s)=\int_{F_v^\times}W_{\phi_1}\overline{W_{\phi_2}}
\,|y|_v^s\,d^\times y$ converges absolutely for $\Re s\ge\sigma_1$, say,
and continues meromorphically \cite[\S10.2]{Baruch}; by
\cite[Corollary~10.4]{Baruch}, $(s^kZ(W_{\phi_1},W_{\phi_2},s))|_{s=0}$
is a nonzero multiple of $\langle\phi_1,\phi_2\rangle_v$, where $k$ is the
maximal order at $s=0$ of the poles of these functions. We have $k=0$.
Indeed, choose a nonzero smooth vector $\phi$ and a Schwartz function
$f$ on $N(F_v)$ with $\widehat f\in C_c^\infty(F_v^\times)$ and
$\pi_v(f)\phi\neq0$; the Kirillov function of $\pi_v(f)\phi$ is
$\widehat f\,W_\phi$, which has compact support, so the corresponding
zeta integral is holomorphic at $s=0$. If $k>0$, the regularized
pairing would vanish on this vector, contradicting its proportionality
to the positive definite invariant inner product. Evaluating on the
same vector, where the regularized pairing is
$\int|\widehat f\,W_\phi|^2\,d^\times y>0$, shows that
$\langle\phi_1,\phi_2\rangle_v=c\,Z(W_{\phi_1},W_{\phi_2},0)$ with $c>0$.
If moreover $\int|W_{\phi_1}W_{\phi_2}|\,d^\times y<\infty$, then, since
$|y|_v^\sigma\le1+|y|_v^{\sigma_1}$ for $0\le\sigma\le\sigma_1$, the
integral $Z(W_{\phi_1},W_{\phi_2},s)$ converges absolutely and is
continuous on the strip $0\le\Re s\le\sigma_1$, so its value at $s=0$ is
$\int W_{\phi_1}\overline{W_{\phi_2}}\,d^\times y$. Hence, at every place,
$\langle\phi_1,\phi_2\rangle_v
=c\int_{F_v^\times}W_{\phi_1}\overline{W_{\phi_2}}\,d^\times y$
with $c>0$ whenever the integral converges absolutely.

Hence, if $\int|W_\phi|\,d^\times y<\infty$, then by Tonelli's
theorem the integral
\[
\int_{F_v^\times}W_\phi(yt)\overline{W_\phi(y)}\,d^\times y
\]
converges absolutely for almost every $t$, and Fubini's theorem gives
\[
\begin{aligned}
\int_{F_v^\times}\langle\pi_v(a(t))\phi,\phi\rangle_v\,d^\times t
&=c\int_{F_v^\times}\int_{F_v^\times}
W_\phi(yt)\overline{W_\phi(y)}\,d^\times y\,d^\times t\\
&=c\,\Bigl|\int_{F_v^\times}W_\phi\,d^\times y\Bigr|^2,
\end{aligned}
\]
with $c>0$ the constant of the Kirillov-model inner product.
If $v$ is finite, take $\phi$ with
$W_\phi\in C_c^\infty(F_v^\times)$, which is possible by
\cite[\S6.1(iii)]{Be84}, and $\int W_\phi\,d^\times y\neq0$. If
$v$ is archimedean, by \cite[Theorem~5.15(ii)]{JL70} for
$F_v=\mathbb R$ and \cite[Theorem~6.4]{JL70} for
$F_v=\mathbb C$ there is a $K_v$-finite $\phi$ with
\[
\int_{F_v^\times}W_\phi(y)|y|_v^{s-1/2}\,d^\times y
=L(s,\pi_v)\,e^{as+b}.
\]
This integral converges absolutely at $s=\tfrac12$ by the proof of
\cite[Theorem~5.15]{JL70} (and its analogue \cite[Theorem~6.4]{JL70}
if $F_v=\mathbb C$): for principal series the left side is the Tate
integral $Z(\mu_1\alpha^s,\mu_2\alpha^s,\Phi)$, with
$|\Re\mu_i|<\tfrac12$ since $\pi_v$ is unitary, while for discrete
series it is the Tate integral $Z(\omega\alpha_{\mathbb C}^s,\Phi)$ over
$\mathbb C^\times$, with $\Phi$ in the model of
\cite[Lemma~5.12]{JL70}. Hence
$\int W_\phi\,d^\times y=L(\tfrac12,\pi_v)e^{a/2+b}\neq0$.
At a finite place $\pi_v(g)\phi$ is again $K_v$-finite, so $g$
may be absorbed into $\phi$.
\end{proof}

By \cite[Proposition~2.3 and (4.1)]{SW1}, the connected
stabilizer $H^\circ_{x_E}$ is isomorphic to
$T_E\times T_E$, where
\[
T_E=R_{E/F}\mathbb G_m/\mathbb G_m
\]
is the torus $T$ of the introduction. Indeed, the
corresponding connected stabilizer in
$\mathrm{GL}_2^3$ is
$(R_{E/F}\mathbb G_m)^2$,
and passage to $H$ amounts to quotienting by
$\ker\rho_V\simeq\mathbb G_m^2$. Let
\[
\mathcal P_E(\xi)
=
\int_{[H^\circ_{x_E}]}
\xi(t_1,t_2)\,d\mu_{H^\circ},
\qquad \xi\in\sigma,
\]
where
\[
d\mu_{H^\circ}
=
L_f(1,\eta)^{-2}|\Delta_F|^{-1}
\prod_v d_v\,|\omega_{H^\circ}|_v
\]
is the measure of \cite[\S4.3]{SW1},
with $|\omega_{H^\circ}|_v$ denoting the local
measure associated with the invariant
gauge form, as in
\S\ref{subsec:invariant-form}.
For $\xi=\phi_1\otimes\phi_2$,
the functional $\mathcal P_E(\xi)$ is
the product of the two toric periods
of $\phi_1$ and $\phi_2$ with this
measure normalization.

The functional $\mathcal P_E$ is invariant
under $T_E(\mathbb A)\times T_E(\mathbb A)$.
By local multiplicity one, its restriction
to factorizable vectors admits a factorization
through the local toric functionals.
Thus there is a scalar $C_P\in\mathbb C$
such that
\begin{equation}\label{eq:CP}
\mathcal P_E\Bigl(\bigotimes_v\xi_v\Bigr)
=
C_P
\prod_{v\in S}L_v(\xi_v)
\prod_{v\notin S}
\frac{L_v(\xi_v)}{r_v},
\end{equation}
for pure tensors with $\xi_v=\xi_v^\circ$ for almost all $v$,
with respect to the same restricted tensor product identification
and local forms $B_v$ as in
\eqref{eq:second-global-pairing}. Since
$L_v(\xi_v^\circ)=r_v$ for $v\notin S$ by \eqref{eq:rv}, the factors
$L_v(\xi_v)/r_v$ equal $1$ at almost all places, and the product is
finite.

If $\mathcal P_E$ vanishes identically,
we take $C_P=0$. Otherwise, the
factorization determines a nonzero scalar
$C_P$. In particular, the vanishing case
includes representations that are not
$E^\times$-distinguished.

Recall that $C_\pi$ is the nonzero scalar
fixed in Section~\ref{sec:second-expansion}
by the factorization
\eqref{eq:second-global-pairing} of the
global diagonal pairing:
\[
B\Bigl(\bigotimes_v\xi_v\Bigr)
=C_\pi\prod_v B_v(\xi_v),
\]
with the local pairings normalized by
$B_v(\xi_v^\circ)=1$ outside $S$.
Thus $C_P$ and $C_\pi$ arise from
factorizations of two different global
bilinear functionals: the toric functional
$\mathcal P_E$ and the diagonal pairing $B$,
respectively. Neither constant is itself
a local period, and neither should be
identified with a Petersson-normalized period.

For every $v$, put
\[
\mathcal J_v(s)=\int_{K_v}
\mathcal Z^{\mathrm{SW}}_{E_v}
\bigl(
\Phi^{\mathrm{loc}}_{\rho(\kappa)\Psi_v},
\xi_{f_v}(\kappa);
s+\tfrac12
\bigr)\,d\kappa .
\]
For $v\notin S$, the spherical data are
independent of $\kappa$, and
$\operatorname{vol}(K_v)=1$.
Consequently, \cite[Thm.~5.7]{SW1}
and \eqref{eq:rv} give
\[
\mathcal J_v(s)
=r_v\,Z^\#_{E_v}(s+\tfrac12),
\]
where
\[
Z^\#_{E_v}(\mathfrak{s}_s)
=
\frac{
\zeta_v(2\mathfrak{s}_s-1)
L_v(2\mathfrak{s}_s-1,\pi,\mathrm{Ad})
D_{E_v}(\pi_v,\mathfrak{s}_s)}
{\zeta_v(2)^3}.
\]
For $\Re s\ge2$, the unramified factors satisfy
$Z^\#_{E_v}(s+\tfrac12)=1+O(q_v^{-2})$
uniformly in $v\notin S$, and each of them is nonzero.
Indeed, with $\sigma=\Re s$ and the Satake bounds
$|\alpha^{\pm2}|\le q_v$, $|\lambda_v|\le2q_v$, the Euler parameters of
$\zeta_v(2s)$ and $L_v(2s,\pi,\mathrm{Ad})$ are at most $q_v^{1-2\sigma}$
in absolute value, while
$D_{E_v}(\pi_v,s+\tfrac12)=\mathcal D_E(s,1)$ satisfies
$|\mathcal D_E(s,1)-1|\le6q_v^{-2\sigma}+q_v^{-(4\sigma+1)}<1$.
Hence their Euler product converges
absolutely and is nonzero in this half-plane.

\subsection{The first residue}
\label{subsec:first-residue}

The Eisenstein-side residue has already been reduced to the
global Suzuki--Wakatsuki integral. We now unfold that integral
for bilinear inducing data and compare the global and local
measure normalizations. The factor $1/2$ appearing below comes
from the component group of the cube stabilizer. It is distinct
from the degree-two orbit-map factor in
Section~\ref{sec:second-expansion}.

\begin{proposition}\label{prop:first-residue}
For $\Re s$ sufficiently large,
\[
\operatorname*{Res}_{s_0=1}
Z'_E(\Psi,f;s,s_0)
=
-\frac12|\Delta_F|^{-5}
\Lambda(1,\eta)^2 C_P
\prod_{v\in S}\mathcal J_v(s)
\prod_{v\notin S}
Z^\#_{E_v}(s+\tfrac12).
\]
\end{proposition}

\begin{proof}
The argument follows the proof of
\cite[Theorem~4.10]{SW1}, except that
the use of Waldspurger's formula
\cite[(1.1)]{SW1} is replaced by the
factorization \eqref{eq:CP}, obtained
from local multiplicity one.

By Corollary~\ref{cor:residue-SW} above,
\[
\operatorname*{Res}_{s_0=1}Z'_E
=
-c_F\int_K
Z^{\mathrm{SW}}_E
\bigl(
\Phi_{\Psi_\kappa},
\xi_f(\kappa);\mathfrak{s}_s
\bigr)\,d\kappa.
\]

\emph{1. Bilinear unfolding.}
The proof of \cite[Lemma~4.9]{SW1} consists of three
steps, none of which uses a Hermitian relation between
the two automorphic factors. The passage from the orbit
sum over $V_E(F)$ to a sum over
$H^\circ_{x_E}(F)\backslash H(F)$, with the factor
$\frac12=|\Pi_{x_E}(F)|^{-1}$ \cite[(4.1)]{SW1}, involves only $\Phi$ and
the orbit map; the automorphic factor depends on
$h\in H(\mathbb A)$ alone and stands outside this sum.
The folding to $H^\circ_{x_E}(F)\backslash H(\mathbb A)$
uses only the left $H(F)$-invariance of $|\omega|^s$, with
$\omega$ the character of \cite[\S2.2]{SW1}, and
of the automorphic factor. The integration over
$[H^\circ_{x_E}]=[T_E\times T_E]$, with the measure
$d\mu_{H^\circ}$ of \S\ref{subsec:comparison-setup}, produces
$\mathcal P_E\bigl((\pi(g_1)\otimes\pi(g_2))\xi\bigr)$
by the definition of $\mathcal P_E$, and is linear in
$\xi$. The final identification of the quotient measure
uses the identification of the domain $Y_{x_E}(\mathbb A)$ in
\cite[Lemmas~4.5 and~4.6]{SW1}, which does not involve the
automorphic factor. Hence the lemma holds for arbitrary
$\xi\in\sigma$ and gives
\[
Z^{\mathrm{SW}}_E(\Phi,\xi;\mathfrak{s}_s)
=
\frac12\int_{Y_{x_E}(\mathbb A)}
\mathcal P_E\bigl(
(\pi(g_1)\otimes\pi(g_2))\xi
\bigr)
|P(\mu(y))|^{\mathfrak{s}_s}
\Phi(\mu(y))\,dy,
\]
where $Y_{x_E}=H^\circ_{x_E}\backslash H$; for
$y=H^\circ_{x_E}h$ with $h=(g_1,g_2,g_3)$ we write
$\mu(y)=x_E\cdot\rho_V(h)$, and the integrand depends
only on $y$;
\[
dy=
c_F^{-1}L(1,\eta)^2
|\Delta_F|^{-4}
\prod_v c_vd_v^{-1}\lvert\omega_Y\rvert_v
\]
is the measure of \cite[(4.6)]{SW1}, with
$\omega_Y$ the
$H$-invariant gauge form on $Y_{x_E}$ of
\cite[\S4.3]{SW1}.

The local counterpart, namely the descent of
$\alpha^B_{\xi_v,v}$ to $V_E(F_v)$ needed in Step~2,
is proved in Section~\ref{sec:local-residue}.

\emph{2. Factorization.}
Insert \eqref{eq:CP}.
If $C_P=0$, both sides below vanish. Otherwise the integrand over
$Y_{x_E}(\mathbb A)$ is $C_P$ times a product of local functions, and
since
\[
|\mathcal P_E(\xi')|
\leq
\int_{[H^\circ_{x_E}]}|\xi'|\,d\mu_{H^\circ},
\]
its absolute integral is bounded by the folded integral with absolute
values inside, which is finite for $\Re s$ sufficiently large by
Lemma~\ref{lem:cube-convergence}. Tonelli's theorem and monotone
convergence over finite sets of places, as in the proof of
Proposition~\ref{prop:second-expansion}, then justify the product of
local absolute integrals. Fubini's theorem consequently gives the
factorization place by place. By
\cite[(4.7)]{SW1}, for each place $v$,
\[
\begin{aligned}
&\int_{Y_{x_E}(F_v)}
\alpha^B_{\xi_v,v}(g_1,g_2)
|P(\mu(y))|_v^{\mathfrak{s}_s}
\Phi_v(\mu(y))\,\lvert\omega_Y\rvert_v\\
&\qquad =
\frac{L_v(1,\eta_v)^2}{c_v}
\mathcal Z^{\mathrm{SW}}_{E_v}
(\Phi_v,\xi_v;\mathfrak{s}_s).
\end{aligned}
\]
The global Schwartz function satisfies
\[
\Phi_\Psi
=
|\Delta_F|^{-1}
\bigotimes_v\Phi^{\mathrm{loc}}_{\Psi_v},
\]
because the self-dual measure on
$\mathbb A^2$ is
$|\Delta_F|^{-1}\prod_v d\lambda_v$.
Consequently, for fixed $\kappa$,
\[
\begin{aligned}
&Z^{\mathrm{SW}}_E
\bigl(
\Phi_{\Psi_\kappa},
\xi_f(\kappa);\mathfrak{s}_s
\bigr)\\
&\quad =
\frac12 C_P c_F^{-1}
L(1,\eta)^2|\Delta_F|^{-5}
\prod_v
\bigl(d_v^{-1}L_v(1,\eta_v)^2\bigr)\\
&\qquad\quad\times
\prod_{v\in S}
\mathcal Z^{\mathrm{SW}}_{E_v}
\bigl(
\Phi^{\mathrm{loc}}_{\rho(\kappa_v)\Psi_v},
\xi_{f_v}(\kappa_v);\mathfrak{s}_s
\bigr)\\
&\qquad\quad\times
\prod_{v\notin S}
r_v^{-1}
\mathcal Z^{\mathrm{SW}}_{E_v}
\bigl(
\Phi^{\mathrm{loc}}_{\rho(\kappa_v)\Psi_v},
\xi_{f_v}(\kappa_v);\mathfrak{s}_s
\bigr).
\end{aligned}
\]

\emph{3. Completion of the $L$-factors.}
The constants $d_v$ belong to the
gauge measure of Suzuki--Wakatsuki.
They reconcile the finite global
$L$-value in that measure with the
local factors of the zeta integral;
this is separate from the self-dual
Schwartz-measure conversion above.

By \cite[\S4.3]{SW1}, at a finite place
where $\eta_v$ is unramified,
\[
d_v
=
(1-\eta_v(\varpi_v)q_v^{-1})^{-2}
=
L_v(1,\eta_v)^2.
\]
At the remaining places $d_v=1$.
At finite ramified places
$L_v(1,\eta_v)=1$, and hence
\[
d_v^{-1}L_v(1,\eta_v)^2=1
\qquad (v<\infty).
\]
At archimedean places this factor is
$L_v(1,\eta_v)^2$. The global value
$L(1,\eta)$ in \cite[(4.6)]{SW1} is the
finite Hecke $L$-value $L_f(1,\eta)$: its
convergence factors are the $d_v$ at finite
places. Since $E$ is a field, $\eta$ is
nontrivial, so $L_f(1,\eta)$ is the finite
nonzero value at $s=1$ of an entire
function; the regularization
$L(1,\eta_E):=c_F$ that \cite[\S4.3]{SW1}
uses for split $E$ does not arise here.
Therefore
\[
\begin{aligned}
&L(1,\eta)^2
\prod_v d_v^{-1}L_v(1,\eta_v)^2\\
&\qquad =
L_f(1,\eta)^2
\prod_{v\mid\infty}L_v(1,\eta_v)^2
=
\Lambda(1,\eta)^2,
\end{aligned}
\]
where
\[
\Lambda(s,\eta)
=
L_f(s,\eta)
\prod_{v\mid\infty}L_v(s,\eta_v),
\]
with the archimedean factors in the
normalization fixed in the introduction.
The same archimedean factors are used in
the measure $dh_{E_v}$ of \cite[(4.4)]{SW1}
and in the normalization of the local
integrals $\mathcal Z^{\mathrm{SW}}_{E_v}$.
These factors cannot be rescaled
independently of the measures: changing
$dh_{E_v}$ also changes the local toric
functionals and $C_P$. See
Remark~\ref{rem:normalization-infinity}.

\emph{4. The $K$-average.}
The local integrals factor over
$\kappa=(\kappa_v)$.
For $\Re s$ sufficiently large,
absolute convergence permits
interchanging the $K$-integration
with the Euler product. Since
$K=\prod_v K_v$ with normalized
Haar measures, this gives
\[
\prod_v\int_{K_v}
\mathcal Z^{\mathrm{SW}}_{E_v}
\bigl(
\Phi^{\mathrm{loc}}_{\rho(\kappa_v)\Psi_v},
\xi_{f_v}(\kappa_v);\mathfrak{s}_s
\bigr)\,d\kappa_v
=
\prod_v\mathcal J_v(s).
\]
For $v\notin S$, the spherical identity
established above gives
\[
r_v^{-1}\mathcal J_v(s)
=
Z^\#_{E_v}(s+\tfrac12).
\]
Finally, multiplying by $-c_F$
cancels the factor $c_F^{-1}$.
We obtain
\[
\operatorname*{Res}_{s_0=1}
Z'_E(\Psi,f;s,s_0)
=
-\frac12|\Delta_F|^{-5}
\Lambda(1,\eta)^2 C_P
\prod_{v\in S}\mathcal J_v(s)
\prod_{v\notin S}
Z^\#_{E_v}(s+\tfrac12),
\]
as asserted.
\end{proof}

\subsection{The second residue}
\label{subsec:second-residue}

The generic-orbit expansion is initially valid in a right
half-plane in $s_0$. We first isolate its global Tate factor
in that region and then continue the remaining factors to
$s_0=1$. This order avoids taking the residue of a formal
Euler product outside its domain of absolute convergence.

\begin{proposition}\label{prop:second-residue}
For $\Re s$ sufficiently large,
\[
\operatorname*{Res}_{s_0=1}
Z'_E(\Psi,f;s,s_0)
=
-\frac14|\Delta_F|^{-11/2}
\Lambda(1,\eta)\,r^S C_\pi
\prod_{v\in S}\mathcal J_v(s)
\prod_{v\notin S}
Z^\#_{E_v}(s+\tfrac12).
\]
\end{proposition}

\begin{proof}
\emph{1. The identity in the initial half-plane.}
Fix $s$ with $\Re s$ sufficiently large.
By Proposition~\ref{prop:second-expansion},
for $\Re s_0>3$,
\[
Z'_E(\Psi,f;s,s_0)
=
\frac12c_F^{-1}|\Delta_F|^{-6}
C_\pi\prod_v2c_v^2\mathscr Z'_v(s,s_0).
\]
Write
\[
\mathscr Z'_v(s,s_0)
=
\mathcal T_v(2s_0-2)R_v(s,s_0).
\]
With $d^\times z=c_v\,dz/|z|_v$ and
the standard functions of
Section~\ref{sec:local-residue},
one has $\mathcal T_v=\zeta_v$ at every
place. Consequently,
\[
\prod_v\mathcal T_v(2s_0-2)
=
\Lambda_F(2s_0-2).
\]

For $v\notin S$,
Theorem~\ref{thm:unramified-factor} gives
$2c_v^2R_v(s,s_0)=Q_v(s,s_0)$, where
\[
\begin{aligned}
Q_v(s,s_0)
={}&
\zeta_v(2)^{-2}
\frac{L_v(s_0-\frac12,\pi_E)}
     {L_v(2s_0-1,\pi,\mathrm{Ad})}\\
&\times
\zeta_v(2s)
\frac{\zeta_v(2s+4s_0-4)}
     {\zeta_v(2s+2s_0-2)}
L_v(2s+2s_0-2,\pi,\mathrm{Ad})\\
&\times\mathcal D_E(s,s_0).
\end{aligned}
\]
Thus, in the initial region of absolute
convergence,
\begin{equation}\label{eq:II-global}
\begin{aligned}
Z'_E
={}&
\frac12c_F^{-1}|\Delta_F|^{-6}
C_\pi\,\Lambda_F(2s_0-2)\\
&\times
\prod_{v\in S}2c_v^2R_v(s,s_0)
\cdot Q^S(s,s_0),
\end{aligned}
\qquad \Re s_0>3,
\end{equation}
where
\[
Q^S(s,s_0)
=\zeta_F^S(2)^{-2}\,
\frac{L^S(s_0-\frac12,\pi_E)}{L^S(2s_0-1,\pi,\mathrm{Ad})}\,
\mathcal H^S(s,s_0),
\]
\[
\mathcal H^S(s,s_0)=\prod_{v\notin S}
\zeta_v(2s)\frac{\zeta_v(2s+4s_0-4)}{\zeta_v(2s+2s_0-2)}
L_v(2s+2s_0-2,\pi,\mathrm{Ad})\,\mathcal D_E(s,s_0).
\]

\emph{2. Continuation to $s_0=1$.}
We continue the factored expression on the
right-hand side of \eqref{eq:II-global},
rather than the original orbital Euler
product term by term. This distinction is
essential: the extraction of the global
$L$-functions isolates the possible poles,
while the remaining Euler product converges
in a larger region.

For each $v\in S$, let $s_{1,v}$ and $\varepsilon_{1,v}$
be the constants supplied by
Appendix~\ref{app:local-continuation} and
Corollary~\ref{cor:local-continuation}.
Put
\[
s_1=\max_{v\in S}s_{1,v},
\qquad
\varepsilon_1=\min_{v\in S}\varepsilon_{1,v},
\qquad
s_2=\max(s_1,2),
\]
and fix $s$ with $\Re s>s_2$.
Choose
\[
0<\varepsilon<
\min\{\varepsilon_1,\tfrac12\}
\]
sufficiently small that the local
continuation statements hold for every
$v\in S$. Define
\[
\Omega=
\{\Re s_0>1\}
\cup
\{|s_0-1|<\varepsilon\}.
\]
Then $\Omega$ is connected, contains the
initial half-plane $\Re s_0>3$ and a
neighbourhood of $s_0=1$, and satisfies
$\Omega\subset\{\Re s_0>\tfrac12\}$.

We examine the factors in
\eqref{eq:II-global} separately.

\begin{enumerate}
\item
The completed Tate factor
$\Lambda_F(2s_0-2)$ is meromorphic
on $\mathbb C$. Its only possible poles
occur at $s_0=1$ and $s_0=\tfrac32$.
In a sufficiently small neighbourhood
of $s_0=1$, only the former occurs,
and it is simple.

\item
The partial base-change factor satisfies
\[
L^S(s_0-\tfrac12,\pi_E)
=
L^S(s_0-\tfrac12,\pi)
L^S(s_0-\tfrac12,\pi\otimes\eta).
\]
Both $\pi$ and $\pi\otimes\eta$ are
cuspidal. Their complete standard
$L$-functions are entire, and the
reciprocals of the finitely many omitted
local Euler factors are entire.
Consequently,
$L^S(s_0-\tfrac12,\pi_E)$ is entire.
Its value at $s_0=1$ may vanish;
no nonvanishing assumption on
$L^S(\tfrac12,\pi_E)$ is required.

\item
The partial adjoint factor
$L^S(2s_0-1,\pi,\mathrm{Ad})$
is meromorphic on $\mathbb C$
\cite{GJ78}. We also have
\[
L^S(1,\pi,\mathrm{Ad})\ne0.
\]
Indeed, by
\cite[Lemma~(4.6)(i), Prop.~(2.3)]{JS81},
as used in the proof of
\cite[Thm.~(4.8)]{JS81}, the local
Rankin--Selberg data at places in $S$
may be chosen so that the corresponding
local integrals are holomorphic and
nonzero near $1$. The partial
Rankin--Selberg $L$-function
$L^S(u,\pi\times\widetilde\pi)$
therefore has a simple pole at $u=1$.
On the other hand, comparison of the
unramified Euler factors gives
\[
L^S(u,\pi\times\widetilde\pi)
=
\zeta_F^S(u)L^S(u,\pi,\mathrm{Ad}).
\]
Since $\zeta_F^S(u)$ has a simple
pole at $u=1$, the adjoint factor is
holomorphic and nonzero there.
Thus
$L^S(2s_0-1,\pi,\mathrm{Ad})^{-1}$
is meromorphic on $\Omega$ and
holomorphic near $s_0=1$.

\item
The remaining unramified Euler product
converges normally on $\Omega$.
Indeed, for $\Re s_0\ge\tfrac12$
and $\Re s>s_2$, each of the factors
\[
\zeta_v(2s),\qquad
\zeta_v(2s+4s_0-4),\qquad
\zeta_v(2s+2s_0-2)^{-1},
\]
\[
L_v(2s+2s_0-2,\pi,\mathrm{Ad}),
\qquad
\mathcal D_E(s,s_0)
\]
is $1+O(q_v^{-2})$, uniformly for
$v\notin S$, as verified below.
Since
\[
\sum_{v\notin S}q_v^{-2}<\infty,
\]
the product converges absolutely and
locally uniformly on $\Omega$.
It therefore defines a holomorphic
function of $s_0$ on $\Omega$.

\item
It remains to treat the finitely many
places $v\in S$. By
Lemma~\ref{lem:local-continuation} and
Corollary~\ref{cor:local-continuation},
the defining local integral
$\mathscr Z'_v(s,s_0)$ converges
absolutely and locally uniformly for
$\Re s_0>1$. Hence
\[
R_v(s,s_0)
=
\zeta_v(2s_0-2)^{-1}
\mathscr Z'_v(s,s_0)
\]
is holomorphic in that half-plane.

The same corollary, together with the
$K_v$-averaging statement of
Corollary~\ref{cor:averaged-local-residue},
gives a holomorphic continuation of
$R_v(s,s_0)$ to
$|s_0-1|<\varepsilon$.
On the intersection of this disk with
$\Re s_0>1$, the continuation agrees
with the original normalized integral.
The two holomorphic functions therefore
glue to a holomorphic function on
$\Omega$.
\end{enumerate}

\emph{3. Uniform estimate for the
unramified remainder.}
Put $\sigma=\Re s$ and
$\sigma_0=\Re s_0\ge\frac12$.
Unitarity and genericity give
\[
|\alpha_v^{\pm1}|\le q_v^{1/2},
\qquad
|\alpha_v^{\pm2}|\le q_v,
\qquad
|\lambda_v|\le2q_v.
\]
Expanding the polynomial in
Theorem~\ref{thm:unramified-factor},
we obtain
\[
\begin{aligned}
|\mathcal D_E(s,s_0)-1|
\le{}&
q_v^{-(2\sigma+2\sigma_0-1)}
+q_v^{-(2\sigma+2\sigma_0-2)}\\
&+
q_v^{-(4\sigma+4\sigma_0-3)}
+2q_v^{-(2\sigma+\sigma_0-1)}\\
&+
2q_v^{-(2\sigma+3\sigma_0-3)}.
\end{aligned}
\]
Every exponent on the right is at
least $2\sigma-2$.
Likewise, the Euler parameters of
\[
\zeta_v(2s),\quad
\zeta_v(2s+4s_0-4),\quad
\zeta_v(2s+2s_0-2),
\]
and
$L_v(2s+2s_0-2,\pi,\mathrm{Ad})$
have absolute value at most
$q_v^{-(2\sigma-2)}$.
Since $\sigma>s_2\ge2$, all these
factors are $1+O(q_v^{-2})$,
uniformly in the indicated region.
The global factors
$L^S(s_0-\frac12,\pi_E)$ and
$L^S(2s_0-1,\pi,\mathrm{Ad})^{-1}$
are retained as meromorphic
functions and are not expanded
into Euler products near $s_0=1$.

\emph{4. Application of the identity theorem.}
Let
\[
\mathcal K=
\{\,s_0\in\mathbb C:
|s_0-\tfrac52|\le2\,\}.
\]
Its interior contains $s_0=1$
and intersects $\Re s_0>3$.
By Theorem~\ref{thm:first-expansion},
for $\Re s\ge c(\mathcal K)$
the left-hand side $Z'_E$
is meromorphic on a neighbourhood
of $\mathcal K$.

Taking $\Re s$ large enough to satisfy
both this condition and $\Re s>s_2$,
both sides of \eqref{eq:II-global}
are meromorphic on the connected
open set
\[
\Omega\cap\operatorname{int}\mathcal K.
\]
They agree on its nonempty open
subset where $\Re s_0>3$.
The identity theorem therefore
extends \eqref{eq:II-global}
throughout this connected set,
in particular to a neighbourhood
of $s_0=1$.

\emph{5. Evaluation of the residue.}
With the multiplicative Haar measure
of Section~\ref{subsec:measures},
Tate's residue formula \cite[Main Theorem~4.4.1]{Tate} gives
\[
\operatorname*{Res}_{u=0}
|\Delta_F|^{-1/2}\Lambda_F(u)
=-c_F.
\]
Consequently,
\[
\operatorname*{Res}_{s_0=1}
\Lambda_F(2s_0-2)
=
-\frac12c_F|\Delta_F|^{1/2}.
\]
The remaining factors in \eqref{eq:II-global} are holomorphic
at $s_0=1$. We specialize the two parts of $Q^S$ separately.
The extracted partial $L$-functions are evaluated through their
meromorphic continuations; by the definition of $r^S$,
\[
\zeta_F^S(2)^{-2}\,
\frac{L^S(\frac12,\pi_E)}{L^S(1,\pi,\mathrm{Ad})}
=\zeta_F^S(2)^{-3}\,r^S L^S(1,\eta).
\]
For $\Re s$ sufficiently large, the Euler product $\mathcal H^S$
converges normally in a neighbourhood of $s_0=1$ by
Step~2(4), and hence may be specialized termwise. For
$v\notin S$, the factors $\zeta_v(2s+4s_0-4)$ and
$\zeta_v(2s+2s_0-2)$ cancel at $s_0=1$, and
$\mathcal D_E(s,1)=D_{E_v}(\pi_v,s+\frac12)$; comparison with the
definition of $Z^\#_{E_v}$ shows that the local factor
$\mathcal H_v$ of $\mathcal H^S$ satisfies
\[
\mathcal H_v(s,1)=\zeta_v(2)^3Z^\#_{E_v}(s+\tfrac12).
\]
Hence
$\mathcal H^S(s,1)=\zeta_F^S(2)^3\prod_{v\notin S}Z^\#_{E_v}(s+\frac12)$, and
\[
Q^S(s,1)
=r^S L^S(1,\eta)\prod_{v\notin S}Z^\#_{E_v}(s+\tfrac12).
\]

For $v\in S$,
Corollary~\ref{cor:averaged-local-residue} gives
\[
2c_v^2R_v(s,1)
=|2|_v^{-1}|4d_E|_v^{1/2}L_v(1,\eta_v)\mathcal J_v(s).
\]

Combining these identities,
\[
\begin{aligned}
\operatorname*{Res}_{s_0=1}Z'_E
&=
-\frac14|\Delta_F|^{-11/2}
C_\pi
\prod_{v\in S}
\bigl(
|2|_v^{-1}|4d_E|_v^{1/2}
\bigr)\\
&\qquad\times
\left(
\prod_{v\in S}L_v(1,\eta_v)
\right)
L^S(1,\eta)\,r^S\\
&\qquad\times
\prod_{v\in S}\mathcal J_v(s)
\prod_{v\notin S}
Z^\#_{E_v}(s+\tfrac12).
\end{aligned}
\]

Finally,
\[
|2|_v^{-1}|4d_E|_v^{1/2}
=
|d_E|_v^{1/2}.
\]
Since $S$ contains every place
where this factor is not $1$,
the product formula gives
\[
\prod_{v\in S}
|2|_v^{-1}|4d_E|_v^{1/2}
=
\prod_v|d_E|_v^{1/2}
=1.
\]
Moreover, since $S$ contains
all archimedean places,
\[
\left(
\prod_{v\in S}L_v(1,\eta_v)
\right)L^S(1,\eta)
=
\Lambda(1,\eta).
\]
We conclude that
\[
\operatorname*{Res}_{s_0=1}
Z'_E(\Psi,f;s,s_0)
=
-\frac14|\Delta_F|^{-11/2}
\Lambda(1,\eta)\,r^S C_\pi
\prod_{v\in S}\mathcal J_v(s)
\prod_{v\notin S}
Z^\#_{E_v}(s+\tfrac12),
\]
as required.
\end{proof}

\subsection{Nonvanishing and comparison}\label{subsec:comparison}

The two residue formulas share the same Euler product.  Cancelling
it requires a choice of data for which it is nonzero.  We establish
this locally before making the comparison; otherwise the equality of
residues would give no information about $C_P$ and $C_\pi$.

\begin{lemma}\label{lem:nonvanishing}
Let $v\in S$, suppose that $\pi_v$ is
$E_v^\times$-distinguished, and fix $s$
with $\Re s$ sufficiently large. Then
there exist $\Psi_v$ and a $K_v$-finite
standard section $f_v$ such that
\[
\mathcal J_v(s)\ne0.
\]
\end{lemma}

\begin{proof}
Choose a $K_v$-finite vector $\phi_v\in\pi_v$ and
$g\in\mathrm{GL}_2(F_v)$ such that
\[
\alpha_{E_v}\bigl(\pi_v(g)\phi_v,\pi_v(g)\phi_v\bigr)>0;
\]
such $\phi_v$ and $g$ exist by Lemma~\ref{lem:alpha-nonzero}. Put
$\xi_0=\phi_v\otimes j_v\phi_v$.

\emph{1. Left $Q_v$-invariance.}
Let $Q_v=G'(F_v)\cap K_v$; equivalently, $Q_v$ is the image of
$(P(F_v)\cap K_{4,v})\times K_{2,v}$, so it contains the outer
factor $K_{2,v}$. In particular, the natural map induces
\[
(P(F_v)\cap K_{4,v})\backslash K_{4,v}
\simeq Q_v\backslash K_v.
\]
Write $q_M\in\widetilde M(F_v)$ for the Levi projection of
$q\in Q_v$, and let $\widetilde M(F_v)$ act on $\sigma_v$
through its $M$-component, as is possible since $f$ is
independent of the outer $\mathrm{GL}_2$-factor and
$\sigma_v$ has trivial central character.
We first observe that the
integrand defining $\mathcal J_v(s)$
is left $Q_v$-invariant in the
$K_v$-variable.

Indeed, for $q\in Q_v$ represented by $(nm,k_2)$
with $m=\operatorname{diag}(a,b)$,
the compact picture of normalized
induction gives
\[
\xi_f(q\kappa)
=
\sigma_v(q_M)\xi_f(\kappa),
\]
since $\delta_P(m)=u(m)=1$ on $Q_v$.
Moreover,
\[
\Phi^{\mathrm{loc}}_{\rho(q)\Psi_v}
=
\Phi^{\mathrm{loc}}_{\Psi_v}
(\,\cdot\,q_M).
\]
To see the latter identity, recall
that $N$ acts on the slice
$\{\ell=0\}$ by translations in
$\lambda$, which preserve $d\lambda$,
while the Levi and outer components of $q$ act on
$\lambda\in F_v^2$ by
$\lambda\mapsto\det(b)\lambda k_2$
with $|\det b|_v=|\det k_2|_v=1$.
The equivariance of
$\mathcal Z^{\mathrm{SW}}_{E_v}$
under $H(F_v)$, together with
$|\omega(q_M)|_v=1$, therefore gives
the asserted invariance.

Consequently,
\[
\mathcal J_v(s)
=
\int_{Q_v\backslash K_v}
\mathcal Z^{\mathrm{SW}}_{E_v}
\bigl(
\Phi^{\mathrm{loc}}_{\rho(\kappa)\Psi_v},
\xi_f(\kappa);s+\tfrac12
\bigr)\,d\dot\kappa.
\]

\emph{2. Choice of the Schwartz function.}
By \cite[Lemma~4.8(3)]{SW1} and its proof,
there exists
$\Phi\in\mathcal S(V(F_v))$,
supported in a small compact neighbourhood of
$x_E\cdot\rho_V(g,g,1)$ contained in $V_E(F_v)$,
such that
\[
\mathcal Z^{\mathrm{SW}}_{E_v}
(\Phi,\xi_0;s+\tfrac12)\ne0.
\]
Choose
\[
\Psi_v=\Phi\otimes\Psi_1
\]
in the coordinates $(x;\ell,\lambda)$,
where
$\Psi_1\in
\mathcal S(F_v^2\times F_v^2)$
satisfies
\[
\int_{F_v^2}\Psi_1(0,\lambda)
\,d\lambda=1.
\]
Then
$\Phi^{\mathrm{loc}}_{\Psi_v}=\Phi$.

\emph{3. Nonarchimedean places.}
Suppose first that $v$ is
nonarchimedean. Choose a sufficiently
small compact open subgroup
$U\subset K_v$ such that
\[
\sigma_v(q_M)\xi_0=\xi_0
\qquad(q\in Q_v\cap U)
\]
and
\[
\rho(u)\Psi_v=\Psi_v
\qquad(u\in U).
\]
Such a subgroup exists by smoothness
of $\sigma_v$ and of the action on
the Schwartz--Bruhat space.

Define a function on $K_v$ by
\[
f_v(qu)=\sigma_v(q_M)\xi_0,
\qquad q\in Q_v,\quad u\in U,
\]
and set $f_v=0$ outside $Q_vU$.
This is well defined: if
$qu=q'u'$, then
$q'^{-1}q=u'u^{-1}\in Q_v\cap U$,
which fixes $\xi_0$ under the
inducing representation.

The function is locally constant
and $Q_v$-equivariant. By Iwasawa
decomposition it extends uniquely
to a standard section of the
normalized induced representation.
Moreover, since $U$ is a subgroup,
the section is right $U$-invariant.
As $U$ has finite index in the
compact group $K_v$, the section
is $K_v$-finite.

For $\kappa=qu\in Q_vU$,
the $Q_v$-invariance established
above and the $U$-invariance of
$\Psi_v$ give
\[
\mathcal Z^{\mathrm{SW}}_{E_v}
\bigl(
\Phi^{\mathrm{loc}}_{\rho(\kappa)\Psi_v},
\xi_f(\kappa);s+\tfrac12
\bigr)
=
\mathcal Z^{\mathrm{SW}}_{E_v}
(\Phi,\xi_0;s+\tfrac12).
\]
Hence
\[
\begin{aligned}
\mathcal J_v(s)
={}&
\operatorname{vol}
(Q_v\backslash Q_vU)\\
&\times
\mathcal Z^{\mathrm{SW}}_{E_v}
(\Phi,\xi_0;s+\tfrac12)
\ne0.
\end{aligned}
\]

\emph{4. Archimedean places.}
Suppose now that $v$ is
archimedean. Choose a smooth local
section
\[
x\longmapsto k_x
\]
of the quotient map
$K_v\to Q_v\backslash K_v$
on a neighbourhood $\mathcal U$
of the identity coset, with
$k_{Q_v}=1$.

Let
$\varphi\in C_c^\infty(\mathcal U)$
be nonnegative and satisfy
\[
\int_{Q_v\backslash K_v}
\varphi(x)\,d\dot\kappa(x)=1.
\]
Define
\[
f_v^{U}(qk_x)
=
\varphi(x)\sigma_v(q_M)\xi_0,
\qquad q\in Q_v,\quad x\in\mathcal U,
\]
and extend it by zero outside
the inverse image of
$\operatorname{supp}\varphi$.

This defines a smooth
$Q_v$-equivariant function on $K_v$
and hence a smooth standard section.
By the quotient integral formula,
\[
\mathcal J_v(\Psi_v,f_v^U;s)
=
\int_{\mathcal U}
\varphi(x)
F(x)\,d\dot\kappa(x),
\]
where
\[
F(x)=
\mathcal Z^{\mathrm{SW}}_{E_v}
\bigl(
\Phi^{\mathrm{loc}}_{\rho(k_x)\Psi_v},
\xi_0;s+\tfrac12
\bigr).
\]
Let $W\subset\sigma_v$ be the span of the vectors
$\sigma_v(q_M)\xi_0$, $q\in Q_v$ (so $q_M$ ranges over $K_{M,v}$);
since $\xi_0$ is $K_{M,v}$-finite, $W$ is finite-dimensional. We
claim that there is a right half-plane $\Re s>\sigma_0$ such that, for
every compact subset $\Omega$ of it, there exist $m_0=m_0(\Omega)\ge0$
and $C_\Omega>0$ with
\begin{equation}\label{eq:SW-seminorm}
\bigl|\mathcal Z^{\mathrm{SW}}_{E_v}(\Phi',\xi;s+\tfrac12)\bigr|
\le C_\Omega\,\|\xi\|_W\sup_{x\in V(F_v)}(1+\|x\|)^{m_0}|\Phi'(x)|
\end{equation}
for all $s\in\Omega$, $\xi\in W$ and $\Phi'\in\mathcal S(V(F_v))$.
Choose a finite-dimensional subspace $U\subset\pi_v$
such that $W\subset U\otimes j_vU$.
By polarization, $U\otimes j_vU$ is spanned by
finitely many vectors of the form
$\phi_i\otimes j_v\phi_i$.
It therefore suffices to establish the claimed
estimate for these finitely many diagonal tensors
and then restrict it to $W$.
By \cite[Lemma~4.8(1)]{SW1}, for $\Re s>\sigma_0$ the
corresponding zeta integrals converge absolutely for the Gaussian
$\Phi_0(x)=e^{-\|x\|^2}$. For each $i$ the integrand is
$\Phi'(x)$ times a weight $\mathcal W_i(x;s)$ on $V_E(F_v)$.
The positive scalars act through the centre of the third
$\mathrm{GL}_2$-factor, which fixes the first two components; since
the toric coefficient depends only on the point of $V_E(F_v)$
\cite[Remark~4.7]{SW1}, it is invariant under $x\mapsto tx$, and
\[
\mathcal W_i(tx;s)=t^{a(s)}\mathcal W_i(x;s),\qquad t>0,
\]
with $a(s)$ independent of $i$.
Writing $x=r\omega$ in real polar coordinates,
absolute convergence for $\Phi_0$ implies
integrability of $|\mathcal W_i(\omega;s)|$
over the unit sphere. For a nonzero weight it
also implies $N+\Re a(s)>0$, where $N=\dim_{\mathbb R}V(F_v)$;
the estimate is trivial for an identically zero weight.
Taking $m_0>N+\Re a(s)$ controls the radial
integral at infinity and proves
\eqref{eq:SW-seminorm} for each fixed $s$.
The estimate is locally uniform in $s$. Indeed, by
the definition of $\mathcal Z^{\mathrm{SW}}_{E_v}$ in
Section~\ref{sec:local-residue}, the weight
$\mathcal W_i(x;s)$ depends on $s$ only through the
factor $|P(x)|_v^{s-3/2}$. For a compact subset
$\Omega$ of the half-plane $\Re s>\sigma_0$, put
$\sigma_-=\min_{s\in\Omega}\Re s$ and
$\sigma_+=\max_{s\in\Omega}\Re s$. Since
$t^\sigma\le t^{\sigma_-}+t^{\sigma_+}$ for $t>0$ and
$\sigma_-\le\sigma\le\sigma_+$, the angular integrals
are bounded uniformly for $s\in\Omega$ by those at
$\sigma_\pm$. Moreover, $a(s)$ is affine in $s$ and
$N+\Re a(s)>0$ throughout the half-plane, so
$\inf_{\Omega}(N+\Re a(s))>0$; choosing
$m_0>N+\sup_{\Omega}\Re a(s)$ controls the radial
integral at infinity. Hence \eqref{eq:SW-seminorm}
holds for $s\in\Omega$ with $m_0=m_0(\Omega)$ and
$C=C_\Omega$.
Since $k\mapsto\rho(k)\Psi_v$ and $\Psi\mapsto\Phi^{\mathrm{loc}}_\Psi$
are continuous in the Schwartz topologies,
\eqref{eq:SW-seminorm} implies that $x\mapsto F(x)$ is continuous.

At the identity coset,
\[
F(Q_v)=
\mathcal Z^{\mathrm{SW}}_{E_v}
(\Phi,\xi_0;s+\tfrac12)\ne0.
\]
By continuity, we may choose
$\operatorname{supp}\varphi$
sufficiently small that
\[
|F(x)-F(Q_v)|
<
\tfrac12|F(Q_v)|
\]
throughout its support.
Since $\varphi\ge0$ and has
integral $1$, it follows that
\[
\left|
\mathcal J_v(\Psi_v,f_v^U;s)
-
F(Q_v)
\right|
<
\tfrac12|F(Q_v)|.
\]
In particular,
$\mathcal J_v(\Psi_v,f_v^U;s)\ne0$.

Finally, $f_v^U$ takes values in $W$. Let $e_n\in C^\infty(K_v)$ be
finite linear combinations of matrix coefficients forming an
approximate identity, and put $f_{v,n}=R(e_n)f_v^U$, where $R$ denotes
right convolution. Each $f_{v,n}$ is a $K_v$-finite standard section
with values in $W$, and $f_{v,n}\to f_v^U$ uniformly on $K_v$.
The family $\{\Phi^{\mathrm{loc}}_{\rho(k)\Psi_v}:k\in K_v\}$ is
compact, hence bounded, in the Schwartz topology. Since
$\mathcal J_v$ depends on the section only through its values on
$K_v$, the seminorm estimate \eqref{eq:SW-seminorm}, together with
the boundedness of this family in the Schwartz topology, gives
$\mathcal J_v(\Psi_v,f_{v,n};s)\to\mathcal J_v(\Psi_v,f_v^U;s)\ne0$,
so $\mathcal J_v(\Psi_v,f_{v,n};s)\ne0$ for some $n$.
This completes the proof.
\end{proof}

\begin{theorem}\label{thm:CP-Cpi}
Suppose that $\pi_v$ is $E_v^\times$-distinguished
for every place $v$. Then
\[
C_P
=
\frac12\,|\Delta_F|^{-1/2}\,
\frac{r^S}{\Lambda(1,\eta)}\,C_\pi.
\]
In particular, under this local hypothesis,
\[
C_P=0
\quad\Longleftrightarrow\quad
\mathcal P_E\equiv0
\quad\Longleftrightarrow\quad
L^S\bigl(\tfrac12,\pi_E\bigr)=0.
\]
\end{theorem}

\begin{proof}
Choose $s_*$ with $\Re s_*$ sufficiently large for
Lemma~\ref{lem:nonvanishing}. For each $v\in S$, choose $\Psi_v$ and $f_v$ as
constructed in the proof of Lemma~\ref{lem:nonvanishing}, so that
$\mathcal J_v(s_*)\ne0$.
For these fixed data, each $\mathcal J_v(s)$ is holomorphic on a right
half-plane containing $s_*$. At a finite place, the construction gives
$\mathcal J_v(s)=\mathrm{vol}(Q_v\backslash Q_vU)\,
\mathcal Z^{\mathrm{SW}}_{E_v}(\Phi,\xi_0;s+\tfrac12)$ for all $s$, with
$\Phi$ supported in a compact subset of $V_E(F_v)$, on which $|P|_v$ is
bounded above and away from zero. Since $|P|_v$ takes only finitely
many values on this support and the toric coefficient is continuous
there, the integral is a finite linear combination of powers of
$q_v^{-s}$; hence $\mathcal J_v$ is entire
(cf.\ \cite[Lemma~4.8(3)]{SW1}). At an archimedean place, the local
Suzuki--Wakatsuki integrals converge absolutely for $\Re s>\sigma_0$
\cite[Lemma~4.8(1)]{SW1}; the restriction of $f_v$ to
$K_v$ takes values in $W$, and the estimate \eqref{eq:SW-seminorm},
together with the boundedness of
$\{\Phi^{\mathrm{loc}}_{\rho(k)\Psi_v}:k\in K_v\}$ in the Schwartz
topology, makes the convergence locally uniform in $s$ and uniform in
$k\in K_v$; hence $\mathcal J_v(s)$ is holomorphic there as well. As $\mathcal J_v(s_*)\ne0$, it is not identically zero. Since $S$ is finite, we may choose $s$ with
$\Re s$ sufficiently large for Propositions~\ref{prop:first-residue}
and~\ref{prop:second-residue}, outside the discrete zero sets of all the
$\mathcal J_v$. Thus $\mathcal J_v(s)\ne0$ for every $v\in S$.
For $v\notin S$, the factors
$Z^\#_{E_v}(s+\tfrac12)$ form an absolutely
convergent, nonzero Euler product.
Consequently,
\[
\mathcal A(s):=
\prod_{v\in S}\mathcal J_v(s)
\prod_{v\notin S}Z^\#_{E_v}(s+\tfrac12)
\ne0.
\]

Equating the residue formulas of
Propositions~\ref{prop:first-residue}
and \ref{prop:second-residue} gives
\[
\begin{aligned}
&-\frac12|\Delta_F|^{-5}
 \Lambda(1,\eta)^2 C_P\,\mathcal A(s)
\\
&\qquad =
-\frac14|\Delta_F|^{-11/2}
 \Lambda(1,\eta)r^S C_\pi\,\mathcal A(s).
\end{aligned}
\]
Cancelling $\mathcal A(s)$ and using
$\Lambda(1,\eta)\ne0$, since $\eta$
is the nontrivial quadratic Hecke
character associated with the field
extension $E/F$, yields
\[
C_P
=
\frac12|\Delta_F|^{-1/2}
\frac{r^S}{\Lambda(1,\eta)}C_\pi.
\]

We now establish the asserted equivalence.
The constant $C_\pi$ is nonzero by the
proof of Proposition~\ref{prop:second-expansion}:
the global diagonal pairing is nonzero
on $\phi\otimes\overline{\phi}$ for any
nonzero cusp form $\phi\in\pi$.
Moreover,
\[
r^S=
\frac{\zeta_F^S(2)
      L^S(\tfrac12,\pi_E)}
     {L^S(1,\pi,\mathrm{Ad})
      L^S(1,\eta)}.
\]
The factors
$\zeta_F^S(2)$,
$L^S(1,\pi,\mathrm{Ad})$, and
$L^S(1,\eta)$ are finite and nonzero,
the nonvanishing of the partial adjoint
factor having been established in the
proof of Proposition~\ref{prop:second-residue}.
Thus
\[
C_P=0
\quad\Longleftrightarrow\quad
L^S(\tfrac12,\pi_E)=0.
\]
Finally, by the convention following
\eqref{eq:CP}, $C_P=0$ if and only if
$\mathcal P_E\equiv0$.
\end{proof}

\begin{remark}[Non-distinguished local components]
\label{rem:nondistinguished}
Without the local distinguishedness hypothesis, the
equivalence with the vanishing of $L^S(\tfrac12,\pi_E)$ is
not claimed. If $\pi_v$ is not $E_v^\times$-distinguished
at some place $v$, the corresponding local toric Hom-space
vanishes; hence $L_v\equiv0$, and the global toric
functional $\mathcal P_E$ vanishes identically. This does
not, by itself, imply the vanishing of
$L^S(\tfrac12,\pi_E)$. Such a place belongs to $S$, since
an unramified representation $\pi_v$ with unramified $E_v$
is distinguished \cite{Tunnell,Saito}.
\end{remark}

The finite set $S$ may be enlarged while retaining the conditions
imposed above. When a place $v$ is added to $S$, both $C_P$ and
$r^S$ are multiplied by $r_v^{-1}$, so the identity of
Theorem~\ref{thm:CP-Cpi} is unchanged. Thus, for any factorizable
$\phi$, we may choose $S$ to contain every place where
$\phi_v\neq\phi_v^\circ$.

\begin{corollary}\label{cor:hermitian-CP}
For $\phi=\bigotimes_v\phi_v\in\pi$ with $\phi_v=\phi_v^\circ$ for $v\notin S$,
\[
\frac{\mathcal P_E(\phi\otimes\bar\phi)}{\langle\phi,\phi\rangle_{\mathrm{Tam}}}
=\frac12\,|\Delta_F|^{-1/2}\,
\frac{\zeta_F^S(2)\,L^S(\tfrac12,\pi_E)}
{\Lambda(1,\eta)\,L^S(1,\pi,\mathrm{Ad})\,L^S(1,\eta)}
\prod_{v\in S}\frac{\alpha_{E_v}(\phi_v,\phi_v)}{\langle\phi_v,\phi_v\rangle_v}.
\]
\end{corollary}

\begin{proof}
If some $\pi_v$ is not $E_v^\times$-distinguished,
then $\mathcal P_E\equiv0$, since a nonzero
global toric functional would induce a
nonzero local $E_v^\times$-invariant
functional. Moreover,
$\alpha_{E_v}\equiv0$, and such a place
belongs to $S$ \cite{Tunnell,Saito}.
Thus both sides vanish.

Suppose now that $\pi_v$ is
$E_v^\times$-distinguished for every $v$.
By \eqref{eq:CP}, with
$L_v(\phi_v\otimes j_v\phi_v)
=\alpha_{E_v}(\phi_v,\phi_v)$ and
the normalization at unramified places,
\[
\mathcal P_E(\phi\otimes\bar\phi)
=
C_P\prod_{v\in S}
\alpha_{E_v}(\phi_v,\phi_v).
\]
Theorem~\ref{thm:CP-Cpi} gives
\[
C_P=
\frac12|\Delta_F|^{-1/2}
\frac{r^S}{\Lambda(1,\eta)}C_\pi.
\]
On the other hand, by
Remark~\ref{rem:hermitian-datum}
and $\langle\phi_v^\circ,\phi_v^\circ\rangle_v=1$
for $v\notin S$,
\[
C_\pi=
\frac{\langle\phi,\phi\rangle_{\mathrm{Tam}}}
{\prod_{v\in S}
\langle\phi_v,\phi_v\rangle_v}.
\]
Substituting these identities and
the definition of $r^S$ yields
the asserted formula.
\end{proof}

\subsection{Conversion to invariant form}\label{subsec:invariant-form}

The comparison theorem uses the coordinate measures of the local orbital
calculation. We make two separate changes of normalization. First, the
connected-stabilizer measure is identified with the product of torus
Tamagawa measures; this introduces no scalar. Second, the local quotient
measures are replaced by quotients of self-dual additive measures; this
introduces the discriminant factor that cancels the one in the comparison
theorem. We derive the constant in Theorem~\ref{thm:main} from
Corollary~\ref{cor:hermitian-CP}, rather than assuming it.

We use the following conventions throughout:
\begin{itemize}
\item At every place, $L_v(s,\eta_v)=\zeta_{E_v}(s)/\zeta_{F_v}(s)$,
with the local zeta factors used in $\Lambda_F$ and $\Lambda(s,\eta)$.
The same factors are used in the measure $dh_{E_v}$ of \cite[(4.4)]{SW1}
in \S\ref{subsec:first-residue}.
\item The local measure in the main theorem is
\[
dt_v=\frac{\zeta_{E_v}(1)|x|_{E_v}^{-1}dx_{E_v}^{\mathrm{sd}}}
{\zeta_{F_v}(1)|x|_v^{-1}dx_v^{\mathrm{sd}}},
\]
where the additive measures are self-dual for $\psi_v$ and
$\psi_v\circ\operatorname{Tr}_{E_v/F_v}$, respectively.
\item $d\tau_T$ is the Tamagawa measure on
$T=R_{E/F}\mathbb G_m/\mathbb G_m$, with
$\operatorname{vol}[T]=\tau(T)=2$; the toric period $P_E$ in the main
theorem uses $d\tau_T$. Indeed, $T$ is split by the cyclic extension
$E/F$, so $\tau(T)=\#H^1(\operatorname{Gal}(E/F),X^*(T))$ by
\cite[\S5, Corollary to the Main Theorem]{Ono63}. Here $X^*(T)$ is the
rank-one lattice of characters of $R_{E/F}\mathbb G_m$ that are trivial
on $\mathbb G_m$, and the nontrivial element of $\operatorname{Gal}(E/F)$
acts on it by $-1$; hence this group has order~$2$.
\end{itemize}

The convention for $L_v(1,\eta_v)$ at archimedean places is fixed
here, independently of the normalization of the local Hecke factors
in \cite{SW1}. It enters only through the measures $dh_{E_v}$ and
$dt_v$, the local toric functionals, and the prefactor
$2c_v/L_v(1,\eta_v)^2$ in the definition of
$\mathcal Z^{\mathrm{SW}}_{E_v}$ in Section~\ref{sec:local-residue},
all of which are formed with this convention. The identity
\cite[(4.7)]{SW1} used in the proof of
Proposition~\ref{prop:first-residue} is a change of variables from
$Y_{x_E}(F_v)$ to $V(F_v)$ and involves no $L$-factor, while the
global measures of \cite[\S4.3]{SW1} contain no archimedean
convergence factors. Thus no identification of the archimedean
convention of \cite{SW1} with ours is required; see also
Remark~\ref{rem:normalization-infinity}.

\emph{1. The connected-stabilizer measure.}
We first make explicit the algebraic-group identification underlying the
measure comparison in \cite[\S4.3, (4.4)--(4.5)]{SW1}. For the standard
representative $x_E=(1,\delta_E)$ of the quaternionic cube, the
stabilizer equation says that the two quaternionic components carry
the $F$-subspace $E=F+F\delta_E$ into itself. The normalizer of
$E^\times$ in $D^\times$ has identity component $E^\times$;
its remaining component induces the nontrivial automorphism of $E/F$.
For any pair $(a,b)\in E^\times\times E^\times$ in the identity
component, the $\mathrm{GL}_2$ component is uniquely determined by
the $F$-linear automorphism $z\mapsto a^{-1}zb$ of $E$ (or its inverse,
according to the convention for the third group action). Thus the
connected stabilizer upstairs is $(R_{E/F}\mathbb G_m)^2$.
The kernel of the effective action intersects this parametrized
stabilizer in exactly the scalar pairs $(a,b)\in(F^\times)^2$:
indeed, a pair maps to the identity in the effective group precisely
when both quaternionic components are scalar, and then the third
component is the corresponding scalar matrix. Therefore the quotient
has no additional finite kernel, and
\[
H^\circ_{x_E}\simeq
\frac{(R_{E/F}\mathbb G_m)^2}{\mathbb G_m^2}
\simeq (R_{E/F}\mathbb G_m/\mathbb G_m)^2=T\times T
\]
as $F$-algebraic groups, not merely up to isogeny. In particular,
no degree factor enters the connected-stabilizer measure.

For $z=a+b\delta_E$, write $N(z)=a^2-d_Eb^2$ and set
$\omega_E=da\,db/N(z)$ and $\omega_F=dz/z$. The quotient rational
gauge form on $T$ is $\omega_T=\omega_E/\omega_F$. The local
measures in \cite[(4.4)]{SW1} are
\[
dh_{E_v}=c_vL_v(1,\eta_v)|\omega_E|_v,
\qquad d^\times z=c_v|\omega_F|_v,
\]
so that their quotient is
\[
dh_v=L_v(1,\eta_v)|\omega_T|_v.
\]
The factors $c_v$ cancel, including at the archimedean places.
With the rational gauge form $\omega_T$, Ono's normalization
\cite{Ono61} is
\[
d\tau_T=L_f(1,\eta)^{-1}|\Delta_F|^{-1/2}
\prod_{v<\infty}L_v(1,\eta_v)|\omega_T|_v
\prod_{v\mid\infty}|\omega_T|_v.
\]
Here the convergence factors are included at finite places only.

On the other hand, \cite[\S4.3]{SW1} defines the connected-stabilizer
measure from its rational invariant gauge form $\omega_{H^\circ}$ by
\[
d\mu_{H^\circ}=L_f(1,\eta)^{-2}|\Delta_F|^{-1}
\prod_v d_v|\omega_{H^\circ}|_v,
\]
where $d_v=L_v(1,\eta_v)^2$ at unramified finite places and $d_v=1$
otherwise. Under the above $F$-isomorphism, $\omega_{H^\circ}$ and
$\omega_T\boxtimes\omega_T$ are nonzero rational invariant top forms
on the same connected group. The one-dimensionality over $F$ of the
space of invariant top forms therefore gives
\[
\omega_{H^\circ}=\kappa\,\omega_T\boxtimes\omega_T,
\qquad \kappa\in F^\times.
\]
By the product formula, $\prod_v|\kappa|_v=1$. At ramified finite
places $L_v(1,\eta_v)=1$, so the $d_v$ are exactly the squares of
the finite-place convergence factors of $d\tau_T$. Both measures use
the same coordinate additive measures at infinity (Lebesgue measure
at real places and $2\,dx\,dy$ at complex places). It follows that
\[
d\mu_{H^\circ}=d\tau_T\boxtimes d\tau_T.
\]
Consequently,
$\mathcal P_E(\phi\otimes\bar\phi)=|P_E(\phi)|^2$.
The factor $1/2$ associated with the disconnected stabilizer is
already accounted for in the orbit unfolding; it is not a factor in
either torus measure.

\emph{2. Comparison of the local torus measures.}
We compare the coordinate measures used in the local
zeta integrals with the quotient measures induced by
self-dual additive Haar measures.

Let $|\omega_{T,v}|_{\mathrm{coord}}$ denote the
quotient measure
\[
\frac{da_v\,db_v/|N|_v}{dz_v/|z|_v}
\]
on $T(F_v)=E_v^\times/F_v^\times$, with the
coordinate measures of \cite[(4.4)]{SW1}.
The local measure occurring in the residue
calculation is
\[
dh_v=L_v(1,\eta_v)|\omega_{T,v}|_{\mathrm{coord}}.
\]
On the other hand, the measure $dt_v$ used in
the invariant formulation is defined by the
same quotient construction, but with the
self-dual additive measures on $E_v$ and $F_v$.
Consequently,
\[
dt_v=\rho_v\,dh_v,
\qquad
\rho_v=
\frac{dx_{E_v}^{\mathrm{sd}}/(da_v\,db_v)}
     {dx_v^{\mathrm{sd}}/dx_v}.
\]

We first compute $\rho_v$. At a finite place,
write
\[
dx_v^{\mathrm{sd}}=c_{\psi,v}\,dx_v,
\qquad
\operatorname{vol}(\mathfrak o_v,dx_v)=1.
\]
In the basis $(1,\delta_E)$, where
$\delta_E^2=d_E$, the trace pairing
$(x,y)\mapsto\operatorname{Tr}_{E_v/F_v}(xy)$
has Gram matrix
\[
\begin{pmatrix}
2&0\\
0&2d_E
\end{pmatrix},
\qquad
\det=4d_E.
\]
The change-of-pairing formula for self-dual
Haar measures therefore gives
\[
dx_{E_v}^{\mathrm{sd}}
=c_{\psi,v}^{\,2}|4d_E|_v^{1/2}
\,da_v\,db_v,
\qquad
\rho_v=c_{\psi,v}|4d_E|_v^{1/2}.
\]
The same formula holds at archimedean places,
with $c_{\psi,v}=1$ under our conventions.

In particular, $\rho_v=1$ whenever the base
additive measure is self-dual and the
coordinate lattice
$\mathfrak o_v+\mathfrak o_v\delta_E$
is self-dual for the trace pairing.
The set $S$ in
\S\ref{subsec:comparison-setup}
contains all exceptional places, including
those dividing $\Delta_F$ and those where
$\mathfrak o_{E_v}\ne
\mathfrak o_v+\mathfrak o_v\delta_E$.
Hence
\[
\rho_v=1\qquad(v\notin S).
\]

We next compute the global product. The
$F$-linear isomorphism
\[
F^2\xrightarrow{\;\sim\;}E,
\qquad (a,b)\longmapsto a+b\delta_E,
\]
identifies $\mathbb A_F^2$ with $\mathbb A_E$.
By \S\ref{subsec:measures}, the coordinate
measures satisfy
\[
\operatorname{vol}
\left(E\backslash\mathbb A_E,
\prod_vda_v\,db_v\right)
=
\operatorname{vol}
\left(F\backslash\mathbb A_F,
\prod_vdx_v\right)^2
=|\Delta_F|.
\]
The global self-dual measure for
$\psi\circ\operatorname{Tr}_{E/F}$ gives
$E\backslash\mathbb A_E$ volume one,
whereas
\[
\prod_vdx_v^{\mathrm{sd}}
=|\Delta_F|^{-1/2}\prod_vdx_v.
\]
It follows that
\[
\prod_vdx_{E_v}^{\mathrm{sd}}
=|\Delta_F|^{-1}\prod_vda_v\,db_v,
\]
and consequently
\[
\prod_v\rho_v=|\Delta_F|^{-1/2}.
\]
This argument does not require
$(1,\delta_E)$ to be an integral basis
of $\mathfrak o_E$.

Since $\rho_v=1$ outside $S$, the relation
$dt_v=\rho_vdh_v$ gives
\[
\prod_{v\in S}
\alpha_{E_v}(\phi_v,\phi_v)
=
\left(\prod_{v\in S}\rho_v^{-1}\right)
\prod_{v\in S}
\alpha_{E_v}^{dt}(\phi_v,\phi_v)
=
|\Delta_F|^{1/2}
\prod_{v\in S}
\alpha_{E_v}^{dt}(\phi_v,\phi_v),
\]
where $\alpha_{E_v}^{dt}$ is the same
local toric integral formed with $dt_v$.

Finally, combining $dt_v=\rho_v\,dh_v$ and
$\prod_v\rho_v=|\Delta_F|^{-1/2}$ with the
formula for $d\tau_T$ in Step~1, the restricted
product of the self-dual quotient measures satisfies
\[
\prod_vdt_v
=
|\Delta_F|^{-1/2}
\prod_vL_v(1,\eta_v)
|\omega_{T,v}|_{\mathrm{coord}}
=
\Lambda(1,\eta)\,d\tau_T.
\]
Thus
\[
d\tau_T=C_T\prod_vdt_v,
\qquad
C_T=\Lambda(1,\eta)^{-1}.
\]
This agrees with the Haar-measure constant
$C_0=\Lambda(1,\eta)^{-1}$ computed in \cite[\S6]{II10}.

\emph{3. Normalized periods.}
Define $\alpha_{E_v}^\#=r_v^{-1}\alpha_{E_v}^{dt}$,
where $r_v$ is given by \eqref{eq:rv}. At $v\notin S$ we have
$dt_v=dh_v$, and the unramified calculation gives
$\alpha_{E_v}^\#(\phi_v^\circ,\phi_v^\circ)=1$. Moreover,
\[
r^S\prod_{v\in S}r_v
=\frac{\Lambda_F(2)\,\Lambda(\frac12,\pi_E)}
{\Lambda(1,\pi,\mathrm{Ad})\,\Lambda(1,\eta)}.
\]

\emph{Conclusion.}
By Step~1 the global bilinear period is the squared Tamagawa toric
period. By Step~2 the factor $|\Delta_F|^{-1/2}$ in
Corollary~\ref{cor:hermitian-CP} cancels the factor
$|\Delta_F|^{1/2}$ from the change of local quotient measures.
Using Step~3 then yields
\begin{equation}\label{eq:main-constant}
\frac{|P_E(\phi)|^2}{\langle\phi,\phi\rangle_{\mathrm{Tam}}}
=\frac12\cdot
\frac{\Lambda_F(2)\,\Lambda(\frac12,\pi_E)}
{\Lambda(1,\pi,\mathrm{Ad})\,\Lambda(1,\eta)^2}
\prod_v\frac{\alpha_{E_v}^\#(\phi_v,\phi_v)}
{\langle\phi_v,\phi_v\rangle_v}.
\end{equation}
This is Theorem~\ref{thm:main}, with constant $\frac12$.
No factor $c_F$ appears after Proposition~\ref{prop:first-residue}.
The Petersson normalization of \cite{SW1} does not enter:
$C_\pi$ is defined directly from the Tamagawa pairing $B$.

\begin{remark}[Normalization conventions at infinity]
\label{rem:normalization-infinity}
The archimedean local $L$-factors and toric
measures must be normalized consistently.
To see the dependence on conventions,
suppose formally that, at each
archimedean place $v$, the factor
$L_v(1,\eta_v)$ is replaced by
$\lambda_v L_v(1,\eta_v)$, with
$\lambda_v>0$, while the invariant gauge
form, the measure-conversion factor
$\rho_v$, and the other local $L$-factors
are held fixed.

Under the corresponding change of
normalization, the measure
\[
dh_v=L_v(1,\eta_v)
|\omega_{T,v}|_{\mathrm{coord}}
\]
and hence $dt_v=\rho_v dh_v$ are
multiplied by $\lambda_v$. Moreover,
the normalizing factor $r_v^{-1}$
is multiplied by $\lambda_v$, since
$L_v(1,\eta_v)$ occurs in the denominator
of $r_v$. Consequently,
\[
\alpha^\#_{E_v}
=r_v^{-1}\alpha^{dt}_{E_v}
\quad\longmapsto\quad
\lambda_v^2\alpha^\#_{E_v}.
\]
On the other hand,
$\Lambda(1,\eta)^{-2}$ is multiplied by
$\prod_{v\mid\infty}\lambda_v^{-2}$.
These changes cancel in
\eqref{eq:main-constant}.

This calculation is a consistency check for the
normalization of $L_v(1,\eta_v)$ and the associated
toric measure $dt_v$. It does not permit changing
this local $L$-factor while retaining the original
measure normalization of \cite[(4.4)]{SW1}. In particular,
our convention
\[
\zeta_{\mathbb C}(s)
=(2\pi)^{1-s}\Gamma(s),
\qquad dz=2\,dx\,dy,
\]
is used consistently in the completed
$L$-factors and the associated local
measures. The other archimedean factors
need no such coordination: $\zeta_v(2)$,
$L_v(\tfrac12,\pi_E)$ and
$L_v(1,\pi,\mathrm{Ad})$ each occur once
in the completed $L$-values of
\eqref{eq:main-constant} and once,
inversely, in $r_v^{-1}$, so their
normalization cancels place by place.
Only $L_v(1,\eta_v)$ also enters the
measure $dt_v$, which is why its
normalization must be fixed together
with the measures.
\end{remark}

\appendix
\section{Proof of Lemma~\ref{lem:levi-measure}}\label{app:levi-measure}

We use the notation and measure conventions of \S\ref{subsec:measures}.

Since $A\simeq\mathbb G_m$ is
split and central, Hilbert's
theorem~90 gives
\[
H^1(L,A)=1
\]
for every field extension $L/F$. Moreover,
\[
H^1(\mathfrak o_v,A)
=
\operatorname{Pic}(\mathfrak o_v)=0.
\]
Exactness on adelic points
follows by lifting place by
place, using integral lifts
at almost all finite places,
as in Section~\ref{subsec:group}.
This proves~(1).

Since $A(\mathbb A)$ is a closed
central subgroup of
$\widetilde M(\mathbb A)$,
Weil's quotient-integral formula
gives a unique Haar measure
$dm$ satisfying the identity
in~(2). The exactness in~(1)
identifies the quotient with
$H(\mathbb A)$.

For~(3), compatibility of
invariant gauge forms with
the quotient by $A$, together
with exactness on $F_v$-points,
gives the local quotient-measure
identity
\[
|\omega_{\widetilde M,v}|
=
|\omega_{A,v}|\,
|\omega_{H,v}|.
\]
Combining this identity with
the normalizations of $dh$
and $d^\times c$ yields
\[
dm=
c_F^{-1}|\Delta_F|^{-5-1/2}
\prod_v c_v^2
|\omega_{\widetilde M,v}|.
\]

Recall that
\[
\chi_P(m)
=
\det\bigl(
\operatorname{Ad}(m)|_{\mathfrak n}
\bigr).
\]
In the coordinates $p=nm$,
the pullback of the left-invariant
gauge form $\omega_{G'}$ is
\[
\chi_P(m)^{-1}
\omega_N\wedge\omega_{\widetilde M}.
\]
Since
\[
\delta_P(m)=|\chi_P(m)|,
\]
we obtain
\[
|\omega_{G',v}|
=
\delta_{P,v}(m)^{-1}
|\omega_{N,v}|
|\omega_{\widetilde M,v}|.
\]
The asserted formula for
$d\mu_{G'}$ follows from the
normalization of $dn$ in \S\ref{subsec:measures}.

Finally, any two nonzero
left-invariant $F$-rational
gauge forms on $H$ differ
by a scalar $\kappa\in F^\times$.
The product formula
\[
\prod_v|\kappa|_v=1
\]
therefore shows that $dh$
is independent of this choice.

The additive and multiplicative
measure conventions, including
the normalization of $c_F$,
agree with those of
Suzuki--Wakatsuki
\cite[\S3.1]{SW1}.
Under the $F$-isomorphism
$H\simeq\mathrm{GL}_2^3/\ker\rho_V$,
the gauge-form normalization
of $dh$ in \S\ref{subsec:measures} is precisely the Tamagawa
measure defined in
\cite[\S4.3]{SW1}.
This proves~(4).

\section{Lattice-point and slice estimates}\label{app:convergence}

This appendix contains the lattice-point estimates and the proofs of
Lemmas~\ref{lem:slice-convergence} and~\ref{lem:outer-cusp}; the former
is used in \S\ref{subsec:absolute-convergence}, the latter in the proof
of Lemma~\ref{lem:truncated-pairing}.

\subsection{Scaling and lattice estimates}
\label{subsec:scaling-lattice}

Put $d=[F:\mathbb Q]$. We use the product sup norm on
$V'(F_\infty)$ and its coordinate subspaces. All implied
constants may depend on the chosen norms, compact sets,
lattices, and Schwartz exponents, but not on the Siegel
parameters.

Recall that $V'=F^2\otimes\bigwedge^2F^4$.
Writing an element as a pair $w=(w_1,w_2)$ of alternating
matrices, we write
\[
p_i:V'\longrightarrow\bigwedge\nolimits^2F^4,
\qquad w\longmapsto w_i\quad(i=1,2),
\]
for the two projections, and we use the coordinates
\[
w=(\ell,x,\lambda),
\qquad
\ell,\lambda\in F^2,\quad x\in M_2(F)^2.
\]
Let $W=\{\ell=0\}\subset V'$.
Then $\dim_F W=10$ and $\dim_F(V'/W)=2$.
Recall from \eqref{eq:au-action} that $a_u$ acts
at each archimedean place by
\[
(\ell,x,\lambda)
\longmapsto
(u^{1/d}\ell,x,u^{-1/d}\lambda).
\]
It is central in $\widetilde M$, but need not
commute with $N$.

We shall use the following elementary estimate. Its uniformity
in the affine translate is important when summing over the
fibres of $V'\to V'/W$.

\begin{lemma}[Translate-uniform lattice counting]
\label{lem:uniform-lattice-counting}
Let $L$ be a full lattice in a finite-dimensional real normed
vector space $V$ of dimension $n$. If $B>n$, then
there is a constant $C_{L,B}$ such that, for every
invertible real linear map $T:V\to V$ and every $b\in V$,
\[
\sum_{\lambda\in L}
(1+\|T(\lambda+b)\|)^{-B}
\le C_{L,B}(1+\|T^{-1}\|)^n.
\]
The constant is independent of $T$ and of the translate $b$.
\end{lemma}

\begin{proof}
Choose $c_L>0$ such that every nonzero element of $L$
has norm at least $c_L$. Distinct points of $T(L+b)$
are separated by at least
$\delta=c_L/\|T^{-1}\|$. A packing argument therefore
gives, uniformly in $b$,
\[
\#\{\lambda\in L:\|T(\lambda+b)\|\le r\}
\ll_L (1+r/\delta)^n\qquad(r\ge0).
\]
Here the implied constant also covers the case when the
ball contains no lattice point. Using
\[
(1+t)^{-B}
=B\int_t^\infty(1+r)^{-B-1}\,dr
\]
and Tonelli's theorem, the required sum is bounded by
\[
\ll_{L,B}\int_0^\infty
(1+r)^{-B-1}(1+r/\delta)^n\,dr
\ll_{L,B}(1+\delta^{-1})^n,
\]
where the last integral converges because $B>n$.
\end{proof}

We begin by reducing the lattice sums to fixed
lattices, uniformly in the finite-place parameters
and the unipotent integration. By strong approximation
for the additive group, we may choose a compact set
of representatives for
\[
[N]=N(F)\backslash N(\mathbb A)
\]
whose archimedean components lie in a fixed compact
set $C_N$ and whose finite components lie in
$N(\widehat{\mathcal O}_F)$. By adelic reduction theory \cite{Bor69},
$\widetilde M(F)\backslash\widetilde M(\mathbb A)$
is covered by finitely many Siegel sets
$\mathfrak S_j=\Omega_jA_j^+K_{\widetilde M}$, where
$\Omega_j$ is compact, $A_j^+$ is the set of elements of
an archimedean torus on which the simple roots of
$\widetilde M$ are bounded below (with no condition in
the central directions), and
$K_{\widetilde M}=K\cap\widetilde M(\mathbb A)$, which is the
image of the block-diagonal subgroup of $K_4\times K_2$, a maximal compact
subgroup of $\widetilde M(\mathbb A)$.
Let $T_{\widetilde M}$ be the image in $\widetilde M$ of the
diagonal torus of $M\times\mathrm{GL}_2$. Since the reduction
theory of \cite{Bor69} is applied to $R_{F/\mathbb Q}\widetilde M$,
the torus underlying $A_j^+$ is the identity component of the
real points of the maximal $\mathbb Q$-split subtorus of
$R_{F/\mathbb Q}T_{\widetilde M}$. In coordinates, its elements
are the classes modulo $Z'_\rho$ of
\[
\bigl(\operatorname{diag}(t_1,t_2,t_3,t_4),\,\operatorname{diag}(t_5,t_6)\bigr),
\qquad t_i\in\mathbb R_{>0},
\]
embedded diagonally at all archimedean places, with trivial
finite components; the inner block $\operatorname{diag}(t_1,t_2)$
lies in $a$, $\operatorname{diag}(t_3,t_4)$ in $b$, and
$\operatorname{diag}(t_5,t_6)$ in the outer $\mathrm{GL}_2$.
In particular, this applies to the norm-one component
$m^1$ in $m=a_um^1$: since $u$ is trivial on
$\widetilde M(F)$ and on $K_{\widetilde M}$, the chosen
representatives of points of $[\widetilde M]^1$ again
lie in $\widetilde M(\mathbb A)^1$. Since $A_j^+$ and
$a_u=e_{c(u)}$ are archimedean, their finite components
are trivial. Thus the finite components of $m=a_um^1$
lie in the fixed compact set
$C_f=\bigcup_j(\Omega_j)_fK_{\widetilde M,f}$.

The condition $(wnm)_f\in\Omega_f$
therefore implies
\[
w_f\in
\Omega_f
\bigl(N(\widehat{\mathcal O}_F)C_f\bigr)^{-1}.
\]
The set on the right is compact.
After fixing an $F$-basis of $V'$, a compact subset of
$V'(\mathbb A_f)$ lies in $\mathfrak a^{-1}\widehat{\mathcal O}_F^{\,12}$
for some nonzero integral ideal $\mathfrak a$, so the denominators
of the contributing rational vectors are uniformly bounded.
Consequently,
there is a fixed fractional $\mathcal O_F$-lattice $L\subset V'(F)$
containing every rational vector contributing to
the lattice sums. Since the projections $p_1,p_2$,
the coordinate map $\ell:V'\to F^2$, and the subspace
$W=\ker\ell$ are defined over $F$,
the projected lattices
\[
L_1=p_1(L),\quad L_2=p_2(L),\quad \ell(L)
\]
and the intersection lattice $L\cap W$
are fixed lattices in their respective rational
vector spaces. Note that $L\subset L_1\oplus L_2$,
and that each nonempty fibre
$\{w\in L:\ell(w)=\ell_0\}$ is a translate of
$L\cap W$.

The archimedean representatives of $[N]$ range
over $C_N$. Their induced linear actions on the
quotient and on the affine fibres, together with
their inverses, are therefore uniformly bounded.
The affine translation parameters need not be
bounded, but Lemma~\ref{lem:uniform-lattice-counting}
is uniform in these translations. Consequently,
all applications of that lemma involve fixed
lattices, with constants uniform in the finite
components, the compact unipotent representatives,
and the affine translates.

We next choose balanced Siegel representatives
and record the corresponding operator-norm bounds.
We distinguish the unipotent variable $n\in[N]$
from the Levi representative. Let
$n=n(U)\in C_N$ be the chosen unipotent
representative. After removing the right compact
factor, write the Levi representative as
\[
m=h\,k\,z,
\qquad
h=\operatorname{diag}(a,b),
\]
where
\[
k=n(\beta)
  \operatorname{diag}(y^{1/2},y^{-1/2})
\]
belongs to the outer $\mathrm{GL}_2$, and $z$
denotes the remaining positive central factor.
Here $a$ and $b$ are the residual inner
$\mathrm{GL}_2$ blocks, after extraction of $z$.
Thus the argument of the theta function is
$w\,n\,m\,a_u$; in particular, the unipotent
variable $n$ is not included a second time
in $m$. The parameters $U$ and $\beta$
range over fixed compact sets, and the
finite components of the representatives
remain in the fixed compact sets chosen above.

We first balance the outer torus parameter $y$.
Since the torus underlying $A_j^+$ is $\mathbb Q$-split,
its archimedean components are equal at all places;
with $\Omega_j$ compact, this gives
\[
y_v=Y\eta_v,\qquad Y\gg1,
\]
where $\eta$ ranges over a fixed compact
subset of
$(F\otimes_{\mathbb Q}\mathbb R)^\times$.
Writing $d_v=[F_v:\mathbb R]$ and
$d=[F:\mathbb Q]$, we obtain
\[
|y_v|_v\asymp Y^{d_v},
\qquad
|y|_{\mathbb A}\asymp Y^d.
\]
Thus the relative sizes of the archimedean
components are bounded, while the common
height $Y$ remains noncompact.

We next normalize the determinant parameters
of the inner blocks. The condition
$u(m^1)=1$ gives equality of their adelic
determinant moduli. The quotient by
\[
Z'_\rho
=
\{(tI_4,t^{-2}I_2):t\in\mathbb G_m\}
\]
allows an outer scalar to be transferred
to the inner factor, since
\[
(I_4,t^{-2}I_2)
\equiv
(t^{-1}I_4,I_2)
\pmod{Z'_\rho}.
\]
We extract the remaining common positive
central modulus into $z$, represented
at every archimedean place by
\[
z_v=r^{1/2}I_4,\qquad r>0,
\]
with trivial finite components. The
residual blocks $a$ and $b$ then have
determinant ideles of adelic modulus one.

For the chosen $\mathbb Q$-split torus, the
coordinates $t_i$ are the same at every
archimedean place. Thus the torus contribution
to the determinant of the first inner block
has local modulus $(t_1t_2)^{d_v}$; here and
below, block determinants are computed on the
chosen representative in $M\times\mathrm{GL}_2$.
The compact factors $\Omega_j$ and
$K_{\widetilde M}$ modify these local
determinant moduli by uniformly bounded
factors. Since the finite components also
range over a fixed compact set, we have
\[
|\det a|_{\mathbb A}
\asymp (t_1t_2)^d.
\]
The residual determinant idele has adelic
modulus one, so $t_1t_2\asymp1$. Consequently,
\[
|\det a_v|_v\asymp1
\qquad(v\mid\infty).
\]
The same argument, with $t_3t_4$ in place
of $t_1t_2$, applies to $b$. The implied
constants are uniform in the chosen
representatives. These bounds concern only
the determinants; the noncentral Siegel
directions of $a$ and $b$ remain unbounded.

To compute the central scaling, recall
that
\[
\chi_\Delta(h,k)
=
(\det h\det k)^2.
\]
Multiplication by $r^{1/2}I_4$
multiplies each inner $2\times2$ block
determinant by $r$ at every archimedean
place. Since the residual determinant
ideles have modulus one and
$|\det k|_{\mathbb A}=1$, we obtain
\begin{equation}\label{eq:chi-Delta-siegel}
|\chi_\Delta(m)|_{\mathbb A}
=r^{4d}.
\end{equation}

For clarity, we record explicit coordinates for
the positive diagonal torus component of these
representatives. Recall that
\[
a_u=e_{c(u)},\qquad
e_c=
\bigl[
(\operatorname{diag}(I_2,c^{-1}I_2),cI_2)
\bigr].
\]
At each archimedean place, $c(u)_v=u^{1/d}>0$.
Multiplying this lift by the element
\[
\bigl(c(u)_v^{1/2}I_4,\;c(u)_v^{-1}I_2\bigr)\in Z'_\rho
\]
gives the equivalent symmetric lift
\[
a_{u,v}=
\bigl[
(\operatorname{diag}(u^{1/(2d)}I_2,
u^{-1/(2d)}I_2),I_2)
\bigr].
\]

Write
\[
D_a=\operatorname{diag}(T_a,T_a^{-1}),
\qquad
D_b=\operatorname{diag}(T_b,T_b^{-1}).
\]
Then the positive diagonal torus component
of the full Levi element $a_um$ has
archimedean coordinates
\begin{equation}\label{eq:siegel-coordinates}
\begin{aligned}
a_{v,\mathrm{tor}}
 &=r^{1/2}u^{1/(2d)}D_a,\\
b_{v,\mathrm{tor}}
 &=r^{1/2}u^{-1/(2d)}D_b,\\
k_{v,\mathrm{tor}}
 &=\operatorname{diag}(Y^{1/2},Y^{-1/2}).
\end{aligned}
\end{equation}
These positive torus coordinates are common
to all archimedean places and have trivial
finite components. The actual Siegel
representatives also contain compact and
unipotent factors.

In these coordinates, $u$ is the relative
determinant scale, $r$ is the common central
scale, $Y$ is the outer Siegel height, and
$T_a,T_b$ are the inner diagonal shape
parameters. In particular,
\[
u(a_um)=u,\qquad
|\chi_\Delta(a_um)|_{\mathbb A}=r^{4d}.
\]
After extracting $a_u$ and the central factor
$z_v=r^{1/2}I_4$, the residual diagonal
inner blocks are $D_a$ and $D_b$.

The balancing concerns only the outer
torus parameter $y$ and the determinant
parameters of the inner blocks. The
remaining noncompact directions are
the central parameters $u$ and $r$,
the outer Siegel height $Y$, and the
inner Siegel heights. These are
controlled separately in the estimates
below.

Finally, the outer Siegel measure
contributes
\begin{equation}\label{eq:outer-siegel-measure}
\ll Y^{-d}\,d^\times Y
\end{equation}
up to bounded factors. For $F=\mathbb Q$,
this is the familiar measure $y^{-2}\,dy$.

The remaining noncompact inner parameters are
controlled by the rapid decay of the cuspidal
factor. The inner parabolic preserves $W_0=\{\ell=0\} \subset\bigwedge^2F^4$,
and hence preserves $W\subset V'$. Let $S_m$
and $L_m$ denote the induced actions on $W$
and $V'/W$, respectively. Recall the height $H(h)$ of \eqref{eq:height-H}, and write
$H_\infty=H(h_\infty)$ for the inner Levi component
$h=\operatorname{diag}(a,b)$. Under the decomposition
\[
\textstyle\bigwedge^2F^4
=
\bigwedge^2F^2
\oplus(F^2\otimes F^2)
\oplus\bigwedge^2F^2,
\]
the action of $\bigwedge^2h$ is given, up to transpose, by
$\det a$, $a\otimes b$, and $\det b$; by \S\ref{subsec:group}
these are the actions
\[
\ell\longmapsto\det(a)\ell,\qquad
x_i\longmapsto{}^tax_ib,\qquad
\lambda\longmapsto\det(b)\lambda
\]
on the coordinates $\ell$, $x$, and $\lambda$. Since the inner
actions $S_h$ and $L_h$ are induced by the restriction and
quotient of $\bigwedge^2h$, respectively, $H_\infty$ dominates
their inverse operator norms. Composition with the compact
archimedean unipotent representatives, which fix $\ell$, changes
these norms by bounded factors, and $H_\infty$ has at most
polynomial growth in the inner Siegel heights. Normalizing the
determinants alone does not bound $H_\infty$: for
$a=\operatorname{diag}(T,T^{-1})$ and $b=I_2$ one has
$\|S_h^{-1}\|\geq T$.

The outer factor acts on the index $i$ and commutes with the
inner Levi. On $V'/W\simeq F^2$, the element $hk$ acts by
\[
\ell\longmapsto\det(a)\,{}^tk\,\ell,
\]
so the inner Siegel root parameters do not enter the quotient
action. Since $\beta$ ranges over a compact set and
$y_v=Y\eta_v$ with $\eta_v$ uniformly bounded above and away
from zero, we have
\[
\|({}^tk)^{-1}\|\ll Y^{1/2}.
\]
Together with the balanced determinant bound
$|\det a_v|_v\asymp1$, the bound $\|S_h^{-1}\|\leq H_\infty$,
and the scalar action of $z$ by $r$ on $\bigwedge^2F^4$, this gives
\begin{equation}\label{eq:induced-norms}
\|L_m^{-1}\|
\ll\frac{Y^{1/2}}r,
\qquad
\|S_m^{-1}\|
\ll\frac{H_\infty Y^{1/2}}r.
\end{equation}

Whenever a right compact factor $\kappa\in K$
occurs, we first use
\eqref{eq:compact-majorant}: $\Psi_B^+(v\kappa) \ll_B\Psi_B^+(v)$.
Thus the counting estimates may be established
with the right compact factor removed. We now
prove the two-component estimate used for
slice convergence.

\begin{lemma}\label{lem:step3}
Let $B>24d$. On the balanced Siegel
representatives, put
\[
X=\frac{rY^{1/2}}{H_\infty},
\qquad
Y'=\frac{H_\infty Y^{1/2}}r.
\]
Then
\[
\Theta_B^{+,N}(m\kappa)
\ll_B
(1+X)^{-B/4}
(1+X^{-1})^{6d}
(1+Y')^{6d},
\]
uniformly in $\kappa\in K$.
\end{lemma}

\begin{proof}
By \eqref{eq:compact-majorant}, we may remove
the right compact factor $\kappa$. Write $w=(w_1,w_2) \in(\bigwedge^2F^4)^2$.
Since $E$ is a field, the associated binary
Pfaffian polynomial is anisotropic. In particular, $w_1\neq0$
for every $w\in V'_E(F)$. By the fixed-lattice
reduction, we may enlarge the summation set to
\[
w_1\in L_1\setminus\{0\},
\qquad
w_2\in L_2.
\]

The outer upper triangular factor acts on
the two components by
\[
(w_1,w_2)\longmapsto
\left(
y^{1/2}w_1,\,
y^{-1/2}(w_2+\beta w_1)
\right).
\]
Including the inner action, the central scalar,
and the compact representatives of $[N]$, we obtain
\[
(w_1,w_2)\longmapsto
\bigl(T_1w_1,T_2(w_2+\beta w_1)\bigr),
\]
where
\[
\|T_1^{-1}\|\ll X^{-1},
\qquad
\|T_2^{-1}\|\ll Y'.
\]
Since $L_1$ is fixed, there exists $c_L>0$
such that
\[
\|w_1\|\geq c_L
\qquad
(w_1\in L_1\setminus\{0\}).
\]
Consequently,
\[
\|T_1w_1\|\gg X.
\]

Put
\[
A=\|T_1w_1\|,
\qquad
C=\|T_2(w_2+\beta w_1)\|.
\]
With the product sup norm,
\[
\|wnm\|=\max(A,C).
\]
Since
\[
1+\max(A,C)
\geq
\bigl((1+A)(1+C)\bigr)^{1/2},
\]
we obtain
\[
\begin{aligned}
(1+\|wnm\|)^{-B}
&\leq
(1+A)^{-B/2}(1+C)^{-B/2}\\
&=
(1+A)^{-B/4}
(1+A)^{-B/4}
(1+C)^{-B/2}.
\end{aligned}
\]
The first factor contributes
\[
\ll(1+X)^{-B/4}.
\]

Since $B/4>6d$, Lemma~\ref{lem:uniform-lattice-counting}
gives
\[
\sum_{w_1\in L_1\setminus\{0\}}
(1+\|T_1w_1\|)^{-B/4}
\ll_B(1+X^{-1})^{6d}.
\]
Likewise,
\[
\sum_{w_2\in L_2}
(1+\|T_2(w_2+\beta w_1)\|)^{-B/2}
\ll_B(1+Y')^{6d},
\]
uniformly in $w_1$ and $\beta$.
Multiplying these estimates and integrating
over the compact unipotent quotient proves
the lemma.
\end{proof}

We next record the effect of the one-parameter
subgroup $a_u$ in the range $u\geq1$.

\begin{lemma}\label{lem:large-u-scaling}
For $u\geq1$ and $B\geq0$,
\[
\Theta_B^{+,N}(a_um)
\ll_B
u^{B/d}\Theta_B^{+,N}(m),
\]
uniformly in $m\in[\widetilde M]^1$.
\end{lemma}

\begin{proof}
At each archimedean place,
\[
\|a_u^{-1}\|=u^{1/d}.
\]
Thus
\[
1+\|va_u\|
\geq
u^{-1/d}(1+\|v\|),
\]
and therefore
\[
\Psi_B^+(va_u)
\leq
u^{B/d}\Psi_B^+(v).
\]
Since $a_u$ is central in $\widetilde M$, $wna_um=wnma_u$.
Applying the preceding inequality termwise,
then summing and integrating over $[N]$,
proves the assertion.
\end{proof}

For the small-$u$ range, we use the filtration
$0\subset W\subset V'$ and the nonvanishing
of the Pfaffian on the field stratum. Write
an alternating matrix as
\[
M=
\begin{pmatrix}
\ell J&x\\
-x^t&\lambda J
\end{pmatrix}.
\]
For
\[
n(U)=
\begin{pmatrix}
I_2&U\\
0&I_2
\end{pmatrix}\in N,
\]
we have
\[
n(U)^tMn(U)=
\begin{pmatrix}
\ell J&x+\ell JU\\
-(x+\ell JU)^t&
\lambda J+U^tx-x^tU+\ell U^tJU
\end{pmatrix}.
\]
In particular, $N$ preserves $\ell$ and
the filtration $0\subset W\subset V'$.
Moreover,
\[
q(Mn(U))=q(M),
\]
since $\det n(U)=1$.

For $w\in V'_E(F)$ with $\ell(w)\neq0$,
anisotropy gives $q(w_1)\in F^\times$.
Thus
\[
q((wn)_1)=q(w_1)
\]
for every $n\in N(\mathbb A)$.
The action of $a_u$ preserves the Pfaffian,
since by \eqref{eq:au-action} it sends
$(\ell,x,\lambda)$ to $(u^{1/d}\ell,x,u^{-1/d}\lambda)$
and $q(M)=\ell\lambda-\det x$. The inner Levi element $h$
multiplies $q$ by $\det h$, the central factor
$z_v=r^{1/2}I_4$ multiplies it by $r^2$, and, since $k$ is
upper triangular, $(wk)_1=y^{1/2}(w)_1$, so that $k$ multiplies
the Pfaffian of the first component by $y$.
Hence, under the balanced Siegel action,
\[
q((wnma_u)_{1,v})
=
r^2y_v\det(h_v)\,q(w_1).
\]
Since $w_1$ belongs to the fixed fractional
lattice $L_1$, the finite denominators of
$q(w_1)$ are uniformly bounded. As
$q(w_1)\in F^\times$, the product formula
therefore gives
\[
\prod_{v\mid\infty}|q(w_1)|_v\gg1.
\]
Since $|r^2|_v=r^{2d_v}$, the balanced determinant
bounds and $|y|_{\mathbb A}\asymp Y^d$ give
\[
\prod_{v\mid\infty}
|q((wnma_u)_{1,v})|_v
\gg r^{2d}Y^d.
\]
On the other hand, since the Pfaffian is
quadratic and $\sum_{v\mid\infty}d_v=d$,
\[
\prod_{v\mid\infty}
|q((wnma_u)_{1,v})|_v
\ll \|(wnma_u)_1\|_\infty^{2d}.
\]
Consequently,
\begin{equation}\label{eq:pfaffian-anisotropy}
\|(wnma_u)_1\|_\infty
\gg rY^{1/2},
\end{equation}
uniformly in $n$, $u$, and the balanced
Siegel parameters.

We can now estimate the lattice sum by
first summing along the fibres of
$V'\to V'/W$ and then over the quotient
lattice.

\begin{lemma}\label{lem:small-u-general}
Let $B>30d$. For $0<u\leq1$,
\[
\begin{aligned}
\Theta_B^{+,N}(a_um)
\ll_B{}&
u^{-2}(1+rY^{1/2})^{-B/3}\\
&\times
\left(1+\frac{Y^{1/2}}r\right)^{2d}
\left(1+\frac{H_\infty Y^{1/2}}r\right)^{10d}.
\end{aligned}
\]
The estimate is uniform on the balanced
Siegel representatives.
\end{lemma}

\begin{proof}
Consider the fibration
\[
V'\longrightarrow V'/W,
\qquad
(\ell,x,\lambda)\longmapsto\ell.
\]
The quotient and fibre have real dimensions
$2d$ and $10d$. For $v=wnma_u$,
we split the Schwartz decay:
\[
\begin{aligned}
(1+\|v\|)^{-B}
\leq{}&
(1+\|v_1\|)^{-B/3}\\
&\times
(1+\|\ell(v)\|)^{-B/3}\\
&\times
(1+\|(x,\lambda)(v)\|)^{-B/3}.
\end{aligned}
\]
Here each coordinate norm is bounded by
$\|v\|$. By
\eqref{eq:pfaffian-anisotropy}, the first
factor is bounded by
\[
(1+rY^{1/2})^{-B/3}.
\]

We next examine the action of $a_u$
on the fibres. On $W=\{\ell=0\}$,
the element $a_u$ acts by
\[
A'_u:(x,\lambda)\longmapsto
(x,u^{-1/d}\lambda).
\]
For $0<u\leq1$,
\[
(A'_u)^{-1}(x,\lambda)
=
(x,u^{1/d}\lambda),
\]
and hence
\[
\|(A'_u)^{-1}\|=1.
\]
If $S_m$ denotes the induced action
of $m$ on $W$, the fibre action
of $ma_u$ is $S_mA'_u$. Therefore
\[
\begin{aligned}
\|(S_mA'_u)^{-1}\|
&\leq
\|(A'_u)^{-1}\|
\|S_m^{-1}\|\\
&=\|S_m^{-1}\|\\
&\ll
\frac{H_\infty Y^{1/2}}r.
\end{aligned}
\]
Thus the affine-fibre counting estimate
introduces no negative power of $u$. This is because
the $\lambda$-direction is expanded by $a_u$ when
$u\le1$: the inverse of the fibre action is
uniformly bounded by the inverse of the unscaled
inner action. In contrast, the quotient
$V'/W$ is contracted by $u^{1/d}$ and its
real dimension $2d$ produces precisely the
factor $u^{-2}$ below.

For each fixed $\ell$, the remaining
coordinates form an affine translate of
the fixed lattice $L\cap W$, of real dimension $10d$.
Since $N$ and the Levi preserve $W$, the linear part
$S_mA'_u$ of the fibre transformation does not depend
on $\ell$; only the translation does.
Since $B/3>10d$, Lemma~\ref{lem:uniform-lattice-counting}
bounds the fibre sum by
\[
\left(
1+\frac{H_\infty Y^{1/2}}r
\right)^{10d},
\]
uniformly in $\ell$, so this bound can be taken outside
the quotient sum.

On the quotient $V'/W$, let $L_m$ denote
the action induced by $m$. By \eqref{eq:induced-norms},
\[
\|L_m^{-1}\|\ll \frac{Y^{1/2}}r.
\]
Since $a_u$ acts on the quotient by
multiplication by $u^{1/d}$, we have
\[
\|(u^{1/d}L_m)^{-1}\|
\ll u^{-1/d}\frac{Y^{1/2}}r.
\]
Since $N$ acts trivially on $V'/W$, the quotient image of
$v=wnma_u$ is $u^{1/d}L_m\xi$ with $\xi=\ell(w)$, independently
of $n$. As the quotient has real dimension $2d$
and $B/3>2d$, Lemma~\ref{lem:uniform-lattice-counting},
applied to the lattice $\ell(L)$, gives
\[
\begin{aligned}
\sum_{\xi\in\ell(L)}
(1+\|u^{1/d}L_m\xi\|)^{-B/3}
&\ll
\left(1+u^{-1/d}\frac{Y^{1/2}}r\right)^{2d}\\
&\leq
u^{-2}
\left(1+\frac{Y^{1/2}}r\right)^{2d},
\end{aligned}
\]
the last step because $u^{-1/d}\geq1$ for $u\leq1$.
Combining the three factors and integrating
over $[N]$ proves the lemma.
\end{proof}

\subsection{Decay of the section on Siegel sets}

Let $\Phi_f$ be as in \S\ref{subsec:absolute-convergence}. On the norm-one slice $u=1$, so the inducing
character is trivial. Moreover, by \S\ref{subsec:zeta-integral}
the section is invariant under the centre of
$\mathrm{GL}_4(\mathbb A)$ and independent of the outer
$\mathrm{GL}_2$-factor. Let $T_a,T_b$ be the inner Siegel
heights appearing in the positive diagonal
coordinates \eqref{eq:siegel-coordinates},
taken modulo the centres of the two inner
$\mathrm{GL}_2$ factors; on the Siegel sets they are
bounded below. Since the remaining noncentral factors
of the chosen Siegel representatives range over fixed
compact sets, and the norm in \eqref{eq:height-H} is a
single operator norm on the full archimedean representation,
the height satisfies
\[
H_\infty\asymp T_aT_b;
\]
indeed, for the diagonal part the singular values of
$\bigwedge^2h_v$ at every archimedean place are
\[
1,\quad1,\quad T_aT_b,\quad T_a/T_b,\quad T_b/T_a,\quad (T_aT_b)^{-1}.
\]

By the rapid decay of the inducing cusp forms on Siegel
sets modulo the centre, for every $N>0$,
\[
|\Phi_f(m\kappa)|
\ll_N
(1+T_a)^{-N}(1+T_b)^{-N},
\]
uniformly in the central parameter $r$, the outer Siegel
parameter $Y$, and $\kappa\in K$. Indeed, the factor $z_v=r^{1/2}I_4$ is central in
$\mathrm{GL}_4$ and the inducing data have trivial central
characters, so the common inner central scaling $r$ does not
affect the section. The section is independent of the outer
$\mathrm{GL}_2$ factor, and hence of $Y$. Finally,
$K_4$-finiteness reduces the dependence on $\kappa$ to finitely
many cuspidal vectors, for which rapid decay on Siegel sets is
uniform. The estimate asserts uniformity, but not decay, in
the central parameter $r$. It concerns the norm-one
slice, where $u=1$; the dependence of the section on $u$
enters only through the section factor $u^{2\Re s_0-4}$,
which is treated separately in
\S\ref{subsec:absolute-convergence}.

The same argument applies to the Levi restriction of a
$K$-finite section whose values lie in a cuspidal
representation of $\widetilde M(\mathbb A)$, provided
the central-character conditions used above are satisfied.
In particular, it applies to the Levi restriction of
$M(s_0-1)f_{s_0}$, whose values lie in $w\sigma$.
The implied constants are controlled by finitely many
cuspidal seminorms, locally uniformly in $s_0$ away
from the poles of $M(s_0-1)$.
Indeed, $f_{s_0}$ is $K$-finite and $M(s_0-1)$ is $K$-equivariant,
so $M(s_0-1)f_{s_0}$ lies in a fixed finite sum of $K$-isotypic
components of the compact picture; this sum is finite-dimensional
because the induced representation is admissible, as in the proof of
Proposition~\ref{prop:intertwining-sign}. Choosing a basis of fixed
flat sections $h_i$, we may write
\[
M(s_0-1)f_{s_0}=\sum_i a_i(s_0)\,(h_i)_{1-s_0},
\]
where the coefficients $a_i$ are meromorphic, with poles among those
of $M(s_0-1)$, by the meromorphy of $M(s_0-1)$ on this
finite-dimensional space \cite[Theorem~IV.1.8(b)]{MW}. On the slice $u=1$ the
Levi restrictions of the $h_i$ are independent of $s_0$, so their
cuspidal seminorms are fixed, and the $a_i(s_0)$ are bounded on compact
sets avoiding the poles; this gives the asserted uniformity. For the
intertwined section, the $u$-dependence enters through
$u^{-2\Re s_0}$ in place of $u^{2\Re s_0-4}$.

\subsection{Proof of Lemma~\ref{lem:slice-convergence}}

By Lemma~\ref{lem:step3},
\[
\Theta_{B_+}^{+,N}(m\kappa)
\ll
(1+X)^{-B_+/4}
(1+X^{-1})^{6d}
(1+Y')^{6d},
\]
where
\[
X=\frac{rY^{1/2}}{H_\infty},
\qquad
Y'=\frac{H_\infty Y^{1/2}}r
=\frac YX.
\]
Since
\[
|\chi_\Delta(m)|_{\mathbb A}^{\sigma_1}
\asymp r^{4d\sigma_1},
\]
make the substitution $r=H_\infty Y^{-1/2}X$.
Then
\[
r^{4d\sigma_1}
=
H_\infty^{4d\sigma_1}
Y^{-2d\sigma_1}
X^{4d\sigma_1}.
\]

It remains to estimate
\[
\begin{aligned}
\int_0^\infty&
X^{4d\sigma_1}
(1+X)^{-B_+/4}\\
&\times
(1+X^{-1})^{6d}
(1+Y/X)^{6d}
\,d^\times X.
\end{aligned}
\]
For $0<X\leq1$, the integrand is
\[
\ll
(1+Y)^{6d}
X^{4d\sigma_1-12d},
\]
which is integrable because $\sigma_1>3$.
For $X\geq1$, it is
\[
\ll
(1+Y)^{6d}
X^{4d\sigma_1-B_+/4},
\]
which is integrable because $B_+>16d\sigma_1$.
Thus integration in $r$ gives
\[
\ll
H_\infty^{4d\sigma_1}
Y^{6d-2d\sigma_1}.
\]

The outer Siegel measure contributes
$Y^{-d}d^\times Y$, so the remaining outer
integral is bounded by
\[
\int_{Y\gg1}
Y^{5d-2d\sigma_1}\,d^\times Y,
\]
which converges since $\sigma_1>3$.
Finally, $H_\infty\ll((1+T_a)(1+T_b))^{C}$ for some fixed $C$,
since $H_\infty$ has polynomial growth in the inner Siegel heights.
Taking $N$ large in the bound for $\Phi_f$, the factor
$H_\infty^{4d\sigma_1}|\Phi_f|$ is integrable against the inner
Siegel measure. Hence the inner integrations converge absolutely.

\subsection{Proof of Lemma~\ref{lem:outer-cusp}}

Since $V'_E(F)$ is $G(F)$-stable, the full theta
majorant $\Theta_B^+$ is left $G(F)$-invariant.
Indeed, for $\gamma\in G(F)$,
\[
\begin{aligned}
\Theta_B^+(\gamma g)
&=
\sum_{w\in V'_E(F)}
\Psi_B^+(w\gamma g)\\
&=
\sum_{w\in V'_E(F)}
\Psi_B^+(wg)
=
\Theta_B^+(g).
\end{aligned}
\]
We may therefore replace $h$ by $\gamma h$
with $\gamma\in\mathrm{GL}_4(F)$.

By the reduction theory used in \S\ref{subsec:scaling-lattice},
applied to $\mathrm{GL}_4(\mathbb A)^1$ \cite{Bor69}, the finite
components of the reduced representatives lie in a fixed
compact subset $C_{4,f}$ of $\mathrm{GL}_4(\mathbb A_f)$.
The positive central parameter $r$ and the outer Siegel
parameter $Y$ are archimedean, while the remaining finite
outer components range over a fixed compact set.
Consequently, the finite components of $hkz$ range over a
fixed compact subset $C_f\subset G(\mathbb A_f)$.
The finite-support condition $(whkz)_f\in\Omega_f$
therefore implies $w_f\in\Omega_fC_f^{-1}$,
a compact subset of $V'(\mathbb A_f)$.
Hence every contributing rational vector belongs
to a fixed fractional $\mathcal O_F$-lattice $L\subset V'(F)$.

The two-component counting argument of
Lemma~\ref{lem:step3} uses only the fixed-lattice
property, the nonvanishing of $w_1$, and bounds
for inverse operators on $\bigwedge^2F^4$.
Since $E$ is a field, its associated binary
Pfaffian polynomial is anisotropic. Thus $w_1\neq0$ for every $w\in V'_E(F)$,
including vectors with $\ell(w)=0$.
The counting argument therefore applies
to the full theta series, without an
$[N]$-average and without requiring $h$
to belong to the parabolic $P$.

Since $|\det h|_{\mathbb A}=1$ and $|\det k|_{\mathbb A}=1$,
\eqref{eq:chi-Delta-siegel} gives
$|\chi_\Delta(hkz)|_{\mathbb A}=r^{4d}$.
Because $k$ is upper triangular, the first component of
$w\,hkz$ is $T_1w_1$ with $T_1=y^{1/2}r\bigwedge^2h$, and
the $w_2$-fibre has linear part $T_2=y^{-1/2}r\bigwedge^2h$.
With
\[
X=\frac{rY^{1/2}}{H(h)},
\]
we have $\|T_1^{-1}\|\ll X^{-1}$ and $\|T_2^{-1}\|\ll Y/X$,
so Lemma~\ref{lem:uniform-lattice-counting} gives
\[
\Theta_B^+(hkz)
\ll
(1+X)^{-B/4}
(1+X^{-1})^{6d}
(1+Y/X)^{6d}.
\]
For $Y\geq1$ we use
$(1+Y/X)^{6d}\leq Y^{6d}(1+X^{-1})^{6d}$.

For $\sigma\in I$, the change of variables $r=H(h)Y^{-1/2}X$
gives
\[
\begin{aligned}
&\int_0^\infty
r^{4d\sigma}
\Theta_B^+(hkz)\,d^\times r\\
&\qquad\ll_I
H(h)^{4d\sigma}
Y^{6d-2d\sigma}.
\end{aligned}
\]
Indeed, the $X$-integral converges uniformly
for $\sigma\in I$, because $\sigma_->3$ and $B>16d\sigma_+$.

Finally, the outer Siegel measure contributes
$Y^{-d}d^\times Y$. Hence
\[
\begin{aligned}
&\int_{Y\gg1}\int_0^\infty
r^{4d\sigma}
\Theta_B^+(hkz)
\,d^\times r\,
Y^{-d}d^\times Y\\
&\qquad\ll_I
H(h)^{4d\sigma}
\int_{Y\gg1}
Y^{5d-2d\sigma}\,d^\times Y\\
&\qquad\ll_I H(h)^{4d\sigma_+},
\end{aligned}
\]
since $\sigma_->3$ and $H(h)\geq1$.
This proves the lemma.

\section{The central intertwining operator}\label{app:intertwining}

We use the notation, normalizations and the identification
$w\sigma=\sigma$ of \S\ref{subsec:central-intertwining}.

\subsection{Proof of Lemma~\ref{lem:local-jacquet}}

We begin with (1) and (2).  The inducing representation
$\sigma_v=\pi_v\boxtimes\pi_v$ is irreducible, admissible,
generic, and unitary.  For nonarchimedean $v$, the entireness and
nonvanishing of the Jacquet functional follow from
\cite[Proposition~3.1]{Sh81}.  At archimedean places we
apply \cite[Proposition~3.2]{Sh81}.  Its fineness hypothesis
holds by \cite[\S~1.3]{Sh81}: for unitary representations it
suffices that the embeddings into principal series be
continuous, continuity of their inverses then following from
Casselman's continuity of intertwining operators. Moreover,
every unitary representation of
$\mathrm{GL}_n(\mathbb R)$ or of a complex group is fine
(Casselman--Wallach); the same argument applies to the Levi
$\mathrm{GL}_2\times\mathrm{GL}_2$.

We identify the conventions in these results with ours.
For the parabolic of type $(2,2)$, Shahidi's Weyl element is
\[
w_0=w_{\ell,G}w_{\ell,M}
=\begin{pmatrix}0&I_2\\ I_2&0\end{pmatrix}
=w_P=w_P^{-1}.
\]
It exchanges the two Levi simple roots; hence the associated
standard parabolic is again $P$, with unipotent radical $N$.
The inversion map $f_z(g)=h_{v,z}(g^{-1})$ identifies our
left-$P$-equivariant, right-translation realization with
Shahidi's right-$P$-equivariant, inverse-translation
realization.  In particular, the latter satisfies
\[
f_z(gnm)=\delta_P(m)^{-1/2}u_v(m)^{-2z}
\sigma_v(m^{-1})f_z(g).
\]
In \cite[\S~3]{Sh81} one fixes a nondegenerate character
$\chi$ of the maximal unipotent subgroup $N_0$; Whittaker
functionals satisfy $\lambda(\pi(u)v)=\overline{\chi(u)}\lambda(v)$,
the Whittaker model consists of the functions
$W_v(g)=\lambda(\pi(g^{-1})v)$, sections $f(g)$ take values in
the Whittaker model of the inducing representation, and
$\langle f(g),e\rangle$ denotes evaluation at the identity
\cite[pp.~331--332]{Sh81}. Take $\chi=\psi_v^{-1}$. Then
$\overline\chi=\psi_v$ on $N_0$, so $\ell_{\sigma_v}$ is a
Whittaker functional for $\sigma_v$ in this sense, and we take
it as Shahidi's fixed functional on $\sigma_v$. Under the
inversion, $\langle f_z(g),e\rangle=\ell_{\sigma_v}\bigl(h_{v,z}^{\,g^{-1}}\bigr)$,
and the integral of \cite[Proposition~3.1]{Sh81} becomes, in
the common region of convergence,
\[
\begin{aligned}
\lambda(\nu,\sigma_v,\theta,\chi)(f_z)
&=\int_{N(F_v)}
  \ell_{\sigma_v}\!\left(h_{v,z}^{\,w_P^{-1}n'^{-1}}\right)
  \overline{\chi(n')}\,dn'\\
&=\int_{N(F_v)}
  \ell_{\sigma_v}\!\left(h_{v,z}^{\,w_P^{-1}n}\right)
  \psi_v(n)^{-1}\,dn
 =\lambda_{v,z}(h_v).
\end{aligned}
\]
Here $n=n'^{-1}$, $\overline{\chi(n')}=\psi_v(n')=\psi_v(n)^{-1}$,
and inversion preserves Haar measure because $N$ is unimodular.
The right-$P$-equivariance formula above identifies
the corresponding inducing parameter in Shahidi's
realization. Since the two Jacquet integrals agree
in their common region of convergence, their
analytic continuations agree identically. In
particular, $\lambda_{v,z}$ is exactly Shahidi's functional
for the nondegenerate character $\chi=\psi_v^{-1}$: this
comparison introduces neither a scalar factor nor a
translation of the section. At nonarchimedean places, this identifies
$\lambda_{v,z}$ directly with the entire, nonzero functional
of \cite[Proposition~3.1]{Sh81}.

At archimedean places, Shahidi constructs the Whittaker
functional by embedding the inducing representation into a
principal series and restricting the analytically continued
principal-series Whittaker functional
\cite[Proposition~3.2]{Sh81}. In the initial region of
convergence this construction is given by the same Jacquet
integral \cite[p.~352]{Sh81}, which agrees with
$\lambda_{v,z}$ under the conventions identified above.
Since $I_v(\sigma_v,z)_\infty$ satisfies multiplicity one
\cite[Proposition~3.2]{Sh81}, uniqueness and analytic
continuation identify the resulting entire functional with
the continuation of $\lambda_{v,z}$.

At nonarchimedean places, every smooth vector is fixed by
an open subgroup $K_v'\subset K_v$, and its $K_v$-translates
span a finite-dimensional space because $K_v/K_v'$ is finite;
thus every smooth vector is $K_v$-finite. At archimedean
places, the Whittaker functional is continuous on the smooth
induced representation and the $K_v$-finite vectors are
dense. Consequently, its nonvanishing at $z=0$ implies
nonvanishing on some $K_v$-finite section. This proves
assertions (1) and (2).

Assume now that $v$ is nonarchimedean and the data are
unramified.  Write $G=\mathrm{GL}_4$ and let $B_G$ and $B_M$
be the standard Borels.  Since $\pi_v$ is spherical and
generic, its unramified normalized $\mathrm{GL}_2$ principal
series is irreducible: reducibility occurs only when the
ratio of the inducing characters is $|\cdot|_v^{\pm1}$,
when the spherical constituent is nongeneric.  Consequently
\[
\sigma_v\simeq\operatorname{Ind}_{B_M}^{M}(\mu_v),
\qquad
I_v(\sigma_v,z)\simeq
\operatorname{Ind}_{B_G}^{G}(\chi_z),
\qquad \chi_z=\mu_vu^{2z},
\]
by normalized induction in stages.  These are isomorphisms
of induced representations for every $z$, including values
at which the outer principal series is reducible.

For $H=G,M$, put
\[
C_H(\chi)=\prod_{\alpha\in\Phi^+(H)}
\left(1-q_v^{-1}\chi(\alpha^\vee(\varpi_v))\right).
\]
Let $J_G$ and $J_M$ denote the standard Borel Jacquet
functionals, with each positive-root subgroup assigned
measure $dt$ satisfying $\operatorname{vol}(\mathfrak o_v)=1$.
The spherical formula of \cite[Theorem~5.4]{CS80} gives
$J_G(f_z^\circ)=C_G(\chi_z)$ and
$J_M(\phi_M^\circ)=C_M(\chi_z)$ for Borel spherical
sections normalized at the identity.  Since $u$ is trivial
on Levi coroots, $C_M(\chi_z)$ is independent of $z$;
it is nonzero by genericity of $\pi_v$.  Our normalization
of $\ell_{\sigma_v}$ therefore amounts to
\[
\ell_{\sigma_v}=C_M(\chi_z)^{-1}J_M.
\]

We compare the full Borel Jacquet integral with successive
Jacquet integrations.  Put
\[
w_P=\begin{pmatrix}0&I_2\\I_2&0\end{pmatrix},\qquad
w_M=\operatorname{diag}(J_2,J_2),\qquad
J_2=\begin{pmatrix}0&1\\1&0\end{pmatrix},
\qquad w_G=w_Mw_P.
\]
On a section supported in the full Borel big cell
$B_Gw_GN_0$, the relevant unipotent-coordinate integrals
are compactly supported.  The identity
\[
w_Muw_Pn=w_G(w_P^{-1}uw_P)n,
\qquad u\in N_M,\quad n\in N,
\]
identifies the iterated integration with integration over
$N_0$.  Conjugation by $w_P$ exchanges the two Levi
simple-root coordinates, preserving both their characters
and their volume-one Haar measures.  The multiplication
map $(w_P^{-1}N_Mw_P)\times N\to N_0$ is triangular in
the six positive-root coordinates with Jacobian of absolute
value one.  Thus $dn_0=du\,dn$.  The normalized modulus factors
are compatible with this decomposition, since
$\delta_{B_G}^{1/2}=\delta_{B_M}^{1/2}\delta_P^{1/2}$ on the diagonal
torus; thus normalized induction in stages introduces no
additional factor in the iterated Jacquet integral, and Fubini
gives the exact identity
\[
J_G=J_M\circ\Omega_{\sigma_v}
   =C_M(\chi_z)\lambda_{v,z}
\]
on the full Borel big-cell test space.  By
\cite[Theorem~1.4, Theorem~1.6, and Corollary~1.7]{CS80},
the big-cell subspace determines the Whittaker functional
on the full induced representation, even at reducible
parameters.  Together with the holomorphy of the canonical
Borel functional \cite[Proposition~2.1]{CS80} and (1),
this yields the same identity on every section.

For completeness, we match the Haar convention with that
of \cite[\S~4]{CS80}.  Let $K=\mathrm{GL}_4(\mathfrak o_v)$,
$B_0=B_G\cap K$, $U=N_0$, and $U_0=U\cap K$.  The Bruhat
coordinate map $B_G\times U\to B_Gw_GU$ is an isomorphism
onto the open cell.  Hence $B_0w_GU_0$ is compact open and
\[
f_0=\mathbf 1_{B_0w_GU_0}\in C_c^\infty(G).
\]
Choose $db$ with $\operatorname{vol}(B_0)=1$.  For the
projection $\mathcal P_{\chi_z}$ to the normalized Borel
principal series, uniqueness of Bruhat coordinates and
unramifiedness give
\[
(\mathcal P_{\chi_z}f_0)(w_Gn)=\mathbf 1_{U_0}(n).
\]
Our Jacquet integral consequently takes the value
\[
\int_U(\mathcal P_{\chi_z}f_0)(w_Gn)
\Psi_G(n)^{-1}\,dn=\operatorname{vol}(U_0)=1.
\]
The open-cell functional of \cite[\S~4]{CS80}, normalized
by $\operatorname{meas}(B_0w_GU_0)=1$, takes the same value
on this test section, since its weight
$\chi_z^{-1}\delta_{B_G}^{1/2}$ is identically $1$ on
$B_0w_GU_0$ in Bruhat coordinates.  Uniqueness of Whittaker functionals
therefore identifies the two normalizations with scalar
one.  The same argument for the Levi identifies $J_M$
with its volume-one root-coordinate normalization.

Evaluating $J_G=C_M(\chi_z)\lambda_{v,z}$ on the
spherical section, and accounting for the normalization
$\ell_{\sigma_v}(h_v^\circ;1)=1$, gives
\[
\lambda_{v,z}(h_v^\circ)
=\frac{C_G(\chi_z)}{C_M(\chi_z)}.
\]
The roots of $G$ outside $M$ are
$e_i-e_{2+j}$, $1\leq i,j\leq2$, and
\[
\chi_z\bigl((e_i-e_{2+j})^\vee(\varpi_v)\bigr)
=q_v^{-2z}\alpha_i\alpha_j^{-1}.
\]
It follows that
\[
\begin{aligned}
\lambda_{v,z}(h_v^\circ)
&=\prod_{i,j=1}^2
 \left(1-q_v^{-1-2z}\alpha_i\alpha_j^{-1}\right)\\
&=L_v(1+2z,\pi_v\times\widetilde\pi_v)^{-1},
\end{aligned}
\]
as required.

\subsection{Proof of Proposition~\ref{prop:intertwining-sign}}

We first show that $M(0)$ acts by a scalar.
For every place $v$, the normalized parabolic
induction $I_v(\sigma_v,0)$ is irreducible
by the irreducibility of unitary parabolic
induction for general linear groups
(cf.\ \cite[\S0.2, Corollary]{Be84}
in the nonarchimedean case and
\cite[\S\S6 and~17]{Vo86} in the archimedean case).
Consequently, the global induced representation
$I(\sigma,0)$, realized as the restricted tensor
product of its local constituents, is irreducible.
Under the identification $w\sigma=\sigma$ fixed in
\S\ref{subsec:central-intertwining}, the operator
$M(0)$ is an intertwining endomorphism of
$I(\sigma,0)$.
Schur's lemma therefore gives
\[
M(0)=c\,\operatorname{Id}
\]
on the $K$-finite vectors, for some
$c\in\mathbb C$.

We determine $c$ using the global
maximal-unipotent Whittaker coefficient.
For a standard global section $h_z$, put
\[
\mathcal W(h,z)
=
W_\psi(E_P(h_z,z)).
\]
Take a factorizable flat section
\[
h=\bigotimes_v h_v,
\]
with $h_v=h_v^\circ$ outside a finite set
$S$ containing all archimedean places,
all places where $\pi_v$ or $\psi_v$ is
ramified, and all places where the local
data or Haar measures are not unramified.
Normalize the global Whittaker functional
on $\sigma$ as the restricted tensor product
of the local functionals, with
\[
\ell_{\sigma_v}(h_v^\circ;1)=1
\qquad (v\notin S),
\]
and use the compatible product Haar
measures in the global unfolding and
the local Jacquet integrals.

Unfolding the global Whittaker coefficient
in the initial region of absolute convergence
(cf.\ \cite[\S~4 and pp.~351--352]{Sh81})
and applying Lemma~\ref{lem:local-jacquet},
we obtain
\[
\mathcal W(h,z)
=
\frac{
\displaystyle\prod_{v\in S}
\lambda_{v,z}(h_v)}
{L^S(1+2z,\pi\times\widetilde\pi)}.
\]
Both sides admit meromorphic continuation,
so the identity holds meromorphically in $z$.

The Rankin--Selberg $L$-function
$L(s,\pi\times\widetilde\pi)$ has a
simple pole at $s=1$ \cite{JS81}.
Moreover, each local factor removed at
a place in $S$ is finite and nonzero
at $s=1$. Indeed, the local parameters
of a unitary generic representation of
$\mathrm{GL}_2$ satisfy the strict
unitary bounds, which exclude poles of
the corresponding local Rankin--Selberg
factor at $s=1$, including at the
archimedean places.
It follows that
$L^S(s,\pi\times\widetilde\pi)$
also has a simple pole at $s=1$.
Set
\[
\rho_S
=
\operatorname*{Res}_{s=1}
L^S(s,\pi\times\widetilde\pi)
\neq0.
\]
Then
\[
L^S(1+2z,\pi\times\widetilde\pi)^{-1}
=
\frac{2z}{\rho_S}+O(z^2).
\]

By the local holomorphy established in
Lemma~\ref{lem:local-jacquet},
the product of the local Jacquet integrals
is holomorphic near $z=0$.
Consequently,
\[
\mathcal W(h,z)
=
\frac{2z}{\rho_S}
\prod_{v\in S}\lambda_{v,0}(h_v)
+O(z^2).
\]
In particular,
\[
\mathcal W(h,0)=0.
\]

Define the linear functional
\[
D(h)
=
\left.
\frac{\partial}{\partial z}
\mathcal W(h,z)
\right|_{z=0}.
\]
For factorizable sections as above,
\[
D(h)
=
\frac{2}{\rho_S}
\prod_{v\in S}\lambda_{v,0}(h_v).
\]
Lemma~\ref{lem:local-jacquet} asserts
that each local functional
$\lambda_{v,0}$ is nonzero on the
$K_v$-finite vectors.
Since $S$ is finite, we may choose
$K_v$-finite sections $h_v$ for
$v\in S$ such that
\[
\lambda_{v,0}(h_v)\neq0
\qquad (v\in S).
\]
For this choice, $D(h)\neq0$.
Thus $D$ is a nonzero linear functional
on the $K$-finite flat-section space.

Finite linear combinations of factorizable
sections span this space. By linearity,
the preceding expansion therefore implies
that, for every $K$-finite flat section $h$,
\[
\mathcal W(h,z)
=
zD(h)+O(z^2).
\]

We now apply the functional equation for Eisenstein
series from \cite[Theorem~IV.1.10(a)]{MW}.
Under the fixed identification
$w\sigma=\sigma$, it takes the form
\[
E_P(f_z,z)
=
E_P(M(z)f_z,-z).
\]
Taking maximal-unipotent Whittaker
coefficients gives
\[
\mathcal W(f,z)
=
\mathcal W(M(z)f,-z).
\]

Fix a $K$-finite flat section $f$ with $D(f)\neq0$,
which exists since $D$ is nonzero.
In the fixed compact-picture realization,
$M(z)$ is $K$-equivariant and hence maps
the finite sum of $K$-isotypic components
containing $f$ into the corresponding
sum in the target compact picture.
By admissibility, these spaces are
finite-dimensional. Furthermore,
the unnormalized global intertwining
operator $M(z)$ is holomorphic at $z=0$
by \cite[Proposition~IV.1.11(b)]{MW},
since the inducing representation is
unitary at the central parameter.

We may therefore choose fixed
$K$-finite flat sections
$h_1,\ldots,h_r$ and holomorphic
functions $a_1(z),\ldots,a_r(z)$
such that
\[
M(z)f_z
=
\sum_{i=1}^r a_i(z)(h_i)_{-z}.
\]
At $z=0$,
\[
\sum_{i=1}^r a_i(0)h_i
=
M(0)f=cf.
\]

Using the expansion of the Whittaker
coefficient at the central point,
we obtain
\[
\begin{aligned}
\mathcal W(M(z)f,-z)
&=
\sum_{i=1}^r
a_i(z)\mathcal W(h_i,-z)\\
&=
-z\sum_{i=1}^r
a_i(0)D(h_i)+O(z^2)\\
&=
-zD(M(0)f)+O(z^2)\\
&=
-zcD(f)+O(z^2).
\end{aligned}
\]
On the other hand,
\[
\mathcal W(f,z)
=
zD(f)+O(z^2).
\]
Comparing the coefficients of $z$
in the functional equation yields
\[
D(f)=-cD(f).
\]
Since $D(f)\neq0$, we conclude that
$c=-1$. Hence
\[
M(0)=-\operatorname{Id}
\]
on the $K$-finite vectors of
$I(\sigma,0)$.

Finally, $M(0)$ is continuous on the
smooth induced representation, and
the $K$-finite vectors are dense.
The identity therefore extends to
all smooth sections.

\section{Local continuation at \texorpdfstring{$s_0=1$}{s0=1}}\label{app:local-continuation}

Fix a place $v$ of $F$, and write $k=F_v$. All measures and absolute
values in this appendix are local, with $d^\times z=c_v\,dz/|z|_v$.
We prove the regularity statement used in the local residue calculation
of Section~\ref{sec:local-residue} and in the Euler-product comparison.
Fix $s_1$ sufficiently large. The parameter $s$ is fixed with
$\Re s>s_1$; more generally, the estimates below are uniform for $s$ in a
prescribed compact subset $\Sigma$ of $\{\Re s>s_1\}$. We do not assert a
continuation in $s$ here.

Let $0\le\theta_v<\tfrac12$ be an exponent for which the
matrix-coefficient bound \cite[(2.8)]{MV10} holds for $\pi_v$; such an
exponent exists because $\pi_v$ is a local component of a cuspidal
automorphic representation of $\mathrm{PGL}_2$. Fix
\[
             0<\beta_0<\tfrac12-\theta_v.
\]
The essential point is to obtain a H\"older estimate in the scaling
variable $z$, uniformly for the finite-dimensional families of inducing
vectors and the compact families of Schwartz functions that occur in
the $K_v$-average.

\subsection{A Tate-quotient lemma}

Let $\Phi_0$ denote the standard local Tate function, normalized so that
\[
 \mathcal T_v(\sigma)=\int_{k^\times}\Phi_0(z)|z|_v^\sigma\,d^\times z
             =\zeta_v(\sigma).
\]
Write
\[
 A_v(\sigma)=\int_{|z|_v\le1}|z|_v^\sigma\,d^\times z
 =\begin{cases}
 (1-q_v^{-\sigma})^{-1},&v<\infty,\\
 2/\sigma,&k=\mathbb R,\\
 2\pi/\sigma,&k=\mathbb C.
 \end{cases}
\]
In the complex case we use the normalized absolute value
$|z|_v=|z|_{\mathbb C}^2$. In each case $A_v$ has a simple pole at
$\sigma=0$.

\begin{lemma}[Tate quotient]\label{lem:tate-quotient}
Fix $0<\varepsilon\le1$, $\sigma_1>0$, and $C>0$.
Suppose $F:k^\times\to\mathbb C$ is measurable, has a limiting value
$F(0)$, and satisfies
\begin{align*}
 |F(0)|&\le C,\\
 |F(z)-F(0)|&\le C|z|_v^\varepsilon &&(|z|_v\le1),\\
 \int_{|z|_v>1}|F(z)|\,|z|_v^{\sigma_1}\,d^\times z&\le C.
\end{align*}
Then
\[
 Z_F(\sigma)=\int_{k^\times}F(z)|z|_v^\sigma\,d^\times z
\]
is absolutely convergent for $0<\Re\sigma\le\sigma_1$ and
$Z_F(\sigma)/\mathcal T_v(\sigma)$ extends holomorphically to a sufficiently
small disc about zero. On that disc the quotient is bounded in terms
of $v,\varepsilon,\sigma_1,C$, and
\[
           \left.\frac{Z_F(\sigma)}{\mathcal T_v(\sigma)}\right|_{\sigma=0}
                    =F(0).
\]
\end{lemma}

\begin{proof}
Subtract the constant term only on the unit ball:
\begin{align*}
 Z_F(\sigma)&=F(0)A_v(\sigma)+h_F(\sigma),\\
 h_F(\sigma)&=\int_{|z|_v\le1}(F(z)-F(0))|z|_v^\sigma\,d^\times z
       +\int_{|z|_v>1}F(z)|z|_v^\sigma\,d^\times z.
\end{align*}
The first integral defining $h_F$ converges normally when
$\Re\sigma>-\varepsilon$ and the second when
$\Re\sigma\le\sigma_1$, with local uniformity on compact subsets of
the interior. Consequently $h_F$ is holomorphic in the strip
$-\varepsilon<\Re\sigma<\sigma_1$ and uniformly bounded on a
smaller closed disc. Similarly
$\mathcal T_v(\sigma)=A_v(\sigma)+h_T(\sigma)$, where $h_T$ is holomorphic
near zero: it vanishes at finite places, while at archimedean places
this follows from $\Phi_0(z)-1=O(|z|_v)$ near the origin and rapid decay
at infinity.

Since $\mathcal T_v=\zeta_v$ has a simple pole at zero and no zeros,
$1/\mathcal T_v$ is holomorphic near zero and vanishes there. Therefore
\[
 \frac{Z_F(\sigma)}{\mathcal T_v(\sigma)}
  =F(0)+\frac{h_F(\sigma)-F(0)h_T(\sigma)}{\mathcal T_v(\sigma)}.
\]
This gives the asserted value and bound.
\end{proof}

\subsection{Polar coordinates and uniform majorants}

Let $\Psi\in\mathcal S(V'(k))$ and $\xi\in\sigma_v$.
We use the polar decomposition from the proof of
Theorem~\ref{thm:local-residue}, choosing a section
$\widehat\ell$ of $k^2\setminus\{0\}\to\mathbb P^1(k)$ whose image
is compact: it has norm one at archimedean places and is primitive at
finite places. With $d\mu$ the projective measure fixed there,
\begin{equation}\label{eq:B-polar}
 \mathcal Z_v^{\mathrm{orb}}(\Psi,\xi;s,s_0)
   =\int_{k^\times}|z|_v^{2s_0-2}\mathcal F(z;s,s_0)\,d^\times z,
\end{equation}
where
\begin{align}\label{eq:B-F}
 \mathcal F(z;s,s_0)=\int_{\mathbb P^1(k)}\int_{V(k)\times k^2}
 &\Psi(z\widehat\ell,x,\lambda)
 |\Delta(z\widehat\ell,x,\lambda)|_v^{s-1}
 \mathbf1_{V'_E}(z\widehat\ell,x,\lambda)\notag\\
 &\quad\cdot|2\det\widehat y|_v^{s_0-2}
 \beta_{\xi,v}(\widehat y)\,dx\,d\lambda\,d\mu,
 \qquad \widehat y=\widehat\ell_2x_1-\widehat\ell_1x_2.
\end{align}
Here $V(k)=M_2(k)\oplus M_2(k)$, and the indicator restricts to the
specified discriminant square class. Values of the integrand on the
zero set of the discriminant or determinant do not affect the integral.

Let $W\subset\sigma_v$ be a fixed finite-dimensional space, equipped
with a norm $\|\cdot\|_W$; at archimedean places its vectors are
$K_{M,v}$-finite. We use the following two uniform inputs.

\noindent\textit{(U1) Matrix coefficients.} For some $N_0\ge0$ and $C_W$,
\begin{equation}\label{eq:B-U1}
 |\beta_{\xi,v}(y)|\le C_W\|\xi\|_W
 \left(\frac{|\det y|_v}{\|y\|^2}\right)^{\beta_0}
 \left(1+\left|\log\frac{|\det y|_v}{\|y\|^2}\right|\right)^{N_0},
 \qquad y\in\mathrm{GL}_2(k).
\end{equation}
Here $\|y\|$ is chosen so that $\|cy\|^2=|c|_v^2\|y\|^2$;
for example one may take $\|y\|=\max_{i,j}|y_{ij}|_v$.
The ratio in \eqref{eq:B-U1} is then invariant under scalar
multiplication and is comparable, up to constants depending only on
$v$, to the Cartan contraction parameter. At complex places this
uses $|a|_v=|a|_{\mathbb C}^2$, rather than the ordinary complex
modulus.
Indeed, the nonzero invariant bilinear form $B_v$ identifies $\pi_v$
with its smooth contragredient, and the unitary structure identifies
the latter with the complex conjugate of $\pi_v$. Consequently, on $W$,
$\beta_{\xi,v}$ is a linear combination, with coefficients
$O_W(\|\xi\|_W)$, of finitely many matrix coefficients
$y\mapsto\langle\pi_v(y)u_j,w_j\rangle$ of fixed vectors
$u_j,w_j\in\pi_v$ depending only on $W$; in particular, their Sobolev
norms depend only on $W$. By the uniform matrix-coefficient estimate of
Michel--Venkatesh \cite[(2.8)]{MV10},
\[
 |\beta_{\xi,v}(y)|\ll_W\|\xi\|_W\,\Xi_v(y)^{1-2\theta_v}.
\]
By the Cartan decomposition, every $y\in\mathrm{GL}_2(k)$ lies in
$Z(k)K_v\operatorname{diag}(t,1)K_v$ for some $t$ with $|t|_v\le1$, and,
with the norm convention above, $|\det y|_v/\|y\|^2\asymp|t|_v$.
Since $\Xi_v$ is bi-$K_v$-invariant, the classical estimate for the
spherical function gives $\Xi_v(y)\ll|t|_v^{1/2}(1+|\log|t|_v|)^{N_v}$
for a fixed $N_v$. Since $\beta_0<\frac12-\theta_v$, the extra power of
$|t|_v$ absorbs the logarithmic factor, which proves \eqref{eq:B-U1}
uniformly for $\xi$ in the unit ball of $W$.

\noindent\textit{(U2) Schwartz control.} Constants may depend on finitely
many Schwartz seminorms of $\Psi$; at finite places they may also
depend on a common compact support and a common open compact subgroup
fixing the functions under consideration. These controls are uniform
for the compact family $\{\rho(\kappa)\Psi:\kappa\in K_v\}$.
At archimedean places we include first derivatives in the
$\ell$-variables among the seminorms.

\begin{lemma}[Uniform H\"older estimate]\label{lem:local-continuation}
Let $\Sigma$ be a fixed compact subset of $\{\Re s>s_1\}$, and let $\mathcal N$ be a compact
subset of $\{\Re s_0>1-\beta_0\}$. For $s\in\Sigma$,
$s_0\in\mathcal N$, and $\xi\in W$ with $\|\xi\|_W\le1$,
the integral \eqref{eq:B-F} is absolutely convergent. There are
$\varepsilon>0$ and $C>0$, uniform in these parameters and in the
Schwartz families just described, such that
\begin{align}
 |\mathcal F(z;s,s_0)|&\le C &&(|z|_v\le1),\label{eq:B-small}\\
 |\mathcal F(z;s,s_0)-\mathcal F(0;s,s_0)|&\le C|z|_v^\varepsilon
                                      &&(|z|_v\le1),\label{eq:B-holder}\\
 \int_{|z|_v>1}|\mathcal F(z;s,s_0)|\,|z|_v^{2\Re s_0-1}
            \,d^\times z&\le C.\label{eq:B-large}
\end{align}
\end{lemma}

\begin{proof}
Put $\sigma=\Re s_0$. For $|z|_v\le1$, polynomiality of
$\Delta$ and compactness of $\Sigma$ give a polynomial bound for
$|\Delta(z\widehat\ell,x,\lambda)|_v^{\Re s-1}$.
Choose a rapidly decreasing weight
\[
 \mathcal W(x,\lambda)=
  (1+\|(x,\lambda)\|)^c
  \sup_{\|\ell\|\le1}
   \bigl(|\Psi(\ell,x,\lambda)|+\|\nabla_\ell\Psi(\ell,x,\lambda)\|\bigr),
\]
with the derivative term omitted in the nonarchimedean case and with
$c$ sufficiently large for $\Sigma$. The coordinate norms in this
weight may be any fixed proper archimedean norms; at finite places
we use the maximum coordinate norm. The Schwartz estimates imply
uniform rapid decrease of $\mathcal W$ in $(x,\lambda)$. Let
\[
 G_\sigma(x,\widehat\ell)
  =|2\det\widehat y|_v^{\sigma-2}
      |\beta_{\xi,v}(\widehat y)|,
 \qquad d\nu=\mathcal W(x,\lambda)\,dx\,d\lambda\,d\mu.
\]
The measure $\nu$ is finite, and the absolute values of the
integrands of both $\mathcal F(z;s,s_0)$ and $\mathcal F(0;s,s_0)$ are bounded by
$\mathcal W G_\sigma$.

We first show that $G_\sigma\in L^p(\nu)$ for some $p>1$,
uniformly on $\mathcal N$. For each $\widehat\ell$, the linear map
$(x_1,x_2)\mapsto\widehat\ell_2x_1-\widehat\ell_1x_2$
is surjective onto $M_2(k)$, with uniformly controlled changes of
coordinates as $\widehat\ell$ ranges through its compact image.
The coefficient bound \eqref{eq:B-U1} reduces the local
integrability question to powers of $\det y$ and logarithms.
Near a smooth point of the rank-one determinantal hypersurface the
required inequality is
\[
                 p(\sigma-2+\beta_0)>-1.
\]
Near $y=0$, decompose into scalar shells $y=r\omega$ with
$\omega$ on a compact unit matrix shell. The determinant is
homogeneous of degree two, and the ratio in \eqref{eq:B-U1}
has degree zero. The radial condition is therefore
\[
                    p(2\sigma-4)>-4.
\]
The angular integral is controlled by the first inequality:
on the unit shell every singular matrix has rank one, where
the determinant hypersurface is smooth. Thus the angular and
radial estimates hold simultaneously, including at the origin.
These shell calculations use the normalized absolute value at
complex places.
Both inequalities are strict at $p=1$ whenever
$\sigma>1-\beta_0$; compactness of $\mathcal N$ allows a single
$p>1$. The logarithmic factors are absorbed by the strict margins.
Rapid decrease handles the remaining variables and the region at
infinity. Consequently,
\begin{equation}\label{eq:B-Lp}
       \sup_{s_0\in\mathcal N}\|G_\sigma\|_{L^p(\nu)}<\infty.
\end{equation}
This proves absolute convergence and \eqref{eq:B-small}.

We next compare the integrands at $z$ and at zero. Write
\[
 D_i=\det x_i,\qquad
 A=\det(x_1+x_2)-\det x_1-\det x_2.
\]
The quadratic form defining the discriminant has coefficients
$q(w_i)=\ell_i\lambda_i-D_i$ and
$B(w_1,w_2)=\ell_1\lambda_2+\ell_2\lambda_1-A$.
Consequently, with $\ell=z\widehat\ell$,
\begin{align*}
 \Delta(z\widehat\ell,x,\lambda)-\Delta(x)
 ={}&z\bigl[-2A(\widehat\ell_1\lambda_2+
                   \widehat\ell_2\lambda_1)
       +4\widehat\ell_1\lambda_1D_2
       +4\widehat\ell_2\lambda_2D_1\bigr]\\
 &+z^2(\widehat\ell_1\lambda_2-
                   \widehat\ell_2\lambda_1)^2.
\end{align*}
Here $\Delta(x)=\Delta(0,x,\lambda)=A^2-4D_1D_2$ is the restriction of
the discriminant to the cube space; it does not depend on $\lambda$.
To state the perturbation estimate uniformly at all places, put
\[
 H_v(x,\lambda)=\max\left\{1,
     \max_{m,i,j}|(x_m)_{ij}|_v,\ \max_i|\lambda_i|_v\right\},
 \qquad P(x,\lambda)=H_v(x,\lambda)^4.
\]
Here $H_v$ is a coordinate-size function, not necessarily a norm
in the usual archimedean sense. In particular, at complex places
$|a|_v=|a|_{\mathbb C}^2$. The coefficient of $z$ has degree
three in $(x,\lambda)$, and the coefficient of $z^2$ has degree
two. Since $\widehat\ell$ stays in a bounded set, the normalized
absolute-value triangle inequality (with a fixed archimedean
constant) gives, for $|z|_v\le1$,
\begin{equation}\label{eq:B-perturb}
 \Delta(z\widehat\ell,x,\lambda)
    =\Delta(x)+E(z,x,\lambda),\qquad
 |E(z,x,\lambda)|_v\le C_1|z|_vP(x,\lambda).
\end{equation}
The fourth power is therefore sufficient also at complex places;
using the ordinary Euclidean coordinate norm there would instead
require a weight of degree at least six.
Define
\[
 I_z=\bigl\{(x,\lambda):
       |\Delta(x)|_v\ge2C_1|z|_v^{1/2}P(x,\lambda)\bigr\}.
\]
On $I_z$ the relative perturbation is at most
$\frac12|z|_v^{1/2}$. For sufficiently small $z$, this preserves
the square class of the discriminant: at nonarchimedean places a
sufficiently small principal-unit neighbourhood lies in the square
subgroup; at real places the sign is preserved; at complex places
there is only one square class. For $s$ in the fixed compact set $\Sigma$, the same relative
bound and the local absolute-value estimates give
\[
 \bigl||\Delta(z\widehat\ell,x,\lambda)|_v^{s-1}
       -|\Delta(x)|_v^{s-1}\bigr|
 \ll_\Sigma |z|_v^{1/4}|\Delta(x)|_v^{\Re s-1}.
\]
The exponent $1/4$ works uniformly at all places. Indeed, at a
complex place the normalized bound on the relative perturbation
is converted to the ordinary complex modulus by taking a square
root. At real places the stronger exponent $1/2$ is available;
at nonarchimedean places the absolute values are actually equal
once the perturbation is sufficiently small.
At archimedean places the mean value theorem bounds the difference
of the Schwartz functions by
$|z|_v^{1/2}$ times a controlled derivative majorant
(the normalized complex absolute value accounts for the exponent
$1/2$). At nonarchimedean places, the family of Schwartz
functions has a common open compact subgroup of
translation invariance in the $\ell$-coordinates,
by (U2). Since the chosen projective section is
bounded, the difference of the Schwartz functions
vanishes for all sufficiently small $z$, uniformly
in the remaining variables and in the family. Thus the contribution from $I_z$ is
$O(|z|_v^{1/4})$ by \eqref{eq:B-Lp}.

On $I_z^c$ we do not use square-class stability. Instead the
absolute difference is bounded by $2\mathcal W G_\sigma$, and
H\"older's inequality gives
\[
 \int_{I_z^c}G_\sigma\,d\nu
 \le \|G_\sigma\|_{L^p(\nu)}\nu(I_z^c)^{1/p'},
 \qquad p'=\frac{p}{p-1}.
\]
We claim that for some $c_0>0$,
\begin{equation}\label{eq:B-negative}
  \int_{\mathbb P^1(k)}\int_{V(k)\times k^2}
  |\Delta(x)|_v^{-c_0}P(x,\lambda)^{c_0}
  \mathcal W(x,\lambda)\,dx\,d\lambda\,d\mu<\infty.
\end{equation}
Indeed, since $\Delta$ is a nonzero polynomial, there exists
$c_0>0$ such that $|\Delta(x)|_v^{-c_0}$
is locally integrable on $V(k)$ (\cite{Igusa} at finite places,
\cite{Atiyah,BG69} at archimedean places).
In particular, its integral over a compact
unit shell is finite. By homogeneity,
\[
 \Delta(rx)=r^4\Delta(x),
\]
and the additive Haar measure on $V(k)$
scales by $|r|_v^8$.
Consequently, the integral over a shell of
radius $|r|_v$ is bounded by a constant
multiple of $|r|_v^{8-4c_0}$.
Choosing $c_0$ sufficiently small gives
convergence of the inner shell sum, while
Schwartz decay, with sufficiently many
seminorms, controls the outer shells and
the additional polynomial weight
$P(x,\lambda)^{c_0}$; the $\lambda$-integral is harmless. Concretely,
a Schwartz weight of order $N>8$ (enlarged if necessary to absorb
$P^{c_0}$) makes the outer shell sum convergent. This proves
\eqref{eq:B-negative} with uniform seminorm control.

Chebyshev's inequality now yields
\[
 \nu(I_z^c)\le (2C_1)^{c_0}|z|_v^{c_0/2}
       \int |\Delta(x)|_v^{-c_0}P^{c_0}\,d\nu
 \ll |z|_v^{c_0/2}.
\]
Combining the two regions gives \eqref{eq:B-holder}, for example
with
\[
           \varepsilon=\min\{\tfrac14,\,c_0/(2p')\}>0.
\]
For $|z|_v\le1$ outside the sufficiently small range, \eqref{eq:B-holder}
follows from \eqref{eq:B-small} after enlarging $C$.

Finally, when $|z|_v>1$, the $\ell$-coordinate of the Schwartz
function is $z\widehat\ell$. Since $\widehat\ell$ stays in a compact
set bounded away from zero, Schwartz decay supplies arbitrarily
large negative powers of $|z|_v$, uniformly in $x,\lambda$ after
allowing an additional rapidly decreasing weight in those variables.
The polynomial growth of the discriminant can be absorbed into this
decay. Integrating against $G_\sigma$ using \eqref{eq:B-Lp} proves
\eqref{eq:B-large}.
\end{proof}

\subsection{Continuation and evaluation at the boundary}

\begin{corollary}\label{cor:local-continuation}
Under the hypotheses of Lemma~\ref{lem:local-continuation}, the orbital
integral
\[
\mathcal Z_v^{\mathrm{orb}}(\Psi,\xi;s,s_0)
\]
is absolutely
convergent and holomorphic for $\Re s_0>1$, with $s$ fixed as above. Moreover,
\[
 R_v^{\mathrm{orb}}(\Psi,\xi;s,s_0)
 =\frac{\mathcal Z_v^{\mathrm{orb}}(\Psi,\xi;s,s_0)}{\mathcal T_v(2s_0-2)}
\]
extends holomorphically to a disc $|s_0-1|<\varepsilon_1$ and
satisfies
\[
 R_v^{\mathrm{orb}}(\Psi,\xi;s,1)=\mathcal F(0;s,1).
\]
On a sufficiently small fixed closed disc, the continuation is
uniformly bounded for $s\in\Sigma$, $\|\xi\|_W\le1$, and the
Schwartz families in Lemma~\ref{lem:local-continuation}. In particular, the
bound applies uniformly to the $K_v$-average in the local residue
identity.
\end{corollary}

\begin{proof}
Every point of $\{\Re s_0>1\}$ has a compact neighbourhood
$\mathcal N\subset\{\Re s_0>1\}$, and Lemma~\ref{lem:local-continuation}
applies to it. By \eqref{eq:B-small} and \eqref{eq:B-large}, the
integral \eqref{eq:B-polar} then converges absolutely, uniformly for
$s_0\in\mathcal N$, and is therefore holomorphic on the interior of
$\mathcal N$. This gives absolute convergence and holomorphy on
$\{\Re s_0>1\}$. For the continuation, we apply the lemma with
$\mathcal N$ a compact neighbourhood of $1$ contained in
$\{\Re s_0>1-\beta_0\}$. The holomorphic
continuation of $R_v^{\mathrm{orb}}$ to a disc $|s_0-1|<\varepsilon_1$
and the identity $R_v^{\mathrm{orb}}(\Psi,\xi;s,1)=\mathcal F(0;s,1)$
are proved in Step~2 of the proof of Theorem~\ref{thm:local-residue}.
Explicitly, Step~2 gives
\[
R_v^{\mathrm{orb}}(\Psi,\xi;s,s_0)
=\mathcal F(0;s,s_0)
+\frac{H_v(s_0)-\mathcal F(0;s,s_0)\,h_v(2s_0-2)}
      {\mathcal T_v(2s_0-2)},
\]
with $H_v$ and $h_v$ holomorphic near $s_0=1$; since
$\mathcal T_v(2s_0-2)^{-1}$ is holomorphic and vanishes at $s_0=1$,
the continuation and the boundary value follow.
That argument uses only Lemma~\ref{lem:local-continuation} and the
decomposition $\mathcal T_v=A_v+h_v$, so it applies verbatim here.
Since the constants in Lemma~\ref{lem:local-continuation} are uniform,
the functions $H_v$ and $\mathcal F(0;s,s_0)$ appearing there are
uniformly bounded on a sufficiently small closed disc about $s_0=1$;
together with the holomorphy of $1/\mathcal T_v$ near $\tau=0$, this
gives the asserted uniform bound for the continuation.

For the $K_v$-average, the vectors $\xi_{f_v}(\kappa)$ lie in a
finite-dimensional subspace: $\xi_{f_v}(\kappa)=\xi_{\rho(\kappa)f_v}(1)$, and
the $K_v$-orbit of the $K_v$-finite section $f_v$ spans a finite-dimensional
space. The family $\{\rho(\kappa)\Psi:\kappa\in K_v\}$
has uniform Schwartz seminorms; at finite places it has a common
compact support and invariance subgroup. Thus the estimates apply
uniformly in $\kappa$, permitting the $K_v$-integration and the
holomorphic continuation to be interchanged.
\end{proof}

\begin{remark}\label{rem:B-application}
The domain $\{\Re s_0>1\}\cup\{|s_0-1|<\varepsilon_1\}$ is the
local domain required for the Euler-product comparison. The uniform
bound in Corollary~\ref{cor:local-continuation} is the analytic input for
the $K_v$-averaged local identity, while the value $\mathcal F(0;s,1)$ is
identified geometrically in Theorem~\ref{thm:local-residue}.
\end{remark}

\section{Proof of Theorem~\ref{thm:unramified-factor}}\label{app:unramified}

The proof of Theorem~\ref{thm:unramified-factor} rests on a separation
of three features of the orbital integral that are coupled in its
original coordinates: the spherical matrix coefficient, the determinant
weight, and the quadratic square-class condition on the discriminant.
We first use integral symmetries to reduce the matrix variable to
primitive Cartan representatives. Completing the square then produces a
Haar-uniform variable independent of a second, unit-invariant variable
whose valuation distribution can be analyzed by finite Fourier methods.
Finally, we sum the resulting series and identify it with the
Euler-factor expression in the theorem by elementary rational-function
identities, separately in the inert and split cases.

Fix a nonarchimedean place $v$ of odd residue characteristic, with valuation ring $\mathfrak o=\mathfrak o_v$, uniformizer $\varpi$, and residue cardinality $q$. Assume that $E_v$ is unramified and $d_E\in\mathfrak o^\times$. The representation $\pi_v$ is unramified with trivial central character, and $\sigma$ denotes its normalized spherical matrix coefficient. Additive Haar measure gives $\mathfrak o$ volume one. We use the spherical Schwartz function $\Psi_v=\mathbf1_{V'(\mathfrak o)}$ and the spherical inducing vector, normalized so that $B_v(\xi_v^\circ)=1$; consequently $\beta=\sigma$ and the compact average in the orbital integral is trivial.

For later use put
\[
\begin{gathered}
a_\pi=\alpha+\alpha^{-1},\qquad \lambda_v=q^{1/2}a_\pi,\qquad
Y=q^{-s},\qquad X=qY,\\
T=q^{-(s_0-\frac12)},\qquad Y_0=q^{-(2s_0-2)}=qT^2,\qquad \rho=qY^2.
\end{gathered}
\]
Here $\alpha,\alpha^{-1}$ are the normalized Satake parameters. We write $\mathbb E$ for expectation with respect to independent
additive Haar-uniform variables on $\mathfrak o$. A variable $w$ is
called \emph{unit-invariant} if its distribution is unchanged under
multiplication by any element of $\mathfrak o^\times$. For every $j$
such that $\mathbb P(\operatorname{ord}(w)=j)>0$, the unit part
$\varpi^{-j}w$, conditional on $\operatorname{ord}(w)=j$, is then
Haar-uniform on $\mathfrak o^\times$: since the event
$\{\operatorname{ord}(w)=j\}$ is stable under $\mathfrak o^\times$, the
conditional distribution is $\mathfrak o^\times$-invariant, hence equal
to the normalized Haar measure.

\subsection{Primitive shells and spherical Cartan coefficients}\label{subsec:unram-shells}

We start from the definition
\[
Z'_v=\int_{V'_E(F_v)}
\mathbf 1_{V'(\mathfrak o)}(w)\,
|\Delta(w)|^{s-1}|R(w)|^{s_0-2}
\sigma(y_w)\,dw.
\]
Recall that, in the coordinates of \S\ref{subsec:pfaffian},
the quadratic form on $\bigwedge^2F_v^4$ is
$q(L,X,T)=LT-\det X$. Writing
$w=(w_1,w_2)$, with $w_i=(\ell_i,x_i,\lambda_i)$,
and denoting its associated bilinear form by $B$, we have
\[
\Delta(w)=B(w_1,w_2)^2-4q(w_1)q(w_2),
\qquad
B(u,v)=q(u+v)-q(u)-q(v).
\]
Let $\langle\ ,\ \rangle$ denote the polar form
of $\det$ on $M_2$, normalized by
$\langle x,x\rangle=\det x$.

The first reduction uses the outer
$\mathrm{GL}_2(\mathfrak o)$-action on the pair
$\ell=(\ell_1,\ell_2)$. This action also transforms
the remaining coordinates, but preserves the integral
lattice, the absolute values of the relative invariants,
and the square class of the discriminant.
It therefore allows us to replace each primitive
$\ell$-vector by a fixed representative.

\begin{lemma}[Primitive-shell formula]
\label{lem:unramified-primitive-shell}
We have
\[
Z'_v=(1-q^{-2})\sum_{k\ge0}Y_0^k
\int_{M_2(\mathfrak o)}
\sigma(x_1)|\det x_1|^{s_0-2}
\mathbb E\Bigl[
|\Delta_k|^{s-1}
\mathbf 1[\Delta_k\in d_E(F_v^\times)^2]
\Bigr]\,dx_1,
\]
where
\[
\Delta_k=
\bigl(\varpi^ku'-2\langle x_1,x_2\rangle\bigr)^2
+4\det x_1
\bigl(\varpi^kv'-\det x_2\bigr),
\]
and the expectation is taken over independent
additive Haar-uniform variables
$x_2\in M_2(\mathfrak o)$ and
$u',v'\in\mathfrak o$.
\end{lemma}

\begin{proof}
\emph{1. Reduction to primitive shells.}
Decompose $\ell=\varpi^k\ell_0$ with $k\ge0$ and
$\ell_0\in\mathfrak o^2$ primitive.
The shell $\varpi^k\bigl(\mathfrak o^2\setminus\varpi\mathfrak o^2\bigr)$
has additive volume
\[
(1-q^{-2})q^{-2k}.
\]

The group $\mathrm{GL}_2(\mathfrak o)$ acts
transitively on primitive vectors.
For each $\ell_0$, choose
$k'\in\mathrm{GL}_2(\mathfrak o)$ carrying
$\ell_0$ to $(0,1)$ under the outer action
$(w_1,w_2)\mapsto(w_1,w_2)k'$.
For fixed $\ell_0$, this is an integral linear
change of variables in the remaining coordinates,
with Jacobian of absolute value one.

More precisely, the relative invariants satisfy
\[
\Delta(w\cdot k')=\det(k')^2\Delta(w),
\qquad
y_{w\cdot k'}=\det(k')y_w,
\]
and consequently $R(w\cdot k')=\det(k')^2R(w)$.
Since $\det(k')\in\mathfrak o^\times$,
the absolute values $|\Delta|$ and $|R|$
are unchanged, as is the condition
$\Delta\in d_E(F_v^\times)^2$.
Moreover, the trivial central character of
$\pi_v$ gives $\sigma(y_{w\cdot k'})=\sigma(y_w)$.
Thus the integral over the remaining coordinates
depends only on $k$, not on the primitive vector
$\ell_0$. Integrating over the shell contributes
the factor $(1-q^{-2})q^{-2k}$.

\emph{2. Evaluation at the shell representative.}
We may now take $\ell=(0,\varpi^k)$.
For this representative,
\[
\begin{aligned}
q(w_1)&=-\det x_1,\\
q(w_2)&=\varpi^k\lambda_2-\det x_2,\\
B(w_1,w_2)&=\varpi^k\lambda_1-2\langle x_1,x_2\rangle.
\end{aligned}
\]
It follows that
\[
\begin{aligned}
\Delta
&=B(w_1,w_2)^2-4q(w_1)q(w_2)\\
&=\bigl(\varpi^k\lambda_1-2\langle x_1,x_2\rangle\bigr)^2
  +4\det x_1\bigl(\varpi^k\lambda_2-\det x_2\bigr).
\end{aligned}
\]
(As a check, $w_\mu$ has $k=0$, $x_1=I_2$, $x_2=0$, $\lambda=(0,\mu)$,
giving $\Delta(w_\mu)=4\mu$.)
Renaming $\lambda_1=u'$ and $\lambda_2=v'$
gives the stated expression for $\Delta_k$.

Furthermore, $y_w=\varpi^kx_1$ and
$R(w)=-2\varpi^{2k}\det x_1$.
Since the residue characteristic is odd,
$|2|=1$, and hence
$|R(w)|^{s_0-2}=q^{-2k(s_0-2)}|\det x_1|^{s_0-2}$.
The trivial central character also gives
$\sigma(\varpi^kx_1)=\sigma(x_1)$.

Finally, the shell volume and the Mellin weight combine to
$q^{-2k}q^{-2k(s_0-2)}=q^{-2k(s_0-1)}=Y_0^k$.
This proves the formula.
\end{proof}

The remaining matrix variable cannot be integrated using determinant valuations alone: the discriminant distribution changes when the elementary divisors of $x_1$ change, even if their sum is fixed. We therefore separate the common scalar power of $\varpi$ from the primitive elementary-divisor difference. The spherical coefficient on each resulting double coset is encoded by the usual Hecke generating function.

\begin{lemma}[Spherical coefficients on primitive Cartan cosets]
\label{lem:unramified-cartan-coefficients}
Let $K=\mathrm{GL}_2(\mathfrak o)$ and
\[
\kappa_{K,v}=(1-q^{-1})(1-q^{-2}),
\qquad
\sum_{m\ge0}\tau_mX^m
=\frac{1}{1-\lambda_vX+qX^2},
\]
with $\tau_m=0$ for $m<0$.
Then the following hold.
\begin{enumerate}
\item For every $m\ge0$,
\[
\int_{\substack{x\in M_2(\mathfrak o)\\
                 \operatorname{ord}(\det x)=m}}
\sigma(x)\,dx
=\kappa_{K,v}q^{-2m}\tau_m.
\]

\item For every $n\ge0$, the primitive matrices
$x\in M_2(\mathfrak o)$ satisfying
$\operatorname{ord}(\det x)=n$ form the single
Cartan double coset
$K\operatorname{diag}(1,\varpi^n)K$.
Their spherical integral is
\[
P_n:=
\int_{K\operatorname{diag}(1,\varpi^n)K}
\sigma(x)\,dx
=\kappa_{K,v}q^{-2n}
(\tau_n-\tau_{n-2}).
\]

\item For every $a,n\ge0$,
\[
\int_{K\operatorname{diag}(\varpi^a,\varpi^{a+n})K}
\sigma(x)\,dx
=q^{-4a}P_n.
\]
\end{enumerate}
\end{lemma}

\begin{proof}
\emph{1. Determinant shells and Hecke eigenvalues.}
The additive volume of $K$ inside $M_2(\mathfrak o)$ is
$\operatorname{vol}_{\mathrm{add}}(K)=(1-q^{-1})(1-q^{-2})=\kappa_{K,v}$.
For $g\in M_2(\mathfrak o)\cap\mathrm{GL}_2(F_v)$,
right multiplication by $g$ scales additive
Haar measure on $M_2(F_v)$ by $|\det g|^2$.
Consequently, each right $K$-coset $gK$
with $\operatorname{ord}(\det g)=m$
has additive volume
$\operatorname{vol}_{\mathrm{add}}(gK)=\kappa_{K,v}q^{-2m}$.

Since $\sigma$ is $K$-bi-invariant, the integral
over the determinant-$m$ shell is
\[
\kappa_{K,v}q^{-2m}
\sum_{\substack{gK\subset M_2(\mathfrak o)\\
                 \operatorname{ord}(\det g)=m}}
\sigma(g).
\]
The sum is the spherical eigenvalue $\tau_m$
of the Hecke operator associated with the
entire determinant-$m$ shell, including all
its Cartan double cosets.

With normalized Satake parameters
$\alpha,\alpha^{-1}$ and trivial central
character, this eigenvalue is
\[
\tau_m=q^{m/2}
\sum_{j=0}^{m}\alpha^{m-2j}.
\]
Thus
\[
\sum_{m\ge0}\tau_mX^m
=\frac{1}
{(1-q^{1/2}\alpha X)(1-q^{1/2}\alpha^{-1}X)}
=\frac{1}{1-\lambda_vX+qX^2},
\]
which proves (1).

\emph{2. Primitive Cartan coefficients.}
A matrix in $M_2(\mathfrak o)$ is primitive
if its entries are not all divisible by
$\varpi$. By the elementary divisor theorem,
a primitive matrix with determinant valuation
$n$ has elementary divisors $(1,\varpi^n)$.
Hence these matrices form precisely
$K\operatorname{diag}(1,\varpi^n)K$.

Every nonsingular integral matrix has a
unique expression $x=\varpi^a x_0$, where
$a\ge0$ and $x_0$ is primitive.
Multiplication by $\varpi^a$ scales additive
measure on $M_2(F_v)$ by $q^{-4a}$.
Since $\pi_v$ has trivial central character,
$\sigma(\varpi^a x_0)=\sigma(x_0)$.
It follows that
\[
\int_{\operatorname{ord}(\det x)=m}
\sigma(x)\,dx
=\sum_{a=0}^{\lfloor m/2\rfloor}
q^{-4a}P_{m-2a}.
\]
Applying (1) at $m$ and $m-2$ and subtracting
gives
\[
\begin{aligned}
P_m
&=\kappa_{K,v}q^{-2m}\tau_m
 -q^{-4}\kappa_{K,v}q^{-2(m-2)}
   \tau_{m-2}\\
&=\kappa_{K,v}q^{-2m}
  (\tau_m-\tau_{m-2}),
\end{aligned}
\]
where $\tau_{m-2}=0$ if $m<2$.

\emph{3. General Cartan cosets.}
Finally,
\[
K\operatorname{diag}(\varpi^a,\varpi^{a+n})K
=\varpi^aK\operatorname{diag}(1,\varpi^n)K.
\]
The same scaling argument used in (2)
therefore gives
\[
\int_{K\operatorname{diag}(\varpi^a,\varpi^{a+n})K}
\sigma(x)\,dx
=q^{-4a}P_n.
\]
\end{proof}

\begin{remark}
Only parts (2) and (3) are used below.
The expectation in
Lemma~\ref{lem:unramified-primitive-shell}
depends on $x_1$ through its Cartan coset
$(a,a+n)$, not merely through
$\operatorname{ord}(\det x_1)=2a+n$.
Indeed, a determinant shell generally contains
several Cartan double cosets, on which the
expectation need not agree.
Thus the determinant-shell formula in (1)
cannot be applied directly to the remaining
$x_1$-integral.
\end{remark}

With $x_1\in K\operatorname{diag}(\varpi^a,\varpi^{a+n})K$, the $x_1$-weights
combine, using $\tau_m=q^{m/2}s_m$ and $\sum_ms_mX^m=(1-a_\pi X+X^2)^{-1}$, to
\[
q^{-4a}P_n|\det x_1|^{s_0-2}
=\kappa_{K,v}\Bigl(\frac{T^2}{q}\Bigr)^{a}\Bigl(s_n-\frac{s_{n-2}}q\Bigr)T^n.
\]
Indeed, $q^{-4a}q^{-(2a+n)(s_0-2)}=q^{-2as_0}q^{-n(s_0-2)}$, and
$q^{-2as_0}=(T^2/q)^a$.

\subsection{The discriminant in diagonal coordinates}\label{subsec:unram-diagonal}

We next show that the expectation in
Lemma~\ref{lem:unramified-primitive-shell}
is constant on each primitive Cartan double coset.
This does not follow from the bi-$K$-invariance of
$\sigma$ alone, since the discriminant also depends
on the remaining integration variables.
The required invariance follows from the integral
Levi action, combined with a compensating scalar
in the outer $\mathrm{GL}_2$.
After passing to a diagonal representative, we
complete the square in the discriminant and obtain
a Haar-uniform variable $z$ independent of a
unit-invariant variable $w$.
These properties will be used in the subsequent
square-class calculation.

\begin{lemma}[Levi invariance and square completion]
\label{lem:unramified-levi-invariance}
\leavevmode
\begin{enumerate}
\item Let $K=\mathrm{GL}_2(\mathfrak o)$.
For $(g,h)\in K^2$, the Levi element
$\operatorname{diag}(g,h)$, combined with the
outer scalar $\det(g)^{-1}$, fixes
$\ell=(0,\varpi^k)$ and acts on the remaining
coordinates by
\[
x_i\longmapsto\det(g)^{-1}\,{}^tgx_ih,
\qquad
\lambda_i\longmapsto\det(h)\det(g)^{-1}\lambda_i.
\]
This transformation preserves the lattice and
Haar measure on $(x_2,u',v')$, multiplies
$\Delta_k$ by a unit square, and preserves
$|\det x_1|$ and $\sigma(x_1)$.
Consequently, the expectation in
Lemma~\ref{lem:unramified-primitive-shell}
is constant on each Cartan double coset.
We may therefore take
\[
x_1=\operatorname{diag}
(\varpi^a,\varpi^{a+n}),
\qquad a,n\ge0.
\]

\item Write
\[
x_2=\begin{pmatrix}p&r\\t&m\end{pmatrix},
\qquad
2\langle x_1,x_2\rangle
=\varpi^am+\varpi^{a+n}p.
\]
Then
\[
\Delta_k=\varpi^{2c}
\bigl(z^2+4\varpi^e w\bigr),
\]
where
\[
\begin{array}{lll}
k\ge a:
&c=a,\quad e=n,
&z=m-\varpi^np-\varpi^{k-a}u',\\
&&w=rt-\varpi^{k-a}u'p+\varpi^kv';
\\[5pt]
k<a:
&c=k,\quad e=2a+n-2k,
&z=u'-\varpi^{a-k}(m+\varpi^np),\\
&&w=rt-pm+\varpi^kv'.
\end{array}
\]

\item In both cases, $z$ is additive
Haar-uniform on $\mathfrak o$ and independent
of $w$. Moreover, the distribution of $w$
is invariant under multiplication by
$\mathfrak o^\times$.
\end{enumerate}
\end{lemma}

\begin{proof}
\emph{1. Invariance under the integral Levi action.}
By the action formulas of
\S\ref{subsec:group}, the Levi element
$\operatorname{diag}(g,h)$ acts by
\[
\ell\longmapsto\det(g)\ell,
\qquad
x_i\longmapsto{}^tgx_ih,
\qquad
\lambda_i\longmapsto\det(h)\lambda_i.
\]
Multiplying by the outer scalar
$\det(g)^{-1}$ fixes $\ell$ and gives
the transformations stated in (1).

Put $\mu=\det(h)/\det(g)\in\mathfrak o^\times$.
The linear map
$x_2\mapsto\det(g)^{-1}\,{}^tgx_2h$ on $M_2(F_v)$
has determinant $\mu^2$.
The map
$(u',v')\mapsto\det(h)\det(g)^{-1}(u',v')$
has determinant $\mu^2$ as well.
Thus the combined transformation on
$M_2(\mathfrak o)\times\mathfrak o^2$
has determinant $\mu^4\in\mathfrak o^\times$.
It preserves both the lattice and its
normalized additive Haar measure.

The quadratic form $q$ transforms by
\[
q\bigl(\det(g)^{-1}\,w_i\cdot\operatorname{diag}(g,h)\bigr)
=\mu\,q(w_i),
\]
since $q(w\cdot\operatorname{diag}(g,h))=\det(g)\det(h)q(w)$ and the
outer scalar multiplies $q$ by $\det(g)^{-2}$.
Consequently, the discriminant of the
binary quadratic form
$(u,v)\mapsto q(uw_1+vw_2)$
is multiplied by $\mu^2$; in particular,
$\Delta_k\mapsto\mu^2\Delta_k$.
Since $\mu$ is a unit, this preserves
$|\Delta_k|$ and the condition
$\Delta_k\in d_E(F_v^\times)^2$.

Finally,
$|\det(\det(g)^{-1}\,{}^tgx_1h)|=|\det x_1|$ and
$\sigma(\det(g)^{-1}\,{}^tgx_1h)=\sigma(x_1)$,
because $\det(g)^{-1}\,{}^tg\in K$, $h\in K$,
and $\sigma$ is bi-$K$-invariant.
Thus the expectation is unchanged under
$x_1\mapsto\det(g)^{-1}\,{}^tgx_1h$. Since
$g\mapsto\det(g)^{-1}\,{}^tg$ is a bijection of $K$,
these maps realize the full $K\times K$-action on $x_1$,
and the Cartan decomposition gives the stated diagonal
representative.

\emph{2. Completion of the square.}
For the diagonal representative,
$\det x_1=\varpi^{2a+n}$ and $\det x_2=pm-rt$.
Substituting into the expression for
$\Delta_k$ gives
\[
\begin{aligned}
\Delta_k
={}&\bigl(\varpi^ku'
-\varpi^am-\varpi^{a+n}p\bigr)^2\\
&+4\varpi^{2a+n}
\bigl(\varpi^kv'-pm+rt\bigr).
\end{aligned}
\]

Suppose first that $k\ge a$.
Expanding and using
\[
(\varpi^am+\varpi^{a+n}p)^2
-4\varpi^{2a+n}pm
=(\varpi^am-\varpi^{a+n}p)^2,
\]
we obtain
\[
\begin{aligned}
\Delta_k
=\varpi^{2a}\Bigl[
&\bigl(m-\varpi^np-\varpi^{k-a}u'\bigr)^2\\
&+4\varpi^n
\bigl(rt-\varpi^{k-a}u'p+\varpi^kv'\bigr)
\Bigr].
\end{aligned}
\]
This gives $c=a$ and $e=n$.

If $k<a$, we instead factor
$\varpi^{2k}$ directly:
\[
\begin{aligned}
\Delta_k
=\varpi^{2k}\Bigl[
&\bigl(u'-\varpi^{a-k}(m+\varpi^np)\bigr)^2\\
&+4\varpi^{2a+n-2k}
\bigl(rt-pm+\varpi^kv'\bigr)
\Bigr].
\end{aligned}
\]
This gives $c=k$ and
$e=2a+n-2k$, proving (2).

\emph{3. Independence and unit-invariance.}
All the variables
$p,r,t,m,u',v'$ are independent and
additive Haar-uniform on $\mathfrak o$.

When $k\ge a$, the variable $w$ does not
involve $m$, whereas
$z=m-\bigl(\varpi^np+\varpi^{k-a}u'\bigr)$.
Conditioning on $(p,r,t,u',v')$
makes $w$ fixed and leaves $z$ as a
translate of the independent Haar-uniform
variable $m$. Hence $z$ is Haar-uniform
and independent of $w$.

When $k<a$, the variable $w$ does not
involve $u'$, whereas
$z=u'-\varpi^{a-k}(m+\varpi^np)$.
Conditioning on $(p,r,t,m,v')$
gives the same conclusion, now using
the Haar-uniform variable $u'$.

It remains to establish unit-invariance
of $w$. Fix $\varepsilon\in\mathfrak o^\times$.

If $k\ge a$, the substitution
\[
(r,u',v')
\longmapsto
(\varepsilon r,\varepsilon u',\varepsilon v')
\]
preserves the product Haar measure and
multiplies every term of
$w=rt-\varpi^{k-a}u'p+\varpi^kv'$
by $\varepsilon$.

If $k<a$, the substitution
\[
(r,m,v')
\longmapsto
(\varepsilon r,\varepsilon m,\varepsilon v')
\]
likewise preserves the product Haar measure
and sends
$w=rt-pm+\varpi^kv'$
to $\varepsilon w$.

Thus the law of $w$ is invariant under
multiplication by every unit
$\varepsilon\in\mathfrak o^\times$,
completing the proof.
\end{proof}

\noindent
Since $\varpi^{2c}$ and $4$ are squares,
$\mathbf 1[\Delta_k\in d_E\square]=\mathbf 1[z^2+4\varpi^ew\in d_E\square]$ and
$|\Delta_k|^{s-1}=X^{2c}\,|z^2+4\varpi^ew|^{s-1}$. Conditionally on $\operatorname{ord}(w)=j$,
the quantity $z^2+4\varpi^ew$ has the law of $Q$ in Lemma~\ref{lem:unramified-square-class-density} with
$N=e+j$.

\subsection{Quadratic square classes}\label{subsec:unram-square-classes}

After fixing the valuation of $w$, the remaining
question is when a random perturbation of a square
belongs to the quadratic class defining $E_v$.
The answer depends on whether $E_v$ is split or
inert. When the square term dominates, the result
is itself a square and contributes only in the
split case. When the perturbation dominates, its
valuation must be even for either class to
contribute. At equal valuation, cancellation may
increase the valuation, and the resulting
square-class distribution requires a separate
calculation.

\begin{lemma}[Quadratic square-class densities]
\label{lem:unramified-square-class-density}
Assume that $F_v$ has odd residue characteristic
and that $E_v/F_v$ is split or unramified quadratic.
Let $N\ge0$ and
$Q=z^2+\varpi^N\varepsilon$,
where $z$ is additive Haar-uniform on $\mathfrak o$,
$\varepsilon$ is Haar-uniform on $\mathfrak o^\times$,
and the two variables are independent. Put
\[
\begin{aligned}
g(N;X)
&=\sum_{n\ge0}
\mathbb P\bigl[
\operatorname{ord}(Q)=n,\,
Q\in d_E(F_v^\times)^2
\bigr]X^n\\
&=\mathbb E\bigl[
|Q|^{s-1}\mathbf 1[Q\in d_E\square]
\bigr],
\end{aligned}
\]
where, as above, $X=q^{1-s}$ and $\rho=q^{-1}X^2$.

Then
\[
g(N;X)=
\alpha_{\mathrm{sc}}
+\beta_{N\bmod2}\rho^{\lfloor N/2\rfloor},
\]
where
\[
\gamma=
\frac{q-2}{2q}
+\frac{(1-q^{-1})X^2}{2(q^2-X^2)},
\]
and the coefficients are given by
\[
\begin{array}{llll}
E_v\text{ split}:&
\displaystyle
\alpha_{\mathrm{sc}}=\frac{1-q^{-1}}{1-\rho},
&
\displaystyle
\beta_0=\gamma-\alpha_{\mathrm{sc}},
&
\displaystyle
\beta_1=-\alpha_{\mathrm{sc}}\rho;
\\[10pt]
E_v\text{ inert}:&
\displaystyle
\alpha_{\mathrm{sc}}=0,
&
\displaystyle
\beta_0=\frac12+
\frac{(1-q^{-1})X^2}{2(q^2-X^2)},
&
\displaystyle
\beta_1=0.
\end{array}
\]
\end{lemma}

The subscript in $\alpha_{\mathrm{sc}}$
distinguishes this square-class coefficient
from the Satake parameter $\alpha$.

\begin{proof}
We first assume $|X|<q$, which makes the expectation and the
geometric series in Step~4 absolutely convergent. The resulting
formulas extend as rational functions of $X$.

Since $E_v/F_v$ is split or unramified quadratic,
membership in $d_E(F_v^\times)^2$ requires
$\operatorname{ord}(Q)$ to be even.
The unit part must be a square in the split case
and a nonsquare in the inert case.

Write $j=\operatorname{ord}(z)$.
Since $z$ is additive Haar-uniform,
$\mathbb P(j=t)=(1-q^{-1})q^{-t}$ for $t\ge0$.
We distinguish the relative valuations of
$z^2$ and $\varpi^N\varepsilon$.

\emph{1. The square term dominates: $2j<N$.} In this range,
\[
Q=z^2\bigl(1+\varpi^{N-2j}\eta\bigr),
\qquad \eta\in\mathfrak o^\times.
\]
Because the residue characteristic is odd,
every element of $1+\varpi\mathfrak o$
is a square. Hence $Q$ is a square and
$\operatorname{ord}(Q)=2j$.
This range contributes only in the split case,
where its contribution is
\[
(1-q^{-1})
\sum_{0\le j<N/2}\rho^j.
\]

\emph{2. The perturbation dominates: $2j>N$.} This event has probability
${\mathbb P(2j>N)}={q^{-(\lfloor N/2\rfloor+1)}}$.
On this event,
\[
Q=\varpi^N
\bigl(\varepsilon+\varpi^{2j-N}z'^2\bigr),
\qquad z'\in\mathfrak o^\times.
\]
The expression in parentheses is a
Haar-uniform unit, since it is obtained
from the independent uniform unit
$\varepsilon$ by adding an element of
$\varpi\mathfrak o$.

Thus $\operatorname{ord}(Q)=N$.
If $N$ is odd, neither quadratic class
contributes. If $N=2m$, the unit belongs
to either prescribed class with probability
$1/2$, giving in both cases
\[
\frac12q^{-(m+1)}X^{2m}
=\frac1{2q}\rho^m.
\]

\emph{3. Equal valuation: $N=2m$ and $j=m$.} This event has probability
$p_N=(1-q^{-1})q^{-m}$.
Writing $z=\varpi^m z'$ with
$z'\in\mathfrak o^\times$, we have
$Q=\varpi^{2m}(z'^2+\varepsilon)$.

Fix $z'$ and put
$\bar s=\bar z'^2\in\mathbb F_q^{\times2}$.
As $\bar\varepsilon$ ranges uniformly
over $\mathbb F_q^\times$,
the residue $\bar s+\bar\varepsilon$
ranges uniformly over
$\mathbb F_q\setminus\{\bar s\}$.

Among these $q-1$ residues, exactly
$(q-3)/2$ are nonzero squares,
$(q-1)/2$ are nonsquares, and one is zero.
Consequently,
\[
\begin{array}{ll}
\mathbb P(z'^2+\varepsilon
 \text{ is a square unit})
 &=\dfrac{q-3}{2(q-1)},\\[6pt]
\mathbb P(z'^2+\varepsilon
 \text{ is a nonsquare unit})
 &=\dfrac12,\\[6pt]
\mathbb P(z'^2+\varepsilon
 \in\varpi\mathfrak o)
 &=\dfrac1{q-1}.
\end{array}
\]

Multiplying the first two probabilities
by $p_NX^{2m}$, we obtain the respective
contributions
\[
\frac{q-3}{2q}\rho^m
\quad\text{and}\quad
\frac{q-1}{2q}\rho^m
\]
in the split and inert cases.

\emph{4. Cancellation at equal valuation.} It remains to account for the event
$z'^2+\varepsilon\in\varpi\mathfrak o$
in the preceding case.

Fix $z'\in\mathfrak o^\times$ and
condition on this event.
Then $\varepsilon$ is additive Haar-uniform
on the residue class
$-z'^2+\varpi\mathfrak o$.
Consequently,
$U=z'^2+\varepsilon$
is additive Haar-uniform on
$\varpi\mathfrak o$.

Conditionally on this event, the valuation
$t=\operatorname{ord}(U)\ge1$ satisfies
$\mathbb P(t=r)=(1-q^{-1})q^{-(r-1)}$ for $r\ge1$,
and its unit part is Haar-uniform on
$\mathfrak o^\times$.

Since $\operatorname{ord}(Q)=2m+t$,
membership in either prescribed quadratic
class requires $t$ to be even.
For each even $t$, the unit part belongs
to the prescribed class with probability
$1/2$.

The equal-valuation event followed by cancellation has
probability $q^{-(m+1)}=p_N/(q-1)$.
Hence the contribution, identical in
the split and inert cases, is
\[
\begin{aligned}
&q^{-(m+1)}X^{2m}
\sum_{u\ge1}
(1-q^{-1})q^{-(2u-1)}
\frac12X^{2u}\\
&\qquad
=\rho^m
\frac{(1-q^{-1})X^2}{2(q^2-X^2)}.
\end{aligned}
\]

\emph{5. Summation.} Suppose first that $N=2m+1$ is odd.
Only the square-dominant range contributes.
Thus
\[
g(2m+1;X)=
\begin{cases}
\displaystyle
(1-q^{-1})\frac{1-\rho^{m+1}}{1-\rho},
& E_v\text{ split},\\[8pt]
0,& E_v\text{ inert}.
\end{cases}
\]

Now let $N=2m$.
In the split case, the four contributions give
\[
\begin{aligned}
g(2m;X)
={}&(1-q^{-1})\frac{1-\rho^m}{1-\rho}\\
&+\rho^m\left(
\frac1{2q}+\frac{q-3}{2q}
+\frac{(1-q^{-1})X^2}{2(q^2-X^2)}
\right)\\
={}&\alpha_{\mathrm{sc}}
+(\gamma-\alpha_{\mathrm{sc}})\rho^m.
\end{aligned}
\]

In the inert case, only the
perturbation-dominant and equal-valuation
ranges contribute, giving
\[
\begin{aligned}
g(2m;X)
&=\rho^m\left(
\frac1{2q}+\frac{q-1}{2q}
+\frac{(1-q^{-1})X^2}{2(q^2-X^2)}
\right)\\
&=\rho^m\left(
\frac12+
\frac{(1-q^{-1})X^2}{2(q^2-X^2)}
\right).
\end{aligned}
\]
These are precisely the stated formulas.
\end{proof}

\subsection{Fourier valuation tails}\label{subsec:unram-fourier}

The preceding reductions express the local orbital integral
as a triple series indexed by the primitive $\ell$-shell,
the common scalar valuation of $x_1$, and its primitive
Cartan index. Each summand consists of an explicit spherical
coefficient and the square-class density calculated in
Lemma~\ref{lem:unramified-square-class-density}.
It remains to determine the valuation distribution of
the auxiliary variable $w$.

Rather than computing successive convolutions directly,
we use additive characters of the finite quotients
$\mathfrak o/\varpi^m$ to calculate the valuation tails
of $w$. Fourier orthogonality converts divisibility
conditions into character averages, which factor over
independent summands. The resulting formula applies
uniformly in the two regions $k\ge a$ and $k<a$
and will be used in the final rational summation.

Combining Lemmas~\ref{lem:unramified-primitive-shell},
\ref{lem:unramified-cartan-coefficients},
\ref{lem:unramified-levi-invariance}, and
\ref{lem:unramified-square-class-density}, we obtain
\[
\begin{aligned}
\frac{Z'_v}{(1-q^{-2})\kappa_{K,v}}
=\sum_{k,a,n\ge0}
&(qT^2)^k
\left(\frac{T^2}{q}\right)^a
\left(s_n-\frac{s_{n-2}}q\right)T^n\\
&\qquad\cdot X^{2c}\,
\mathbb E\bigl[
g(e+\operatorname{ord}(w);X)
\bigr],
\end{aligned}
\]
where $(c,e,w)$ are as in
Lemma~\ref{lem:unramified-levi-invariance}.

In the region $k\ge a$, the variable $w$ has
the same distribution as $rt-\varpi^{k-a}u'p+\varpi^kv'$,
whereas in the region $k<a$ it has the same
distribution as $rt-pm+\varpi^kv'$.
Since all the variables involved are independent
and additive Haar-uniform on $\mathfrak o$,
the signs may be absorbed into the variables.
Thus in both regions $w$ has the law of
\[
W_{i,j}
=b_1c_1+\varpi^ib_2c_2+\varpi^ju,
\]
where $b_1,c_1,b_2,c_2,u$ are independent
and additive Haar-uniform on $\mathfrak o$, with
\[
(i,j)=
\begin{cases}
(k-a,k),&k\ge a,\\
(0,k),&k<a.
\end{cases}
\]

We calculate the divisibility probabilities of
$W_{i,j}$ by Fourier orthogonality.
For a character of exact conductor $\varpi^\ell$,
the average of its composition with a product
of two independent uniform variables is
$q^{-\ell}$. This elementary observation
makes the character calculation explicit.

\begin{lemma}[Fourier valuation tails]
\label{lem:unramified-fourier-tails}
For $i,j,m\ge0$, one has
\[
\mathbb P\bigl(
\operatorname{ord}(W_{i,j})\ge m
\bigr)
=q^{-m}\left(
1+(1-q^{-1})
\sum_{\ell=1}^{\min(m,j)}
q^{-\max(\ell-i,0)}
\right).
\]

For $|t|<q$, put
\[
V_{i,j}(t)
=\mathbb E\bigl[
t^{\operatorname{ord}(W_{i,j})}
\bigr],
\qquad r=\frac tq.
\]
Then
\[
V_{i,j}(t)
=\frac{q-1}{q-t}
+\frac{(t-1)(1-q^{-1})}{q-t}
\sum_{\ell=1}^{j}
r^{\ell-1}q^{-\max(\ell-i,0)}.
\]
The right-hand side gives the rational
continuation of $V_{i,j}(t)$.
\end{lemma}

\begin{proof}
\emph{1. Fourier orthogonality.}
Let $A_m=\mathfrak o/\varpi^m\mathfrak o$,
and let $\widehat A_m$ denote its group
of additive characters.
Character orthogonality gives
\[
\mathbf 1[x\in\varpi^m\mathfrak o]
=q^{-m}\sum_{\psi\in\widehat A_m}\psi(x).
\]
Taking expectations, we obtain
\[
\mathbb P\bigl(
\operatorname{ord}(W_{i,j})\ge m
\bigr)
=q^{-m}\sum_{\psi\in\widehat A_m}
\mathbb E[\psi(W_{i,j})].
\]

Since the three summands defining
$W_{i,j}$ are independent,
\[
\mathbb E[\psi(W_{i,j})]
=\mathbb E[\psi(b_1c_1)]\,
\mathbb E[\psi(\varpi^ib_2c_2)]\,
\mathbb E[\psi(\varpi^ju)].
\]

\emph{2. Character averages.}
Suppose that $\psi$ has exact conductor
$\varpi^\ell$ with $0\le\ell\le m$.

For independent additive Haar-uniform
$b,c\in\mathfrak o$, orthogonality gives
\[
\int_{\mathfrak o}\psi(bc)\,dc
=\mathbf 1[b\in\varpi^\ell\mathfrak o].
\]
Averaging over $b$ therefore yields
$\mathbb E[\psi(bc)]=q^{-\ell}$. The same argument with the factor
$\varpi^i$ gives $\mathbb E[\psi(\varpi^ibc)]=q^{-\max(\ell-i,0)}$.
Finally, $\mathbb E[\psi(\varpi^ju)]=\mathbf 1[\ell\le j]$.

Consequently,
\[
\mathbb E[\psi(W_{i,j})]
=q^{-\ell}
q^{-\max(\ell-i,0)}
\mathbf 1[\ell\le j].
\]

\emph{3. Summation over conductors.}
The trivial character contributes $1$
to the character sum.
For $\ell\ge1$, the number of characters
of exact conductor $\varpi^\ell$ is
$q^\ell-q^{\ell-1}=(q-1)q^{\ell-1}$.
Their total contribution is
\[
\begin{aligned}
&(q-1)q^{\ell-1}
q^{-\ell}
q^{-\max(\ell-i,0)}
\mathbf 1[\ell\le j]\\
&\qquad
=(1-q^{-1})
q^{-\max(\ell-i,0)}
\mathbf 1[\ell\le j].
\end{aligned}
\]
Summing over $1\le\ell\le m$ gives
\[
\mathbb P\bigl(
\operatorname{ord}(W_{i,j})\ge m
\bigr)
=q^{-m}\left(
1+(1-q^{-1})
\sum_{\ell=1}^{\min(m,j)}
q^{-\max(\ell-i,0)}
\right).
\]

\emph{4. The valuation generating function.}
For $|t|<q$, the tail identity gives
\[
V_{i,j}(t)
=1+(t-1)\sum_{m\ge1}t^{m-1}
\mathbb P\bigl(
\operatorname{ord}(W_{i,j})\ge m
\bigr).
\]
The contribution of the trivial character is
\[
1+(t-1)\sum_{m\ge1}t^{m-1}q^{-m}=\frac{q-1}{q-t}.
\]

For the remaining terms, we interchange
the absolutely convergent sums:
\[
\begin{aligned}
&(t-1)(1-q^{-1})
\sum_{m\ge1}t^{m-1}q^{-m}
\sum_{\ell=1}^{\min(m,j)}
q^{-\max(\ell-i,0)}\\
&\quad
=(t-1)(1-q^{-1})
\sum_{\ell=1}^{j}
q^{-\max(\ell-i,0)}
\sum_{m\ge\ell}t^{m-1}q^{-m}.
\end{aligned}
\]
The inner geometric series is
${\sum_{m\ge\ell}t^{m-1}q^{-m}}={(t/q)^{\ell-1}/(q-t)}$.
Substitution yields
\[
V_{i,j}(t)
=\frac{q-1}{q-t}
+\frac{(t-1)(1-q^{-1})}{q-t}
\sum_{\ell=1}^{j}
\left(\frac tq\right)^{\ell-1}
q^{-\max(\ell-i,0)},
\]
as required.
\end{proof}

\subsection{Summation and the rational identity}\label{subsec:unram-summation}

We first sum the three indices, obtaining an explicit rational function, and then identify it with the Euler-factor expression of Theorem~\ref{thm:unramified-factor}.

\emph{1. Summation over the three indices.}
To sum the three indices, we first package the spherical coefficients into
\[
 C(z)=\sum_{n\ge0}\left(s_n-\frac{s_{n-2}}q\right)z^n
 =\frac{1-z^2/q}{1-a_\pi z+z^2},
 \qquad s_{-1}=s_{-2}=0.
\]
The other indices are geometric once the valuation kernel is inserted. Put
\[
M=q^2Y^2T^4,\quad N=qT^2,\quad D=T^2/q,
\]
and define $C_0(t)=(q-1)/(q-t)$,
$C_1(t)=(t-1)(1-q^{-1})/(q-t)$, and $r=t/q$.
In the range $k\ge a$ write $k=a+i$. Then $c=a$, $e=n$, $(i,j)=(i,i+a)$, and
the weight is $(qT^2)^{a+i}(T^2/q)^aX^{2a}=N^iM^a$. In the range $k<a$ write
$a=k+j'$ with $j'\ge1$. Then $c=k$, $e=n+2j'$, $(i,j)=(0,k)$, and the weight is
$M^kD^{j'}$. (The weight $N=qT^2$ is unrelated to the exponent $N$ of
Lemma~\ref{lem:unramified-square-class-density}.) The sums over the Cartan indices are therefore the two
series below. To evaluate them, split the finite conductor sum in
Lemma~\ref{lem:unramified-fourier-tails} at $\ell=i$. For $j=i+b$,
\[
\begin{aligned}
\sum_{\ell=1}^{i+b}r^{\ell-1}q^{-\max(\ell-i,0)}
&=\sum_{\ell=1}^{i}r^{\ell-1}+\sum_{h=1}^{b}r^{i+h-1}q^{-h}
=\frac{1-r^i}{1-r}+\frac{r^i}{q}\,\frac{1-(r/q)^b}{1-r/q}.
\end{aligned}
\]
The two summands contribute
\[
\begin{aligned}
\sum_{i,b\ge0}N^iM^b\frac{1-r^i}{1-r}
&=\frac1{(1-r)(1-M)}\left(\frac1{1-N}-\frac1{1-Nr}\right),\\
\sum_{i,b\ge0}N^iM^b\frac{r^i}{q}\,\frac{1-(r/q)^b}{1-r/q}
&=\frac1{q(1-r/q)(1-Nr)}\left(\frac1{1-M}-\frac1{1-Mr/q}\right),
\end{aligned}
\]
and the same calculation with $i=0$ gives the row sum. Hence
\[
\begin{aligned}
\mathscr D(M,N;t)
:={}&\sum_{i,b\ge0}N^iM^bV_{i,i+b}(t)\\
={}&\frac{C_0(t)}{(1-M)(1-N)}
+\frac{C_1(t)}{(1-r)(1-M)}
\left(\frac1{1-N}-\frac1{1-Nr}\right)\\
&+\frac{C_1(t)}{q(1-r/q)(1-Nr)}
\left(\frac1{1-M}-\frac1{1-Mr/q}\right),
\end{aligned}
\]
\[
\begin{aligned}
\mathscr R(M;t):={}&\sum_{k\ge0}M^kV_{0,k}(t)\\
={}&\frac{C_0(t)}{1-M}
+\frac{C_1(t)}{q(1-r/q)}
\left(\frac1{1-M}-\frac1{1-Mr/q}\right).
\end{aligned}
\]
These formulas hold as identities of rational functions, initially derived where the geometric series converge absolutely.

It remains to incorporate the parity dependence of the square-class density. Introducing a formal square root $t$ of $\rho$ lets us express its even and odd parts as a sum of two ordinary exponential sequences. This is a bookkeeping device: after adding the two signs, the result lies in the original rational-function field.

With $\alpha_{\mathrm{sc}},\beta_0,\beta_1$ as in Lemma~\ref{lem:unramified-square-class-density}, which are
explicit rational functions of $q$ and $X$ depending on $\eta_v$,
\[
g(m;X)=\alpha_{\mathrm{sc}}+\beta_{m\bmod2}\rho^{\lfloor m/2\rfloor},
\qquad \rho=qY^2.
\]
Let $t^2=\rho$ and $c_\pm=(\beta_0\pm\beta_1/t)/2$.
Then $g(m;X)=\alpha_{\mathrm{sc}}+c_+t^m+c_-(-t)^m$. Inserting $g(e+\operatorname{ord}(w);X)$, summing over
$n$ with $C$, over the Cartan indices with $\mathscr D$ and $\mathscr R$, and
over $j'$ geometrically, using $(\pm t)^{2j'}=\rho^{j'}$, the complete triple
series is the rational expression
\[
\begin{aligned}
\frac{Z'_v}{(1-q^{-2})\kappa_{K,v}}
={}&\alpha_{\mathrm{sc}}C(T)
\left(\frac1{(1-M)(1-N)}+\frac{D}{(1-M)(1-D)}\right)\\
&+\sum_{\varepsilon=\pm1}c_\varepsilon C(\varepsilon tT)
\left(\mathscr D(M,N;\varepsilon t)
+\frac{D\rho}{1-D\rho}\mathscr R(M;\varepsilon t)\right).
\end{aligned}
\]
The right-hand side has been derived from the orbital integral without using the closed formula in Theorem~\ref{thm:unramified-factor}. To finish the proof, it must be compared with the rational expression obtained by writing that theorem in the variables $Y,T$. For clarity, put $L(z)=1-a_\pi z+z^2$ and
\[
 L_E(T)=\begin{cases}
 L(T)^2,&\eta_v(\varpi)=1,\\
 L(T)L(-T),&\eta_v(\varpi)=-1.
 \end{cases}
\]
Thus $L_E(T)^{-1}=L_v(s_0-\tfrac12,\pi_E)$. Set $U=qY^2T^2$ and
\[
 \mathcal D'=(1+Y^2T^2)(1+qY^2T^2)
 -\eta_v(\varpi)a_\pi Y^2T(1+qT^2).
\]
The expression to be proved equal to the preceding triple-series sum is
\[
 \frac{\frac12(1-q^{-1})(1-T^2)
 (1-(a_\pi^2-2)T^2+T^4)\mathcal D'}
 {(1-qT^2)L_E(T)
 (1-(a_\pi^2-2)U+U^2)(1-Y^2)(1-q^2Y^2T^4)}.
\]

\emph{2. A common simplification of the Fourier contribution.}
It remains to identify the preceding rational expression with the
Euler factors in Theorem~\ref{thm:unramified-factor}. We first simplify
the Fourier contribution independently of the splitting behaviour of
$E_v$, then eliminate the auxiliary square root (Step~3) and treat the
inert and split cases separately (Steps~4 and~5). The main point is that the two Fourier generating
functions admit a common simplification before the split and
inert cases are distinguished.

For the remainder of the proof, write $a=a_\pi$ and $x=Y^2$,
with $M$, $N$, $D$, and $L$ as in Step~1, and $t^2=qx$.
The geometric sums defining $\mathscr D$ and $\mathscr R$
simplify to
\[
\mathscr D(M,N;t)=
\frac{(q-1)P(M,N;t)}
{(1-M)(1-N)(q-t)(q^2-Mt)(q-Nt)}
\]
and
\[
\mathscr R(M;t)=
\frac{(q-1)(q^2-M)}
{(1-M)(q-t)(q^2-Mt)},
\]
where
\[
P(M,N;t)=q^3-Nq^2-Mq+MN(q-qt+t).
\]
These identities follow by putting the geometric terms
in the preceding formulas over common denominators.

The combination occurring in the orbital integral is
\[
G(t):=C(tT)\left(
\mathscr D(M,N;t)
+\frac{Dt^2}{1-Dt^2}\mathscr R(M;t)
\right).
\]
Before multiplication by $C(tT)$, the numerator over
the common denominator is $(q-1)H(t)$, with
\[
\begin{aligned}
H(t)={}&(1-Dt^2)P(M,N;t)\\
&+Dt^2(1-N)(q-Nt)(q^2-M).
\end{aligned}
\]
Substituting $M,N,D$ and reducing with $t^2=qx$
gives the factorization
\[
H(t)=q^3(1-T^2)(1-qxT^4)(1-xT^4t).
\]
Now $q^2-Mt=q^2(1-xT^4t)$, while
$C(tT)=(1-t^2T^2/q)/L(tT)=(1-Dt^2)/L(tT)$.
After these cancellations, and using
$q-Nt=q(1-T^2t)$, we obtain
\begin{equation}\label{eq:A-G-simplified}
G(t)=
\frac{(q-1)(1-T^2)(1-qxT^4)}
{(1-M)(1-N)(q-t)(1-T^2t)L(tT)}.
\end{equation}

\emph{3. Eliminating the formal square root.}
The square-class density contributes the combination
$c_+G(t)+c_-G(-t)$, where
$c_\pm=\frac12\bigl(\beta_0\pm\beta_1/t\bigr)$.
Thus
\[
c_+G(t)+c_-G(-t)
=
\beta_0\frac{G(t)+G(-t)}2
+\beta_1\frac{G(t)-G(-t)}{2t}.
\]
We calculate these even and odd parts explicitly.

Put
\[
B(t)=(q-t)(1-T^2t)L(tT).
\]
Writing $z=t^2$, direct multiplication gives
\[
B(t)=B_0(z)+tB_1(z),
\]
where
\begin{equation}\label{eq:A-B-parts}
\begin{aligned}
B_0(z)
&=(q+T^2z)(1+T^2z)
  +aTz(1+qT^2),\\
B_1(z)
&=-aT(q+T^2z)
  -(1+qT^2)(1+T^2z).
\end{aligned}
\end{equation}
Consequently,
\[
\begin{aligned}
\frac12\left(\frac1{B(t)}+\frac1{B(-t)}\right)
&=\frac{B_0(t^2)}{B(t)B(-t)},\\
\frac1{2t}\left(\frac1{B(t)}-\frac1{B(-t)}\right)
&=-\frac{B_1(t^2)}{B(t)B(-t)}.
\end{aligned}
\]

Set
\[
A=L(tT)L(-tT)
=1-(a^2-2)qxT^2+q^2x^2T^4.
\]
Since $t^2=qx$,
\[
B(t)B(-t)
=(q^2-qx)(1-qxT^4)A.
\]
The factor $1-qxT^4$ cancels the corresponding
factor in \eqref{eq:A-G-simplified}. We therefore obtain
\begin{equation}\label{eq:A-even-odd}
\begin{aligned}
\frac{G(t)+G(-t)}2
&=\frac{(q-1)(1-T^2)B_0(qx)}
{(1-M)(1-N)(q^2-qx)A},\\
\frac{G(t)-G(-t)}{2t}
&=-\frac{(q-1)(1-T^2)B_1(qx)}
{(1-M)(1-N)(q^2-qx)A}.
\end{aligned}
\end{equation}
In particular, the formal square root $t$ has
disappeared. The remaining calculation depends only
on the square-class coefficients
$\alpha_{\mathrm{sc}},\beta_0,\beta_1$.

\emph{4. The inert case.}
Suppose $\eta_v(\varpi)=-1$.
By Lemma~\ref{lem:unramified-square-class-density},
\[
\alpha_{\mathrm{sc}}=\beta_1=0,
\qquad
\beta_0=\frac{1-x/q}{2(1-x)}.
\]
Hence $c_+=c_-=\beta_0/2$, and the complete
triple-series sum reduces to
\[
\mathcal F_{-1}
:=\frac{Z'_v}{(1-q^{-2})\kappa_{K,v}}
=\beta_0\frac{G(t)+G(-t)}2.
\]

Define
\[
\mathcal D_-=
(1+xT^2)(1+qxT^2)
+axT(1+qT^2).
\]
Equation~\eqref{eq:A-B-parts} gives
\[
B_0(qx)=q\mathcal D_-.
\]
Substituting into \eqref{eq:A-even-odd} and using
$q^2-qx=q^2(1-x/q)$
yields
\begin{equation}\label{eq:A-inert-result}
\mathcal F_{-1}
=
\frac{\frac12(1-q^{-1})(1-T^2)\mathcal D_-}
{(1-qT^2)(1-x)(1-q^2xT^4)A}.
\end{equation}

For an inert unramified quadratic extension,
$L_E(T)=L(T)L(-T)=1-(a^2-2)T^2+T^4$.
Thus the corresponding numerator factor in the
claimed rational expression cancels $L_E(T)$,
and \eqref{eq:A-inert-result} is exactly its
inert specialization.

\emph{5. The split case.}
Suppose $\eta_v(\varpi)=1$. The square-class
coefficients are
\[
\begin{aligned}
\alpha_{\mathrm{sc}}
&=\frac{q-1}{q(1-qx)},\\
\gamma&=\frac{q-2+x}{2q(1-x)},\\
\beta_0&=\gamma-\alpha_{\mathrm{sc}},
\qquad
\beta_1=-\alpha_{\mathrm{sc}}qx.
\end{aligned}
\]

Unlike the inert case, the constant term
$\alpha_{\mathrm{sc}}$ contributes. Its sum is
\[
\begin{aligned}
&\alpha_{\mathrm{sc}}C(T)
\left(
\frac1{(1-M)(1-N)}
+\frac{D}{(1-M)(1-D)}
\right)\\
&\qquad=
\frac{\alpha_{\mathrm{sc}}(1-T^4)}
{(1-M)(1-N)L(T)},
\end{aligned}
\]
where we used $DN=T^4$ and
$C(T)=(1-D)/L(T)$.

Combining this constant contribution with the
even and odd parts in \eqref{eq:A-even-odd},
we obtain
\begin{equation}\label{eq:A-split-reduced}
\begin{aligned}
\mathcal F_1
=\frac{1-T^2}{(1-M)(1-N)}
\Bigg[
&\frac{\alpha_{\mathrm{sc}}(1+T^2)}{L(T)}\\
&+\frac{(q-1)
(\beta_0B_0(qx)-\beta_1B_1(qx))}
{(q^2-qx)A}
\Bigg].
\end{aligned}
\end{equation}

Define
\[
\mathcal D_+
=(1+xT^2)(1+qxT^2)
-axT(1+qT^2).
\]
We claim that the expression in brackets in
\eqref{eq:A-split-reduced} equals
\begin{equation}\label{eq:A-split-key}
\frac{q-1}{2q(1-x)}
\frac{L(-T)}{L(T)}
\frac{\mathcal D_+}{A}.
\end{equation}

To verify this identity, substitute the explicit
values of $\alpha_{\mathrm{sc}},\beta_0,\beta_1$
and clear denominators. Multiplication by
$2q(1-x)(1-qx)(q-x)L(T)A/(q-1)$
reduces \eqref{eq:A-split-key} to
\[
\begin{aligned}
&2(1-x)(q-x)(1+T^2)A\\
&\quad+\frac{L(T)}q
\Bigl[
(1-qx)(q-2+x)B_0(qx)\\
&\hspace{35mm}
-2(q-1)(1-x)
\bigl(B_0(qx)-qxB_1(qx)\bigr)
\Bigr]\\
&\qquad=(1-qx)(q-x)L(-T)\mathcal D_+.
\end{aligned}
\]
Both sides are polynomials of degree at most
two in $a$. Direct expansion shows that their
coefficients of $1$, $a$, and $a^2$ agree
and are, respectively,
\[
\begin{aligned}
&(1+T^2)(q-x)(1+xT^2)
  (1-qx)(1+qxT^2),\\
&T(q-x)(1-x)(1-qx)(1-qxT^4),\\
&-T^2x(q-x)(1+qT^2)(1-qx).
\end{aligned}
\]
This proves \eqref{eq:A-split-key}.

Substituting \eqref{eq:A-split-key} into
\eqref{eq:A-split-reduced} gives
\[
\mathcal F_1
=
\frac{
\frac12(1-q^{-1})(1-T^2)
L(-T)\mathcal D_+}
{(1-qT^2)(1-x)(1-q^2xT^4)L(T)A}.
\]
In the split case, $L_E(T)=L(T)^2$, while
$L(T)L(-T)=1-(a^2-2)T^2+T^4$.
Therefore this is exactly the split
specialization of the claimed rational expression.

\emph{6. Identification with the Euler factors.}
We finally compare the rational expressions
obtained above with the local factors in
Theorem~\ref{thm:unramified-factor}.

Recall that $Y=q^{-s}$, $T=q^{-(s_0-1/2)}$, and $U=qY^2T^2$.
The relevant local factors are
\[
\begin{aligned}
\zeta_v(2s_0-2)
&=\frac1{1-qT^2},\\
L_v(s_0-\tfrac12,\pi_E)
&=\frac1{L_E(T)},\\
\zeta_v(2s)
&=\frac1{1-Y^2},\\
L_v(2s_0-1,\pi,\mathrm{Ad})^{-1}
&=(1-T^2)
  \bigl(1-(a_\pi^2-2)T^2+T^4\bigr),\\
\frac{\zeta_v(2s+4s_0-4)}
{\zeta_v(2s+2s_0-2)}
&=\frac{1-U}{1-q^2Y^2T^4},\\
L_v(2s+2s_0-2,\pi,\mathrm{Ad})
&=\frac1{(1-U)
  \bigl(1-(a_\pi^2-2)U+U^2\bigr)}.
\end{aligned}
\]

The factors $1-U$ cancel. Moreover, since $(tT)^2=qxT^2=U$,
\[
A=L(tT)L(-tT)
=1-(a_\pi^2-2)U+U^2,
\]
and $\mathcal D_E(s,s_0)=\mathcal D'$, with $\mathcal D'$ as in Step~1.
Thus, in both the split and inert cases,
the preceding calculation gives the claimed
Euler-factor expression divided by
$(1-q^{-2})\kappa_{K,v}$.

Finally, the normalization
$(1-q^{-2})\kappa_{K,v}\cdot\frac12(1-q^{-1})=c'_v$
recovers the constant in
Theorem~\ref{thm:unramified-factor}.

We have proved the generating-function identity in both cases
using rational-function manipulations. All series used in the
derivation converge absolutely when $\Re s$ and $\Re s_0$ are
sufficiently large. In particular, throughout the region
\[
\Re s_0\ge 3+\delta,\qquad
\Re s\ge 1+\delta,\qquad \delta>0,
\]
of Corollary~\ref{cor:unramified-normalization}(2), one has
\[
|M|,|N|,|D|,|D\rho|,|\rho|<1,\qquad
|X|<q,\qquad q^{1/2}|tT|<1.
\]
These inequalities ensure that the geometric, Fourier,
spherical, and square-class generating series used above
converge absolutely. The orbital integral also converges
absolutely in this region, justifying the interchanges of
integration and summation. Since the original integral and
the resulting rational expressions agree there, the identity
extends meromorphically in the parameters. Together with the
geometric and Fourier reductions above, this completes the
proof of Theorem~\ref{thm:unramified-factor}.

\begin{remark}[Consistency checks]
Theorem~\ref{thm:unramified-factor} is consistent with two facts
established independently of this appendix.
\begin{itemize}
\item At $s_0=1$ it reduces to the local residue identity
(Theorem~\ref{thm:local-residue}).
\item The ratio $Z'_v(s,s_0)/Z'_v(s+2s_0-2,\,2-s_0)$ is independent
of $s$, since the factor
\[
\zeta_v(2s)\frac{\zeta_v(2s+4s_0-4)}{\zeta_v(2s+2s_0-2)}
L_v(2s+2s_0-2,\pi,\mathrm{Ad})\,\mathcal D_E(s,s_0)
\]
is invariant under $(s,s_0)\mapsto(s+2s_0-2,\,2-s_0)$. This is
consistent with the functional equation of the Eisenstein series
recalled in (S1) of \S\ref{subsec:standard-inputs}, since
$M(s_0-1)$ acts on the spherical section by a scalar depending
only on $s_0$.
\end{itemize}
\end{remark}

\printbibliography

@article{SW1,
  author  = {Suzuki, Miyu and Wakatsuki, Satoshi},
  title   = {Zeta functions and nonvanishing theorems for toric periods on
             {$\mathrm{GL}_2$}},
  journal = {Tohoku Math. J. (2)},
  volume  = {78},
  year    = {2026},
  pages   = {311--366},
  doi     = {10.2748/tmj.20250107},
  eprint  = {2005.02017},
  archivePrefix = {arXiv}
}

@article{SW2,
  author  = {Suzuki, Miyu and Wakatsuki, Satoshi},
  title   = {Explicit mean value theorems for toric periods and automorphic
             {$L$}-functions},
  journal = {J. Math. Soc. Japan},
  volume  = {78},
  number  = {4},
  year    = {2026},
  pages   = {1075--1119},
  doi     = {10.2969/jmsj/93829382},
  note    = {With an appendix by M. Suzuki, S. Wakatsuki and S. Yokoyama},
  eprint  = {2103.04589},
  archivePrefix = {arXiv}
}

@article{Sh81,
  author  = {Shahidi, Freydoon},
  title   = {On certain {$L$}-functions},
  journal = {Amer. J. Math.},
  volume  = {103},
  number  = {2},
  pages   = {297--355},
  year    = {1981}
}

@book{MW,
  author    = {M{\oe}glin, Colette and Waldspurger, Jean-Loup},
  title     = {Spectral Decomposition and Eisenstein Series:
               une paraphrase de l'\'Ecriture},
  series    = {Cambridge Tracts in Mathematics},
  volume    = {113},
  publisher = {Cambridge University Press},
  address   = {Cambridge},
  year      = {1995}
}

@article{Ar78,
  author  = {Arthur, James},
  title   = {A Trace Formula for Reductive Groups I:
             Terms Associated to Classes in {$G(\mathbb{Q})$}},
  journal = {Duke Math. J.},
  volume  = {45},
  number  = {4},
  pages   = {911--952},
  year    = {1978}
}

@article{Ar80,
  author  = {Arthur, James},
  title   = {A Trace Formula for Reductive Groups II:
             Applications of a Truncation Operator},
  journal = {Compositio Math.},
  volume  = {40},
  number  = {1},
  pages   = {87--121},
  year    = {1980}
}

@article{CS80,
  author  = {Casselman, William and Shalika, Joseph},
  title   = {The Unramified Principal Series of
             {$p$}-Adic Groups. {II}.
             The Whittaker Function},
  journal = {Compositio Math.},
  volume  = {41},
  number  = {2},
  pages   = {207--231},
  year    = {1980}
}

@article{JS81,
  author  = {Jacquet, Herv{\'e} and Shalika, Joseph A.},
  title   = {On {Euler} Products and the Classification
             of Automorphic Representations. {I}},
  journal = {Amer. J. Math.},
  volume  = {103},
  number  = {3},
  pages   = {499--558},
  year    = {1981}
}

@incollection{Be84,
  author    = {Bernstein, Joseph},
  title     = {{$P$}-invariant Distributions on
               {$\mathrm{GL}(N)$} and the Classification of
               Unitary Representations of {$\mathrm{GL}(N)$}
               (Non-Archimedean Case)},
  booktitle = {Lie Group Representations II},
  series    = {Lecture Notes in Mathematics},
  volume    = {1041},
  pages     = {50--102},
  year      = {1984},
  publisher = {Springer}
}

@article{Vo86,
  author  = {Vogan, Jr., David A.},
  title   = {The Unitary Dual of {$\mathrm{GL}(n)$}
             over an Archimedean Field},
  journal = {Invent. Math.},
  volume  = {83},
  number  = {3},
  pages   = {449--505},
  doi     = {10.1007/BF01394418},
  year    = {1986}
}

@article{Wal85,
  author  = {Waldspurger, Jean-Loup},
  title   = {Sur les valeurs de certaines fonctions {$L$} automorphes en leur
             centre de sym{\'e}trie},
  journal = {Compositio Math.},
  volume  = {54},
  number  = {2},
  year    = {1985},
  pages   = {173--242}
}

@article{GJ78,
  author  = {Gelbart, Stephen and Jacquet, Herv{\'e}},
  title   = {A relation between automorphic representations of {$\mathrm{GL}(2)$}
             and {$\mathrm{GL}(3)$}},
  journal = {Ann. Sci. {\'E}cole Norm. Sup. (4)},
  volume  = {11},
  number  = {4},
  year    = {1978},
  pages   = {471--542}
}

@article{Ono61,
  author  = {Ono, Takashi},
  title   = {Arithmetic of algebraic tori},
  journal = {Ann. of Math. (2)},
  volume  = {74},
  year    = {1961},
  pages   = {101--139}
}

@article{Ono63,
  author  = {Ono, Takashi},
  title   = {On the {T}amagawa number of algebraic tori},
  journal = {Ann. of Math. (2)},
  volume  = {78},
  year    = {1963},
  pages   = {47--73}
}

@article{Tunnell,
  author  = {Tunnell, Jerrold B.},
  title   = {Local {$\epsilon$}-factors and characters of {$\mathrm{GL}(2)$}},
  journal = {Amer. J. Math.},
  volume  = {105},
  number  = {6},
  year    = {1983},
  pages   = {1277--1307}
}

@article{Saito,
  author  = {Saito, Hiroshi},
  title   = {On {T}unnell's formula for characters of {$\mathrm{GL}(2)$}},
  journal = {Compositio Math.},
  volume  = {85},
  number  = {1},
  year    = {1993},
  pages   = {99--108}
}

@article{II10,
  author  = {Ichino, Atsushi and Ikeda, Tamotsu},
  title   = {On the periods of automorphic forms on special orthogonal groups
             and the {G}ross--{P}rasad conjecture},
  journal = {Geom. Funct. Anal.},
  volume  = {19},
  number  = {5},
  year    = {2010},
  pages   = {1378--1425},
  doi     = {10.1007/s00039-009-0040-4}
}

@incollection{Tate,
  author    = {Tate, John T.},
  title     = {Fourier analysis in number fields, and {H}ecke's zeta-functions},
  booktitle = {Algebraic Number Theory (Proc. Instructional Conf., Brighton, 1965)},
  editor    = {Cassels, J. W. S. and Fr{\"o}hlich, A.},
  publisher = {Thompson, Washington, D.C.},
  year      = {1967},
  pages     = {305--347}
}

@book{Igusa,
  author    = {Igusa, Jun-ichi},
  title     = {An Introduction to the Theory of Local Zeta Functions},
  series    = {AMS/IP Studies in Advanced Mathematics},
  volume    = {14},
  publisher = {American Mathematical Society},
  address   = {Providence, RI},
  year      = {2000}
}

@article{Atiyah,
  author  = {Atiyah, Michael F.},
  title   = {Resolution of singularities and division of distributions},
  journal = {Comm. Pure Appl. Math.},
  volume  = {23},
  year    = {1970},
  pages   = {145--150}
}

@article{BG69,
  author  = {Bernstein, I. N. and Gelfand, S. I.},
  title   = {Meromorphy of the function {$P^\lambda$}},
  journal = {Funct. Anal. Appl.},
  volume  = {3},
  number  = {1},
  year    = {1969},
  pages   = {68--69}
}

@book{Bor69,
  author    = {Borel, Armand},
  title     = {Introduction aux groupes arithm{\'e}tiques},
  series    = {Publications de l'Institut de Math{\'e}matique de l'Universit{\'e} de Strasbourg, XV; Actualit{\'e}s Scientifiques et Industrielles},
  number    = {1341},
  publisher = {Hermann},
  address   = {Paris},
  year      = {1969}
}

@article{MW89,
  author  = {M{\oe}glin, Colette and Waldspurger, Jean-Loup},
  title   = {Le spectre r{\'e}siduel de {${\rm GL}(n)$}},
  journal = {Ann. Sci. {\'E}cole Norm. Sup. (4)},
  volume  = {22},
  number  = {4},
  pages   = {605--674},
  year    = {1989}
}

@book{Weil82,
  author    = {Weil, Andr{\'e}},
  title     = {Adeles and Algebraic Groups},
  series    = {Progress in Mathematics},
  volume    = {23},
  publisher = {Birkh{\"a}user},
  address   = {Boston},
  year      = {1982}
}

@article{Ono65,
  author  = {Ono, Takashi},
  title   = {On the relative theory of {T}amagawa numbers},
  journal = {Ann. of Math. (2)},
  volume  = {82},
  year    = {1965},
  pages   = {88--111},
  doi     = {10.2307/1970563}
}

@article{Jac86,
  author  = {Jacquet, Herv{\'e}},
  title   = {Sur un r{\'e}sultat de {W}aldspurger},
  journal = {Ann. Sci. {\'E}cole Norm. Sup. (4)},
  volume  = {19},
  number  = {2},
  year    = {1986},
  pages   = {185--229}
}

@article{Jac87,
  author  = {Jacquet, Herv{\'e}},
  title   = {Sur un r{\'e}sultat de {W}aldspurger. {II}},
  journal = {Compositio Math.},
  volume  = {63},
  number  = {3},
  year    = {1987},
  pages   = {315--389}
}

@article{MartinWhitehouse,
  author  = {Martin, Kimball and Whitehouse, David},
  title   = {Central {$L$}-values and toric periods for {${\rm GL}(2)$}},
  journal = {Int. Math. Res. Not. IMRN},
  volume  = {2009},
  number  = {1},
  year    = {2009},
  pages   = {141--191}
}

@article{SakBE2,
  author  = {Sakellaridis, Yiannis},
  title   = {Beyond endoscopy for the relative trace formula {II}:
             global theory},
  journal = {J. Inst. Math. Jussieu},
  volume  = {18},
  number  = {2},
  year    = {2019},
  pages   = {347--447},
  doi     = {10.1017/S1474748017000032}
}

@incollection{SakBE1,
  author    = {Sakellaridis, Yiannis},
  title     = {Beyond endoscopy for the relative trace formula {I}:
               local theory},
  booktitle = {Automorphic Representations and {L}-functions},
  publisher = {Tata Institute of Fundamental Research},
  address   = {Mumbai},
  year      = {2013},
  pages     = {521--590},
  eprint    = {1207.5761},
  archivePrefix = {arXiv}
}

@article{MV10,
  author  = {Michel, Philippe and Venkatesh, Akshay},
  title   = {The subconvexity problem for {${\rm GL}_2$}},
  journal = {Publ. Math. Inst. Hautes \'Etudes Sci.},
  volume  = {111},
  year    = {2010},
  pages   = {171--271},
  doi     = {10.1007/s10240-010-0025-8}
}

@article{JR92,
  author  = {Jacquet, Herv{\'e} and Rallis, Stephen},
  title   = {Symplectic periods},
  journal = {J. Reine Angew. Math.},
  volume  = {423},
  year    = {1992},
  pages   = {175--198}
}

@article{Jia98,
  author  = {Jiang, Dihua},
  title   = {{$G_2$}-periods and residual representations},
  journal = {J. Reine Angew. Math.},
  volume  = {497},
  year    = {1998},
  pages   = {17--46}
}

@article{PWZ21,
  author  = {Pollack, Aaron and Wan, Chen and Zydor, Micha{\l}},
  title   = {On the residue method for period integrals},
  journal = {Duke Math. J.},
  volume  = {170},
  number  = {7},
  year    = {2021},
  pages   = {1457--1515},
  doi     = {10.1215/00127094-2020-0078}
}

@article{FMP17,
  author  = {File, Daniel and Martin, Kimball and Pitale, Ameya},
  title   = {Test vectors and central {$L$}-values for {${\rm GL}(2)$}},
  journal = {Algebra Number Theory},
  volume  = {11},
  number  = {2},
  year    = {2017},
  pages   = {253--318},
  doi     = {10.2140/ant.2017.11.253}
}

@book{YZZ13,
  author    = {Yuan, Xinyi and Zhang, Shou-Wu and Zhang, Wei},
  title     = {The {G}ross--{Z}agier Formula on {S}himura Curves},
  series    = {Annals of Mathematics Studies},
  volume    = {184},
  publisher = {Princeton University Press},
  address   = {Princeton, NJ},
  year      = {2013}
}

@article{Bha04,
  author  = {Bhargava, Manjul},
  title   = {Higher composition laws {I}: {A} new view
             on {G}auss composition, and quadratic
             generalizations},
  journal = {Ann. of Math. (2)},
  volume  = {159},
  number  = {1},
  year    = {2004},
  pages   = {217--250},
  doi     = {10.4007/annals.2004.159.217}
}

@article{Baruch,
  author  = {Baruch, Ehud Moshe},
  title   = {A proof of {K}irillov's conjecture},
  journal = {Ann. of Math. (2)},
  volume  = {158},
  number  = {1},
  year    = {2003},
  pages   = {207--252},
  doi     = {10.4007/annals.2003.158.207}
}

@book{JL70,
  author    = {Jacquet, Herv{\'e} and Langlands, Robert P.},
  title     = {Automorphic Forms on {${\rm GL}(2)$}},
  series    = {Lecture Notes in Mathematics},
  volume    = {114},
  publisher = {Springer-Verlag},
  address   = {Berlin},
  year      = {1970}
}

\clearpage

\section*{Notation and conventions}
\markboth{\MakeUppercase{\shortauthors}}{\MakeUppercase{Notation and conventions}}
\label{sec:notation}

The principal notation is listed below by topic, together with
the place where it is defined. The entries are meant as a guide;
the definitions in the text are authoritative.

\begingroup
\small
\renewcommand{\arraystretch}{1.15}
\newcommand{\ngroup}[1]{\multicolumn{3}{@{}l}{\textit{#1}}\\*[2pt]}
\begin{longtable}{@{}
  >{\raggedright\arraybackslash}p{0.27\textwidth}
  >{\raggedright\arraybackslash}p{0.47\textwidth}
  >{\raggedright\arraybackslash}p{0.19\textwidth}@{}}

\ngroup{Global data and conventions}
$F$, $E$, $\eta=\eta_{E/F}$ &
number field; quadratic field extension; its quadratic character &
\S\ref{sec:intro} \\
$T=R_{E/F}\mathbb G_m/\mathbb G_m$, $d\tau_T$ &
torus, embedded in $\mathrm{PGL}_2$ via $1,\delta_E$; its Tamagawa
measure, $\operatorname{vol}[T]=2$ &
\S\ref{sec:intro}, \S\ref{subsec:invariant-form} \\
$\pi$, $\pi_E$ &
cuspidal automorphic representation of $\mathrm{PGL}_2(\mathbb A)$;
its base change to $\mathrm{GL}_2(\mathbb A_E)$ &
\S\ref{sec:intro} \\
$P_E$, $\langle\ ,\ \rangle_{\mathrm{Tam}}$ &
toric period (measure $d\tau_T$); Petersson product (Tamagawa
measure) &
\S\ref{sec:intro} \\
$\alpha^\#_{E_v}$ &
normalized local toric period &
\S\ref{sec:intro}, \S\ref{subsec:invariant-form} \\
$\Lambda$, $\Lambda_F$ &
completed $L$-functions: Euler products over all places, with no
discriminant or conductor factor &
\S\ref{sec:intro} \\
$\zeta_{\mathbb R}$, $\zeta_{\mathbb C}$ &
archimedean zeta factors $\pi^{-s/2}\Gamma(s/2)$ and
$(2\pi)^{1-s}\Gamma(s)$ &
\S\ref{sec:intro}, \S\ref{subsec:measures} \\
$L_v(s,\eta_v)$ &
$\zeta_{E_v}(s)/\zeta_{F_v}(s)$, at every place, including the
archimedean ones &
\S\ref{sec:intro} \\
$\Delta_F$, $\psi_F$ &
absolute discriminant of $F$; standard additive character &
\S\ref{subsec:measures} \\
$d_E$, $\delta_E$ &
$E=F+F\delta_E$ with $\delta_E^2=d_E$ &
\S\ref{subsec:generic-orbits} \\
$c_v$ &
$\zeta_{F_v}(1)$, at every place &
\S\ref{subsec:measures} \\
$\operatorname{ord}=\operatorname{ord}_v$ &
normalized valuation of $F_v$ at a finite place,
$\operatorname{ord}(\varpi)=1$ &
\S\ref{sec:unramified} \\
$d^\times x$, $\mathbb A^1$, $d^1x$ &
multiplicative measure on $\mathbb A^\times$; norm-one ideles and
the measure with quotient $dt/t$ &
\S\ref{subsec:measures} \\
$c_F$ &
$\operatorname{vol}(F^\times\backslash\mathbb A^1,d^1x)$, the residue
at $s=1$ of the finite Dedekind zeta function &
\S\ref{subsec:measures} \\[6pt]

\ngroup{Spaces, groups and invariants}
$q$, $B$ &
Pfaffian quadratic form on $\bigwedge^2F^4$; its polar bilinear form &
\S\ref{subsec:pfaffian} \\
$V'$ &
the fusion space $F^2\otimes\bigwedge^2F^4$, of dimension $12$ &
\S\ref{subsec:pfaffian} \\
$w=(\ell,x,\lambda)$ &
Witt coordinates on $V'$: $\ell,\lambda\in F^2$,
$x=(x_1,x_2)\in M_2(F)^2$ &
\S\ref{subsec:pfaffian} \\
$V$ &
the cube space $M_2\oplus M_2$, embedded by $x\mapsto(0,x,0)$ &
\S\ref{subsec:pfaffian} \\
$P=P_{2,2}$, $M$, $N$ &
$(2,2)$ parabolic subgroup of $\mathrm{GL}_4$; its Levi subgroup and
unipotent radical &
\S\ref{subsec:group} \\
$Z'_\rho$ &
kernel $\{(zI_4,z^{-2}I_2)\}$ of the action on $V'$ &
\S\ref{subsec:group} \\
$G$, $G'$ &
effective groups $(\mathrm{GL}_4\times\mathrm{GL}_2)/Z'_\rho$ and
$(P\times\mathrm{GL}_2)/Z'_\rho$ &
\S\ref{subsec:group} \\
$\widetilde M$, $h(g_1,g_2)$ &
Levi subgroup $(M\times\mathrm{GL}_2)/Z'_\rho$ of $G'$;
$h(g_1,g_2)=\operatorname{diag}({}^tg_1^{-1},g_2)$ &
\S\ref{subsec:group} \\
$A$, $e_c$ &
central torus $\{e_c\}$ of $\widetilde M$ &
\S\ref{subsec:group} \\
$H$ &
$\widetilde M/A\simeq\mathrm{GL}_2^3/\ker\rho_V$, the cube group of
Suzuki--Wakatsuki &
\S\ref{subsec:group} \\
$\Delta$, $R$ &
discriminant and second relative invariant on $V'$ &
\S\ref{subsec:invariants} \\
$\chi_\Delta$, $\chi_R$ &
characters of $\Delta$ and $R$ ($\chi_R$ on $G'$ only) &
Prop.~\ref{prop:relative-invariants} \\
$y_w$ &
$\ell_2x_1-\ell_1x_2$, with $R(w)=-2\det y_w$ &
Lemma~\ref{lem:R-determinant} \\
$\delta_P$, $u$ &
modulus character $|\det a/\det b|^2$; positive character
$|\det a/\det b|^{1/2}$ &
\S\ref{subsec:invariants} \\
$w_\mu$ &
generic base point, with $[4\mu]=[d_E]$ &
\S\ref{subsec:invariants}, \S\ref{sec:second-expansion} \\
$S$, $S^\circ$ &
stabilizer $G'_{w_\mu}\simeq\mathrm{PGL}_2\times\mu_2$ of $w_\mu$;
its identity component &
Prop.~\ref{prop:generic-stabilizer}, \S\ref{sec:second-expansion} \\
$\iota$ &
$t\mapsto h(t,t)e_{\det t}$, identifying $\mathrm{PGL}_2$ with
$S^\circ$ &
Prop.~\ref{prop:generic-stabilizer} \\
$V'_E(F)$, $V'_E(F_v)$ &
points $w$ with $\Delta(w)\in d_E(F^\times)^2$, resp.\
$d_E(F_v^\times)^2$ &
\S\ref{subsec:generic-orbits} \\
$\Theta_E$, $\Theta_E^{\ell\neq0}$ &
theta series over $V'_E(F)$; its part with $\ell\neq0$ &
\S\ref{subsec:theta} \\
$W$ &
the boundary subspace $\{\ell=0\}$ &
\S\ref{subsec:R-cutoff} \\[6pt]

\ngroup{Measures}
$dx$, $dx_v^{\mathrm{sd}}$ &
self-dual additive measures on $\mathbb A$ and $F_v$ &
\S\ref{subsec:measures} \\
$dw$, $d\lambda$ &
measures on $V'(\mathbb A)$ and $\mathbb A^2$: products of local
measures with factors $|\Delta_F|^{-6}$ and $|\Delta_F|^{-1}$ &
\S\ref{subsec:measures} \\
$dt$ &
Tamagawa measure on $\mathrm{PGL}_2(\mathbb A)$, quotient of
volume $2$ &
\S\ref{subsec:measures} \\
$dh$ &
Tamagawa measure on $H(\mathbb A)$ &
\S\ref{subsec:measures} \\
$dm$, $dm^1$, $d^\times u$ &
$dm=dh\,d^\times c=d^\times u\,dm^1$ on $\widetilde M(\mathbb A)$ &
Lemma~\ref{lem:levi-measure}, \S\ref{subsec:zeta-integral} \\
$dn$ &
Haar measure on $N(\mathbb A)$ with
$\operatorname{vol}(N(F)\backslash N(\mathbb A))=1$ &
\S\ref{subsec:measures} \\
$d\mu_{G'}$ &
Haar measure on $G'(\mathbb A)$ &
\S\ref{subsec:measures}, Lemma~\ref{lem:levi-measure} \\
$dh_v$, $dh_{E_v}$ &
quotient measure on $T_v=F_v^\times\backslash E_v^\times$; Haar
measure on $E_v^\times$ &
\S\ref{sec:local-residue} \\
$d\mu_{H^\circ}$ &
Suzuki--Wakatsuki measure on the connected stabilizer
$H^\circ_{x_E}(\mathbb A)$ &
\S\ref{subsec:comparison-setup} \\
$d_v$ &
convergence factors of the Suzuki--Wakatsuki gauge measure &
\S\ref{subsec:first-residue} \\
$dt_v$, $c_{\psi,v}$ &
self-dual local torus measure; $dx_v^{\mathrm{sd}}=c_{\psi,v}\,dx_v$ &
\S\ref{subsec:invariant-form} \\
$C_T$ &
$d\tau_T=C_T\prod_vdt_v$, with $C_T=\Lambda(1,\eta)^{-1}$ &
\S\ref{subsec:invariant-form} \\[6pt]

\ngroup{Eisenstein data and parameters}
$\sigma$ &
$\pi\boxtimes\pi$, the inducing representation on the Levi &
\S\ref{subsec:zeta-integral} \\
$I(s')$, $I(\sigma,z)$ &
normalized induced representations &
\S\ref{subsec:zeta-integral}, \S\ref{subsec:central-intertwining} \\
$s$, $s_0$ &
the two complex variables of $Z'_E$ &
\S\ref{subsec:zeta-integral} \\
$s'$, $z$ &
normalized inducing parameter, $s'=z=s_0-1$ &
\S\ref{subsec:zeta-integral}, \S\ref{subsec:central-intertwining} \\
$f_{s_0}$ &
flat section of $I(s_0-1)$ &
\S\ref{subsec:zeta-integral} \\
$E_P(f,s_0)$ &
Eisenstein series on $\mathrm{GL}_4$ attached to $f_{s_0}$ &
\S\ref{subsec:standard-inputs} \\
$M(z)$, $r(z)$, $N(z)$ &
standard intertwining operator, factored as normalizing factor times
normalized operator &
\S\ref{subsec:central-intertwining} \\
$\widetilde M(\mathbb A)^1$, $[\widetilde M]^1$ &
norm-one Levi subgroup and its automorphic quotient &
\S\ref{subsec:zeta-integral} \\
$a_u$ &
central scaling element of $A(\mathbb A)$ with $u(a_u)=u$ &
\S\ref{subsec:zeta-integral} \\[6pt]

\ngroup{Zeta integrals}
$Z'_E(\Psi,f;s,s_0)$ &
the two-variable zeta integral &
\eqref{eq:zeta-integral} \\
$\Phi_\Psi$ &
cube Schwartz function $x\mapsto\int\Psi(0,x,\lambda)\,d\lambda$ &
\S\ref{subsec:N-constant-term} \\
$\Theta_{\mathrm{cube},E}$ &
cube theta series over $V_E(F)$ &
Lemma~\ref{lem:N-constant-term} \\
$\Phi_f$ &
restriction of a section $f$ to $\widetilde M(\mathbb A)^1K$ &
\S\ref{subsec:absolute-convergence} \\
$\mathfrak{s}$ &
cube zeta parameter; $\mathfrak{s}=s+s_0-\frac12$ in the first
expansion &
\S\ref{subsec:first-expansion} \\
$\mathfrak{s}_s$ &
specialization $\mathfrak{s}_s=s+\frac12$ at $s_0=1$ &
\S\ref{subsec:comparison-setup} \\
$\mathcal Z(\Phi_f;\mathfrak{s})$ &
boundary cube zeta integral &
\eqref{eq:boundary-cube-integral} \\
$Z_E^{\mathrm{SW}}$ &
global Suzuki--Wakatsuki zeta integral (Tamagawa normalization) &
Cor.~\ref{cor:residue-SW} \\
$\xi_f(\kappa)$, $\xi_{f_v}(\kappa)$ &
global and local inducing vectors of the section &
\S\ref{sec:second-expansion} \\
$B_v$, $B$ &
local invariant bilinear forms on $\pi_v\otimes\pi_v$; global
diagonal period &
\S\ref{sec:second-expansion} \\
$\beta_{\xi,v}$ &
local matrix coefficient $y\mapsto B_v((1\otimes\pi_v(y))\xi)$ &
\eqref{eq:second-beta} \\
$\mathscr Z'_v$ &
local orbital zeta integral, averaged over $K_v$ &
\eqref{eq:local-second} \\
$\mathcal Z_v^{\mathrm{orb}}$ &
local orbital integral on $V'_E(F_v)$ &
\S\ref{sec:local-residue} \\
$x_E$, $P$, $V_E(F_v)$ &
cube base point $(1,\delta_E)$; Suzuki--Wakatsuki invariant on $V$;
local cube orbit &
\S\ref{sec:local-residue} \\
$\alpha^B_{\xi,v}$ &
bilinear local toric integral &
\S\ref{sec:local-residue} \\
$\mathcal Z^{\mathrm{SW}}_{E_v}(\Phi,\xi;\mathfrak{s})$ &
bilinear local Suzuki--Wakatsuki zeta integral &
\S\ref{sec:local-residue} \\
$\mathcal T_v(\sigma)$ &
standard local Tate integral, equal to $\zeta_v(\sigma)$ &
\S\ref{sec:local-residue} \\
$\Phi^{\mathrm{loc}}_{\Psi_v}$ &
local cube function $x\mapsto\int\Psi_v(0,x,\lambda)\,d\lambda$ &
Thm.~\ref{thm:local-residue} \\
$\mathcal D_E(s,s_0)$, $c'_v$ &
unramified correction polynomial and constant &
Thm.~\ref{thm:unramified-factor} \\
$Z^\#_{E_v}$, $D_{E_v}(\pi_v,\cdot)$ &
normalized unramified Suzuki--Wakatsuki factor; their unramified
polynomial &
Cor.~\ref{cor:unramified-normalization} \\[6pt]

\ngroup{Constants and the comparison}
$C_\pi$ &
scalar comparing the global diagonal pairing with the product of
the normalized local invariant pairings &
\eqref{eq:second-global-pairing} \\
$\kappa_0$ &
rational Jacobian scalar of the orbit map, $\kappa_0\in F^\times$ &
\eqref{eq:second-gauge-identity} \\
$\kappa_{K,v}$ &
volume of $\mathrm{GL}_2(\mathfrak o_v)$ in $M_2(F_v)$ &
proof of Thm.~\ref{thm:unramified-factor} \\
$S$ &
finite set of places, with the conditions listed there &
\S\ref{subsec:comparison-setup} \\
$T_E$, $H^\circ_{x_E}$ &
the torus $T$; the connected stabilizer $T_E\times T_E$ of $x_E$
in $H$ &
\S\ref{subsec:comparison-setup} \\
$r_v$ &
local $L$-ratio &
\eqref{eq:rv} \\
$r^S$ &
partial global $L$-ratio, defined after \eqref{eq:rv} &
\S\ref{subsec:comparison-setup} \\
$\mathcal P_E$ &
global bilinear toric period, with measure $d\mu_{H^\circ}$ &
\S\ref{subsec:comparison-setup} \\
$C_P$ &
scalar comparing the global bilinear toric period with the product
of the local toric integrals &
\eqref{eq:CP} \\
$\mathcal J_v(s)$ &
local Suzuki--Wakatsuki zeta integral, averaged over $K_v$ &
\S\ref{subsec:comparison-setup} \\
$L_f(1,\eta)$, $\Lambda(s,\eta)$ &
finite Hecke $L$-value; completed Hecke $L$-function &
\S\ref{subsec:first-residue} \\
$\mathcal A(s)$ &
common Euler product of the two residues &
Thm.~\ref{thm:CP-Cpi} \\
\end{longtable}
\endgroup

\end{document}